\documentclass[reqno,12pt]{amsart}
  
\usepackage[noadjust]{cite}
\usepackage{amssymb,amsmath,amsthm,}
\usepackage{amscd}
\usepackage{amsfonts}
\usepackage{amssymb}
\usepackage{latexsym}
\usepackage{color}
\usepackage{esint}
\usepackage{graphicx}
\usepackage{caption}
\usepackage{subcaption}
\usepackage{amsthm}
 \usepackage{verbatim}
 \usepackage{MnSymbol} 
\usepackage{hyperref}

\usepackage[makeroom]{cancel}
\usepackage{mathtools}
\usepackage{pgfplots}

\let\hide\iffalse
\let\unhide\fi

\newtheorem{theorem}{Theorem}[section]

\newtheorem{corollary}[theorem]{Corollary}
\newtheorem{definition}[theorem]{Definition}
\newtheorem{lemma}[theorem]{Lemma}
\newtheorem{proposition}[theorem]{Proposition}
\newtheorem{remark}[theorem]{Remark}

\let\e=\varepsilon

\let\p=\partial

\let\O=\Omega

\let\o=\omega

\let\b=\beta

\newcommand{\R}{\mathbb{R}}
\newcommand{\T}{\mathbb{T}}
\renewcommand{\P}{\mathbf{P}}
\renewcommand{\S}{\mathbb{S}}

\newcommand{\be}{\begin{equation}}
\newcommand{\bm}{\begin{multline}}
\newcommand{\ee}{\end{equation}}
\newcommand{\dd}{\mathrm{d}}

\newcommand{\xb}{x_{\mathbf{b}}}

\newcommand{\tb}{t_{\mathbf{b}}}
\newcommand{\vb}{v_{\mathbf{b}}}

\newcommand{\xf}{x_{\mathbf{f}}}
\newcommand{\tf}{t_{\mathbf{f}}}
\newcommand{\vf}{v_{\mathbf{f}}}
\newcommand{\VN}{\mathbb{V}^{N}}
\newcommand{\X}{\mathbf{x}}
\newcommand{\V}{\mathbf{v}}

\newcommand{\ephalf}{\frac{\varepsilon}{2}}

\newcommand{\Bes}{\begin{eqnarray*}}
	\newcommand{\Ees}{\end{eqnarray*}}
\newcommand{\Be}{\begin{equation}}
\newcommand{\Ee}{\end{equation}}

\def\p{\partial}

\def\O{\Omega}
\def\R{\mathbb{R}}

\def\B{\begin{equation}}
\def\E{\end{equation}}
\def\BN{\begin{eqnarray*}}
\def\EN{\end{eqnarray*}}

\numberwithin{equation}{section}

\usepackage{tikz-cd}
\usepackage{ragged2e}
\usepackage{etoolbox}
\usepackage{lipsum}

\usepackage{color}

\newcommand{\Green}{\textcolor{green}}

\allowdisplaybreaks

\begin{document}
\date{ \today}

\title[Dynamical Billiard and the Boltzmann equation]{Dynamical Billiard and a long-time behavior of the Boltzmann equation in general 3D toroidal domains}


\author{Gyounghun Ko, Chanwoo Kim,\and Donghyun Lee}


 \address{Department of Mathematics, Pohang University of Science and Technology, South Korea, gyounghun347@postech.ac.kr}
 \address{Department of Mathematics, University of Wisconsin, Madison, WI 53706 USA, chanwookim.math@gmail.com}
 \address{Department of Mathematics, Pohang University of Science and Technology, South Korea, donglee@postech.ac.kr}
\begin{abstract}
	
	
	
\hide	Global well-posedness and convergence to equilibrium of the Boltzmann equation with specular reflection boundary condition in 3D non-convex domain is an outstanding open problem in kinetic theory. Inspired by \cite{guo10}, low regularity $L^{\infty}$ solutions of the Boltzmann equation for boundary conditions have been studied and general $C^{3}$ convex domain case was completely solved in \cite{KimLee}. However, it is still widely open for non-convex domains because of the existence of inflection points.  Two dimensional non-convexity was studied in \cite{cylinder} by considering 3D cylindrical domain with analytic non-convex cross-section. In this paper, we study the problem in general solid torus, i.e., solid torus with analytic convex cross-section. This is the first result for the domains which contain 3D nontrivial non-convex structure. 
	\unhide
	 
Establishing global well-posedness and convergence toward equilibrium of the Boltzmann equation with specular reflection boundary condition has been one of the central questions in the subject of kinetic theory. Despite recent significant progress in this question when domains are strictly convex, as shown by Guo and Kim-Lee, the same question without the strict convexity of domains is still totally open in 3D. The major difficulty arises when a billiard map has an infinite number of bounces in a finite time interval or when the map fails to be Lipschitz continuous, both of which happen generically when the domain is non-convex. In this paper, we develop a new method to control a billiard map on a surface of revolution generated by revolving any planar analytic convex closed curve (e.g., typical shape of tokamak reactors' chamber). In particular, we classify and measure the size (to be small) of a pullback set (along the billiard trajectory) of the infinite-bouncing and singular-bouncing cases. As a consequence, we solve the open question affirmatively in such domains. To the best of our knowledge, this work is the first construction of global solutions to the hard-sphere Boltzmann equation in generic non-convex 3-dimensional domains. In Appendix, we introduce a novel method for constructive coercivity of a linearized collision operator $L$ when the specular boundary condition is imposed. In particular, this method works for a periodic cylindrical domain with an annulus cross-section.
\end{abstract}

\maketitle

\tableofcontents

\section{Introduction}
\textit{The Boltzmann equation} is one of the fundamental kinetic models of collisional particles. Mathematically, the Boltzmann equation is a PDE for $F(t,x,v)\geq 0$, a probability distribution function of particles in the phase space. Without external fields, the Boltzmann equation is expressed as
\begin{align} 
			\p_{t}F + v\cdot\nabla_{x} F = Q(F,F) \ \ \text{in} \ \R_+ \times \O \times \R^3,\label{Boltzmann}\\
				F |_{t=0}=F_0 \ \ \text{on} \ \O \times \R^3. \label{initial data}
\end{align}
Here, $F(t,x,v)$ is the probability distribution function of gas particles with velocity $v\in \R^3$ at time $t\in \R_+= [0, \infty)$ and position $x\in \O$ where $\O$ is spatial domain. The Boltzmann collision operator $Q(\cdot ,\cdot)$ is given by 
\begin{equation}\begin{split} \label{collision operator}
	Q(F,G) = \frac{1}{2}\iint_{\R^3\times \S^2} B(v-u,\omega)&[F(u')G(v') +F(v')G(u')  \\
	 -&   F(u)G(v)-G(u)F(v)]\; d\omega du.
\end{split}\end{equation}
The collision kernel $B(v-u,\omega)$ in \eqref{collision operator} depends on $\vert v-u\vert$ and $\cos \theta:= \frac{v-u}{\vert v-u\vert} \cdot \omega$. In this paper, we consider the hard-sphere:
\Be\label{B_h}
B(v-u,\omega)=  |(v-u) \cdot \o|, 
\Ee
as well as the hard-potential with the Grad's angular cut-off assumption:
\Be\label{B}
	B(v-u,\omega)= \vert v-u \vert^{\zeta} b(\theta),  \ \ \text{with}  \ \  0<\zeta \leq 1, \; 0\leq b(\theta)\leq C \vert \cos \theta \vert,
\Ee
for some positive constant $C$. 

In \eqref{collision operator}, the post-collision velocity $(u',v')$ and the pre-collision velocity $(u,v)$ have the following relations: 
\begin{align*}
	u'=u+[(v-u)\cdot\omega]\,\omega,\quad v'=v-[(v-u)\cdot\omega]\,\omega, \quad \text{for } \omega \in \S^2.  
\end{align*}
Note that above expression implies binary elastic collision:
\begin{align*}
u+v=u'+v',\quad |u|^2+|v|^2=|u'|^2+|v'|^2.
\end{align*}
Inheriting from these local conservation laws in the collision, the Boltzmann operator enjoys so-called the collisional invariant property:  
\Be\label{Coll_Inv}
\int_{\R^3} 
\left(\begin{matrix}
 1 \\ 
 v\\
 |v|^2
\end{matrix}\right)
 Q(F,F) (v)\dd v=\left(\begin{matrix}
 0 \\ 
 	0\\
0
 \end{matrix}\right).
\Ee

 In many physical applications, e.g., tokamak reactor, solar wind, etc., particles interact with a physical boundary. The interaction is governed by a boundary condition among various types (\cite{CIP}). In this paper, we focus on a basic ideal law, a specular reflection boundary condition: when a gas particle hits the boundary, then it bounces back with the opposite normal velocity and the same tangential velocity, as a \textit{billiard}:  
\begin{align}\label{specular}
	F(t,x,v)=F(t,x,R_xv),  \ \ \text{on} \  \R_+\times \partial\O \times \R^3,
\end{align}
with a reflection operator $R_x v:=v-2({n}(x)\cdot v){n}(x)$ where ${n}(x)$ is the outward unit normal vector at $x\in \partial \O$. In particular, the specular reflection at the boundary conserves the local mass and energy in general. Therefore, the collisional invariance successfully leads conservation of total mass and total energy: for all $t \geq 0$,
\begin{align}
	\iint_{\O \times \R^3} F(t,x,v) 	 \dd v \dd x & = \iint_{\O \times \R^3} F(0,x,v) 	 \dd v \dd x ,\label{cons_mass}\\
	\iint_{\O \times \R^3} |v|^2 F(t,x,v) 	 \dd v \dd x& = \iint_{\O \times \R^3} |v|^2 F(0,x,v) 	 \dd v \dd x. \label{cons_energy}
\end{align}
The momentum can be conserved in a special case. A domain $\O$ is axis-symmetric if there are vectors $x_0$ and $\varpi$ such that
\Be\label{axis-sym}
[(x-x_0) \times \varpi] \cdot n(x ) =0 \ \ \text{for all } \ x \in \p\O.
\Ee
We note that the specular reflection at the boundary conserves a local angular momentum when the domain is axis-symmetric. Hence, when the domain is axis-symmetric \eqref{axis-sym}, the total angular momentum is conserved: for all $t \geq 0$,
\Be
\begin{split}
	&\iint_{\O \times \R^3}   [(x-x_0) \times \varpi]  \cdot v F(t,x,v) 	 \dd v  \dd x\\
	&= \iint_{\O \times \R^3} [(x-x_0) \times \varpi]  \cdot v  F(0,x,v) 	 \dd v \dd x.\label{cons_energyL}
\end{split}\Ee

Perhaps the most important feature of the Boltzmann equation is its time-irreversibility. In 1872, Boltzmann derived a celebrated \textit{H-theorem}, namely 
 \Be\label{H-theorem}
 \begin{split}
  \frac{d}{dt} \iint_{\O \times  \R^3} F \ln F \dd v  \dd x + D _H(F) =0 ,\end{split}
 \Ee
 with an entropy dissipation $D_H(F):= -\iint _{\O \times \R^3} Q(F,F) \ln F \dd v \dd x  \geq 0$. This non-negativity of the entropy dissipation comes from an algebraic structure of the Boltzmann operator. 
 Suppose $F(t,x,v)$ attains its limit $F_\infty(x,v)$ as $t \rightarrow \infty$. Then we may expect that an entropy dissipation vanishes at the limit, $D_H(F_\infty)=0$, which implies that $F_\infty(u')F_\infty(v')=F_\infty(u)F_\infty(v)$. 
In this case, we know that $F$ should take a form of local Maxwellian (\cite{CIP}), namely 
 \Be\label{def:lMax}
 F_\infty = \exp \left\{ a + b \cdot v + c |v|^2\right\}  \ \text{with $a,c \in \R$ and $b \in \R^3$. }
 \Ee
In general, we can determine the parameters $a,b,$ and $c$ using the conservative quantities (e.g., \eqref{cons_mass}, \eqref{cons_energy} \eqref{cons_energyL}). In this paper, we are interested in a normalized global Maxwellian with zero angular momentum:
 \begin{align}\label{def:gMax}
 	\mu(v)=\frac{1}{(2\pi)^{3/2}}\exp\left(-\frac{\vert v \vert^2}{2}\right). 
 \end{align}

Asymptotic stability of equilibria naturally arises as a central subject in mathematical physics and kinetic theory due to the H-theorem. There are basically two types of available results. The first type is a small perturbation framework based on linear analysis and strong control of the perturbation. Classical works of Ukai \cite{Ukai} and Guo, Strain \cite{GuoVPB, StrainJAMS} refer to this type. In particular, Guo initiated a novel nonlinear energy method in high Sobolev space in collisional kinetic theory without a physical boundary in \cite{GuoVPB}. Using this method, he has solved many open problems, namely constructing a smooth global-in-time solution near the global Maxwellian and proving its convergence as $t\rightarrow \infty$ in various collisional models.

\hide
An asymptotic stability of equilibria (due to the H-theory, essentially) naturally arises as a central subject in the mathematical physics and kinetic theory. 
There are basically two types of available results. The first type is a small perturbation framework based on a linear analysis and strong control of the perturbation. We refer to classical works of Ukai \cite{Ukai} and Guo, Strain et al. \cite{GuoVPB, StrainJAMS} and the reference therein. In particular Guo initiated, in \cite{GuoVPB}, a novel nonlinear energy method in high Sobolev space in the collisional kinetic theory without a physical boundary. Using the method, he has solved many open problems, namely constructing a smooth global-in-time solution near the global Maxwellian and proving its convergence as $t\rightarrow \infty$ in various collisional models.  
\unhide

The second type of results is an asymptotic stability study of possibly-large solutions that are assumed to exist globally a priori. Desvillettes-Villani obtained a convergence of the solutions to equilibria polynomially fast as $t \rightarrow \infty$, under strong smoothness assumptions and a Gaussian lower bound for the solutions in \cite{DV}.
Later, Imbert-Silvestre replaces the strong smoothness assumption of Desvillettes-Villani with a priori pointwise bound of local mass, momentum, and energy by proving a global regularity of solutions in \cite{IS} under the same assumption when the angular cross section $b(\theta)$ is not integrable. In all these results, they work on assumed-to-exist global-in-time solutions with strong a priori uniform-in-time conditions. However, even a weaker condition of a local energy conservation law for the weak global solution `a la DiPerna-Lions (\cite{Diperna-Lion}) has not been settled affirmative over three decades. On the other hand, such class of solutions is non-empty due to Guo’s smooth solutions (and due to Strain-Gressman’s solutions in \cite{StrainJAMS} when the angular cross section is not integrable).


We also refer to \cite{DHWY2017,DKL2020,KLP2022} for an interesting recent construction of global solutions with large amplitude but small-in-$L^p$ data. Unlike small perturbation problems, it is very hard to study large amplitude problems for other kinetic equations such as BGK, of which nonlinear order (in collision term) of perturbation is higher than that of the Boltzmann equation. Therefore, constructing large amplitude solutions for such higher nonlinear models would be very interesting. Recently, in \cite{BKLY_BGK}, they obtained the large amplitude solutions in the BGK model by careful analysis to treat highly nonlinear structure.

  \subsection{Coercivity in $L^2$}\label{sec:coercivity} As the H-theorem suggests, the main stabilizing mechanism comes from the entropy dissipation. It is convenient to look at a relative entropy: with $\mathfrak{h}(z) = (1+ z ) \ln (1+ z)  -z$,
\Be\label{def:rH}
H(F | \mu) : = \iint_{\O \times \R^3}  \mu \mathfrak{h} \left(\frac{F-\mu}{\mu}\right)\dd v \dd x.
\Ee 
The relative entropy has the same dissipation $D_H$ of \eqref{H-theorem}.  
As the variation of relative entropy gives rise to an energy $\iint_{\O \times \R^3} \frac{1}{2\mu} (F(t,x,v) - \mu(v))^2 \dd v \dd x$ (since $\mathfrak{h}(z)\sim \frac{z^2}{2}$), \hide
\Be
 \frac{d}{dt}\iint_{\O \times \R^3} \frac{1}{\mu} (F(t,x,v) - \mu(v))^2 \dd v \dd x - \iint_{\O \times \R^3} Q(\mu, F- \mu ) (t,x,v) \frac{F(t,x,v)- \mu(v)}{\mu(v)}\dd v \dd x .\notag\label{DrH}
\Ee\unhide
a natural perturbation form turns to be 
\Be\label{pert}
F= \mu + \sqrt{\mu} f.
\Ee 
On the other hand, the entropy dissipation gives rise to a linearized Boltzmann operator  
\Be\label{Lf}
Lf := \frac{-1}{\sqrt{\mu}} Q(\mu, \sqrt{\mu} f).
\Ee
Due to the inherited collisional invariance \eqref{Coll_Inv}, $L$ has a 5-dimensional null space, whose projection is given by 
\Be\label{nullL}
\P 
f
=\big\{
a_f (t,x) +    b_f (t,x) \cdot v  +c_f (t,x)  \frac{|v|^2 -3}{\sqrt{6} }
\big\}\sqrt \mu.
\Ee

Furthermore, for the hard-potential models, $Lf = \nu(v)f - Kf$ has a natural decomposition by a compact operator $K
$ and the collisional frequency $\nu(v)$, which is given by
\begin{equation}\label{def:nu}
	\nu(v) := \int_{\mathbb{R}^{3}} \int_{\mathbb{S}^{2}} |v-u|^{\zeta}q_{0}(\theta) {\mu}(u) d\o du
	\sim \langle v \rangle^{\zeta}
	:=  [1+|v|^{2}]^{\zeta/2}
	.
\end{equation}

Weyl's theorem implies that $L$ is non-negative with respect to the local energy modulo $\P f$: 
\Be\label{lowerL}
\int_{  \R^3} fL f d v   \gtrsim 
\| (\mathbf{I}- \P)
f \|_{L^2_{\nu} (\R^3)}^2
:=
\int_{\R^3} \nu  | 
(\mathbf{I}- \P)
f|^2 \dd v .
\Ee 
Clearly, it originally comes from the non-negativity of the entropy dissipation $D_H$ in \eqref{H-theorem}, which is the major $L^2$-relaxation mechanism. By the missing part in the lower bound of \eqref{lowerL}, we need coercivity estimate to obtain $L^2$-decay. In \cite{Guo10}, Guo developed his micro-macro decomposition method in the presence of the boundary, in which the proof is non-constructive. In \cite{EGKM}, Esposito-Guo-Kim-Marra devised a constructive method to prove the coercivity in the presence of diﬀusive reﬂection boundary condition. It is natural to ask the same question of the constructive coercivity proof in the presence of specular reﬂection boundary. To the best of our knowledge, however, this is not available yet. The key diﬃculty is the lack of clear decomposition of tangential momentum and normal momentum, which is crucial to cancel a boundary terms eﬀectively. In Appendix, we present a special case study in which we can prove the coercivity constructively.


\hide

For an asymptotic stability study in the perturbation framework, the major challenge is to recover a missing control of $\P f$ in the dissipation lower bound. To overcome it, Guo initiated a novel nonlinear energy method in high Sobolev norms in the Boltzmann study (\cite{GuoVPB, GuoVMB}). The novelty is proving that $\iint fLf dv dx $ control the full $\|  f\|_{\nu}^2$ for small-in-$H^k_x$ solutions $f$ to the nonlinear Boltzmann equation. The key discovery in his 2002 paper \cite{GuoVPB} is so-called macroscopic equations for $(a_f,b_f,c_f)$ connecting $\P f$ and $(\mathbf{I} - \P) f$ via the Boltzmann equation as 
\Be
\Delta \P f \sim  \partial^2 (\mathbf{I} - \P) f +  \cdots. 
\Ee
Using such hidden ellipticity (see page 1129 in \cite{GuoVPB}), Guo can control the missing $\mathbf{P} f$ in \eqref{lowerL} via the microscopic part $(\mathbf{I} - \P) f$, so that $\iint fLf dv dx $ can control the full $\| f\|_{\nu}^2$. Liu-Yu used a similar decomposition with respect to a local Maxwellian of traveling waves in their 2004 paper \cite{LiuYu}.

The method of macroscopic equations \`a la Guo has been applied to several boundary problems. In a paper of Esposito-Guo-Kim-Marray \cite{EGKM}, they devise a method of test functions when the diffusive reflection boundary condition is imposed, inspired by the macroscopic equations. It is worth to note that they derive the elliptic operator in a slight different way than \cite{GuoVPB, GuoVMB}, so that the method of \cite{EGKM} can apply to 1D. Later, the method of test functions is generalized to the specular reflection boundary condition in using a special boundary condition for the Poisson equation. See Section \ref{sec:coercivity}.

\unhide

\subsection{$L^\infty$-Control} Now let us count the nonlinear operator
\begin{equation}\label{E_eqtn}
	\p_{t}f + v\cdot \nabla_{x} f + L f 
	= \Gamma(f,f):= \frac{1}{\sqrt{\mu}} Q(\sqrt{\mu}f, \sqrt{\mu}f ).
\end{equation} 
A major challenge to this nonlinear problem is then to achieve a pointwise-in-$x$ bound of $f$ in order to control the quadratic nonlinearity. It should be noted that a strategy to use high Sobolev norms via the embedding (e.g. \cite{GuoVPB, DV}) breaks down in the presence of boundary in general. In fact, solutions are singular in the presence of boundary \cite{Kim2011, cylinder,GKTT2017, GKTT2016,ChenK}, in general.

To overcome this difficulty, Guo developed a method of characteristics for the Boltzmann equation to bootstrap the $L^2$ bound into a pointwise bound, in \cite{Guo10}. The basic idea follows: Let $U(t)$ be the semigroup generated by $v \cdot \nabla_x + L$ and $G(t)$ be the semigroup generated by $v\cdot \nabla_x + \nu$ with the specular reflection boundary condition. By \textit{two} iterations, one can establish
\Be\begin{split}
U(t) &= G(t) + \int^t_0 G(t-s) K G(s) d  s\\
& \quad + 
\int^t_0 \int^{s}_0  G(t-s) KG (s-s_1) KU (s_1)ds_1 ds.\notag
\end{split}\Ee
From the compactness property of $K$, the main contribution is roughly
\Be\label{main_Guo}
\begin{split}
&\int^t_0  ds_1 e^{-\nu (t-s)} \int^{s_1}_0 ds   e^{-\nu (s-s_1)}  \\
& \times  \iint_{|v^\prime| \leq N, |v^{\prime\prime}| \leq N }
| f(s_1, X(s_1; s, X(s;t,x,v), v^\prime), v^{\prime\prime})| d v^\prime d v^{\prime\prime}
,
\end{split}
\Ee
where $(X(s;t,x,v), V(s;t,x,v))$  denotes the generalized characteristics associated with the specular reflection boundary condition for given $t\geq s \geq 0$ (See Definition \ref{def:XV}), and $N$ is some chosen large number.

The key idea of Guo in \cite{Guo10}, inspired by the Vidav's idea \cite{Vidav},  is to control the main term \eqref{main_Guo} by $L^2$-in-\textit{space}-norm of $f$ via a crucial change of variables 
\Be\label{COV_intro}
v^\prime \mapsto  X(s_1; s, X(s;t,x,v), v^\prime)
 \ \  \ \ \ \text{({a billiard mixing map})}.
\Ee
 In other words, a mixing effect of the transport operator transfers the $L^p_{x}$-control in Section \ref{sec:coercivity} to a pointwise control if the billiard mixing map has full rank in the following sense 
 \Be\label{C}
 \begin{split}
 &\det \left(\frac{\p X(s_1; s,X(s;t,x,v), v^\prime)}{\p v^\prime}\right)\gtrsim_\e 1, \\
& \ \ \  \ \ \  \text{except a small subset of  } (s, v^\prime)  \in [0, s_1] \times \R^3 .
 \end{split}
 \Ee
  If the trajectory is a straight line, then the change of variables is valid for the most of time $s_1$. However, in the presence of a reflexive boundary, verifying such a change of variables is obviously extremely difficult for its possible \textit{singularity} and \textit{complexity} in a general domain.

  In our opinion, verifying \eqref{C} (or \eqref{C'} in general) is not just a technical matter of some particular method. We call the authors' attention to the fact that the exponential relaxation mechanism of Boltzmann comes from the H-theorem of relative entropy. In general domains, any spatial derivative of solutions does not hold the H-theorem as they do not satisfy a boundary condition. Moreover, the characteristics of the specular reflection do not propagate $C^1$-norm (only propagates a $C^{0,\frac{1}{2}}$-H\"older norm is known \cite{KL_holder}) whenever it has hit the boundary tangentially (a grazing bounce). Therefore we expect that derivatives of solutions could oscillate (See \cite{KL_holder}), while the $L^2$-relaxation comes through (cf. relaxation of the gradient solutions in the case of diffuse reflection boundary \cite{ChenK2}). In this situation, if the change of variables \eqref{COV_intro} fails when $X(s;t,x,v)$ belongs to a codimension-$1$ subset of the phase space, then the $L^2$-relaxation does not transfer to a pointwise control and the pointwise bound of the nonlinear problem may not be available.

 In various areas of mathematics such as differential topology, geometric measure theory, and symplectic geometry, several weaker theorems close to the nondegeneracy condition \eqref{C} appear as important subjects. For example, 
 in the field of differential topology, the transversality theorem is a major result that describes the transverse intersection properties of a family of smooth maps. 
 \begin{theorem}[The Transversality Theorem, page 68 in Guillemin-Pollack \cite{GP}]
 	Suppose 
 	\Be\label{frakF}
 	\mathfrak{F} :  \mathfrak{X} \times  \mathfrak{S} \rightarrow \mathfrak{Y}
 	\Ee
 	is a smooth map of manifold $\mathfrak{X},  \mathfrak{S}, \mathfrak{Y}$ without boundary. Let $\mathfrak{Z}$ be any boundaryless submanifold of $\mathfrak{Y}$. Suppose $	\mathfrak{F} $ is transversal to $\mathfrak{Z}$, i.e. (see page 28, Section 5 of Chapter 1 in the Guillemin-Pollack \cite{GP}) 
 	\Be\label{def:transverse}
 	D_a \mathfrak{F} (T_a ( \mathfrak{X} \times  \mathfrak{S} ) )
 	+ T_{\mathfrak{F} (a)}\mathfrak{Z} =  T_{\mathfrak{F} (a)} \mathfrak{Y} \ \ \text{for every }
 	a \in 	\mathfrak{F}^{-1} (\mathfrak{Z}).
 	\Ee
 	Then for almost every $s \in \mathfrak{S}$, the map $	\mathfrak{F} _s : = \mathfrak{F} (\cdot, s) : \mathfrak{X}  \rightarrow \mathfrak{Y}$ is transversal to $\mathfrak{Z}.$
 \end{theorem}
 
 Now let us see how this theorem can apply to our case. Assume that the generalized characteristics $(X(s;t,x,v), V(s;t,x,v))$ is smooth locally.  
 Then the map $
 (   x,v) \mapsto  (   \mathcal {X}(x,v),    \mathcal  {V}(x,v))
 := (X(s;t,x,v), V(s;t,x,v))
 \in \O \times \R^3\label{def:Z}$ is a local diffeomorphism. Because of that, a map $( \mathcal X, 0): (x,v) \mapsto ( X(s;t,x,v), 0)$ is \textit{transversal} to $\{0\} \times \R^3_v$ (see \eqref{def:transverse}), i.e. for all $(x,v) \in \O \times \R^3$,
 \Be\label{transversal}
 \begin{split}
 	& \frac{\p (\mathcal  X, 0)}{ \p(x,v)} (T_{(x,v) }(\O \times \R^3) ) + T_{ (\mathcal  X, 0)(x,v)} ( \{0\} \times  \R^3_v) \\
 	&=\text{Proj}_{\R^3_x \times \{0\}} 
 	\left(\frac{\p (\mathcal  X, \mathcal  V)}{ \p(x,v)} (\R^6) \right)+  \{0\} \times  \R^3_v  \\
 	& = \R^3_x \times \{0\} + \{0\} \times  \R^3_v   = \R^6 =  T_{\mathcal (\mathcal  X,0)(x,v)}(\O \times \R^3) .
 \end{split}\Ee
 Here, a project on the space variables is denoted by $\text{Proj}_{\R^3_x \times \{0\}} = \begin{bmatrix}
 	Id_{3 \times 3}  & 0_{3 \times 3} \\
 	0_{3 \times 3}  & 0_{3 \times 3} 
 \end{bmatrix}$. From the second to the third line in \eqref{transversal}, we have used that $(x,v)  \mapsto (\mathcal{X}(x,v), \mathcal{V}(x,v))$ is a local diffeomorphism so that the Jacobian matrix $\frac{\partial (\mathcal  X, \mathcal  V)}{ \partial (x,v)} $ has a full rank 6, and hence $\text{Proj}_{\R^3_x \times \{0\}}  \frac{\partial (\mathcal  X, \mathcal  V)}{ \partial (x,v)} $ has a full rank 3 in $\R^3_x \times \{0\}$. Now we apply the transversality theorem. Set $\mathfrak{F} = (\mathcal{X}, 0)$, $\mathfrak{X} = \R^3_v$, $\mathfrak{S} =\O$, $\mathfrak{Y}  = \O \times \R^3_v$, and $\mathfrak{Z} = \R^3_v$. We can conclude that the map $(\mathcal X, 0): v \mapsto (\mathcal X(x,v), 0 )$ is transversal to $\{0 \} \times \R^3_v$ for almost every $x \in \O$, assuming the non-grazing condition \eqref{non-grazing} holds. This means that, under the same assumption of \eqref{non-grazing},   
 \Be
 \frac{\p ( \mathcal{X}, 0)}{\p v} ( \R^3\times \R^3) + \{0\} \times  \R^3_v  =  \R^3 \times \R^3 \ \ \text { for almost every }  x  .\notag
 \Ee
 Therefore we conclude that if the non-grazing condition \eqref{non-grazing} holds then 
 \Be\label{conc:Trans}
 \det \left(\frac{ \p X(s;t,x,v)}{\p v}\right) \neq 0  \ \ \text { for almost every $x$, and all $v$.} 
 \Ee
 It is important to understand that this \textit{almost-every-in-x} result \eqref{conc:Trans} is not sufficient for the $L^2$-$L^\infty$ bootstrap scheme! 
 For example if the nondegeneracy condition fails in some subset of codimension-$1$ then a pointwise control may not be possible in 3D via the $L^2$-$L^\infty$ bootstrap argument.

 \subsection{Major difficulty-I: Singularity of the billiard mixing map} \label{subsec:SS}

 
 In general, the map $v^\prime \mapsto X(s_1; s, *, v^\prime)$ in \eqref{COV_intro} is not Lipschitz continuous, and hence the Jacobian matrix $\p X / \p v^\prime$ is not always well-defined (the spatial argument is $*= X(s;t,x,v)$, in our case). When the boundary is smooth such a singularity happens if and only if when the following \textit{non-grazing condition} \eqref{non-grazing} fails:
 \Be\label{non-grazing}
 (X(\tau;s,*,v^\prime), V(\tau;s,*,v^\prime)) \notin \gamma_0 \  \text{ for all }  \tau \in (s_1, s).
 \Ee
 Here, the grazing phase boundary is defined as $\gamma_0 = \{ (x,v) \in \p\O \times \R^3: n(x) \cdot v=0\}$. 
 
 
 In a general non-convex domain, the generalized characteristics emanating from the grazing boundary $\gamma_0$ propagate into the interior of the phase space. It is very hard to locate such singularity due to loss of regularity of the billiard map and a possible infinite number of bouncing. We think that these two difficulties are best understood in the language of the billiard theory. Following \cite{Chernov}, let a collision space $\mathcal{M} : =   \{ (x,v)\in \partial \O \times  \mathbb{S}^2: n(x) \cdot v \leq 0 \}$ with a natural measure $|n(x) \cdot v|\dd S_x \dd v$ and a collision map $\mathcal{F}: \mathcal{M} \rightarrow \mathcal{M}$ such that $(x,v) \mapsto (\xb(x,v), R_{\xb(x,v)} v)$. The collision map $\mathcal{F}$ preserves a measure on $\mathcal{M}$. Here, we used a backward exit time and position defined as 
 \begin{align}\label{def:tb}
 	\begin{split}
 		t_{\mathbf{b}}(x,v)  & :=\sup\{\tau \geq 0 : x-sv \in \Omega \textrm{ for all } 0\leq s \leq \tau\},  \\
 		x_{\mathbf{b}}(x,v) &  :=x-\tb(x,v)v \in \p\O,
 	\end{split}
 \end{align}
 with a special convention $\sup \emptyset=0$. Now we can construct singular sets inductively. Set $\mathcal{S}_0 = \p \mathcal{M}$, which actually is the grazing phase boundary $\gamma_0$. Now we define a singular set for the map $\mathcal{F}$ which is $\mathcal{S}_1 := \mathcal{S}_0 \cup
 \mathcal{F}^{-1} (\mathcal{S}_0)$. We also set $\mathcal{S}_{-1} := \mathcal{S}_0 \cup
 \mathcal{F} (\mathcal{S}_0)$. In other words, $ \mathcal{F}^{-1} (\mathcal{S}_0) = \{(x,v) \in \mathcal{M}: n( x_{\mathbf{b}} (x,v)  ) \cdot v = 0 \}$, which is indeed a codimension-1 subset of $\mathcal{M}$ and a smooth submanifold away from $\mathcal{S}_0$. Therefore we can view $ \mathcal{F}^{-1} (\mathcal{S}_0) $ as a disjoint union of codimension-1 smooth submanifolds whose intersections are always on $\mathcal{S}_0$. These submanifolds fail to be $C^1$ on the intersection. Then we have that $\mathcal{F} : \mathcal{M}\backslash \mathcal{S}_{1} \rightarrow \mathcal{M} \backslash \mathcal{S}_{-1}$ is a smooth diffeomorphism. Inductively we construct a singular set of the k$^{\text{th}}$ composition $\mathcal{F}^{k}:=   {\mathcal{F} \circ  \cdots \circ \mathcal{F}} $ which is $\mathcal{S}_k = \mathcal{S}_{k-1} \cup \mathcal{F}^{-1} (\mathcal{S}_{k-1})$. Similarly to $\mathcal{F}^{-1} (S_0)$, the set  $\mathcal{F}^{-1} (\mathcal{S}_{k-1}): = \{ 
 (x,v) \in \mathcal{M}: n(x ^{k} (x,v)) \cdot  v ^{k-1} (x,v)=0
 \}$ is a disjoint union of submanifolds intersecting each other on $\mathcal{S}_{k-1}$. Here, we have used a billiard cycle, which is defined as $(t^0, x^0, v^0)= (t,x,v)$ and $(t^1, x^1, v^1)= (t-t_\mathbf{b}(x,v),x_\mathbf{b}(x,v), R_{x_\mathbf{b} (x,v)} v )$, and successively 
  \Be\label{bcycle}
 (t^{k+1},x^{k+1},v^{k+1})=(t^k-t_{\mathbf{b}}(x^k,v^k),x^k- t_{\mathbf{b}} (x^k, v^k) v^k,R_{x^{k+1}}v^k). 
 \Ee
 
 Finally, we can identify a singular set for given $s_1, s$ and arbitrary position $*$, in which \eqref{non-grazing} fails, with 
 \Be\label{singular_S}
\bigcup_{k: t^{k+1} (s, *, v^\prime) \geq s_1   }^\infty \left\{ v^\prime:  \left( x_\mathbf{b}(*, v^\prime), \frac{v_\mathbf{b}(*, v^\prime) }{| v_\mathbf{b}(*, v^\prime)|}\right) \in \mathcal{S}_k  \right\}.
 \Ee
 
 To realize Vidav's idea in the $L^\infty$ estimate of \eqref{main_Guo}, it is essential to construct a continuous $\e$-neighborhood of the singular set \eqref{singular_S} which has small measure in $v^\prime \in \R^3$. This is an extremely hard task because of the following reasons: The first reason is that $\left( x_\mathbf{b}(*, v^\prime), \frac{v_\mathbf{b}(*, v^\prime) }{| v_\mathbf{b}(*, v^\prime)|}\right) \in \mathcal{S}_k$ could be a large set in $v^\prime$ for $k\geq 1$ (it has zero measure for $k=0$ by Sard's theorem \cite{Guo10}). Second, $S_k$ is not a uniformly $C^1$ smooth manifold, but a disjoint union of many ($\# \sim (C_{\O})^k$) codimension-1 submanifolds which are H\"older regular at their intersections (\cite{KL_holder}). Lastly, $\bigcup_{k=0}^\infty \mathcal{S}_k$ is dense in $\mathcal{M}$ (\cite{Chernov})! In other words, the singular set could be dense, large-in-measure, and ill-behaved, in general!



 \subsection{Major difficulty-II: Number of bouncing} We already have seen why we have to control the number of bouncing in the construction of the singular set, otherwise the singular set could be dense! More precisely, we want to bound the maximal number $k \in \mathbb{N}$ satisfying $t^{k+1} (s, *, v^\prime) \geq s_1$.

 Now we discuss another reason which is closely related to our method (\cite{KimLee}) studying the lower bound estimate of Jacobian in \eqref{C}. In \cite{KimLee}, Kim-Lee develops a method to investigate the Jacobian \eqref{conc:Trans} using the billiard map $\mathcal{F}^k$. The key of the method is to relate the Jacobian $v \mapsto X(s;t,x,v)$ to the billiard map $(x,v) \mapsto \mathcal{F}^k(x,v)$ using a triple iteration of \eqref{main_Guo}. Hence they are able to prove a lower bound of Jacobian \eqref{conc:Trans}, which depends on a pointwise bound of Jacobian matrix $(x,v) \mapsto \mathcal{F}^k(x,v)$. Unfortunately, such a map blows up on the singular set $\mathcal{S}_k$. In a strictly $C^3$ convex domain, they can control the number of bounces using the velocity lemma. On the other hand, the number of bouncing could be unbounded if the domain fails either convexity or a $C^3$ boundedness \cite{Halpern}. Therefore it is very challenging and important to control the number of bounces. 
 
Both difficulties are also major topics in the contents of billiard theory. It would be worth mentioning a general aspect of our problem in billiard theory. First of all, our domains are essentially 3-dimensional, while most of the billiard theory is built for 2D domains. Second, our billiard is not hyperbolic, which means Lyapunov exponents could be zero and all dispersing, focusing, and degenerating billiard could happen in a single cycle. In particular, the billiard map is not even defined when it hits the inflection point of the boundary tangentially. Moreover, this set could have a large measure in the collision space. Thirdly, modern billiard theory based on the ergodic theory and many available results hold almost everywhere, which is not sufficient for our purpose as we have discussed across \eqref{conc:Trans}.

\hide where the nonlinear term becomes
\begin{equation} \label{Gamma}
	\Gamma(f,f) = \frac{1}{\sqrt{\mu}} Q(\sqrt{\mu}f, \sqrt{\mu}f ).
\end{equation}\unhide


 \subsection{Decomposition of 3D general toroidal domains}In this paper, we consider a \textit{toroidal} domain with \textit{convex} planar cross section whose boundary is a \textit{strictly convex analytic} closed curve. This type of domain appears in many applications, such as modern tokamak reactors.

 \begin{definition}[Domains $\O$ and their boundary $\p\O$]\label{def:domain}
 	Suppose a closed curve is sitting on the $xz$-plane away from the $z$-axis: with $\gamma(a)=\gamma(b)$ and $\gamma_1>0$,
 	\Be\label{def:gamma}
 	\gamma(\tau)=(\gamma_1(\tau),0,\gamma_2(\tau)) \ \ \text{for} \ a\leq \tau < b.
 	\Ee 
 	We assume that this curve is positively oriented, \textit{real analytic}, and \textit{unit-speed}, namely $\gamma_i :  \tau \in [a,b) \rightarrow xz\text{-plane}$ is real-analytic with respect to $\tau$; and $(\gamma_1'(\tau))^2+(\gamma_2'(\tau))^2=1$ for all $\tau \in [a,b)$. Furthermore, we assume that the curve is strictly convex, i.e.  
 	\Be\label{gamma:convex}
 	\gamma_1^\prime (\tau)\gamma_2^{\prime\prime} (\tau)- \gamma_2^{\prime} (\tau)\gamma_1^{\prime\prime} (\tau)>0 \ \ \text{for} \ a\leq \tau < b.
 	\Ee
 	Here and after, we always regard $[a,b)$ as a periodic interval. Then there exist $\tau_{1,*}, \tau_{2,*}$, and $\lambda_{*}$ in $[a,b)$ such that $\gamma_2'(\tau_{1,*}) =0=\gamma_2'(\tau_{2,*})$, $\gamma_2'(\tau)<0$ for $\tau_{1,*}<\tau<\tau_{2,*}$, and $\lambda_*\in(\tau_{1,*},\tau_{2,*})$ is the unique zero of the function $\gamma_1'$ in $(\tau_{1, *}, \tau_{2, *})$: $\gamma_1' (\lambda_*)=0$. Also it is easy to check that $\gamma_1'(\tau)<0$ for $\tau \in (\tau_{1,*},\lambda_*)$ and $\gamma_1'(\tau)>0$ for $\tau \in (\lambda_*,\tau_{2,*})$. See Figure \ref{analytic convex curve}. 
 	
 	We obtain our boundary $\p\O$ by rotating this curve about the $z$-axis, which can be parameterized as 
 	\Be \label{sigma para} 
 	\sigma(\tau,\varphi)=(\gamma_1(\tau)\cos\varphi,\gamma_1(\tau)\sin\varphi,\gamma_2(\tau)),  \ \  a\leq \tau < b, \; 0\leq \varphi < 2\pi.
 	\Ee
 	When we define the boundary, it suffices to restrict $\varphi \in [0, 2\pi)$. However, later we will extend $\varphi$ to the whole real number $\R$ to count a winding number of billiard trajectory winding around the $z$-axis. Finally, we define our domain $\O$ to be the bounded open subset of $\R^3$ surrounded by the boundary $\p\O$.  
 \end{definition}

 \begin{figure}[h]
	\centering
	\includegraphics[width=16cm]{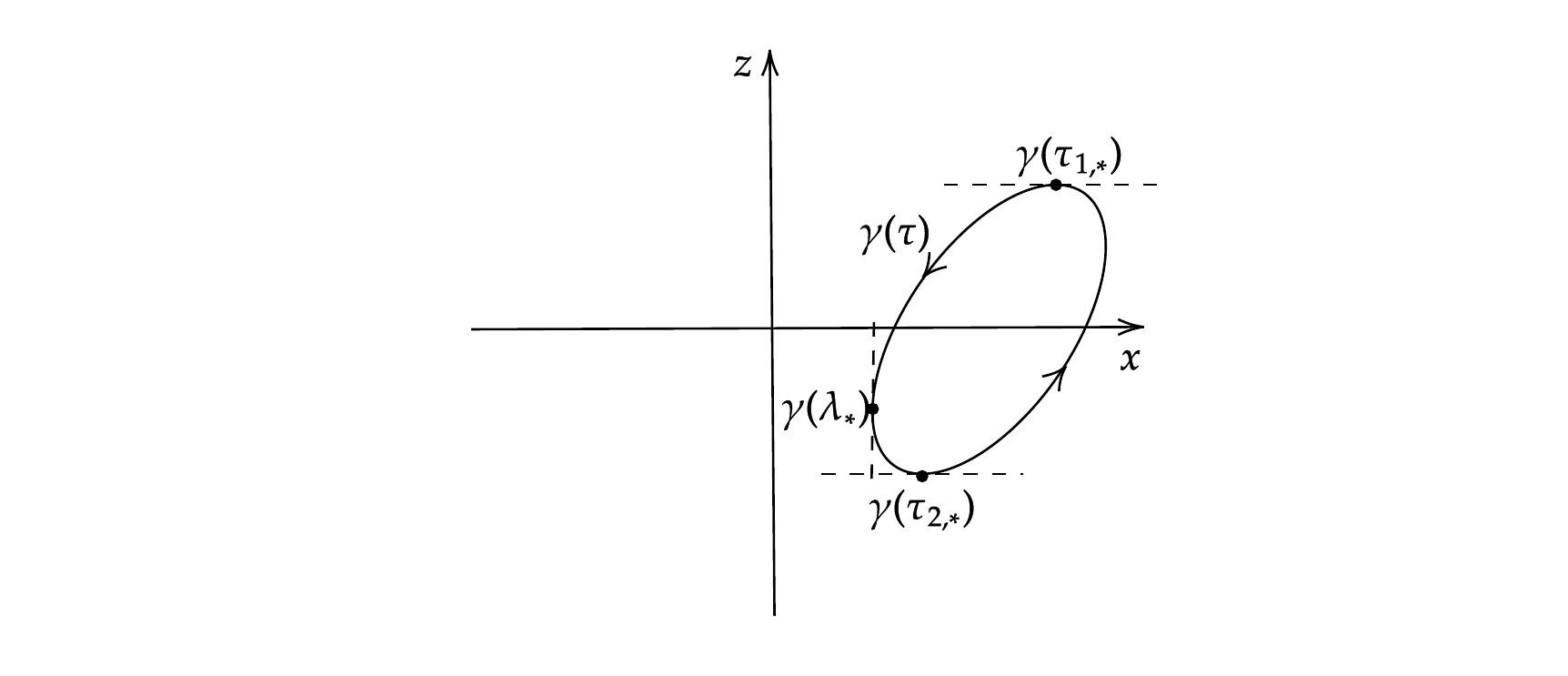} 
	\vspace{-25pt}
	\caption{Analytic convex curve $\gamma$ in $xz$-plane}
	\label{analytic convex curve_figure}
\end{figure}

 Although the above $\O$ and $\p\O$ are well-defined, it turns out that introducing an auxiliary indication function of the domain is convenient. Let $\overline{\xi}: xz\text{-plane} \rightarrow \R$ be a real-analytic function and satisfies that $\overline{\xi}<0$ in $\O \cap xz\text{-plane}$, $\overline{\xi}>0$ in $ \O^C \cap xz\text{-plane}$, and 
 \Be \label{analytic convex curve}
 \overline{\xi}=0 \ \text{on} \   \p\O \cap xz\text{-plane}.
 \Ee
For the construction of such functions, we refer to \cite{ChenK}. With this choice, $\overline{\xi}$ is real-analytic and convex. Finally we set an indication function in $\R^3$, which is real-analytic and convex, by 
 \Be \label{surface}
 \xi(x,y,z) :=\overline{\xi}(\sqrt{x^2+y^2},z).
 \Ee
 Now, we have the following relations
 \Be \label{def domain}
 \begin{split}
 	\O&:=\{(x,y,z)\in \R^3: \xi(x,y,z) :=\overline{\xi}(\sqrt{x^2+y^2},z)<0\},\\
 	\partial \O&:=\{(x,y,z)\in \R^3:\xi(x,y,z) := \overline{\xi}(\sqrt{x^2+y^2},z)=0\}.
 \end{split}
 \Ee

 \hide
 Since $\O$ do not include any points on $z$-axis, i.e., $\gamma_{1}(\tau) > 0$, $\xi$ is also anaytic.

 Now, let us introduce general solid torus domain $\O$. $\O$ is a kind of solid torus whose cross-section is a general analytic convex curve instead of a circle. First we define a general analytic convex curve. Let $z$-axis be rotational axis. (see Figure \ref{solid torus}) We define an analytic convex curve in $xz$ plane WLOG as follows:
 \begin{align}\label{analytic convex curve}
 	\overline{\xi}(x,z)=0,
 \end{align}
 where $\overline{\xi}$ is analytic and (uniformly) convex. The analytic convex curve $\overline{\xi}$ can be parametrized by analytic regularized(unit speed) curve, 
 \begin{align*}
 	\gamma(\tau)=(\gamma_1(\tau),0,\gamma_2(\tau)), 
 \end{align*}
 for $a\leq \tau \leq b$ with $\gamma(a)=\gamma(b)$. Here, $\gamma$ is a unit-speed curve, $(\gamma_1'(\tau))^2+(\gamma_2'(\tau))^2=1$ for all $\tau$. There exists unique $\tau_{1,*}, \tau_{2,*}$, and $\lambda_{*}$ such that $\gamma_2'(\tau_{1,*})=\gamma_2'(\tau_{2,*})=0$, $\gamma_2'(\tau)<0$ for $\tau_{1,*}<\tau<\tau_{2,*}$, and $\lambda_*\in(\tau_{1,*},\tau_{2,*})$ is the zero of the function $\gamma_1'$. See Figure \ref{analytic convex curve}. Also it is easy to check that $\gamma_1'(\tau)<0$ for $\tau \in (\tau_{1,*},\lambda_*)$ and $\gamma_1'(\tau)>0$ for $\tau \in (\lambda_*,\tau_{2,*})$. Rotating the curve about $z$-axis yields a $z$-axial symmetric surface
 \begin{align} \label{surface}
 	\overline{\xi}(\sqrt{x^2+y^2},z)=0.
 \end{align}

 The boundary $\p\O$ can be also parametrized by 
 \begin{align} \label{sigma para} 
 	\sigma(\tau,\varphi)=(\gamma_1(\tau)\cos\varphi,\gamma_1(\tau)\sin\varphi,\gamma_2(\tau)), \quad a\leq \tau < b, \; 0\leq \varphi < 2\pi.
 \end{align}
 Here, we might as well assume that $\gamma_1(\tau)>0$ for all $\tau \in[a,b)$, so that $\gamma_1(\tau)$ is the distance of $\sigma(\tau,\varphi)$ from the $z$-axis. When we restrict $0\leq \varphi < 2\pi$, note that $\sigma$ is surjective. (However, we will extend $\varphi$ to be larger than $2\pi$ later to consider multiple winding number around $z$-axis.) Now, a general solid torus domain $\O$ is 3D non-convex $z$-axial symmetric domain 
 \Be \label{def domain}
 \O:=\{(x,y,z)\in \R^3: \xi(x,y,z) :=\overline{\xi}(\sqrt{x^2+y^2},z)<0\},\quad \partial \O:=\{(x,y,z)\in \R^3:\xi(x,y,z) := \overline{\xi}(\sqrt{x^2+y^2},z)=0\}.
 \Ee
 Since $\O$ do not include any points on $z$-axis, i.e., $\gamma_{1}(\tau) > 0$, $\xi$ is also anaytic.  \\
 \unhide

 For $(t,x,v)\in [0,\infty) \times \overline\Omega \times \R^3$, ($\overline{\O}$ is a closure of $\O$)
 \hide
 \begin{align}
 	\begin{split}
 		X(s;t,x,v)&:=\textrm{position of the particle at time s which was at position $x$ with velocity $v$ at time $t$,}\\
 		V(s;t,x,v)&:=\textrm{velocity of the particle at time s which was at position $x$ with velocity $v$ at time $t$,}
 	\end{split}
 \end{align}
 where $X(s;t,x,v),V(s;t,x,v)$ satisfy 
 \begin{equation*}
 	\frac{d}{ds} X(s;t,x,v)=V(s;t,x,v), \quad \frac{d}{ds} V(s;t,x,v)=0,
 \end{equation*}
 with initial condition $X(t;t,x,v)=x, V(t;t,x,v)=v$. \\ 
 \unhide
 we let $\gamma=\partial \Omega \times \R^3$ denote the boundary of phase space $  \Omega \times \R^3$. We further split $\gamma$ into an outgoing boundary $\gamma_+$, an incoming boundary $\gamma_-$, and a grazing boundary $\gamma_0$:
 \Be\begin{split}\label{gamma}
 	\gamma_+&:=\{(x,v)\in \partial \Omega \times \R^3 : n(x)\cdot v >0\}, \\
 	\gamma_-&:=\{(x,v)\in \partial \Omega \times \R^3 : n(x)\cdot v<0\}, \\ 
 	\gamma_0&:=\{(x,v) \in \partial \Omega \times \R^3 : n(x)\cdot v=0\}. 
 \end{split}\Ee
 Here, $n(x)$ denotes the outward normal at $x \in \p\O$ such that $n(x) = \frac{\nabla \xi (x)}{|\nabla \xi (x)|}$. 
 
 Recall the backward exit time $\tb(x,v)$ in \eqref{def:tb}. Remark that $(x,v) \in \p\O \times \S^2  \mapsto (	x_{\mathbf{b}}(x,v), R_{	x_{\mathbf{b}}(x,v)} v) \in \p\O \times \S^2$ is the billiard map in the mathematical theory of billiard \cite{Chernov}. Here, the reflection operator is denoted by $R$:
\Be\label{reflectionO}
R_x v:=v-2({n}(x)\cdot v){n}(x).
\Ee

 We set $(t^0,x^0,v^0)=(t,x,v)$ and define backward cycles successively, for $k\geq 0$,  
 \Be\label{bcycle}
 	(t^{k+1},x^{k+1},v^{k+1})=(t^k-t_{\mathbf{b}}(x^k,v^k),x^k-(t^k-t^{k+1})v^k,R_{x^{k+1}}v^k), 
 \Ee
 where $R_{x^{k+1}} = I - 2 {n}(x^{k+1})\otimes {n}(x^{k+1}) $. We also define the forward exit time and position  \begin{align*}
 	t_{\mathbf{f}}(x,v)&:= \sup\{ \tau\geq 0 :  x+ s v \in \Omega \textrm{ for all } s \in (0,\tau)\},\\
 	x_{\mathbf{f}}(x,v)&:= x+t_{\mathbf{f}}(x,v)v\in \partial \Omega.
 \end{align*}
 We set $(t_{\mathbf{f}}^0,x_{\mathbf{f}}^0,v_{\mathbf{f}}^0)=(0,x,v)$. Then, we define forward cycles, for $k\geq 0$,
 \Be\label{fcycle}
 	(t_{\mathbf{f}}^{k+1},x_{\mathbf{f}}^{k+1},v_{\mathbf{f}}^{k+1}) =(t_{\mathbf{f}}^k+t_{\mathbf{f}}(x_{\mathbf{f}}^k,v_{\mathbf{f}}^k),\xf^k-(\tf^{k+1}-\tf^k)\vf^k,R_{x_{\mathbf{f}}^{k+1}}v_{\mathbf{f}}^k).
 \Ee

 In our analysis, it is important to classify points in the phase space where $v \mapsto  x^k - (t^k- s) v^k$ is Lipschitz continuous. It turns out that such a map fails to be Lipschitz continuous whenever there exists $\ell \leq k $ such that $n(x^\ell) \cdot v^\ell =0$, namely the billiard trajectory hits the grazing set $\gamma_0$ in past. The major difficulty arises from the fact that the billiard trajectory emanating from the grazing set is not confined near the boundary but impact on the interior dynamics globally in general. As the first step to analyze this case, we categorize the grazing set into several subsets:
 \begin{definition}
 	We split the grazing set $\gamma_0$ into three parts: the convex grazing $\gamma_0^V$, concave grazing $\gamma_0^C$, and inflection grazing $\gamma_0^I$ of 
 	\begin{align} \label{grazing}
 		\begin{split}
 			\gamma_0^V&:=\{(x,v)\in \gamma_0: t_{\mathbf{b}}(x,v)=0 \textrm{ and } t_{\mathbf{b}}(x,-v)=0\}, \\
 			\gamma_0^C&:=\{(x,v) \in \gamma_0: t_{\mathbf{b}}(x,v) \neq 0 \textrm{ and } t_{\mathbf{b}}(x,-v)\neq0\}, \\ 
 			\gamma_0^I &:=\gamma_0^{I_+} \cup \gamma_0^{I_-}, 
 		\end{split}
 	\end{align}
 	where the inflection grazing $\gamma_0^I$ is further decomposed as 
 	\begin{align} \label{inflection +-}
 		\begin{split}
 			\gamma_0^{I_+}&:= \left\{ (x,v) \in \gamma_0:
 			\begin{array}{cc}
 			 t_{\mathbf{b}}(x,v) \neq 0,  \ t_{\mathbf{b}}(x,-v)=0,\\ \textrm{ and } \exists \; \delta>0 \textrm{ s.t. } x+sv \in \R^3\backslash \overline\Omega \  \textrm{ for } s \in (0,\delta)
 			 \end{array}
 			\right\},\\
 			\gamma_0^{I_-}&:= \left\{ (x,v) \in \gamma_0: 
 				\begin{array}{cc}
 			t_{\mathbf{b}}(x,v) = 0, \ t_{\mathbf{b}}(x,-v)\neq0, \\
 			\textrm{ and } \exists \; \delta>0 \textrm{ s.t. } x+sv \in \R^3\backslash \overline\Omega \  \textrm{ for } s \in (-\delta,0)
 			\end{array}
 			\right\}.
 		\end{split}
 	\end{align}
 	The inflection grazing $\gamma_0^I$ can be redefined as follows using the analytic function $\xi$ defined in \eqref{def domain}. 
 \end{definition}

 \begin{remark}\label{inflection point} (Inflection grazing) For an arbitrary boundary point $x \in \p\O$, let $\sigma(\tau,\varphi)=x$ in \eqref{sigma para}. The fact $(\sigma(\tau,\varphi),v)\in \gamma_0^I$ is equivalent to the following: $(\sigma(\tau,\varphi),v)\in \gamma_0$ and  
 	there exists $\delta>0$ such that
 	$$
 	\xi(\sigma(\tau,\varphi)+sv)>0 \ \text{ and } \ \xi(\sigma(\tau,\varphi)-sv)<0 \ \text{ for all } \  s\in (0,\delta).$$
 	In this case, $(\sigma(\tau,\varphi),v) \in \gamma_0^{I_+}$ and $(\sigma(\tau,\varphi),-v) \in \gamma_0^{I_-}$. 
 \end{remark}
 
 \begin{remark}\label{cycles_not_defined}Due to our special convention $\sup \emptyset=0$, either $t_{\mathbf{b}}(x,v)$ or $t_{\mathbf{f}} (x,v)$ could be zero so that the forward or backward cycles can stop. There are two cases:
 	\begin{itemize}
 		\item[1.] If $(x,v)\in \gamma_0^{V}$, then $t_{\mathbf{b}}(x,v) = t_{\mathbf{f}}(x,v)$. In this case, $t^\ell (t,x,v) = t$ and $t^\ell _{\mathbf{f}}(t,x,v) = t$ for all $\ell \in \mathbb{N}$, and hence $x^\ell (x,v) =x  = x^\ell _{\mathbf{f}}(x,v) $ and $v^\ell (x,v) =v = v^\ell _{\mathbf{f}}(x,v) $  for all $\ell$. 
 		\item[2.]  If $(x,v)\in \gamma_0^{I+}$ (resp, $(x,v)\in \gamma_0^{I-}$), then $t_\mathbf{f}(x,v) =0$ (resp, $t_\mathbf{b}(x,v) =0$) and hence $ x^\ell _{\mathbf{f}}(x,v) =x$ and $ v^\ell _{\mathbf{f}}(x,v) =v$ (resp, $ x^\ell  (x,v) =x$ and $ v^\ell  (x,v) =v$) for all $\ell \in \mathbb{N}$. 
 	\end{itemize}
 \end{remark}

 \begin{definition}[Billiard Trajectory]\label{def:XV}
 	Let $(t,x,v) \in [0, \infty) \times \overline{\O} \times \R^3$. 
 	We define a billiard trajectory (or generalized specular characteristics)
 	\Be\begin{split} \label{XV}
 		X(s;t,x,v)&= \sum_{\ell=0}^\infty \mathbf{1}_{s \in [t^{\ell +1} (t,x,v),t^\ell(t,x,v))} \big( x^\ell(x,v)-(t^\ell(t,x,v)-s)v^\ell (x,v) \big), \\
 		V(s;t,x,v)&=  \sum_{\ell=0}^\infty \mathbf{1}_{s \in [t^{\ell+1}(t,x,v),t^\ell(t,x,v))}  v^\ell(x,v).
 	\end{split}\Ee
 	Hence the billiard trajectory solves weakly, for $s\in(t^{\ell+1},t^\ell)$, 
 	\begin{align*}
 	\begin{split}
 		\frac{d}{ds} X(s;t,x,v)=V(s;t,x,v), \quad \frac{d}{ds} V(s;t,x,v)=0, \quad &   \\
 		X(s;t,x,v)=x, \quad V(s;t,x,v)=v , \quad &\text{for } s=t.
 	\end{split}
	\end{align*}
 \end{definition}
\begin{remark}
It is a standard application of the	Poincare recurrence theorem (\cite{Chernov}) to guarantee that 
$\sum_{\ell=0}^\infty |t^{\ell} (t,x,v) - t^{\ell+1} (t,x,v)| = \infty$ for almost every $(x,v) \in \O \times \R^3$ for each $t \geq s \geq 0$. Therefore \eqref{XV} is well-defined almost everywhere in the phase space. 
	\end{remark}

 \hide
 \Green{[move somewhere]]
 	\[
 	(X(s;t,x,v), V(s;t,x,v)) = (x,v),\quad \text{ for all $s\geq t$ (resp, $s\leq t$) if $(x,v)\in \gamma_0^{I+}$ (resp, $(x,v)\in \gamma_0^{I-}$). }
 	\]
 	Since our domain $\O$ is axial symmtric with general analytic convex cross-section, $(\sigma(\tau,\varphi), v) \in \p\O \times \{ \{T_{\sigma(\tau,\varphi)} - \sigma(\tau,\varphi)\} \cap \S^{2}\}$ which satisfies
 	\begin{align*}
 		\xi(\sigma(\tau,\varphi)+sv)=0, \quad \forall s\in (0,\delta)
 	\end{align*}
 	does not exist. Thus, for both $\gamma_0^{I_+}$ and $\gamma_0^{I_-}$, it does not count the case $x+s v \in \p \O$ for all $s \in (0,\delta)$. Note that backward (resp, forward) in time trajectory cannot belong to $\gamma_0^{I+}$ (resp, $\gamma_0^{I-}$) within finite bounces.} 
 
\unhide

\subsection{Main Theorem} 
Global well-posedness of the Boltzmann equation and the long-time behavior of solutions have been central subjects (e.g., Boltzmann's H-theorem) in mathematical kinetic theory. Unfortunately, any rigorous answer to the \textit{hard-sphere} Boltzmann equation of the \textit{specular reflection} at a boundary of general convex domains had not been provided until a recent date, and yet much is unknown for the general domains! 
	
	 In a famous 1977's announcement \cite{SA}, Shizuta-Asano claimed to construct a global-in-time solution of the Boltzmann equation in a general smooth convex domain with a specular reflection boundary, while no proof of their logic has been available yet. As a matter of fact, a real breakthrough happened after several decades. In 2010, Guo developed a new method, called $L^2$-$L^\infty$ framework, and first rigorously constructed a global-in-time solution to the Boltzmann equation with the specular reflection boundary condition and proved its convergence toward the global Maxwellian, when a domain is both \textit{real-analytic} and convex (\cite{Guo10}). In 2018, Kim-Lee removed a crucial real-analytic condition (\cite{KimLee}), in which they completely resolved a longstanding open problem after the famous unsettled 1977's announcement. They made a major advance in the study of billiard trajectory by discovering its generic relation to the billiard map.

\hide
Global wellposedness of the hard-sphere Boltzmann equation and a long time behavior of solutions have been extensively studied for various physical boundary conditions. However, for the specular reflction boundary condition, much is unknown for the general domains due to its singular nature and complex characteristics. In 2010, Guo first contructed a unique global-in-time solution to the Boltzmann equation with the specular reflection boundary condition, when a domain is both real-analytic and strictly convex (\cite{Guo10}). In 2018, Kim-Lee removes a crucial real-analytic condition (\cite{KimLee}), in which they resolves a longstanding open problem after the famous unsettled announcement of Shizuta-Asano \cite{SA} in 1977. 
\unhide

 \hide

introduced $L^{2}L^{\infty}$ bootstrap argument to derive low regularity $L^{\infty}$ solution which converges to the global Maxwellian exponentially for some smooth bounded domains with several types of boundary conditions. In the case of specular reflection case, however, unlike to diffuse reflection condition, analytic convexity assumption was imposed to the domain. The difficulty of specular reflection essentially comes from complexity of trajectory up to initial time. Analyticity assumption was removed in \cite{KimLee} by the first and third author. But the problem for general non-convex domain is still widely open. 

In the context of low regularity theory introduced above, regularity of the Boltzmann solution with boundary conditions is very natural question. {\color{blue} Write about regularity works : convex, BV, Holder etc...}    

As mentioned before, global well-posedness and asymptotic behavior with specular reflection is also very important problem. Specular reflection condition is more sensitive to geometric property of domains. This is mainly because, it is required to keep track the backward in time trajectory until $t=0$. 

\unhide

The same problem without the convex condition of domain is very challenging due to the presence of infinitely-many bouncing trajectories and interior grazing sets. 
In this question, there is only one result so far by Kim-Lee \cite{cylinder} in 2018, in which they construct a global solution and prove its asymptotic stability near the Maxwellian in cylindrical domains whose cross-section is an analytic and non-convex closed planar curve. This result is essentially two dimensional and it is crucial in the proof of control number of bounces and interior grazing sets. 

 The hard-sphere model of specular reflection boundary is particularly hard in the study of global behavior because the problem is very hyperbolic! If the boundary reflection is not specular but either diffuse reflection of inflow then the problem turns out to be dissipative even without the intermolecular collision \cite{KimVP, JJKim}. If the intermolecular collision is grazing (e.g., Landau equation), then the analysis is not suffered from infinitely many bounces as the hard-sphere model due to a regularity gain \cite{GHJO}. On the other hand, it is worth noting that the hard-sphere model of specular reflection boundary is yet the sole boundary problem that can be rigorously derived from the particle systems (e.g., Lanford's theory of Boltzmann-Grad limit) \cite{EGM:BG}.



 In this paper, we consider a 3D toroidal domain, whose cross-section is a general analytic convex curve. To the best of the authors' knowledge, this is the first result on the global solution of the Boltzmann in a generic 3D non-convex domain with the specular reflection boundary. 

\begin{theorem}\label{main theorem}
	Let $w(v) = (1+|v|)^{\beta}$ with $\beta > \frac{5}{2}$. Consider a general toroidal domain $\O \subset \R^3$ in Definition \ref{def:domain}. We assume $F_{0} = \mu + \sqrt{\mu} f_{0} \geq 0$ with $\mu$ in \eqref{def:gMax} satisfies a normalization condition:
\Be	\begin{split} \label{normalization}
		\iint_{\O\times \R^3} (F_0(x,v) - \mu) dxdv& =0,\\
  \iint_{\O \times \R^3} \vert v \vert^2 (F_0(x,v) - \mu) dxdv  &= 0,\\
  \iint_{\O \times \R^3} \{x\times \hat{z}\} \cdot v F_0(x,v) dxdv & = 0.
	\end{split}\Ee





	


	Then, there exists $0 < \delta \ll 1$ such that if 
	\[
	\|wf_{0}\|_{L^\infty(\O \times \R^3)} < \delta, 
	\]
	then the Boltzmann equation \eqref{Boltzmann} with the specular boundary condition \eqref{specular} has a unique global solution $F(t) = \mu + \sqrt{\mu}f(t) \geq 0$. Moreover, there exists $\lambda > 0$ such that
	\[	
	\sup_{t \geq 0} e^{\lambda t}\|wf(t)\|_{\infty} \lesssim \|wf_{0}\|_{\infty},
	\] 
	and the following conservation laws hold:
	\begin{align*} 
	&	\iint_{\O\times \R^3} (F(t,x,v) - \mu) dxdv = \iint_{\O \times \R^3} \{x\times \hat{z}\} \cdot v F(t,x,v) dxdv\\
	&	 =   \iint_{\O \times \R^3} \vert v \vert^2 (F(t,x,v) - \mu) dxdv = 0.  
	\end{align*}
\end{theorem}

\begin{remark}
	Using the method in \cite{YunKim}, our proof also works for a non-zero angular momentum case.  For the sake of simplicity, we only give a proof for the zero angular momentum case.
	\end{remark}

\begin{remark}
	Extending the theorem to a general non-convex domain beyond the category of Definition \ref{def:domain} is challenging. We discuss two types of domains and anticipated difficulties. 
	\begin{itemize}
\item {\it No structural condition}: In this paper, axis-symmetric structure is used crucially several ways, in particular when we prove Proposition \ref{mB zero}. The proposition says that ``Sticky Grazing'' points, which was introduced in \cite{cylinder}, do not exist. (See Figure 2 and Lemma 8 in \cite{cylinder}.) Without any special symmetric structures, it seems hard to control the number (or the measure) of such Sticky Grazing points. The presence of inflection points is another major difficulty, in which the billiard map is not a union of straight lines. In general, billiard trajectory emanating from such inflection points could be large set in the velocity. 

\item  {\it No real analytic condition}: From the example of Halpern's billiard \cite{Halpern}, infinitely many bounces would happen in finite travel length ,i.e., accumulation of bouncing points. This is totally local phenomenon, and without real analytic condition of the boundary, this can happen generically.  
	\end{itemize}
	
\end{remark}

\hide

\indent Global well-posedness and asymptotic behavior of the Boltzmann equation with general large data may be one of the single most famous problem in the Boltzmann theory. Although there were several seminal results \cite{Diperna-Lion, Desvillettes-Villani, Sylestre}, complete anwer the for question is still open. (See \cite{Sylestre note}). In the perturbation framework, however, there have been lots of complete results in this direction. We refer classical works of Ukai \cite{Ukai} and Guo \cite{VPB, VMB} in which the authors used high order regularity setup. Inspired these works, many fruitful results have been studied in $\T^{3}_{x}$ or $\R^{3}_{x}$. We refer {\color{blue} \cite{Several works...}.}

it bounces back with the opposite normal velocity and the same tangential velocity, as a billiard

one of the most basic conditions, 

Various important phenomena occur when gas particles in- teract with the boundary, such as the formation and propagation of singularities ([15, 10, 11]). In the presence of the boundary, a kinetic equation has to be sup- plemented with boundary conditions modeling the interaction between the particles and the boundary. Among other boundary conditions (See [3, 13]), in this paper, we focus on one of the most basic conditions, a so-called specular reflection boundary condition. This mathematical model takes into account a case that if a gas particle hits the boundary, then it bounces back with the opposite normal velocity and the same tangential velocity, as a billiard:

e.g. dilute gases passing objects and a plasma
inside tokamak devices,

	 \unhide
	\hide
	{\color{blue}	
	(Cylindrical) Despite extensive developments in the study of the Boltzmann theory, many basic boundary problems, especially regarding the specular reflection BC with general domains, have remained open. In 1977, in \cite{SA}, Shizuta and Asano  announced the global existence of the Boltzmann equation with the specular boundary condition in smooth convex domain without a complete proof. The first mathematical proof of such problem was given by Guo in \cite{Guo10}, but with a strong extra assumption that the boundary should be a level set of a real analytic function. Very recently the authors proved the unique existence and asymptotic stability of the specular boundary problem for general smooth convex domains (with or without external potential) in \cite{KimLee}, using triple iteration method and geometric decomposition of particle
	trajectories. This marks a complete resolution of a 40-years open question after \cite{SA}.
	
	Meanwhile, there were even fewer results for general non-convex domains with the specular boundary condition. An asymptotic stability of the global Maxwellian is established in \cite{DV}, provided certain \textit{a-priori} strong Sobolev estimates can be verified. However, such strong estimates seem to fail especially when the domain is non-convex (\cite{GKTT1, GKTT2, Kim11}). Actually the authors believe that the solution cannot be in $C^1$ (but in $C^{0,\alpha}$) when the domain is non-convex because if a non-convex grazing interferes, the bouncing point is not differentiable. We refer forthcoming paper \cite{Holder}. To the best of our knowledge, our work is the first result on the global well-posedness and decay toward Maxwellian results for any kind of non-convex domains with the specular boundary condition! One of the intrinsic difficulties of a non-convex domain problem is the (billiard) trajectory is very complicated to control (e.g. infinite bouncing, grazing).  
	The problems of general smooth non-convex domains or three-dimensional non-convex domains are still open.   \\
	
	------------------------------------\\
	
	(Large amp) In contrast to the DiPerna-Lions' large-data theory for general initial data mentioned before, the global well-posedness has been well understood in the perturbation framework and many results exist in this direction. In particular, the problem was first solved by Ukai \cite{Ukai} in case when the space domain is a periodic box and initial data are chosen as $F_0(x,v)=\mu+\sqrt{\mu}f_0(x,v)\geq 0$ such that 
	$\sup_{v}\langle v\rangle^\beta \|f_0(\cdot,v)\|_{H^N_x}$
	is sufficiently small for suitably large $\beta$ and $N$. Later, Shizuta-Asano \cite{SA} announced an extension to general smooth bounded convex domains for specular reflection boundary condition without much details of the proof. For general bounded convex domains with specular reflection boundary condition, Guo \cite{Guo} gave a complete proof via a novel $L^2$-$L^\infty$ approach to construct the global unique solution provided that $\sup_{x,v}\langle v\rangle^\beta |f_0(x,v)|$ is sufficiently small for suitably large $\beta$ and the boundary of domain is a level set of a real analytic function. Notice that besides specular reflection boundary condition, other types of boundary conditions such as the inflow, bounce-back and diffuse reflection were also treated in \cite{Guo}. In case of the specular reflection boundary condition, Kim-Lee \cite{KL} removed the analyticity assumption in \cite{Guo} and replaced it by the general $C^3$ regularity for a bounded convex domain. Recently, Guo-Zhou \cite{GZ} constructed the global-in-time Boltzmann solution in the diffusive limit for the general Maxwell boundary condition with the full range of accommodation coefficients even in the case when the space domain is not necessarily convex and the boundary admits $C^2$ regularity only; see also Briant-Guo \cite{BG} on the construction of global close-to-equilibrium solutions for the finite Knudsen number and for the Maxwell boundary condition of the partial range of accommodation coefficients without the convexity and analyticity assumptions on $\Omega$. We also refer \cite{EGKM, EGKM-hy, GKTT, GL, KL-nc} for other recent results about non-convex domain, discontinuity, regularity,  and hyrodynamic limit results for the Boltzmann boundary problems.

	On the other hand, in those aforementioned results, the initial perturbation $f_0(x,v)$ has to have a small oscillation in space variables due to smallness of $L^\infty$ norm. An extension of \cite{Guo} to a class of large amplitude initial data in the sense that the $L^\infty$ norm was allowed to be arbitrarily large but the $L^p$ norm for some $p<\infty$ is suitably small was recently made in Duan-Huang-Wang-Yang \cite{DHWY17} and Duan-Wang \cite{DW} either for the whole space and periodic box or for the bounded domain with diffusion reflection boundary condition, respectively. Inspired by \cite{DHWY17,DW,KL}, we are going to explore in the paper the possibility of constructing the global Boltzmann solution not only for the specular reflection boundary condition with the $C^3$ boundary regularity but also for a class of large amplitude initial data. It turns out that the direct combination of those approaches in \cite{DHWY17,DW,KL} does not work well. We will have to figure out new difficulties and strategies to treat the problem under consideration.  \\
} 
\unhide
%

\subsection{Billiard Trajectory}
\begin{definition} (Velocity set $\VN$ and winding number)
For fixed $N\geq 1 $, 
\begin{equation} \label{V^N}
	\mathbb{V}^N:= \{ v \in \R^3 : \frac{1}{N} \leq \vert v \vert \leq N\}. 
\end{equation}
When $(X(s;t,x,v), V(s;t,x,v))$ travels inside of $\O$ under the specular reflection boundary condition up to total travel length $L>0$, we define a uniformly finite upper bound for the winding number around the $z$-axis: 
\begin{align}\label{winding number}
\mathbf{w} := \left(\left[\frac{L}{2\pi r}\right]+1\right), 
\end{align}
where $r=\gamma_{1}(\lambda_{*})$ is the minimum distance between $z$-axis and $\p\O$. Here $[\cdot]$ is the Gauss function. 
\end{definition}
\begin{remark} \label{LNT}
We often consider only a normalized speed $\frac{v}{|v|}$ with fixed travelling length $L$. We can easily study the trajectory of unscaled $v\in \VN$, by setting $L = NT$ (in Theorem \ref{bad set thm}) where $NT$ is the maximal travel length of the trajectory for fixed time interval $[0,T]$ and $v\in \mathbb{V}^N$. See Theorem \ref{bad set thm}.  
Also, note that the smallest perimeter around $z$-axis of the domain $\O$ is $2\pi r$. Hence, the winding number $\mathbf{w}$ of \eqref{winding number} is the maximum number that a trajectory can rotate $z$-axis during whole travel length $L$.  
\end{remark}

Since a trajectory can rotate around the $z$-axis multiple times, it is important to distinguish two trajectory on the same cross-section depending on how many times the trajectory rotated around the $z$-axis. So, it is natural to extend $\varphi$ into $\R$.
\begin{definition} \label{varphi cross section}
Let $\varphi \in \R$. We define $\varphi$ cross-section $S_{\varphi}$ as follows
\begin{align} \label{cross section}
	S_{\varphi}:=\{x \in \O : \phi(x)=\varphi \pmod{2\pi} \}, 
\end{align}
where $\phi(x)$ is the angle between the $x$-axis and the projected point $(x_1,x_2,0)$ of $x=(x_1,x_2,x_3)$. In other words, the map $\phi:\overline{\O}\mapsto [0,2\pi)$ satisfies
\begin{align*}
	\phi(x)=\cos^{-1} \left(\frac{(x_1,x_2,0)}{\sqrt{x_1^2+x_2^2}}\cdot (1,0,0)\right)=\cos^{-1}\left(\frac{x_1}{\sqrt{x_1^2+x_2^2}}\right). 
\end{align*}
We define the closure $\overline{S_\varphi}$ of $S_{\varphi}$ as
\begin{align*}
	\overline{S_\varphi}:=\{x\in \overline{\O}: \phi(x)=\varphi \pmod{2\pi}\},
\end{align*} 
and the boundary $\p S_{\varphi}$ of $S_{\varphi}$
\begin{align*}
	\p S_{\varphi}:=\{x\in \p\O : \phi(x)=\varphi \pmod{2\pi}\}.
\end{align*}
We also define a function $\hat{\varphi}=\hat{\varphi}(\varphi):(-2\pi\mathbf{w},0]\rightarrow \S^2$ as a direction which is perpendicular to $S_{\varphi}$ with positive orientation w.r.t the $z$-axis. Since the cross section $S_0$ is in the $xz$-plane, $\hat{\varphi}(0)=(0,1,0)$ and 
\begin{align*}
	\hat{\varphi}(\varphi)=\begin{pmatrix}
	\cos \varphi &-\sin \varphi & 0 \\ \sin \varphi & \cos \varphi & 0 \\ 0 & 0 & 1
	\end{pmatrix} \hat{\varphi}(0). 
\end{align*} 
\end{definition}

\begin{remark}
	In this paper, we carefully distinguish the cross-sections $S_{\varphi}$ and $S_{\varphi+2\pi}$. For fixed $x\in \O$, let us choose a $\varphi\in (-2\pi\mathbf{w},0]$ such that $\phi(x)=\varphi \pmod{2\pi}$. Then for $(x,\varphi,v)\in \overline{\O}\times (-2\pi\mathbf{w},0]\times \S^2$, we will define the arrival time function $s(x,\varphi,v)$ in Lemma \ref{lem time s} which represents time taking from $x\in S_{\varphi}$ with velocity $v$ to $S_0$ along the trajectory\eqref{XV}, rotating around the $z$-axis multiple times if necessary. Then, the arrival time function is uniquely determined with respect to $(x,\varphi,v)$, instead of $(x,v)$. 
\end{remark}

\begin{definition}
\noindent {(Definition for $x,y,\varphi$ coordinates)}  
We define a new coordinate on fixed $S_{0}$, $\varphi=0$ cross-section. We use the coordinate for $v$:
\Be \label{coord}
	\begin{pmatrix}
	v_x	  \\ v_{\varphi} \\ v_{y} 
\end{pmatrix}.
\Ee
Here, $(v_{x}, v_y)$ are perpendicular coordinates in the cross-section $S_0$ and $v_{\varphi}$ is coordinate of $\hat{\varphi}(0)$ direction which is perpendicular (positive orientation WLOG) to $S_{0}$ cross-section. See below Figure \ref{solid torus}.   
\end{definition}

\begin{figure}[h]
\centering
\includegraphics[width=12cm]{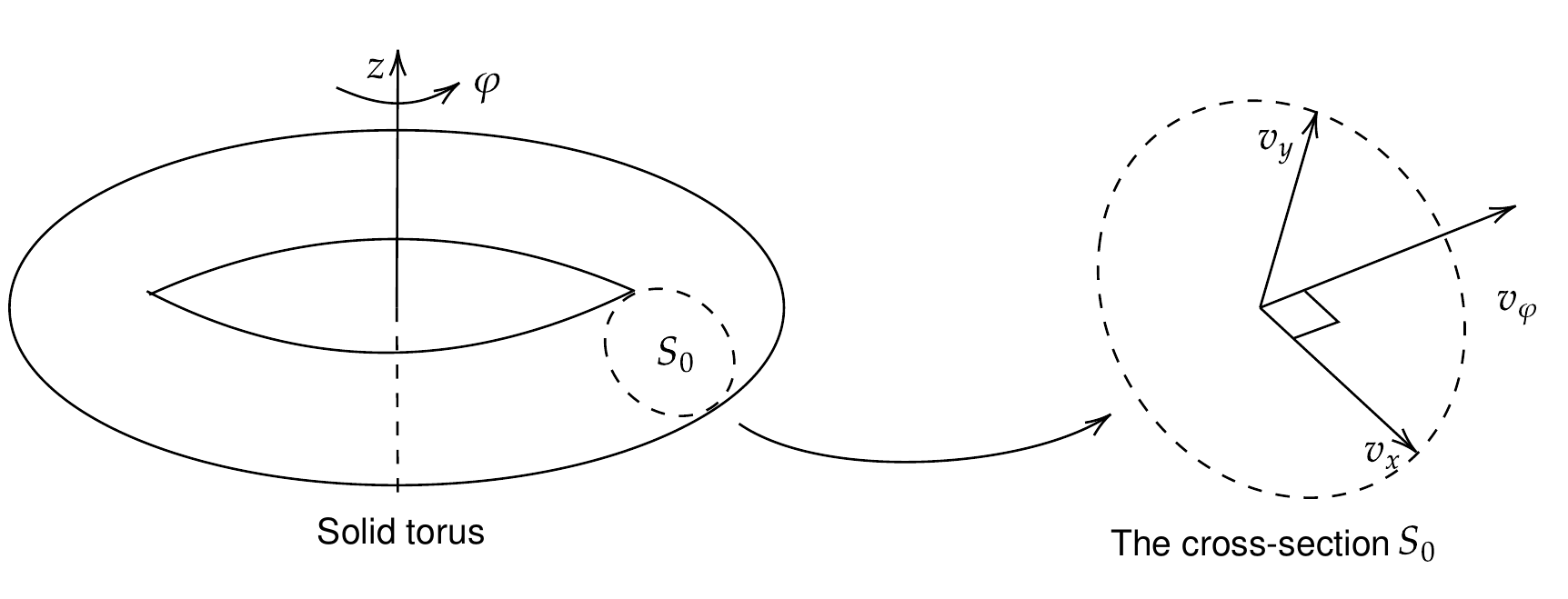}
\caption{Solid torus and the cross section}
\label{solid torus}
\end{figure}
In the figure above, the left-hand side represents the $\varphi=0$ cross section $S_0$ in \eqref{cross section}. On the right-hand side, we express the new coordinates as $v=(v_x,v_{\varphi},v_y)$. Here, $v_x$ and $v_y$ are perpendicular to each other, and $v_\varphi$-direction is perpendicular to $S_0$, as shown.   

\section{Ideas and Scheme to overcome the Main Difficulties}The major difficulty arises as the map $v^{\prime} \mapsto X(s_1;s,X(s;t,x,v),v^{\prime})$ can fail to be Lipschitz in a positive measure subset when the domain is not convex, in general. To overcome this difficulty, we first carefully study the inflection points of $\p\O$. In \cite{cylinder}, the billiard trajectory is two-dimensional essentially in the cylindrical domain. As the boundary is a real-analytic curve, inflection points (and corresponding tangential direction) are at most finitely many. On the other hand, in our case of Definition \ref{def:domain}, structure of inflection points and directions are complex. 
In fact, every point of the inner part (where $\tau_{1,*} < \tau < \tau_{2,*}$) is a saddle point, and inflection grazing always exists at all of these points. Therefore, in contrast to the cylindrical case of \cite{cylinder}, {\bf there are uncountably many inflection grazing in $\O$}. Moreover, it is impossible to keep track of the chaotic billiard trajectory emanating from such inflection grazing in $\O$. 
 We resolve this difficulty by classifying and measuring the size of trajectories of the inflection grazing using axial symmetric property and analyticity crucially. We explain the scheme for the trajectory with positive angular momentum without loss of generality: 
 \\ 

 {\bf Step 1.}  (Study the structure of inflection grazing and nonexistence of converging bounces.) In Section 3, we show all the point of  the inner part, i.e., $\tau_{1,*} < \tau < \tau_{2,*}$ (See Figure \ref{analytic convex curve_figure}), are saddle points with exactly two directions $I^{1}, I^{2}$, of which normal curvatures vanish. For another directions, we have positive or negative normal curvatures. (See Figure \ref{inflec-directions}.) Moreover, it is important to distinguish whether an inflection grazing is $\gamma_0^{I_{+}}$ or $\gamma_0^{I_{-}}$ as the backward in time trajectory must stop at $\gamma_0^{I_{-}}$ and a particle trajectory cannot propagate forward in time at $\gamma_0^{I_{+}}$ by definition \eqref{inflection +-}. We classify such cases using the intersection of $\p\O$ and tangential plane $T_{p}(\p\O)$ at $p\in\p\O$, which is determined by local convexity/concavity at the point. In Lemma \ref{inflection}, using analyticity and saddle structure of the boundary, we will prove that every saddle point $p\in\p\O$ has one $\gamma_0^{I_{-}}$ direction and one $\gamma_0^{I_{+}}$ direction on the tangential plane $T_{p}(\p\O)$ away from some finite points $\gamma(\tau)$ of $\tau \in (\tau_{1,*}, \tau_{2,*})$.   

For simplicity, let us consider standard solid torus (of which cross-section is a circle) with inner radius $r$ and revolving radius $R$. Then, boundary $\p\O$ can be parametrized by 
\begin{align*}
	\sigma(\theta,\varphi)=((R+r\cos \theta)\cos \varphi,(R+r\cos \theta)\sin \varphi,r\sin \theta), \quad 0\leq \theta <2\pi, \; 0\leq \varphi<2\pi,
\end{align*}
where $\theta$ represents rotation around the cross-section circle and $\varphi$ represents rotation around $z$-axis of revolution. Figure \ref{Intersection} below visualizes the intersection of tangent plane $T_p(\partial \O)$ at $p=\sigma(\theta,0)\in \p\O$ and the torus $(\p\O)$ for each $\theta=\frac{\pi}{2}+\frac{\pi}{12}, \frac{\pi}{2}+\frac{\pi}{3}$, and $\pi$. Except for $\theta=\pi$, there are two directions $I^1(\theta,\varphi)\in \gamma_0^{I_+}$ and $I^2(\theta,\varphi)\in \gamma_0^{I_-}$ on the tangent plane $T_p(\p\O)$. (See (A) and (B) of Figure \ref{Intersection} and Figure \ref{inflec-directions}.) The point $\theta=\pi$ is the case where there is no inflection grazing. ((C) of Figure \ref{Intersection}.)

\begin{figure}[h]
\centering
\begin{subfigure}[b]{0.3\textwidth}
\includegraphics[width=5cm]{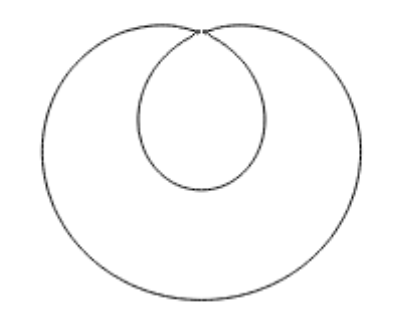}
\caption{$\theta=\frac{\pi}{2}+\frac{\pi}{12}$}
\end{subfigure}
\quad \quad \quad
\begin{subfigure}[b]{0.25\textwidth}
\includegraphics[width=4cm]{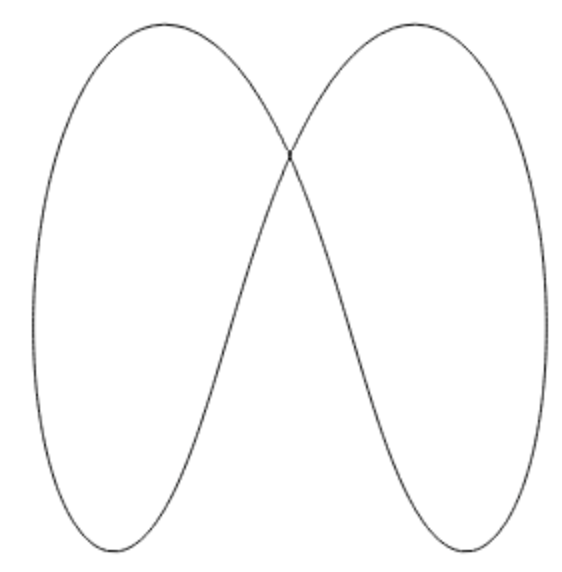}
\caption{$\theta=\frac{\pi}{2}+\frac{\pi}{3}$}
\end{subfigure}
\quad \quad \quad
\begin{subfigure}[b]{0.24\textwidth}
\includegraphics[width=4cm]{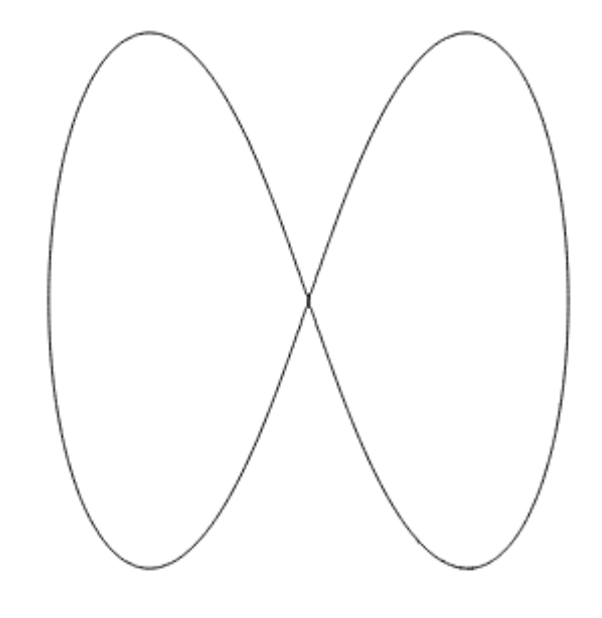}
\caption{$\theta=\pi$}
\end{subfigure}
\caption{Intersection between solid torus ($\p\O$) and tangent plane $T_p$ at $p=\sigma(\theta,0) \; (\varphi=0)$}
\label{Intersection}
\end{figure}

\hide
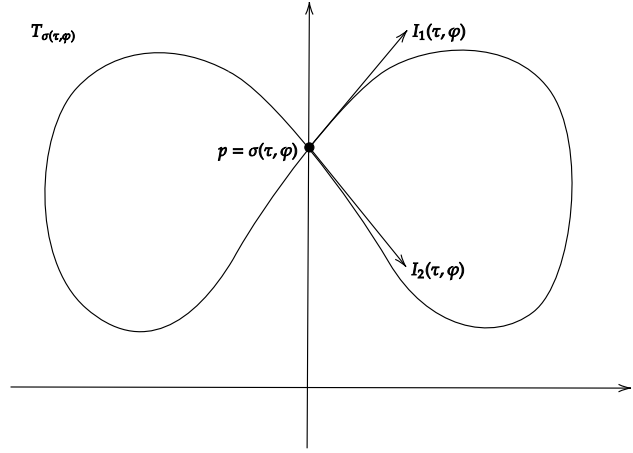
\begin{figure}
\begin{center}
\tikzset{every picture/.style={line width=0.75pt}} 
\begin{tikzpicture}[x=0.75pt,y=0.75pt,yscale=-0.8,xscale=0.8]

\draw    (275,606) .. controls (309,564) and (359,532) .. (408,532) .. controls (457,532) and (478,588) .. (477,641) .. controls (476,694) and (466,719) .. (421,720) .. controls (376,721) and (309,645) .. (275,606) .. controls (241,567) and (203,540) .. (159,541) .. controls (115,542) and (82,605) .. (83,648) .. controls (84,691) and (103,723) .. (147,724) .. controls (191,725) and (236,655) .. (275,606) -- cycle ;
\draw    (275,606) -- (326.93,523.69) ;
\draw [shift={(328,522)}, rotate = 122.25] [color={rgb, 255:red, 0; green, 0; blue, 0 }  ][line width=0.75]    (10.93,-3.29) .. controls (6.95,-1.4) and (3.31,-0.3) .. (0,0) .. controls (3.31,0.3) and (6.95,1.4) .. (10.93,3.29)   ;
\draw    (275,606) -- (326.86,680.36) ;
\draw [shift={(328,682)}, rotate = 235.11] [color={rgb, 255:red, 0; green, 0; blue, 0 }  ][line width=0.75]    (10.93,-3.29) .. controls (6.95,-1.4) and (3.31,-0.3) .. (0,0) .. controls (3.31,0.3) and (6.95,1.4) .. (10.93,3.29)   ;

\draw (282,685) node [anchor=north west][inner sep=0.75pt]    {$I^{2}( \theta ,\varphi )\in \gamma_0^{I_-}$};
\draw (298,499) node [anchor=north west][inner sep=0.75pt]    {$I^{1}( \theta ,\varphi )\in \gamma_0^{I_+}$};
\draw (191,595) node [anchor=north west][inner sep=0.75pt]    {$p=\sigma ( \theta ,\varphi )$};
\end{tikzpicture}
\end{center}
	\caption{General intersection structure of $T_p \cap \p\O$ and inflection directions} \label{inflec-directions}
\end{figure}
\unhide

\begin{figure}[h]
	\centering
	\includegraphics[width=10cm]{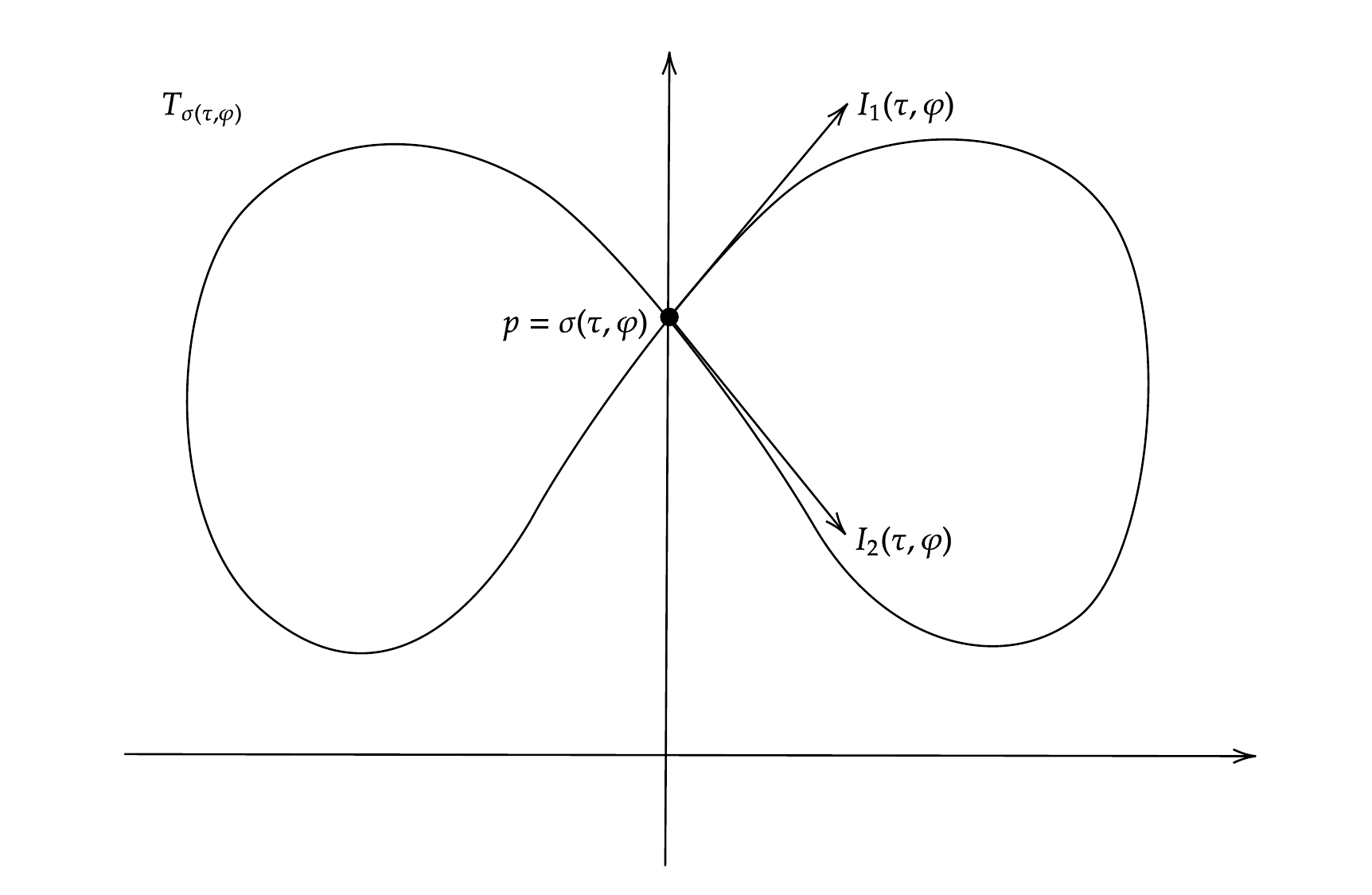}
	\caption{General intersection structure of $T_p \cap \p\O$ and inflection directions}
	\label{inflec-directions}
\end{figure}



Without loss of generality, let us fix a point $x\in \overline{\O}$ and consider velocity $v$ with unit speed $|v|=1$. We consider \eqref{XV}, the backward in time trajectory under specular reflection BC from $(x,v)$ and will study how many bounces are available up to finite travel length $L<\infty$. 
If there are infinitely many bounces in length $L$, bouncing points must be convergent. Such a phenomenon cannot happen in the outer part (where $\tau_{2,*} < \tau < \tau_{1,*}$) by uniform convexity (See velocity lemma : Lemma 1 of \cite{Guo10} or Lemma 2.7 of \cite{KimLee}, for example). Therefore we can assume the trajectory converges to a point in the inner part. See Figure \ref{dvarphidtau}. From axial symmetric structure, specular reflection, and convexity of the cross-section, we can derive (See Lemma \ref{difference})
\begin{align*}
	\Delta \tau_{i+1}-\Delta \tau_{i}= \mathcal{O}(\Delta \tau_{i+1}^2) +\mathcal{O}(\Delta \tau_{i}^2) +\mathcal{O}(\Delta \varphi_{i+1}^2) +\mathcal{O}(\Delta \varphi_{i}^2), \quad \text{for } \; i=1,2,\cdots,
\end{align*}
 for sufficiently small difference $\Delta\theta$ and $\Delta\tau$ except some cases : for example, if $\tau$ is too close to $\tau_{1,*}$ or $\tau_{2,*}$ of Figure \ref{analytic convex curve}, we remove inflection directions uniformly as in Figure \ref{remark 1}. If $p\in\p\O$ is an inflection point on the intersection $T_{p}(\p\O)\cap \p\O$, we also remove inflection direction as in Figure \ref{remark 2}. (Also see Figure \eqref{fig5:b}.) 
 The above estimate implies that $\varphi$ (axisymmetric) directional perturbation affects in higher order with small effect. Or equivalently, we can say that for converging bouncing points (small $\Delta\varphi$ and $\Delta\tau$), $\varphi$-directional perturbation is too weak to change the behavior of the $\tau$-directional property. 
However, this is a contradiction, because if $\Delta\tau$ is dominant, the trajectory cannot converge by the convexity of the cross-section with help of the velocity lemma again. 
In fact, we obtain $\sum_{j=1}^{N}\Delta \tau_{i}\rightarrow\infty$, and hence infinite bounce cannot happen. See Lemma \ref{no infty bounce} for a precise statement. \\

\hide
\begin{figure}
\begin{center}
\tikzset{every picture/.style={line width=0.75pt}} 
\begin{tikzpicture}[x=0.75pt,y=0.75pt,yscale=-1,xscale=1]

\draw    (64,822) .. controls (94,771) and (234,691) .. (311,706) ;
\draw    (304,821) .. controls (273.39,772.53) and (130.56,696.49) .. (52,710.75) ;
\draw    (179,669) -- (233,669) ;
\draw    (179,736) -- (233,669) ;
\draw    (179,736) -- (179,669) ;

\draw (233,660) node [anchor=north west][inner sep=0.75pt]    {$P_{i+1} =\sigma ( \tau _{i+1} ,\varphi _{i+1})$};
\draw (95,725) node [anchor=north west][inner sep=0.75pt]    {$P_{i} =\sigma ( \tau _{i} ,\varphi _{i})$};
\draw (186,635) node [anchor=north west][inner sep=0.75pt]    {$ \begin{array}{l}
\Delta \varphi _{i}\\
\ \ \wr 
\end{array}$};
\draw (141,690) node [anchor=north west][inner sep=0.75pt]    {$\Delta \tau _{i}  \sim $};
\end{tikzpicture}
\end{center}
\caption{Perturbation $\Delta\tau_{i}$ and $\Delta\varphi_{i}$} \label{dvarphidtau}
\end{figure}
\unhide

\begin{figure}[h]
	\centering
	\includegraphics[width=13cm]{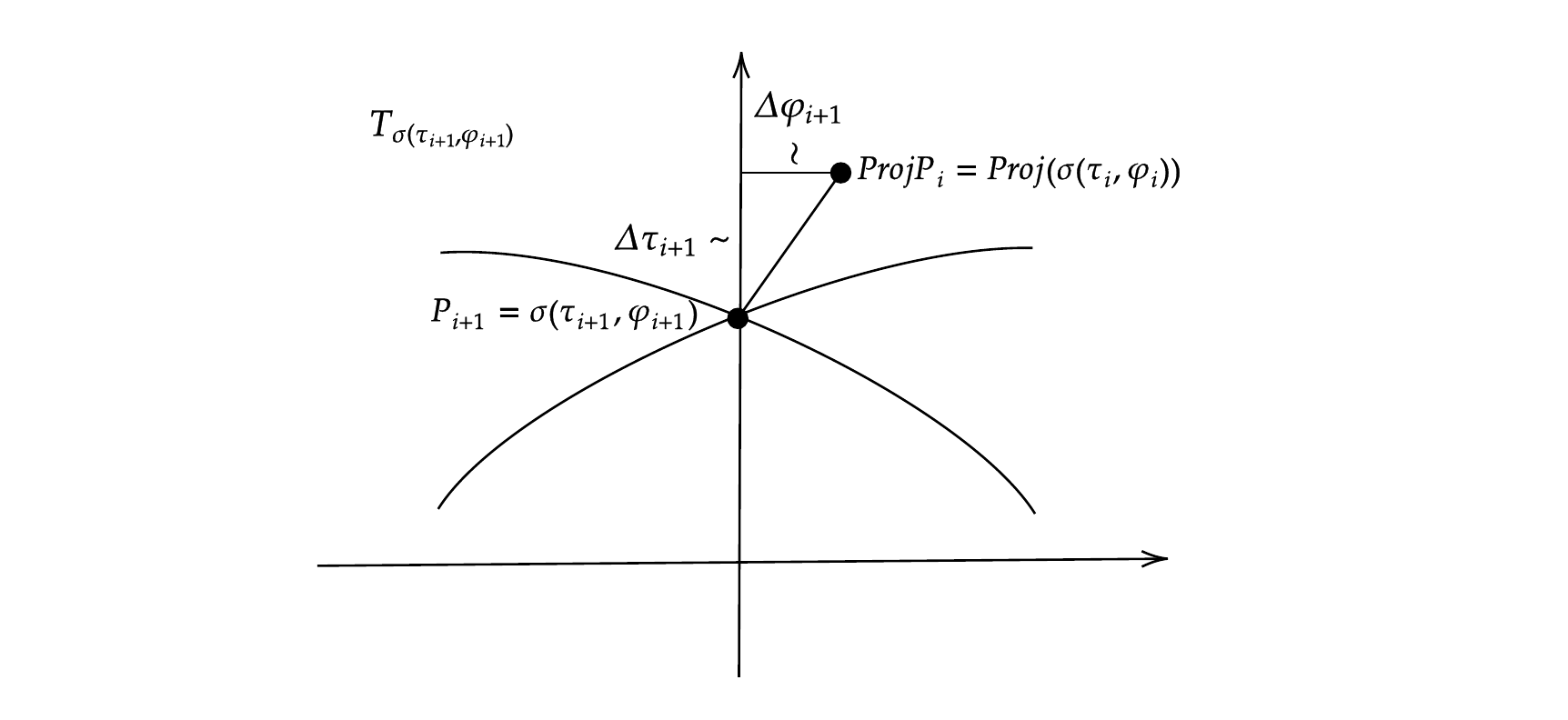}
	\caption{Perturbation $\Delta\tau_{i}$ and $\Delta\varphi_{i}$} 
	\label{dvarphidtau}
\end{figure}

{\bf Step 2.} (Uniform number of bounce and small neighborhood of bad directions.) In general toroidal domains, billiard trajectory \eqref{XV} is extremely chaotic with unpredictable encounter with inflection grazing. We will rely on analyticity, axis-symmetry, convex cross-section of $\p\O$, and compactness argument to derive some properties of the linear trajectory in $\O$. Of course, we note that the trajectory $(X(s;t,x,v), V(s;t,x,v))$ is not analytic in general if the trajectory belongs to $\gamma_0^{C}$ or $\gamma_0^{I}$. First, in previous {\bf Step 1}, we introduced some bad direction sets which have something to do with the possibility of infinite bounce in finite travel length, e.g., inflection directions at some points on $\p\O$. By conservation of angular momentum, such velocities look like a ring on $\S^{2}$ centered at $x\in \overline{\O}$ as in \eqref{small_ang} of Definition \ref{def ring}. Since we can confine such bad directions in some small sets on $\S^{2}$, we can define an open cover $\bigcup_{i=1}^{l_{R}} \{B(x_{i}^{R}, r_{i}^{R})\cap \overline{S_0}\}$ for $\overline{S_0}$ and the corresponding small bad directions $\mathcal{O}_{i}^{R}\subset \S^{2}$ such that the backward in time trajectory from $(x,v)$ avoids becoming bad phase points which described in {\bf Step 1},
\[
	 \text{ if $x\in B(x_{i}^{R}, r_{i}^{R})\cap \overline{S_0}$ and $v\notin \bigcap_{j\in I_x^R}\mathcal{O}_{j}^{R}$},
\]
where $I_x^R$ defined in \eqref{index_R}.\\

Using $ \{B(x_{i}^{R}, r_{i}^{R})\cap\overline{S_0}\}_{i=1}^{l_{R}}$ and $\{\mathcal{O}_{i}^{R}\}_{i=1}^{l_{R}}$, we carefully construct compact set $X^{\varepsilon}$ from which the backward in time trajectory for given finite length is ``good". (See Lemma \ref{lem Xe} for precise statement.) Unfortunately, however, $X^{\varepsilon}$ still contains some phase points whose trajectory becomes inflection grazing or concave grazing. For these cases, the backward in time trajectory is not analytic. In fact, we can expect H\"older regularity at most. So we define
\Be \notag
\begin{split}
	B^{\varepsilon}_{L} &:= \left \{ (x,v)\in X^{\varepsilon} : \exists k \in \mathbb{N} : (x^{k}, v^{k-1}(x,v))\in \gamma_{0}^{I-}, \ \text{and} \  \sum_{j=1}^k \vert x^j(x,v)-x^{j-1}(x,v) \vert \leq L \right \}, \\
\end{split}	
\Ee
which is proved to be closed set. By compactness of $B_{L}^{\varepsilon}$ and continuity of trajectory, we claim that there exist $K_{\varepsilon, L} < \infty$ such that the backward in time trajectory from $(x,v) \in B_{L}^{\varepsilon}$ does not generate inflection grazing after its $K_{\varepsilon, L}$ bounces up to length $L$. Motivated by this fact, we define the following sets 
\begin{align*} 
	(G_1)_x&:= \{v \in \S^2 : (x,v)\in X^{\varepsilon}, \ (x^1(x,v),v^0(x,v))\notin \gamma_0^C \cup \gamma_0^I\},\\ 
	(B_1)_x&:=\{v \in \S^2: (x,v)\in X^{\varepsilon}, \ (x^1(x,v),v^0(x,v))\in \gamma_0^C \cup \gamma_0^I\},\\ 
	&\vdots \notag \\
	(G_j)_x&:=\{v \in (G_{j-1})_x: (x^j(x,v),v^{j-1}(x,v))\notin \gamma_0^C \cup \gamma_0^I\}, \notag \\ 
	(B_j)_x&:=\{v \in (G_{j-1})_x: (x^j(x,v),v^{j-1}(x,v))\in \gamma_0^C \cup \gamma_0^I\}, \notag \end{align*}
for all $1\leq j \leq K_{\varepsilon, L}$ for a fixed point $x\in \overline{\O}$ . We note that the backward in time trajectory does not graze until its $j$-th bounce and is analytic if $(x,v) \in (G_{j})_{x}$. And then combining all $\{(B_{j})_{x}\}_{j=1}^{K}$ and using continuity, we construct a finite open cover $\{ B(x_{i}^{\mathcal{K}}, r_{i}^{\mathcal{K}})\cap\overline{S_0} \}_{i=1}^{l_{\mathcal{K}}}$ and corresponding bad direction sets $\{\mathcal{O}_{i}^{\mathcal{K}}\}_{i=1}^{l_{\mathcal{K}}}$ such that the backward in times trajectory from $(x,v)$ avoids being $\gamma_0^{I}$ or $\gamma_0^{C}$
\[
	\text{ if $x\in B(x_{i}^{\mathcal{K}}, r_{i}^{\mathcal{K}})\cap \overline{S_0}$ and $v\notin \bigcap_{j \in I_x^{\mathcal{K}}}\mathcal{O}_{j}^{\mathcal{K}}$},
\]
where $I_x^{\mathcal{K}}$ is defined in \eqref{index_K}. As we constructed $X^{\varepsilon}$ before, we can also construct an open set $\mathcal{K}$ and compact set $\mathcal{K}^{c}$ so that for $\mathcal{K}^{c}$, trajectory does not contain converging bounces and does not generate inflection grazing up to length $L$. And using compactness and continuity again, we obtain the uniform number of bounces $M_{\varepsilon, L}$ away from $\mathcal{K}$. See Lemma \ref{bounce K}. \\

{\bf Step 3.} ($(B_j)_{x}$ has measure zero) Roughly speaking, each $\mathcal{O}_{i}^{\mathcal{K}}$ must contain all $\bigcup_{x\in B(x_{i}^{\mathcal{K}}, r_{i}^{\mathcal{K}})}\bigcup_{j=1}^{M} (B_{j})_{x}$. 
Therefore, we should claim $\mathfrak{m}_{2}((B_{j})_{x})=0$ to conclude that $\mathcal{O}_{i}^{\mathcal{K}}$ can be chosen arbitrarily small. Here, let us briefly explain $\gamma_0^{I}$ case which is easier than $\gamma_0^{C}$ case. (See Proposition \ref{mB zero}.) We crucially use axial symmetry of $\O$. Throughout one-to-one corresponding property between $(\theta, \varphi)$ (parametrization of inflection grazing) or $(x,v)\in B_{L}^{\varepsilon}$ (when trajectory is well-defined and there are no converging bounces), we define
\begin{equation}\label{Fj}
	F_j(\theta,\varphi):=X(s(\theta,\varphi);0,\sigma(\theta,\varphi),I^{2}(\theta,\varphi)) - x , 
\end{equation}
when forward in time trajectory from $(\sigma(\theta,\varphi),I^{2}(\theta,\varphi))$ passes $\overline{S_0}$ after $j$-th bouncing. Here $s(\theta, \varphi)$ is arrival time function which measures traveling time from $(y,u) = (\sigma(\theta,\varphi),I^{2}(\theta,\varphi)) \in \gamma_{0}^{I_{-}}$ to $S_0$. See Lemma \ref{lem time s} for detail. Once we prove that $F_{j}$ is analytic, we have the following dichotomy (a) or (b) : \\

(a) If the above function is identically zero, all the inflection grazing which pass $x$ yield $\mathfrak{m}_{2}(B_{j})_{x} > 0$ in general. Fortunately, however, we can exclude such cases, away from small sets, using the axial symmetry structure of $\O$. From the following observation,
\[
\frac{\partial X(s;0,\tau,\varphi,\eta)}{\partial \varphi} = C \hat{\varphi}(0), \quad \text{where $\hat{\varphi}(0)$ is defined in Definition \ref{cross section}},
\]
the direction of $V(s(\theta, \varphi); 0, \sigma(\theta,\varphi),I^{2}(\theta,\varphi))$ must satisfy some specific direction which we can exclude in the definition of $X^{\varepsilon}$.  \\

(b) Therefore, from analyticity, the zero set of  \eqref{Fj} has measure zero in $(\theta, \varphi)$ space. And by Lusin's Theorem (Lemma \ref{Lusin}), corresponding $v=V(s(\theta, \varphi); 0, \sigma(\theta,\varphi),I^{2}(\theta,\varphi))$ has also measure zero in $\S^{2}$. \\

Treating concave grazing $\gamma_0^{C}$ in $(B_{j})_{x}$ is more tricky since concave grazing is parametrized by $(\theta, \varphi, \tau)$ for $0 < \tau < 1$. We carefully apply analytic version of implicit function theorem to identify dimension between domain and target space. (Note that we need the same dimension between domain and range to apply Lusin's theorem.) See Proposition \ref{mB zero} for detail. At last, after taking a small neighborhood and performing some simple scaling for general speed $\frac{1}{N} \leq |v| \leq N$,  we obtain the phase decomposition Theorem \ref{bad set thm} which says our trajectory up to fixed length is nongrazing away from small sets as follows.   

\begin{theorem}[Informal statement of Theorem \ref{bad set thm}]
 There exists a compact set $\mathcal{Z}^{\varepsilon} \subset \overline{S_{\varphi}}\times \VN$ such that if $(x,v) \in \mathcal{Z}^{\varepsilon}$, then the backward in time trajectory from $(x,v)$ up to length $L$ is uniformly nongrazing with uniformly finite number for bounce. Moreover, for each $x\in \overline{\O}$,
\[
	\mathfrak{m}_{3}\{ v\in \R^{3} \ : \ (x,v)\in \{\overline{S_{\varphi}}\times \VN\}\backslash \mathcal{Z}^{\varepsilon}  \} \lesssim \varepsilon \ll 1.
\]
\end{theorem}
We lastly give some comments about {\it sticky grazing points} which were introduced in \cite{cylinder}. Sticky grazing points are a point $x\in\overline{\O}$ from which backward in time trajectory with some directions of considerable measure in $\VN$ hits boundary grazingly, $\gamma_0^{I}$ or $\gamma_0^{C}$. In the case of a general 2D analytic non-convex domain such points exist in general, but at most uniformly finite number of such points only. Such an issue can be resolved by choosing an arbitrarily small neighborhood around the sticky grazing point and considering small traveling time of the trajectory near the point. In this paper, however, axial symmetry plays a crucial role to confine all such directions in a measure zero set as shown in $\mathfrak{m}_{2}((B_{j})_{x})=0$. \\



{\bf Step 4.} (Triple iteration technique and constructive $L^{2}$-coercivity)
Nondegeneracy condition
\Be \label{C'}
\Big| \det \frac{\p X(s;t,x,v)}{\p v}\Big|  \geq \epsilon > 0
\Ee
is the heart of $L^{2}$-$L^{\infty}$ bootstrap argument. The most generalized tool to obtain \eqref{C'} is the triple iteration technique which was introduced in \cite{KimLee}. Once we have phase decomposition theorem stated above (or see Theorem \ref{bad set thm} for full statement), we can apply triple Duhamel iteration analysis to obtain nondegeneracy \eqref{C'} away from some bad sets. Performing triple Duhamel iteration in \cite{KimLee} is very complicated and lengthy. In this paper, we do not provide full detail of triple iteration process, because we can directly adopt many key computation and Lemmas from \cite{KimLee}. 


In this paper, we often use the following convenient notation: 
\begin{definition}(Big O notation) \label{Big O}
	Let $f(x)$ and $g(x)$ be two real-valued functions defined on some subset $\subset \R$. For some $x_0 \in \R$, we say
	\begin{align*}
		f(x)=\mathcal{O}(g(x)), \quad as \; x\rightarrow x_0,
	\end{align*}
	if there exist positive numbers $\delta$ and $M$ such that 
	\begin{align*}
		\vert f(x) \vert \leq M \vert g(x) \vert \quad when \; 0<\vert x-x_0 \vert <\delta,
	\end{align*}
	where $M$ is a generic constant. For some variable $y \in \R$, we define 
	\begin{align*}
		f(x)=\mathcal{O}_{y} (g(x)), \quad as \; x\rightarrow x_0,
	\end{align*}
	if there exist positive numbers $\delta$ and $M_y$(depending on $y$) such that 
	\begin{align*}
		\vert f(x) \vert \leq M_y \vert g(x) \vert \quad when \; 0<\vert x-x_0 \vert<\delta. \\
	\end{align*}
\end{definition}

\section{Phase Boundary Decomposition}

Recall that the analytic convex curve $\gamma(\tau)=(\gamma_1(\tau),0,\gamma_2(\tau))$ satisfies $\gamma_2'(\tau)<0$ for $\tau\in (\tau_{1,*},\tau_{2,*})$ where $\tau_{1,*}$ and $\tau_{2,*}$ be the zeros of $\gamma_2'$, see Figure \ref{analytic convex curve}. Also, we assume that $\gamma_1'(\tau) < 0$ for $\tau \in (\tau_{1,*},\lambda_*)$ and $\gamma_1'(\tau)>0$ for $\tau \in (\lambda_*,\tau_{2,*})$. Here, $\lambda_*$ be the zero of $\gamma_1'$.   

\begin{lemma} \label{analytic_zeros} \cite{Guo10}
	Suppose that $f$ is a non-constant real-analytic function on a connected open domain $D\subset\R^n$. Then, the zero set 
	\begin{align*}
		Z_f:=\{x\in D : f(x)=0\}
	\end{align*}
	has zero $n$-dimensional Lebesgue measure. If $n=1$, the zero set $Z(f)$ has no limit point. Moreover, $f$ has only finitely many zeros on compact subsets of $\R$. 
	\end{lemma}
	\begin{proof}
	From Lemma 4 in \cite{Guo10}, the zero set $Z_f$ has zero $n$-dimensional Lebesgue measure and  has no limit point in $\R$. Thus, it suffices to show that a non-constant real analytic function $f$ in a compact region has finitely many zeros. If $f$ has infinitely many zeros in a compact region $K\subset \R$, then there is a sequence $(x_n)_{n\in \mathbb{N}}$ such that 
\begin{align*}
	   f(x_n)=0 \quad \forall n\in \mathbb{N} \quad \text{and} \quad x_n \rightarrow x_*\in K 
\end{align*}
by compactness. Due to continuity of $f$, it holds that $f(x_*)=0$. However, in $\R$, the zero set of identically nonzero real-analytic function has no limit point, which leads to a contradiction.   
	\end{proof}

\begin{lemma} \cite{Diff_Geom} \label{Euler thm} 
	(Euler's formula) Let $\gamma$ be a regularized (unit speed) curve on an oriented surface $\mathcal{S} \subset \R^{3}$.  The scalar
\begin{align*}
	\kappa_n(\tau) :=\ddot{\gamma}(\tau)\cdot n(\gamma(\tau)),
\end{align*}
is called the normal curvature of $\gamma$ where $n(\gamma(\tau))$ is a fixed unit normal vector of $\mathcal{S}$ at $\gamma(\tau)\in \mathcal{S}$. If $\kappa_1$ and $\kappa_2$ be the principal curvatures with non-zero principal vectors $\mathbf{t}_1$ and $\mathbf{t}_2$, then the normal curvature $\kappa_n$ of $\gamma$ is 
	\begin{align*}
		\kappa_n=\kappa_1 \cos^2 \vartheta + \kappa_2 \sin^2 \vartheta,
	\end{align*}
	where $\vartheta$ is the oriented angle between vectors $\mathbf{t}_1$ and $\dot \gamma$.  
\end{lemma}

The following lemma is obvious every point is a saddle point in inner part $\tau_{1,*}< \tau <  \tau_{2,*}$.
\begin{lemma} \label{lem:nograzing_V}
	 $(\gamma(\tau),\dot\gamma(\tau))\in\gamma_0^{C} \cup \gamma_0^{I} $ only if $\tau_{1,*}< \tau <  \tau_{2,*}$. (See Section 1.1 and Figure \ref{analytic convex curve} for the definition of $\tau_{1,*}, \tau_{2,*}$.)
\end{lemma}

\begin{definition} \label{tangent plane}
At $\sigma(\tau_0,0)=(\gamma_1(\tau_0),0,\gamma_2(\tau_0))\in \p \O \; (\varphi=0)$, the outward normal vector $N_{\sigma(\tau_0,0)}$ is 
\begin{align*}
	N_{\sigma(\tau_0,0)}=(\gamma_2'(\tau_0),0,-\gamma_1'(\tau_0)).
\end{align*}
In general, $N_{\sigma(\tau,\varphi)}$ denotes the outward normal vector at $\sigma(\tau,\varphi)$ throughout this paper. We define the tangent plane $T_{\sigma(\tau_0,0)}$ at $\sigma(\tau_0,0)$ as 
\begin{align*}
	T_{\sigma(\tau_0,0)}:=\Big \{ (x,y,z): x=\gamma_1(\tau_0)+\frac{\gamma_1'(\tau_0)}{\gamma_2'(\tau_0)}(z-\gamma_2(\tau_0))\Big \}.
\end{align*}
We define a new 2D coordinate of the tangent plane $T_{\sigma(\tau_0,0)}$,
\begin{align} \label{new coord in tp}
	(\tilde{y},\tilde{z}):=\Big(y, \frac{z}{\vert \gamma_2'(\tau_0)\vert}\Big), \quad \text{for } (x,y,z)\in T_{\sigma(\tau_0,0)}.
\end{align}
Moreover, direction $\hat{\tilde{y}},\hat{\tilde{z}}$ of $\tilde{y}, \tilde{z}$ are
\begin{align} \label{new direction}
	\hat{\tilde{y}}=\hat{y}, \quad \hat{\tilde{z}}=\begin{pmatrix}
	\cos \theta_* & 0 & \sin \theta_* \\ 0 & 1 & 0 \\ -\sin \theta_* & 0 & \cos \theta_*
	\end{pmatrix} \hat{z},
\end{align}
where $\hat{y}=(0,1,0),\hat{z}=(0,0,1)$ and $\theta_*$ satisfies that $ \cos \theta_* = \vert \gamma_2'(\tau_0)\vert$. Using the new coordinate \eqref{new coord in tp} and direction \eqref{new direction} of the tangent space $T_{\sigma(\tau_0,0)}$, we have the following coordinate map 
\begin{align} \label{coord map}
	(x,y,z)\mapsto y \hat{\tilde{y}}+ \frac{z}{\vert \gamma_2'(\tau_0)\vert} \hat{\tilde{z}},
\end{align}
for $(x,y,z) \in T_{\sigma(\tau_0,0)}$.   
\end{definition}

\begin{remark}
In \eqref{coord map}, $\hat{\tilde{y}}$ and $\hat{\tilde{z}}$ satisfy that 
\begin{align*}
&	N_{\sigma(\tau_0,0)}\cdot \hat{\tilde{y}}=(\gamma_2'(\tau_0),0,-\gamma_1'(\tau_0))\cdot (0,1,0) = 0, \\
& N_{\sigma(\tau_0,0)}\cdot \hat{\tilde{z}} =(\gamma_2'(\tau_0),0,-\gamma_1'(\tau_0))\cdot (\sin\theta_*,0,\cos \theta_*)=0,
\end{align*}
where we have used $\cos\theta_*=-\gamma_2'(\tau_0)$ and $\sin \theta_*=-\gamma_1'(\tau_0)$. Hence, we can consider $y\hat{\tilde{y}}+\frac{z}{\vert \gamma_2'(\tau_0)\vert} \hat{\tilde{z}}$ in \eqref{coord map} as a coordinate in the tangent plane $T_{\sigma(\tau_0,0)}$. The reason why we set $\tilde{z}$ in \eqref{new coord in tp} will be explained through Figure \ref{new coordinate}. Figure \ref{new coordinate} shows the tangent plane $T_{\sigma(\tau_0,0)}$ at $\sigma(\tau_0,0)$ in $xz$-plane $(\varphi=0)$. Since $\cos \theta_*=N_{\sigma(\tau_0,0)}\cdot (-1,0,0)=\vert \gamma_2'(\tau_0) \vert$, $\tilde{z}$ was put as $\frac{z}{\vert \gamma_2'(\tau_0)\vert}$. Throughout this paper, points in $T_{\sigma(\tau_0,0)}$ can be expressed through a new coordinate system such as \eqref{coord map}.   
\end{remark}

\begin{figure}[h]
\centering
\includegraphics[width=15cm]{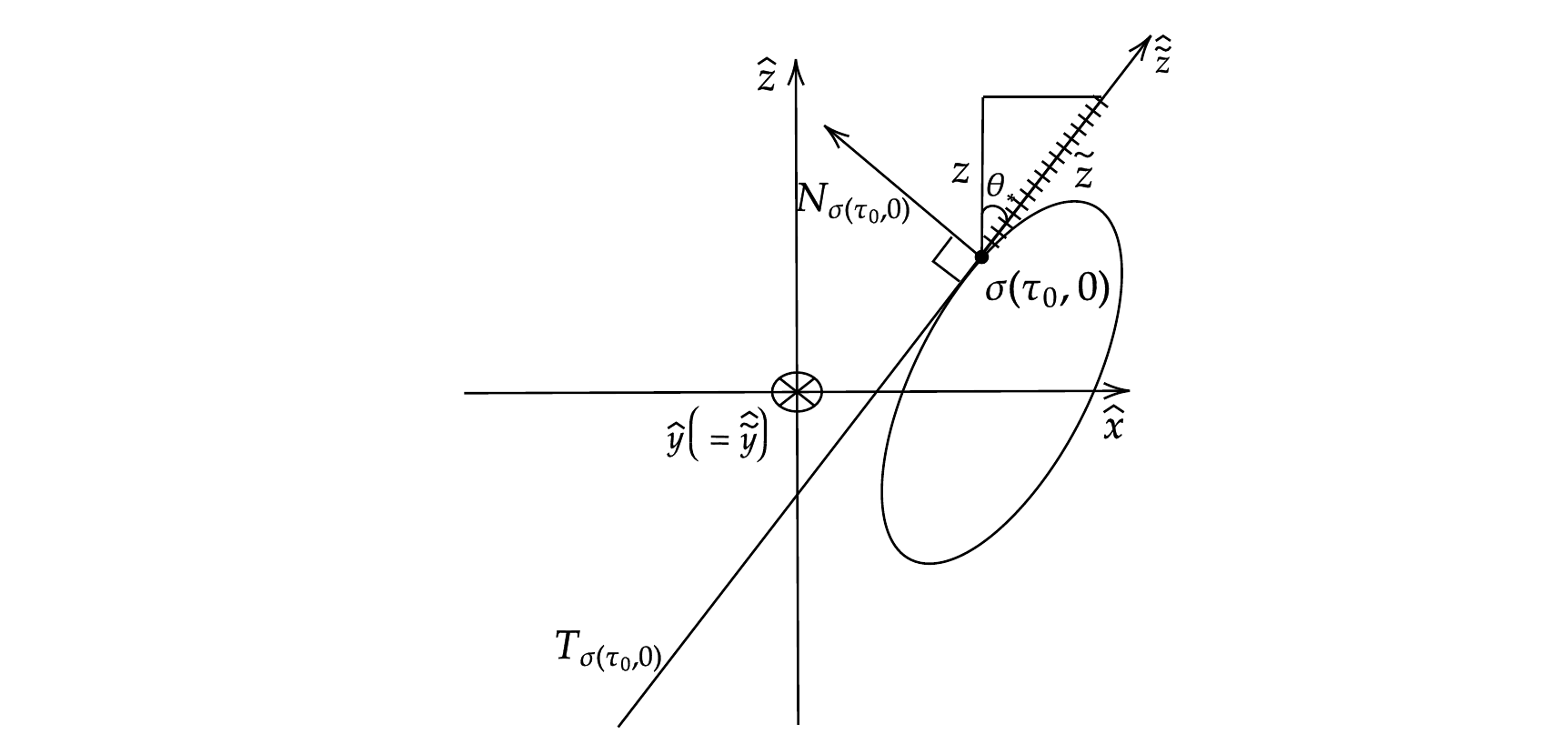}
\caption{New coordinate $\tilde{y},\tilde{z}$ of the tangent plane $T_{\sigma(\tau_0,0)}$ (direction of $\hat{y}=\hat{\tilde{y}}$ is inward to the paper)}
\label{new coordinate}
\end{figure}

\subsection{Inflection grazing set}

Now, we study local inflection structure.
\begin{lemma} \label{inflection}
Assume that $\tau_{1,*}$ and $\tau_{2,*}$ are zeros of $\gamma_2'$ for the unit-speed analytic convex curve $\gamma(\tau)=(\gamma_1(\tau),0,\gamma_2(\tau))$. Let us define
\begin{align} \label{def h}
	Z_{h} := \Big\{ \tau\in(\tau_{1,*},\tau_{2,*}) \ : \ h(\tau):= \frac{\gamma_1'(\tau)}{\gamma_1(\tau)}(\gamma_1(\tau)\kappa(\tau)+\vert \gamma_2'(\tau)\vert)+\frac{\vert \gamma_2'(\tau)\vert \kappa'(\tau)}{3\kappa(\tau)} = 0 \Big\},
\end{align}
where 
\begin{align} \label{curvature}
\kappa(\tau) :=\sqrt{(\gamma_1''(\tau))^2+(\gamma_2''(\tau))^2} 
\end{align}
is the curvature of a unit-speed curve $\gamma$ at $\tau$. Then, the zero set $Z_h$ is a finite union of points, denoted by 
\begin{align} \label{Zh_ell}
Z_h=\bigcup_{j=1}^{\ell}\{\tau_j^{\mathbf{z}}\}, \quad \text{for some finite } \ell =\ell_{\Omega} <\infty.
\end{align}
Moreover, for $(\tau,\varphi) \in (\tau_{1,*},\tau_{2,*})\backslash Z_h \times [0,2\pi)$, there exist two directions $I^1(\tau,\varphi),I^2(\tau,\varphi)$\\
$\in \S^2$ which satisfy 
\begin{align*}
	(\sigma(\tau,\varphi),I^1(\tau,\varphi))\in \gamma_0^{I_+}, \quad (\sigma (\tau,\varphi),I^2(\tau,\varphi))\in \gamma_0^{I_-},
\end{align*}
and have positivie angular momentums with respect to $z$-axis. (See Figure \ref{inflec-directions}.) Here, $\gamma_0^{\pm}$ are defined in \eqref{inflection +-}. In other words, $(\sigma(\tau,\varphi),I^k(\tau,\varphi))\in \gamma_0^I$ for $k=1,2$.  
\end{lemma}
\begin{proof}
\textbf{Step 1.} In this step, we consider the intersection of the boundary $\p\O$ and the tangent plane $T_{\sigma(\tau_0,0)}$ at point $\sigma(\tau_0,0)$. Here, the tangent plane $T_{\sigma(\tau_0,0)}$ was defined in Definition \ref{tangent plane}. The aim of this step is to calculate a tangent line of the intersection (or $\frac{d\tilde{z}}{dy}$) at $\sigma(\tau_0,0)$. \\
To find where inflection grazing occurs, we only focus the domain $\tau \in (\tau_{1,*},\tau_{2,*})$ by Lemma \ref{lem:nograzing_V}. Also, to ensure $\gamma_1'(\tau),\gamma_2'(\tau)<0$, we restrict our domain from $(\tau_{1,*},\tau_{2,*})$ to $(\tau_{1,*},\lambda_*)$ where $\gamma_1'(\lambda_*)=0$. Since our domain $\O$ is symmetric about the $z$-axis, it would be enough to consider $\varphi=0$ case without loss of generality. Recall that $\gamma$ is a unit speed curve. For $\tau_0\in (\tau_{1,*},\lambda_*)$, we calculate a outward normal vector $N_{\sigma(\tau_0,0)}$ at $\sigma(\tau_0,0)=(\gamma_1(\tau_0),0,\gamma_2(\tau_0))\in \partial \Omega$
\begin{align} \label{out N}
	N_{\sigma(\tau_0,0)}=(\gamma_2'(\tau_0),0,-\gamma_1'(\tau_0)).
\end{align}
Then, the tangent plane $T_{\sigma(\tau_0,0)}$ at $\sigma(\tau_0,0)$ is 
\begin{align} \label{cond1}
	\{(x,y,z)\in \R^3: x=\gamma_1(\tau_0)+\tan\theta_0(z-\gamma_2(\tau_0))\}, \quad \text{where }\tan \theta_0=\frac{\gamma_1'(\tau_0)}{\gamma_2'(\tau_0)}.
\end{align}
In general, a smooth curve $(x(\tau),y(\tau),z(\tau))\in \partial \Omega$ on the general toroidal domain satisfies 
\begin{align} \label{cond2}
\sqrt{x^2(\tau)+y^2(\tau)}=\gamma_1(\tau), \quad z(\tau)=\gamma_2(\tau). 
\end{align}

As we mentioned before, we consider the intersection of the boundary $\p\O$ and the tangent plane $T_{\sigma(\tau_0,0)}$ at $\sigma(\tau_0,0)$. Since our domain is axisymmetric about the $z$-axis, we firstly consider the case $\left (y>0, \tilde{z}>\frac{\gamma_2(\tau_0)}{\vert \gamma_2'(\tau_0)\vert}\right)$ where $\tilde{z}$ in Definition \ref{tangent plane}. From \eqref{cond1} and \eqref{cond2}, we denote the intersection as 
\begin{align} \label{para intersection}
\mathcal{I}_{\sigma(\tau_0,0)}:=\left\{(x,y,z)=(x(\tau),y(\tau),z(\tau)):
\begin{array}{cc}
x(\tau)=\gamma_1(\tau_0)+\tan\theta_0(z(\tau)-\gamma_2(\tau_0)),\\
y(\tau)=\sqrt{\gamma_1^2(\tau)-x^2(\tau)}
\end{array}
\right\}.
\end{align}
Using the new coordinate \eqref{new coord in tp} and \eqref{new direction} in the tangent plane $T_{\sigma(\tau_0,0)}$, the intersection $\mathcal{I}_{\sigma(\tau_0,0)}$ becomes 
\begin{align*}
	\Pi_{\widetilde{yz}} I_{\sigma(\tau_0,0)} := \Big \{(\tilde{y},\tilde{z})_{\widetilde{yz}}=\Big(y,\frac{z}{\vert \gamma_2'(\tau_0)\vert}\Big)_{\widetilde{yz}}: (x,y,z)\in \mathcal{I}_{\sigma(\tau_0,0)}\Big\},
\end{align*}
where $(a,b)_{\widetilde{yz}}=a\hat{\tilde{y}}+b\hat{\tilde{z}}$. Especially, the point $\sigma(\tau_0,0)=(\gamma_1(\tau_0),0,\gamma_2(\tau_0))\in T_{\sigma(\tau_0,0)}$ becomes 
\begin{align} \label{projected point}
\Pi_{\widetilde{yz}}\sigma(\tau_0,0):=\frac{\gamma_2(\tau_0)}{\vert \gamma_2'(\tau_0) \vert}\hat{\tilde{z}}.
\end{align} 
 We check whether the tangent line $T_\ell(y)$ in \eqref{tangent line} to the intersection $\Pi_{\widetilde{yz}} \mathcal{I}_{\sigma(\tau_0,0)}$ at $\Pi_{\widetilde{yz}}\sigma(\tau_0,0)$ has the direction of which the sign of normal curvature $\kappa_n$ in Lemma \ref{Euler thm} changes. \\
Recall that boundary $\p\O$ can be represented by $\{(x,y,z): \xi(\sqrt{x^2+y^2},z)=0\}$ using \eqref{analytic convex curve}. Then, equation of the intersection $\Pi_{\widetilde{yz}} \mathcal{I}_{\sigma(\tau_0,0)}$ becomes
\begin{align} \label{implicit inter}
	\xi\left(\sqrt{\left(\gamma_1(\tau_0)+\tan\theta_0\left(\vert \gamma_2'(\tau_0)\vert \tilde{z}-\gamma_2(\tau_0)\right)\right)^2+y^2},\vert \gamma_2'(\tau_0)\vert \tilde{z}\right)=0.
\end{align}
Since it is hard to calculate $\frac{d\tilde{z}}{dy}$ at $\tau=\tau_0$ directly from \eqref{implicit inter}, we calculate $\frac{d\tilde{z}}{dy}$ at $\tau=\tau_0$ using the Chain rule and \eqref{cond2}. Due to $\left. \frac{d\tilde{z}}{dy}\right \vert_{\tau=\tau_0} = \frac{1}{\vert \gamma_2'(\tau_0)\vert} \left. \frac{dz}{dy}\right \vert_{\tau=\tau_0}$, it suffices to compute $ \left. \frac{dz}{dy}\right \vert_{\tau=\tau_0}$. To apply the Chain rule, we calculate $\frac{dz}{d\tau}$ and $\frac{dy}{d\tau}$. From \eqref{cond2},  
\begin{align} \label{deriv y,z}
	y(\tau)=\sqrt{\gamma_1^2(\tau)-x^2(\tau)}, \quad 
	y'(\tau)=\frac{\gamma_1(\tau)\gamma_1'(\tau)-x(\tau)x'(\tau)}{\sqrt{\gamma_1^2(\tau)-x^2(\tau)}}, \quad z'(\tau)=\gamma_2'(\tau).
\end{align}
Then,
\begin{align}\label{dz/dy}
	\begin{split}
	\left. \frac{dz}{dy}\right \vert_{\tau=\tau_0}= \left. \frac{\frac{dz}{d\tau}}{\frac{dy}{d\tau}} \right \vert_{\tau=\tau_0}&=\lim_{\tau \rightarrow \tau_0-} \frac{\sqrt{\gamma_1^2(\tau)-x^2(\tau)}}{\gamma_1(\tau)\gamma_1'(\tau)-x(\tau)x'(\tau)}\gamma_2'(\tau) \\
	&=\lim_{\tau \rightarrow \tau_0-} \frac{\sqrt{\vert \gamma_1(\tau)-x(\tau)\vert}\sqrt{\gamma_1(\tau)+x(\tau)}}{\gamma_1(\tau)\gamma_1'(\tau)-x(\tau)x'(\tau)} \gamma_2'(\tau).
	\end{split}
\end{align}
Using Taylor's expansion at $\tau=\tau_0$, \eqref{cond2}, and \eqref{para intersection}, we have 
\begin{align} \label{calc1}
	\gamma_1(\tau)-x(\tau) &= \left(\gamma_1(\tau_0)+\gamma'_1(\tau_0)(\tau-\tau_0)+\frac{\gamma_1''(\tau_0)}{2}(\tau-\tau_0)^2 +\mathcal{O}_{\gamma_1}(\vert\tau-\tau_0\vert^3)\right) \nonumber \\
	&\quad -(\gamma_1(\tau_0)+\tan\theta_0(\gamma_2(\tau)-\gamma_2(\tau_0)) \nonumber\\
	&=\left(\gamma_1(\tau_0)+\gamma'_1(\tau_0)(\tau-\tau_0)+\frac{\gamma_1''(\tau_0)}{2}(\tau-\tau_0)^2 +\mathcal{O}_{\gamma_1}(\vert \tau-\tau_0\vert^3)\right)\nonumber\\
	&\quad -\left(\gamma_1(\tau_0)+\tan\theta_0 \gamma'_2(\tau_0)(\tau-\tau_0)+\tan\theta_0 \frac{\gamma''_2(\tau_0)}{2}(\tau-\tau_0)^2 +\mathcal{O}_{\gamma_2}(\vert\tau-\tau_0\vert^3)\right)\nonumber\\
	&=\frac{1}{2}\left(\gamma_1''(\tau_0)-\tan\theta_0\gamma_2''(\tau_0)\right)(\tau-\tau_0)^2+\mathcal{O}_{\gamma}(\vert\tau-\tau_0\vert^3)\nonumber\\
	&=\frac{\kappa(\tau_0)}{2\vert \gamma_2'(\tau_0)\vert}(\tau-\tau_0)^2+ \mathcal{O}_{\gamma}(\vert\tau-\tau_0\vert^3),
\end{align}
where we used 
\begin{align} \label{second deriv gamma}
(\gamma_1''(\tau_0),\gamma_2''(\tau_0))=\kappa(\tau_0)(-\gamma_2'(\tau_0),\gamma_1'(\tau_0)),
\end{align}
$\tan\theta_0=\frac{\gamma_1'(\tau_0)}{\gamma_2'(\tau_0)}$ from \eqref{cond1}, and $\gamma_2'(\tau_0)=-\vert \gamma_2'(\tau_0)\vert$ for $\tau_0\in (\tau_{1,*},\lambda_*)$. Using $\tan\theta_0=-\frac{\gamma_1'(\tau_0)}{\vert \gamma_2'(\tau_0)\vert}$, \eqref{para intersection}, \eqref{second deriv gamma} and Taylor expansion at $\tau=\tau_0$, we derive 
\begin{align} \label{calc2}
&	\gamma_1(\tau)\gamma_1'(\tau)-x(\tau)x'(\tau)\notag \\ &=\Big(\gamma_1(\tau_0)+\gamma_1'(\tau_0)(\tau-\tau_0)+\mathcal{O}_{\gamma_1}(\vert\tau-\tau_0\vert^2)\Big)\times\Big(\gamma_1'(\tau_0)+\gamma_1''(\tau_0)(\tau-\tau_0)+\mathcal{O}_{\gamma_1}(\vert\tau-\tau_0\vert^2)\Big) \nonumber \\
	&\quad -\Big(\gamma_1(\tau_0)+\tan\theta_0(\gamma_2(\tau)-\gamma_2(\tau_0))\Big)\times(\tan \theta_0\gamma_2'(\tau))\nonumber \\
	&=\Big(\gamma_1(\tau_0)\gamma_1'(\tau_0)+\gamma_1(\tau_0)\gamma_1''(\tau_0)(\tau-\tau_0)+(\gamma_1'(\tau_0))^2(\tau-\tau_0)+\mathcal{O}_{\gamma_1}(\vert\tau-\tau_0\vert^2)\Big)\nonumber\\
	&\quad -\Big(\gamma_1(\tau_0)+\tan\theta_0\gamma_2'(\tau_0)(\tau-\tau_0)+\mathcal{O}_{\gamma_2}(\vert\tau-\tau_0\vert^2)\Big)\nonumber \\
	&\qquad \times\Big(\tan\theta_0 \gamma_2'(\tau_0)+\tan\theta_0\gamma_2''(\tau_0)(\tau-\tau_0)+\tan\theta_0\mathcal{O}_{\gamma_2}(\vert\tau-\tau_0\vert^2)\Big)\nonumber \\
	&=\Big(\gamma_1(\tau_0)\gamma_1''(\tau_0)-\tan\theta_0\gamma_1(\tau_0)\gamma_2''(\tau_0)\Big)(\tau-\tau_0)+\mathcal{O}_{\gamma}(\vert\tau-\tau_0\vert^2)\nonumber \\
	&=\frac{\kappa(\tau_0)\gamma_1(\tau_0)}{\vert \gamma_2'(\tau_0)\vert}(\tau-\tau_0)+\mathcal{O}_{\gamma}(\vert\tau-\tau_0\vert^2).
\end{align}
Then, plugging \eqref{calc1} and \eqref{calc2} into $\frac{\sqrt{ \vert \gamma_1(\tau)-x(\tau)\vert}}{\gamma_1(\tau)\gamma_1'(\tau)-x(\tau)x'(\tau)}$ of \eqref{dz/dy}, we have for $\tau-\tau_0<0$ 
\begin{align} \label{calc3}
	 \frac{\sqrt{\vert \gamma_1(\tau)-x(\tau)\vert}}{\gamma_1(\tau)\gamma_1'(\tau)-x(\tau)x'(\tau)}=\frac{\sqrt{\frac{\kappa(\tau_0)}{2\vert \gamma_2'(\tau_0)\vert}+\mathcal{O}(\vert\tau-\tau_0\vert)}}{-\frac{\kappa(\tau_0)\gamma_1(\tau_0)}{\vert \gamma_2'(\tau_0)\vert }+\mathcal{O}(\vert\tau-\tau_0\vert)}.
\end{align}
From \eqref{dz/dy} and \eqref{calc3}, we derive that
\begin{align} \label{deriv tilde z}
	\lim_{\tau\rightarrow \tau_0-}\frac{dz}{dy}=\sqrt{\frac{\vert \gamma_2'(\tau_0)\vert}{\kappa(\tau_0)\gamma_1(\tau_0)}} \vert\gamma_2'(\tau_0)\vert, \quad \lim_{\tau\rightarrow \tau_0-}\frac{d\tilde{z}}{dy}=\sqrt{\frac{\vert \gamma_2'(\tau_0)\vert}{\kappa(\tau_0)\gamma_1(\tau_0)}},
\end{align}
where $\tilde{z}=\frac{z}{\vert \gamma_2'(\tau_0)\vert}$.  Using the first and second fundamental forms of the surface patch $\sigma$ at $(\tau,\varphi)=(\tau_0,0)$, the principal curvatures are the roots of the equation
\begin{align*}
&\left. \det \begin{pmatrix}
\sigma_{\tau\tau}\cdot N_{\sigma(\tau,\varphi)}-\kappa (\sigma_{\tau}\cdot \sigma_{\tau})& \sigma_{\tau\varphi}\cdot N_{\sigma(\tau,\varphi)}-\kappa (\sigma_{\tau}\cdot \sigma_{\varphi})\\ \sigma_{\tau\varphi}\cdot N_{\sigma(\tau,\varphi)}-\kappa (\sigma_{\tau}\cdot \sigma_{\varphi})& \sigma_{\varphi\varphi}\cdot N_{\sigma(\tau,\varphi)}-\kappa (\sigma_{\varphi}\cdot \sigma_{\varphi})
\end{pmatrix}\right \vert_{(\tau,\varphi)=(\tau_0,0)}\\
&=\det \begin{pmatrix}
\kappa(\tau_0)-\kappa & 0 \\ 0 & \gamma_1(\tau_0)\gamma_2'(\tau_0)-\kappa(\gamma_1(\tau_0))^2
\end{pmatrix}=0,
\end{align*}
where $\kappa$ is curvature of $\gamma$ as in \eqref{curvature}. Thus, the principal curvatures of $\sigma$ at $(\tau,\varphi)=(\tau_0,0)$ are 
\begin{align*}
	\kappa_1(\tau_0) = \frac{\gamma_2'(\tau_0)}{\gamma_1(\tau_0)}, \quad \kappa_2(\tau_0)=\kappa(\tau_0). 
\end{align*}
Note that the principal vectors corresponding to the principal curvature $\kappa$ are $\xi\sigma_{\tau}+\eta \sigma_{\varphi}$ such that 
\begin{align*}
\begin{pmatrix}
\kappa(\tau_0)-\kappa & 0 \\ 0 & \gamma_1(\tau_0)\gamma_2'(\tau_0)-\kappa(\gamma_1(\tau_0))^2
\end{pmatrix}\begin{pmatrix} \xi \\ \eta \end{pmatrix}= \begin{pmatrix} 0 \\ 0 \end{pmatrix}. 
\end{align*}
Since $\kappa_2(\tau_0)-\kappa_1(\tau_0)\neq0$, the principal vector corresponding to $\kappa_1(\tau_0)=\frac{\gamma_2'(\tau_0)}{\gamma_1(\tau_0)}$ is $\sigma_\varphi(\tau_0,0)=(0,\gamma_1(\tau_0),0)$. Then, by using Lemma \ref{Euler thm} and $\gamma_2'(\tau_0)=-\vert \gamma_2'(\tau_0) \vert$, the normal curvature $\kappa_n$ of $\gamma$ at $\tau=\tau_0$ is zero if
\begin{align} \label{tan vartheta}
	\tan\vartheta = \sqrt{\frac{\vert \gamma_2'(\tau_0)\vert}{\kappa(\tau_0)\gamma_1(\tau_0)}},
\end{align} 	
where $\vartheta$ is an angle between the $y$-axis (the principal vector corresponding to $\kappa_1$) and $\dot \gamma$. Hence, the tangent line $T_{\ell}$ to the intersection $\Pi_{\widetilde{yz}} \mathcal{I}_{\sigma(\tau_0,0)}$ at $\Pi_{\widetilde{yz}}\sigma(\tau_0,0)$ in \eqref{projected point} is 
\begin{align} \label{tangent line}
	T_{\ell}(y)=\left. \frac{d\tilde{z}}{dy}\right \vert_{\tau=\tau_0} y + \frac{\gamma_2(\tau_0)}{\vert \gamma_2'(\tau_0)\vert} = (\tan \vartheta) y +\frac{\gamma_2(\tau_0)}{\vert \gamma_2'(\tau_0)\vert}. 
\end{align}
and from \eqref{dz/dy} and \eqref{deriv tilde z}, we have 
\begin{align*}
	\tan\vartheta = \left. \frac{d\tilde{z}}{dy}\right \vert_{\tau=\tau_0}=\frac{1}{\vert \gamma_2'(\tau_0)\vert}\left. \frac{dz}{dy}\right \vert_{\tau=\tau_0}=\sqrt{\frac{\vert \gamma_2'(\tau_0)\vert}{\kappa(\tau_0)\gamma_1(\tau_0)}},
\end{align*}
and the normal curvature $\kappa_n=0$ on the tangent line $T_\ell$ define in \eqref{tangent line}. See $I_{1}$ direction in the Figure \ref{remark 1}. \\

 
In terms of defining normal curvature in Lemma \ref{Euler thm}, let us choose inward normal vector $-N_{\sigma(\tau_0,0)}$ instead of outward normal vector for convenience. Then, in the upper region of the tangent line $T_{\ell}(y)$ 
\begin{align*}
	\left\{({y}, \tilde{z}) : {y}>0, \tilde{z} > \frac{\gamma_2(\tau_0)}{\vert \gamma_2'(\tau_0)\vert}, \tilde{z} > T_{\ell}({y})\right\}, 
\end{align*}
a direction of which slope larger than $\tan\vartheta$ in $\tilde{y}\tilde{z}$-plane has a positive normal curvature ($\kappa_n>0$) at $\tau=\tau_0$. 
Similarly, in the lower region of the tangent line $T_{\ell}$ 
\begin{align*}
	\left\{(y, \tilde{z}) : y > 0, \tilde{z} > \frac{\gamma_2(\tau_0)}{\vert \gamma_2'(\tau_0)\vert}, \tilde{z} < T_{\ell}(y)\right\}, 
\end{align*}
a direction of which slope lesser than $\tan\vartheta$ in $\tilde{y}\tilde{z}$-plane has a negative normal curvature ($\kappa_n<0$) at $\tau=\tau_0$. \\

\textbf{Step 2.} Next, we calculate $\frac{d^2\tilde{z}}{dy^2}$ at $\tau=\tau_0$ and check when it becomes zero. 
Using chain rule, 
\begin{align} \label{d^2z/dy^2}
	\frac{d^2z}{dy^2}=\frac{d}{d\tau} \left(\frac{\frac{dz}{d\tau}}{\frac{dy}{d\tau}}\right) \frac{1}{\frac{dy}{d\tau}}=\frac{z''(\tau)y'(\tau)-z'(\tau)y''(\tau)}{(y'(\tau))^3}.
\end{align}
Notice that, from \eqref{deriv y,z} and \eqref{second deriv gamma}, as $\tau\rightarrow \tau_0$, 
\begin{align} \label{limit}
	\begin{split}
	z''(\tau)&=\gamma_2''(\tau)=\kappa(\tau)\gamma_1'(\tau)\rightarrow \kappa(\tau_0)\gamma_1'(\tau_0),\\
	y'(\tau)&=\frac{\gamma_1(\tau)\gamma_1'(\tau)-x(\tau)x'(\tau)}{\sqrt{\gamma_1^2(\tau)-x^2(\tau)}}\rightarrow -\sqrt{\frac{\kappa(\tau_0)\gamma_1(\tau_0)}{\vert \gamma_2'(\tau_0)\vert}},\\
	z'(\tau)&=\gamma_2'(\tau)\rightarrow\gamma_2'(\tau_0),
	\end{split}
\end{align}
where the limit of $y'(\tau)$ comes from \eqref{calc3}. Thus, to obtain $\frac{d^2z}{dy^2}$ at $\tau=\tau_0$, it suffices to compute $\lim_{\tau\rightarrow \tau_0}y''(\tau)$. From \eqref{deriv y,z}, we have 
\begin{align*}
	y(\tau)y'(\tau)=\gamma_1(\tau)\gamma_1'(\tau)-x(\tau)x'(\tau). 
\end{align*}
Taking derivative yields that 
\begin{align} \label{I+II}
	\begin{split}
	y(\tau)y''(\tau)+(y'(\tau))^2 &=(\gamma_1'(\tau))^2+\gamma_1(\tau)\gamma_1''(\tau)-(x'(\tau))^2-x(\tau)x''(\tau)\\
	&=(\gamma_1'(\tau)-x'(\tau))(\gamma_1'(\tau)+x'(\tau))+(\gamma_1(\tau)\gamma_1''(\tau)-x(\tau)x''(\tau))\\
	&:=I+II. 
	\end{split}
\end{align}
For part $I$ above, we use Taylor expansion at $\tau=\tau_0$, $\tan\theta_0=-\frac{\gamma_1'(\tau_0)}{\vert \gamma_2'(\tau_0)\vert}$, and \eqref{para intersection}: 
\begin{align} \label{calc4}
	\gamma_1'(\tau)-x'(\tau)&=\left(\gamma_1'(\tau_0)+\gamma_1''(\tau_0)(\tau-\tau_0)+\mathcal{O}(\vert\tau-\tau_0\vert^2)\right)-\left(\tan\theta_0\gamma_2'(\tau)\right) \nonumber\\
	&=\left(\gamma_1'(\tau_0)+\gamma_1''(\tau_0)(\tau-\tau_0)+\mathcal{O}(\vert\tau-\tau_0\vert^2)\right)\nonumber\\
	&\quad - \left(\tan\theta_0 \gamma_2'(\tau_0)+\tan\theta_0\gamma_2''(\tau_0)(\tau-\tau_0)+\mathcal{O}(\vert\tau-\tau_0\vert^2)\right)\nonumber\\
	&=\left(\gamma_1''(\tau_0)-\tan\theta_0\gamma_2''(\tau_0)\right)(\tau-\tau_0)+\mathcal{O}(\vert\tau-\tau_0\vert^2) \nonumber\\
	&=\frac{\kappa(\tau_0)}{\vert \gamma_2'(\tau_0)\vert}(\tau-\tau_0)+\mathcal{O}(\vert\tau-\tau_0\vert^2).
\end{align}
Thus, from \eqref{calc4}, the part $I$ be 
\begin{align} \label{part I}
	I=(\gamma_1'(\tau)+x'(\tau))\frac{\kappa(\tau_0)}{\vert \gamma_2'(\tau_0)\vert}(\tau-\tau_0)+\mathcal{O}(\vert\tau-\tau_0\vert^2).
\end{align}
For part $II$, using \eqref{para intersection} and \eqref{second deriv gamma}, each term in $II$ will be represented by Taylor's expansion:
\begin{align}
&	\gamma_1(\tau)\gamma_1''(\tau)=-\kappa(\tau)\gamma_1(\tau)\gamma_2'(\tau)\notag\\
&=-\kappa(\tau)\left(\gamma_1(\tau_0)+\gamma_1'(\tau_0)(\tau-\tau_0)+\mathcal{O}(\vert\tau-\tau_0\vert^2)\right)\nonumber\\
	& \quad \times \left(\gamma_2'(\tau_0)+\gamma_2''(\tau_0)(\tau-\tau_0)+\mathcal{O}(\vert\tau-\tau_0\vert^2)\right)\nonumber\\
	&=-\kappa(\tau)\left(\gamma_1(\tau_0)\gamma_2'(\tau_0)+(\gamma_1(\tau_0)\gamma_2''(\tau_0)+\gamma_1'(\tau_0)\gamma_2'(\tau_0))(\tau-\tau_0)\right)+\mathcal{O}(\vert\tau-\tau_0\vert^2) \label{calc5}, \\
	\notag\\
&	x(\tau)x''(\tau)=\tan\theta_0x(\tau)\gamma_2''(\tau)\notag\\
	&=\tan\theta_0\kappa(\tau)x(\tau)\gamma_1'(\tau)\nonumber\\
	&=\tan\theta_0\kappa(\tau)(\gamma_1(\tau_0)+\gamma'_1(\tau_0)(\tau-\tau_0)+\mathcal{O}(\vert\tau-\tau_0\vert^2))\nonumber\\
	&\quad \times \left(\gamma_1'(\tau_0)+\gamma_1''(\tau_0)(\tau-\tau_0)+\mathcal{O}(\vert\tau-\tau_0\vert^2)\right)\nonumber\\
	&=\kappa(\tau)\tan\theta_0\left(\gamma_1(\tau_0)\gamma'_1(\tau_0)+(\gamma_1(\tau_0)\gamma_1''(\tau_0)+(\gamma_1'(\tau_0))^2)(\tau-\tau_0)\right)+\mathcal{O}(\vert \tau-\tau_0\vert^2). \label{calc6}
\end{align}
Thus, by \eqref{calc5} and \eqref{calc6}, the part $II$ becomes 
\begin{align} \label{part II}
	II&=-\kappa(\tau)\gamma_1(\tau_0)\left(\gamma_2'(\tau_0)+\tan\theta_0\gamma_1'(\tau_0)\right)\\
	&  -\kappa(\tau)\left(\gamma_1(\tau_0)\gamma_2''(\tau_0)+\tan\theta_0\gamma_1(\tau_0)\gamma_1''(\tau_0)+\gamma_1'(\tau_0)\gamma_2'(\tau_0)+\tan\theta_0(\gamma_1'(\tau_0))^2\right)(\tau-\tau_0)\nonumber\\
	&\quad +\mathcal{O}((\tau-\tau_0)^2)\nonumber \\
	&=\frac{\kappa(\tau)\gamma_1(\tau_0)}{\vert \gamma_2'(\tau_0)\vert}+\frac{\kappa(\tau)\gamma_1'(\tau_0)}{\vert \gamma_2'(\tau_0)\vert}(\tau-\tau_0)+\mathcal{O}((\tau-\tau_0)^2).\notag
\end{align}
Hence, from \eqref{I+II},\eqref{part I}, and \eqref{part II}, one obtains that 
\begin{align} \label{conc1}
	&y(\tau)y''(\tau)+(y'(\tau))^2=I+II
	\\
	&=(\gamma_1'(\tau)+x'(\tau))\frac{\kappa(\tau_0)}{\vert \gamma_2'(\tau_0)\vert}(\tau-\tau_0)+\frac{\kappa(\tau)\gamma_1(\tau_0)}{\vert \gamma_2'(\tau_0)\vert}+\frac{\kappa(\tau)\gamma_1'(\tau_0)}{\vert \gamma_2'(\tau_0)\vert}(\tau-\tau_0)+\mathcal{O}((\tau-\tau_0)^2).\notag
\end{align}
Also, from the calculation \eqref{calc1} and Taylor expansion at $\tau=\tau_0$, we have 
\begin{align} \label{calc7}
	\begin{split}
	y(\tau)&=\sqrt{\gamma_1^2(\tau)-x^2(\tau)}=\sqrt{\frac{\kappa(\tau_0)}{2\vert \gamma_2'(\tau_0)\vert}(\tau-\tau_0)^2+\mathcal{O}((\tau-\tau_0)^3)}\sqrt{\gamma_1(\tau)+x(\tau)},\\
	(y'(\tau))^2&=(y'(\tau_0))^2+2y'(\tau_0)y''(\tau_0)(\tau-\tau_0)+\mathcal{O}((\tau-\tau_0)^2)\\
	& =\frac{\kappa(\tau_0)\gamma_1(\tau_0)}{\vert \gamma_2'(\tau_0)\vert}+2y'(\tau_0)y''(\tau_0)(\tau-\tau_0)+\mathcal{O}((\tau-\tau_0)^2).
	\end{split}
\end{align}
We substitute \eqref{calc7} for the left-hand side in \eqref{conc1}, use \eqref{limit}, and then divide both sides by $\vert \tau-\tau_0\vert=-(\tau-\tau_0)$ to get 
\begin{align*}
	&\sqrt{\frac{\kappa(\tau_0)}{2\vert \gamma_2'(\tau_0)\vert}+\mathcal{O}(\tau-\tau_0)}\sqrt{\gamma_1(\tau)+x(\tau)}y''(\tau)+2\sqrt{\frac{\kappa(\tau_0)\gamma_1(\tau_0)}{\vert \gamma_2'(\tau_0)\vert}}y''(\tau_0)\\
	&=-\frac{\gamma_1(\tau_0)}{\vert \gamma_2'(\tau_0)\vert}\frac{\kappa(\tau)-\kappa(\tau_0)}{\tau-\tau_0}-(\gamma_1'(\tau)+x'(\tau))\frac{\kappa(\tau_0)}{\vert \gamma_2'(\tau_0)\vert}-\frac{\kappa(\tau)\gamma_1'(\tau_0)}{\vert \gamma_2'(\tau_0)\vert}
+\mathcal{O}(\tau-\tau_0).
\end{align*}
As $\tau\rightarrow \tau_0$, one obtains that 
\begin{align} \label{y''}
	3\sqrt{\frac{\kappa(\tau_0)\gamma_1(\tau_0)}{\vert \gamma_2'(\tau_0)\vert}}y''(\tau_0)&=-\frac{\gamma_1(\tau_0)\kappa'(\tau_0)}{\vert \gamma_2'(\tau_0)\vert}-3\frac{\kappa(\tau_0)\gamma_1'(\tau_0)}{\vert \gamma_2'(\tau_0)\vert}, \nonumber \\
	y''(\tau_0)&=-\frac{1/3\gamma_1(\tau_0)\kappa'(\tau_0)+\kappa(\tau_0)\gamma_1'(\tau_0)}{\sqrt{\kappa(\tau_0)\gamma_1(\tau_0)\vert \gamma_2'(\tau_0)\vert}}.
\end{align}
Hence, \eqref{d^2z/dy^2},\eqref{limit}, and \eqref{y''} yield $\frac{d^2z}{dy^2}$ at $\tau=\tau_0$:
\begin{align*}
	\left. \frac{d^2z}{dy^2}\right \vert_{\tau=\tau_0}&=\frac{z''(\tau_0)y'(\tau_0)-z'(\tau_0)y''(\tau_0)}{(y'(\tau_0))^3}\\
	&=\frac{-\kappa(\tau_0){\gamma_1'(\tau_0)}\sqrt{\frac{\kappa(\tau_0)\gamma_1(\tau_0)}{\vert \gamma_2'(\tau_0)\vert}}+\gamma_2'(\tau_0)\frac{1/3\gamma_1(\tau_0)\kappa'(\tau_0)+\kappa(\tau_0){\gamma_1'(\tau_0)}}{\sqrt{\kappa(\tau_0)\gamma_1(\tau_0)\vert \gamma_2'(\tau_0)\vert}}}{\left(-\sqrt{\frac{\kappa(\tau_0)\gamma_1(\tau_0)}{\vert \gamma_2'(\tau_0)\vert}}\right)^{3}}\\
	&=\frac{\gamma_1'(\tau_0)\vert \gamma_2'(\tau_0)\vert}{\kappa(\tau_0)\gamma_1^2(\tau_0)}(\gamma_1(\tau_0)\kappa(\tau_0)+\vert \gamma_2'(\tau_0)\vert)+\frac{\vert \gamma_2'(\tau_0)\vert^2 \kappa'(\tau_0)}{3\kappa^2(\tau_0)\gamma_1(\tau_0)},\\
	\left.\frac{d^2\tilde{z}}{dy^2}\right \vert_{\tau=\tau_0}&=\frac{\gamma_1'(\tau_0)}{\kappa(\tau_0)\gamma_1^2(\tau_0)}(\gamma_1(\tau_0)\kappa(\tau_0)+\vert \gamma_2'(\tau_0)\vert)+\frac{\vert \gamma_2'(\tau_0)\vert \kappa'(\tau_0)}{3\kappa^2(\tau_0)\gamma_1(\tau_0)},
\end{align*}
where $\tilde{z}=\frac{z}{\vert \gamma_2'(\tau_0)\vert}$. Finally, at $\tau=\tau_0$, we derive
\begin{align*}
	 \left.\frac{d^2\tilde{z}}{dy^2}\right \vert_{\tau=\tau_0}&=\frac{\gamma_1'(\tau_0)}{\kappa(\tau_0)\gamma_1^2(\tau_0)}(\gamma_1(\tau_0)\kappa(\tau_0)+\vert \gamma_2'(\tau_0)\vert)+\frac{\vert \gamma_2'(\tau_0)\vert \kappa'(\tau_0)}{3\kappa^2(\tau_0)\gamma_1(\tau_0)}.
\end{align*}

So far, we have only considered $\left(y>0,\tilde{z}>\frac{\gamma_2(\tau_0)}{\vert \gamma_2'(\tau_0)\vert}\right)$ case. Since our domain is $z$-axial symmetry, we have $\tilde{z}(y)=\tilde{z}(-y)$, which implies that  $\frac{d\tilde {z}(y)}{dy}=-\frac{d\tilde {z}(-y)}{dy}$ and $\frac{d^2\tilde{z}(y)}{dy^2}=\frac{d^2 \tilde{z}(-y)}{dy^2}$. Thus, for the case  $\left(y<0, \tilde{z}>\frac{\gamma_2(\tau_0)}{\vert \gamma_2'(\tau_0)\vert}\right)$, we have 
\begin{align*}
	&\left.\frac{d\tilde{z}}{dy}\right \vert_{\tau=\tau_0}=-\sqrt{\frac{\vert \gamma_2'(\tau_0)\vert}{\kappa(\tau_0)\gamma_1(\tau_0)}}, \\
	&\left.\frac{d^2\tilde{z}}{dy^2}\right \vert_{\tau=\tau_0}=\frac{\gamma_1'(\tau_0)}{\kappa(\tau_0)\gamma_1^2(\tau_0)}(\gamma_1(\tau_0)\kappa(\tau_0)+\vert \gamma_2'(\tau_0)\vert)+\frac{\vert \gamma_2'(\tau_0)\vert \kappa'(\tau_0)}{3\kappa^2(\tau_0)\gamma_1(\tau_0)} .
\end{align*}
Because two cases $\left(y>0,\tilde{z}>\frac{\gamma_2(\tau_0)}{\vert \gamma_2'(\tau_0)\vert}\right)$ and $\left(y<0,\tilde{z}<\frac{\gamma_2(\tau_0)}{\vert \gamma_2'(\tau_0)\vert}\right)$ share the same values $\frac{d\tilde{z}}{dy}$ and $\frac{d^2\tilde{z}}{dy^2}$ at $\tau_0$ , we obtain the values for the region  $\left(y<0,\tilde{z}<\frac{\gamma_2(\tau_0)}{\vert \gamma_2'(\tau_0)\vert}\right)$ . On the other hand, for the case $\left(y>0, \tilde{z}<\frac{\gamma_2(\tau_0)}{\vert \gamma_2'(\tau_0)\vert}\right)$, values $\frac{d\tilde{z}}{dy}$ and $\frac{d^2\tilde{z}}{dy^2}$ at $\tau_0$ can be deduced from the case $\left(y<0, \tilde{z}>\frac{\gamma_2(\tau_0)}{\vert \gamma_2'(\tau_0)\vert}\right)$. \\

\textbf{Step 3.} In this argument, we change the direction of the normal vector $N_{\sigma(\tau_0,0)}$ from outward to inward. Thus, in Lemma \ref{Euler thm}, we consider 
\begin{align*}
	n(\gamma(\tau_0))=(-\gamma_2'(\tau_0),0,\gamma_1'(\tau_0)). 
\end{align*}
By Lemma \ref{Euler thm}, the sign of the normal curvature $\kappa_n$ at $\sigma(\tau_0,0)$ does not change if the direction $v\in T_{\sigma} \cap\S^2$ at $\sigma(\tau_0,0)$ is between tangent lines. When an angle between $\hat{\tilde{y}}$ and the direction $v$ in the tangent space $T_{\sigma(\tau_0,0)}$ is in $(-\vartheta,\vartheta)\cup(\pi-\vartheta,\pi+\vartheta)$ where $\vartheta$ is defined in \eqref{tan vartheta}, the normal curvature $\kappa_n$ is negative. Thus, there exists $\varepsilon>0$ such that $\xi(\sigma(\tau_0,0)+tv)<0$ for any $t\in (-\varepsilon,0)\cup (0,\varepsilon) $. By Remark \ref{inflection point}, $(\sigma(\tau_0,0),v)$ cannot be an inflection point. \\
Similarly, when an angle between $\hat{\tilde{y}}$ and the direction $v$ in the tangent space $T_{\sigma(\tau_0,0)}$ is in $(\vartheta,\pi-\vartheta)\cup(\pi+\vartheta,2\pi-\vartheta)$, the normal curvature $\kappa_n$ is positive. Hence, there exists $\varepsilon>0$ such that $\xi(\sigma(\tau_0,0)+tv)>0$ for any $t\in (-\varepsilon,0)\cup(0,\varepsilon)$, so that $(\sigma(\tau_0,0),v)$ cannot be an inflection point. Thus, $(\sigma(\tau_0,0),v)$ can be an inflection point only if the direction $v$ have angles $\vartheta$ and $2\pi-\vartheta$ in $T_{\sigma(\tau_0,0)}$.  

To occur an inflection grazing at $\sigma(\tau_0,0)$, $\frac{d^2\tilde{z}}{dy^2}$ at $\tau=\tau_0$ must not be zero. Thus, we exclude the part that satisfies $\left.\frac{d^2\tilde{z}}{dy^2}\right\vert_{\tau=\tau_0}=0$. Since $\gamma$ is an analytic convex curve and $\gamma_1(\tau)>0$ for $\tau\in(\tau_{1,*},\tau_{2,*})$,
\begin{align*}
	\frac{d^2\tilde{z}}{dy^2}(\tau)=0 \Leftrightarrow \frac{\gamma_1'(\tau)}{\gamma_1(\tau)}(\gamma_1(\tau)\kappa(\tau)+\vert \gamma_2'(\tau)\vert)+\frac{\vert \gamma_2'(\tau)\vert \kappa'(\tau)}{3\kappa(\tau)}=0,
\end{align*}
and the function 
\begin{align*}
	 h(\tau):= \frac{\gamma_1'(\tau)}{\gamma_1(\tau)}(\gamma_1(\tau)\kappa(\tau)+\vert \gamma_2'(\tau)\vert)+\frac{\vert \gamma_2'(\tau)\vert \kappa'(\tau)}{3\kappa(\tau)}
\end{align*}
is also analytic. Notice that if $1$-dimensional real analytic function is not identically zero, then the zero set has no limit point in $(\tau_{1,*},\tau_{2,*})$ by Lemma \ref{analytic_zeros}. Also, 
\begin{align*}
	h(\tau_{i,*})=\gamma_1'(\tau_{i,*})\kappa(\tau_{i,*})\neq 0, \quad \text{for} \; i=1,2,
\end{align*}
which means that the zero set $Z_h$ of $h$ in $(\tau_{1,*},\tau_{2,*})$ is equal to the zero set in $[\tau_{1,*},\tau_{2,*}]$. Then, the zeros of $h(\tau)$ in $(\tau_{1,*},\tau_{2,*})$ consists of only finite points, say $\bigcup_{j=1}^\ell\{\tau_j^{\mathbf{z}}\}$ for some constant $\ell<\infty \;(\tau_{1,*}<\tau_1^{\mathbf{z}}<\tau_2^{\mathbf{z}}<\cdots<\tau_\ell^{\mathbf{z}}<\tau_{2,*})$.\\

Recall that we set an angle $\vartheta$ satisfying $\tan \vartheta = \sqrt{\frac{\vert \gamma_2'(\tau)\vert}{\kappa(\tau)\gamma_1(\tau)}}$. Due to axis-symmetry about $z$-axis and the specular reflection boundary condition, trajectory preserves the angular momentum. Thus, WLOG, we only consider trajectories with positive angular momentum.  Then, for $\tau \in (\tau_{1,*},\tau_{2,*}) \backslash \cup_{j=1}^l \{\tau_j^{\mathbf{z}}\}$ and  $\varphi\in [0,2\pi)$, we only consider two cases $(\sigma(\tau,\varphi),I^1(\tau,\varphi))$ and $(\sigma(\tau,\varphi),I^2(\tau,\varphi))$ where directions $I^1(\tau,\varphi)$ and $I^2(\tau,\varphi)$ have angles $\vartheta$ and $2\pi-\vartheta$ respectively in the tangent plane. By Remark \ref{inflection point}, we have $(\sigma(\tau,\varphi),I^1(\tau,\varphi))\in \gamma_0^{I_+}$ and $(\sigma(\tau,\varphi),I^2(\tau,\varphi))\in \gamma_0^{I_-}$.   
\end{proof}

Next, we introduce concave grazing directions using $I^1(\tau,\varphi)$ and $I^2(\tau,\varphi)$ defined in Lemma \ref{inflection}. Although $I^1(\tau,\varphi)$ and $I^2(\tau,\varphi)$ were not defined for $\tau \in Z_h=\cup_{j=1}^{\ell}  \{ \tau_z^{\textbf{z}}\}$ in Lemma \ref{inflection}, we will define them as the directions in the tangent plane $T_{\sigma(\tau,\varphi)}$ that make angles of $\vartheta$ and $2\pi-\vartheta$ (defined in the proof of Lemma \ref{inflection}), respectively, to define the concave grazing direction at $\sigma(\tau,\varphi)$ for $\tau \in Z_h$. Note that $(\sigma(\tau,\varphi),I^1(\tau,\varphi)), (\sigma(\tau,\varphi),I^2(\tau,\varphi)) \notin \gamma_0^{I}$ for $(\tau,\varphi) \in Z_h \times [0,2\pi)$.

\begin{definition}
At $\sigma(\tau,\varphi)\in \p \O$, let the direction $v_c\in\S^2$ be concave grazing direction which is parametrized by parameter $\eta\in\R^1$. For $(\sigma(\tau,\varphi), v_c(\eta))\in\gamma_0^C$, $v_c$ must be between directions $I^1(\tau,\varphi)$ and $I^2(\tau,\varphi)$ defined in Lemma \ref{inflection} and above. Thus, $v_c$ can be parametrized as 
\begin{align} \label{concave para}
v_c=v_c(\eta):= \frac{\eta I^1(\tau,\varphi)+(1-\eta)I^2(\tau,\varphi)}{|\eta I^1(\tau,\varphi)+(1-\eta)I^2(\tau,\varphi)|},
\end{align}
for $0<\eta <1$. See Figure \ref{remark 1}. Directions between $I_{1}$ and $I_{2}$ in the shaded region corresponds to concave grazings.  
\end{definition}

Notice that the sign of the normal curvature is always negative if the direction at $\sigma(\tau,\varphi)$ has $v_c$ from Euler's theorem. Thus, there exists $\varepsilon>0$ such that $\xi(\sigma(\tau,\varphi)+tv_c(\eta))<0$ for all $t\in (-\varepsilon,0)\cup(0,\varepsilon)$. In other words, $(\sigma(\tau,\varphi),v_c(\eta))\in \gamma_0^C$.   

\section{Control the Number of Bounces}
\subsection{Finite number of bounce away from inflection grazing}
Recall the definition of velocity set \eqref{V^N}. 
Let $(t,x,v)\in [0,T]\times \overline\Omega\times \mathbb{V}^N$. For each $i\in \mathbb{N}$, $x^i(x,v)\in \p\Omega$ can be parametrized as
\begin{align*}
	x^i(x,v) = \sigma(\tau_i, \varphi_i), \quad (\tau_i,\varphi_i) \in[a,b)\times (-2\pi \mathbf{w}, 0 ] := U.
\end{align*}
where $\mathbf{w}$ defined in \eqref{winding number} with maximal travel length $L=NT$. \\

We define non-negative angular momentum (since we consider only trajectory with positive orientation WLOG) of $(x,v)$, i.e,
\Be \label{def_ang moment}
\begin{split}
	\omega(x,v) &:= |(x\times \hat{z})\cdot v|  \\
	&= \text{angular momentum of a trajectory with respect to $z$-axis which passes $(x,v)$} \\
	&= \omega( X(s; t,x,v), V(s;t,x,v)).
\end{split}
\Ee
This implies that the angular momentum is conserved due to $z$-axial symmetric domain and specular reflection BC. For $x\in \O$, this \eqref{def_ang moment} can be easily checked:
\begin{align*}
\frac{d}{ds} \omega( X(s; t,x,v), V(s;t,x,v)) = \{(V \times \hat{z}) \cdot V 
+ (X \times \hat z) \cdot 0
\}\frac{X \times \hat z \cdot V}{|X \times \hat z \cdot V |}=0.
\end{align*}
Now, we only consider the case $x\in \partial \O$. WLOG we assume $x\in \partial S_0$ where $S_0$ is $\varphi=0$ cross section. It suffices to check that 
\begin{align*}
	\vert (x\times \hat{z}) \cdot v \vert = \vert (x\times \hat{z}) \cdot R(x)v\vert,
\end{align*}
where $R(x)=I-2n(x)\otimes n(x)$. Since $x\times \hat{z}$ is parallel to $\hat{\varphi}$ which is normal to $n(x)$, 
\begin{align*}
	\vert (x\times \hat{z})\cdot R(x)v \vert = \vert(x\times \hat{z})\cdot(v-2(n(x)\cdot v)n(x))\vert =\vert (x\times \hat{z})\cdot v\vert.  \\
\end{align*}

\begin{lemma} \label{compare}
Define $\Delta \tau_i=\tau_i-\tau_{i-1}$ and $\Delta \varphi_i=\varphi_i-\varphi_{i-1}$ for $i=1,2,\cdots$. Let $\sigma_i:=\sigma(\tau_i,\varphi_i)\in \p\O$ with $\tau_{i}\in [\tau_{1,*}-\varepsilon, \tau_{2,*}+\varepsilon]$ for $i=0,1,\cdots$. If the trajectory hits $\sigma_{i-1}$ and then $\sigma_i$, we denote the projection of $\sigma_{i-1}$ into the tangent plane $T_{\sigma(\tau_i,\varphi_i)}$ as $Proj \sigma_{i-1}$. We define an angle $\eta_i$ between $\overrightarrow{\sigma_{i-1,\widetilde{yz}}\sigma_{i,\widetilde{yz}}}$ and $\tilde{z}$-axis where $\sigma_{i-1,\widetilde{yz}}$ and $\sigma_{i,\widetilde{yz}}$ are representation of $Proj \sigma_{i-1}$ and $\sigma_i$ in $T_{\sigma(\tau_i,\varphi_i)}$ using \eqref{new coord in tp} and \eqref{new direction}. See Figure \ref{remark 1}. (Note that we will specify the angle $\eta_i$ in this proof \eqref{proj ang}.) 
Assume 
\begin{align} \label{assump 1}
	\eta_i \leq
		\frac{\pi}{2} - C\varepsilon \quad &\text{if } \tau_i\in[\tau_{1,*}-\varepsilon,\tau_{1,*}+\varepsilon] \cup[\tau_{2,*}-\varepsilon,\tau_{2,*}+\varepsilon],
\end{align}
in addition for all $i\geq 1$ and fixed small $\varepsilon\ll 1$, where $C$ is some generic constant. 
There exists $\delta \ll 1$ such that if there exists $I\in \mathbb{N}$ such that 
\begin{align} \label{assump 2}
	\vert \Delta \tau_i\vert < \delta \quad \text{and} \quad \vert \Delta \varphi_i\vert < \delta, \quad \forall i\geq I, 
\end{align}
we have
\begin{align} \label{tau and varphi}
	\vert \Delta \varphi_i \vert \leq C_{\varepsilon} \vert \Delta \tau_i\vert, \quad \forall i \geq I, 
\end{align}
where $C_{\varepsilon}$ is a constant depending on $\varepsilon$.  
\end{lemma}

\begin{remark} 
In Lemma \ref{compare}, we consider the trajectory segment between $\sigma_0$ and $\sigma_1$ (See Figure \ref{remark 1}). If $\tau_1\in[\tau_{1,*}-\varepsilon,\tau_{1,*}+\varepsilon] \cup[\tau_{2,*}-\varepsilon,\tau_{2,*}+\varepsilon]$, then $\vartheta_1=\tan^{-1}\left(\sqrt{\frac{\kappa(\tau_1)\gamma_1(\tau_1)}{\vert \gamma_2'(\tau_1)\vert}}\right)$ is near $\frac{\pi}{2}$. So, the two directions $I_1(\tau_1,\varphi_1)$ and $I_2(\tau_1,\varphi_1)$ are very close and almost identical. Later, in Lemma \ref{lem Xe}, we will exclude $I_{1}$ and $I_{2}$ directions uniformly using the angular momentum conservation. 
(If $\varepsilon$ is sufficiently small, near $\vartheta = \frac{\pi}{2}$ directions are removed as in shaded region in Figure \ref{remark 1}.) Hence, for $\tau_1\in[\tau_{1,*}-\varepsilon,\tau_{1,*}+\varepsilon] \cup[\tau_{2,*}-\varepsilon,\tau_{2,*}+\varepsilon]$, we assume that $\eta_1 \leq \frac{\pi}{2}-C\varepsilon$ for some constant $C>0$ and $\varepsilon \ll 1$. If $\tau_{1} \in [\tau_{1,*}+\varepsilon, \tau_{2,*}-\varepsilon]$, then direction of $I_{1}$ and $I_{2}$ are uniformly away from $\eta=\frac{\pi}{2}$ angle.  
\begin{figure}[h]
\centering
\includegraphics[width=13cm]{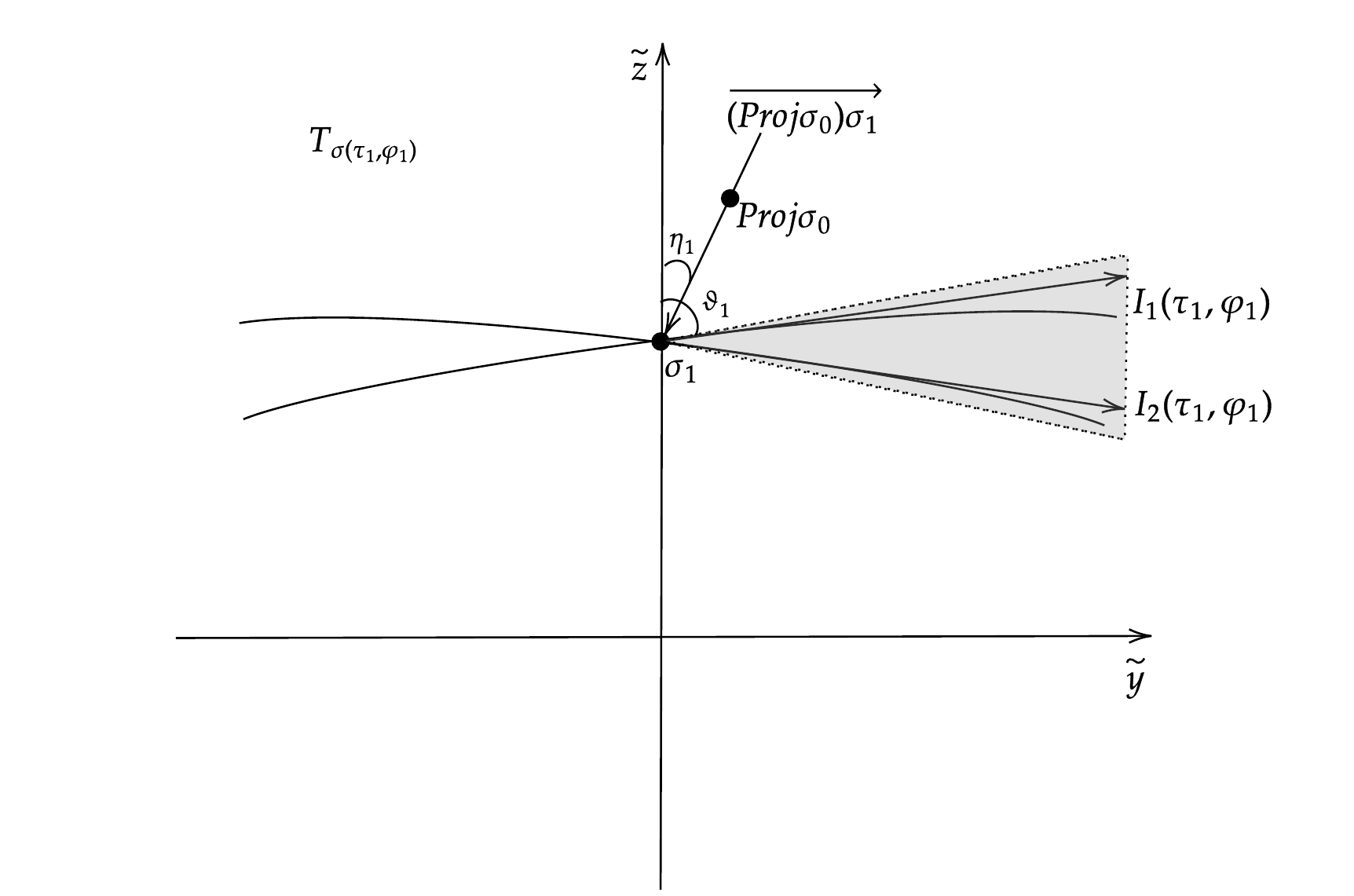}
\caption{The trajectory between $\sigma_0$ and $\sigma_1$ in the tangent plane $T_{\sigma(\tau_1,\varphi_1)}$ ($\tau_1$ near $\tau_{1,*}$ or $\tau_{2,*}$)}
\label{remark 1}
\end{figure}
\end{remark}

\begin{proof}
For notational convenience, we set $I=1$ and consider two boundary points $\sigma_0$ and $\sigma_1$
\begin{align*}
	\sigma_i:=\sigma(\tau_i,\varphi_i)=(\gamma_1(\tau_i)\cos\varphi_i,\gamma_1(\tau_i)\sin\varphi_i,\gamma_2(\tau_i)), \quad i=0,1,
\end{align*} 
satisfying $\vert \Delta \tau_1 \vert, \vert \Delta \varphi_1\vert \ll 1$. Since our domain $\O$ is symmetric with respect to $\varphi$, we set $\varphi_0=-\varphi$ and $\varphi_1=0$ to satisfy $\Delta \varphi_1=\varphi$. The normal vector $N_{\sigma(\tau_1,0)}$ at $\sigma_1$ is
\begin{align*}
	N_{\sigma(\tau_1,0)}=(\gamma_2'(\tau_1),0,-\gamma_1'(\tau_1)),
\end{align*}
and the tangent plane at $\sigma_1$ is 
\begin{align*}
	\{(x,y,z)\in\R^3:\gamma_2'(\tau_1)(x-\gamma_1(\tau_1))-\gamma_1'(\tau_1)(z-\gamma_2(\tau_1))=0\}. 
\end{align*}
We project the point $\sigma_0$ into the tangent plane at $\sigma_1$ and calculate the projected point
\begin{align*}
	Proj \sigma_0 & = \sigma_0-((\sigma_0-\sigma_1)\cdot N_{\sigma(\tau_1,0)})N_{\sigma(\tau_1,0)} \\
	& =(\gamma_1(\tau_0)\cos \varphi -A\gamma_2'(\tau_1),-\gamma_1(\tau_0)\sin\varphi,\gamma_2(\tau_0)+A\gamma_1'(\tau_1)),
\end{align*}
where $A=(\sigma_0-\sigma_1)\cdot N_{\sigma(\tau_1,0)}=(\gamma_1(\tau_0)\cos\varphi-\gamma_1(\tau_1))\gamma_2'(\tau_1)+(\gamma_2(\tau_1)-\gamma_2(\tau_0))\gamma_1'(\tau_1)$. Using the new coordinate \eqref{new coord in tp} and \eqref{new direction} in the tangent plane $T_{\sigma(\tau_1,0)}$ at $\sigma_1$, the points $Proj \sigma_0$ and $\sigma_1$ become 
\begin{align*}
	Proj \sigma_0 \rightarrow \sigma_{0,\widetilde{yz}}=\Big(-\gamma_1(\tau_0)\sin\varphi,\frac{\gamma_2(\tau_0)+A\gamma_1'(\tau_1)}{\vert \gamma_2'(\tau_1)\vert}\Big)_{\widetilde{yz}}, \quad \sigma_1 \rightarrow \sigma_{1,\widetilde{yz}}=\Big(0,\frac{\gamma_2(\tau_1)}{\vert \gamma_2'(\tau_1)\vert}\Big)_{\widetilde{yz}},
\end{align*}
where $(a,b)_{\widetilde{yz}}=a\hat{\tilde{y}}+b\hat{\tilde{z}}$ for $a,b\in\R$. Let $\eta_1$ be the angle between the $\tilde{z}$-axis and $\overrightarrow{\sigma_{0,\widetilde{yz}}\sigma_{1,\widetilde{yz}}}$, and then 
\begin{align} \label{proj ang}
	\tan\eta_1 = \frac{\gamma_1(\tau_0)\vert \gamma_2'(\tau_1)\vert \sin \varphi}{\gamma_2(\tau_1)-\gamma_2(\tau_0)-A\gamma_1'(\tau_1)}.
\end{align}
Using Taylor expansion yields that 
\begin{align*}
	 &\gamma_2(\tau_1)-\gamma_2(\tau_0)-A\gamma_1'(\tau_1)\\
	  &= (\gamma_2(\tau_1)-\gamma_2(\tau_0))(1-(\gamma_1'(\tau_1))^2)+(\gamma_1(\tau_1)-\gamma_1(\tau_0)\cos\varphi)\gamma_1'(\tau_1)\gamma_2'(\tau_1)\\
	&=(\gamma_2(\tau_1)-\gamma_2(\tau_0))(\gamma_2'(\tau_1))^2+(\gamma_1(\tau_1)-\gamma_1(\tau_0))\gamma_1'(\tau_1)\gamma_2'(\tau_1)\\
	&\quad+(1-\cos\varphi)\gamma_1(\tau_0)\gamma_1'(\tau_1)\gamma_2'(\tau_1)\\
	&=\gamma_2'(\tau_1)(\Delta \tau_1 +\mathcal{O}(\Delta \tau_1^2)+\mathcal{O}(\varphi^2)),
\end{align*}
where we have used $(\gamma_1'(\tau_1))^2+(\gamma_2'(\tau_1))^2=1$. \\
\indent Meanwhile, note that we are assuming $|\Delta \varphi_{1}|, |\Delta \tau_{1}| \ll 1$.  In Figure \ref{proof 1}, imagine a sufficiently small circle $B_1$ near $\sigma_1$ so that $\sigma_0$ locates inside of $B_1$. Since the trajectory hits $\sigma_0$ and $\sigma_1$, $(\sigma_1,  \overrightarrow{\sigma_0\sigma_1})$ must belong to $\gamma_0$ or $\gamma_{-}$, i.e., $(\sigma_1-\sigma_0)\cdot N_{\sigma(\tau_1,0)} \geq 0$. (Shaded region in Figure \ref{proof 1})

\begin{figure}[h]
\centering
\includegraphics[width=13cm]{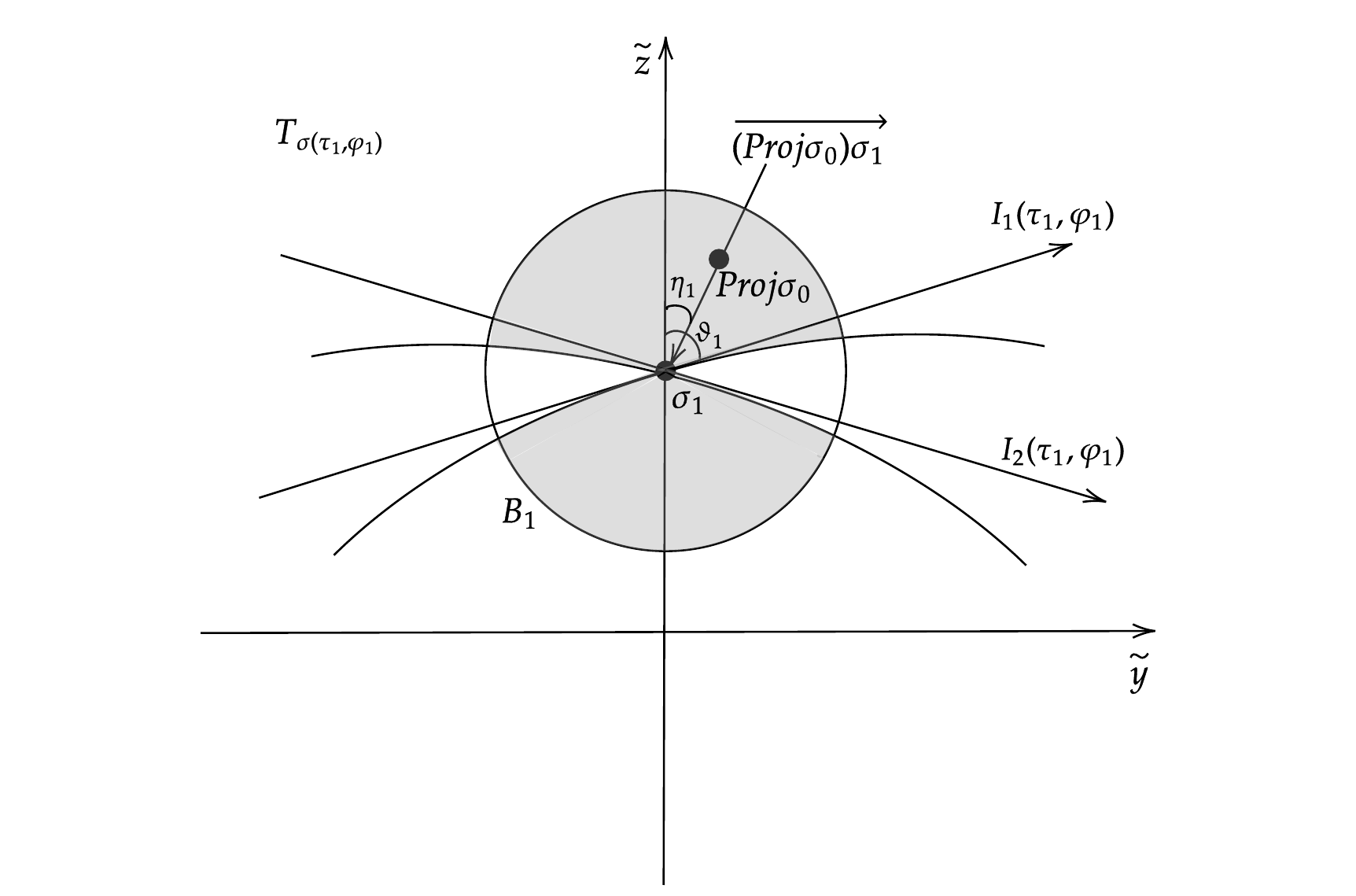}
\caption{The trajectory between $\sigma_0$ and $\sigma_1$ in the tangent plane $T_{\sigma(\tau_1,\varphi_1)}$ in a small ball $B_1$}
\label{proof 1}
\end{figure}

First, let  us consider $\tau_{i}\in [\tau_{1,*}+\varepsilon, \tau_{2,*}-\varepsilon]$ case. Considering the direction of zero normal curvature, (limit of $y^{\prime}(\tau)$ in \eqref{limit}) and sufficiently small circle $B_1$ (with sufficiently small radius $\delta$ by \eqref{assump 2}), we obtain 
\Be \label{angle bound}
	\tan\eta_1 \leq (1 + C\delta) \sqrt{\frac{\kappa(\tau_1)\gamma_1(\tau_1)}{\vert \gamma_2'(\tau_1)\vert}}, 
\Ee
for some $C>0$ when $\tau_1\notin [\tau_{1,*}-\varepsilon,\tau_{1,*}+\varepsilon]\cup[\tau_{2,*}-\varepsilon,\tau_{2,*}+\varepsilon]$. Coefficient $(1+C\delta)$ comes from the fact that tangential direction of $I_{1,2}$ and intersection curve $T_{\sigma_{1}}\cap \p\O$ are nearly parallel in sufficiently small ball $B_1$. (See Figure \ref{proof 1}
.)
 If $\tau_1\notin [\tau_{1,*}-\varepsilon,\tau_{1,*}+\varepsilon]\cup[\tau_{2,*}-\varepsilon,\tau_{2,*}+\varepsilon]$, then 
\begin{align*}
	\sqrt{\frac{\kappa(\tau_1)\gamma_1(\tau_1)}{\vert \gamma_2'(\tau_1)\vert}}\leq C_{1,\varepsilon},
\end{align*} 
by continuity of $\gamma_2'(\tau)$ and $\sup_{\tau \in [a,b]} \vert \kappa(\tau)\gamma_1(\tau)\vert <\infty$. \\
\indent On the other hand, for $\tau_1\in [\tau_{1,*}-\varepsilon,\tau_{1,*}+\varepsilon]\cup[\tau_{2,*}-\varepsilon,\tau_{2,*}+\varepsilon]$, we have
\begin{align*}
	\tan \eta_1 \leq \tan(\frac{\pi}{2}-C\varepsilon)=C_{2,\varepsilon},
\end{align*}
due to our assumption \eqref{assump 1}. 
 Since $\varphi =\Delta \varphi_1$, we have 
\begin{align*}
	\left \vert \frac{\sin(\Delta \varphi_1)}{\Delta \tau_1+\mathcal{O}(\Delta \tau_1^2)+\mathcal{O}(\Delta \varphi_1^2)}\right \vert\leq C_{\varepsilon},
\end{align*}
where $C_{\varepsilon}:=\max\{(1+C\delta)C_{1,\varepsilon},C_{2,\varepsilon}\}$. Since we assume 
smallness condition $\vert \Delta \tau_1\vert,\vert\Delta \varphi_1\vert < \delta \ll 1$, we have $\vert \Delta \varphi_1\vert \leq C_{\varepsilon} \vert \Delta \tau_1\vert$. By applying the above argument for every bound, we get \eqref{tau and varphi}.  
\end{proof}

The following analyltic implicit function theorem will be used in the proof of Lemma \ref{difference}.
\begin{lemma} \label{analytic IFT} \cite{KP2002}
	Let $f:\R^{n+m}\rightarrow \R^m$ be a real analytic function in a neighborhood of $(x_0,y_0)$ for some $(x_0,y_0)\in \R^n \times \R^m$. If $f(x_0,y_0)=0$ and the Jacobian matrix 
\begin{align*}
	\left[ \frac{\partial f_i}{\partial y_j}(x_0,y_0) \right]
\end{align*}	
is invertible, then there exists a function $g:\R^n \rightarrow \R^m$ which is real-analytic in a neighborhood of $x_0$ and satisfies 
\begin{align*}
	f(x,g(x))=0
\end{align*}
in a neighborhood of $x_0$.   
\end{lemma}

\begin{lemma} \label{difference} 
We suppose the same assumption in Lemma \ref{compare}. Using the same notation in Lemma \ref{compare}, we also assume that 
\begin{align} \label{assump 3}
	\vert \tan \eta_i - \tan \vartheta_i \vert \geq \varepsilon, \quad \text{for } \tau_i \in (Z_h)_{\varepsilon},
\end{align}
where $\tan\vartheta_i:=\sqrt{\frac{\kappa(\tau_i)\gamma_1(\tau_i)}{\vert \gamma_2'(\tau_i)\vert}}$ and $(Z_h)_{\varepsilon}$ is $\varepsilon$-neighborhood of $Z_h$ in Lemma \ref{inflection}
\begin{align} \label{Z_h nbh}
	(Z_h)_{\varepsilon}:=\bigcup_{i=1}^{\ell} (\tau_i^{\mathbf{z}}-\varepsilon, \tau_i^{\mathbf{z}}+\varepsilon), \quad Z_h = \{\tau_i^{\mathbf{z}}\}_{i=1}^{\ell}, \quad \textrm{for some constant } \ell<\infty,
\end{align}
so that $(\sigma_i,\overrightarrow{\sigma_{i-1}\sigma_i})$ is not an inflection of the zero curvature curve $\p\O\cap T_{\sigma(\tau_i,\varphi_i)}$ on tangent plane $T_{\sigma(\tau_i,\varphi_i)}$, i.e., $\p\O\cap T_{\sigma(\tau_i,\varphi_i)}$ is concave or convex in $T_{\sigma(\tau_i,\varphi_i)}$. Then, we have the following relation:
\begin{align*}
\vert \Delta \tau_{i+1}-\Delta \tau_{i}\vert = \mathcal{O}_{\varepsilon}(\Delta \tau_{i+1}^2) +\mathcal{O}_{\varepsilon}(\Delta \tau_{i}^2) +\mathcal{O}_{\varepsilon}(\Delta \varphi_{i+1}^2) +\mathcal{O}_{\varepsilon}(\Delta \varphi_{i}^2), \quad \forall i \geq I,
\end{align*}
where the notation $\mathcal{O}_{\varepsilon}$ was defined in Definition \ref{Big O} and $I\in \mathbb{N}$ in Lemma \ref{compare}.
\end{lemma}
	\begin{remark} In Lemma \ref{difference}, we assumed \eqref{assump 3} to derive the relation about $\Delta \tau_{i+1}-\Delta \tau_i$. Assumption \eqref{assump 3} means that we do not consider shading region in Figure \ref{remark 2}. The case $\tau_i \in (Z_h)_{\varepsilon}$ is covered through these two pictures in Figure \ref{remark 2}. Later, we will remove this region through the angular momentum conservation of the trajectory which passes $(\sigma(\tau_i,\varphi_i),I_2(\tau_i,\varphi_i))$ whenever $\tau_i \in (Z_h)_{\varepsilon}$.
	\end{remark}
\begin{figure}[h]
	\centering
	\begin{subfigure}[h]{0.495\textwidth}
	\includegraphics[width=\linewidth]{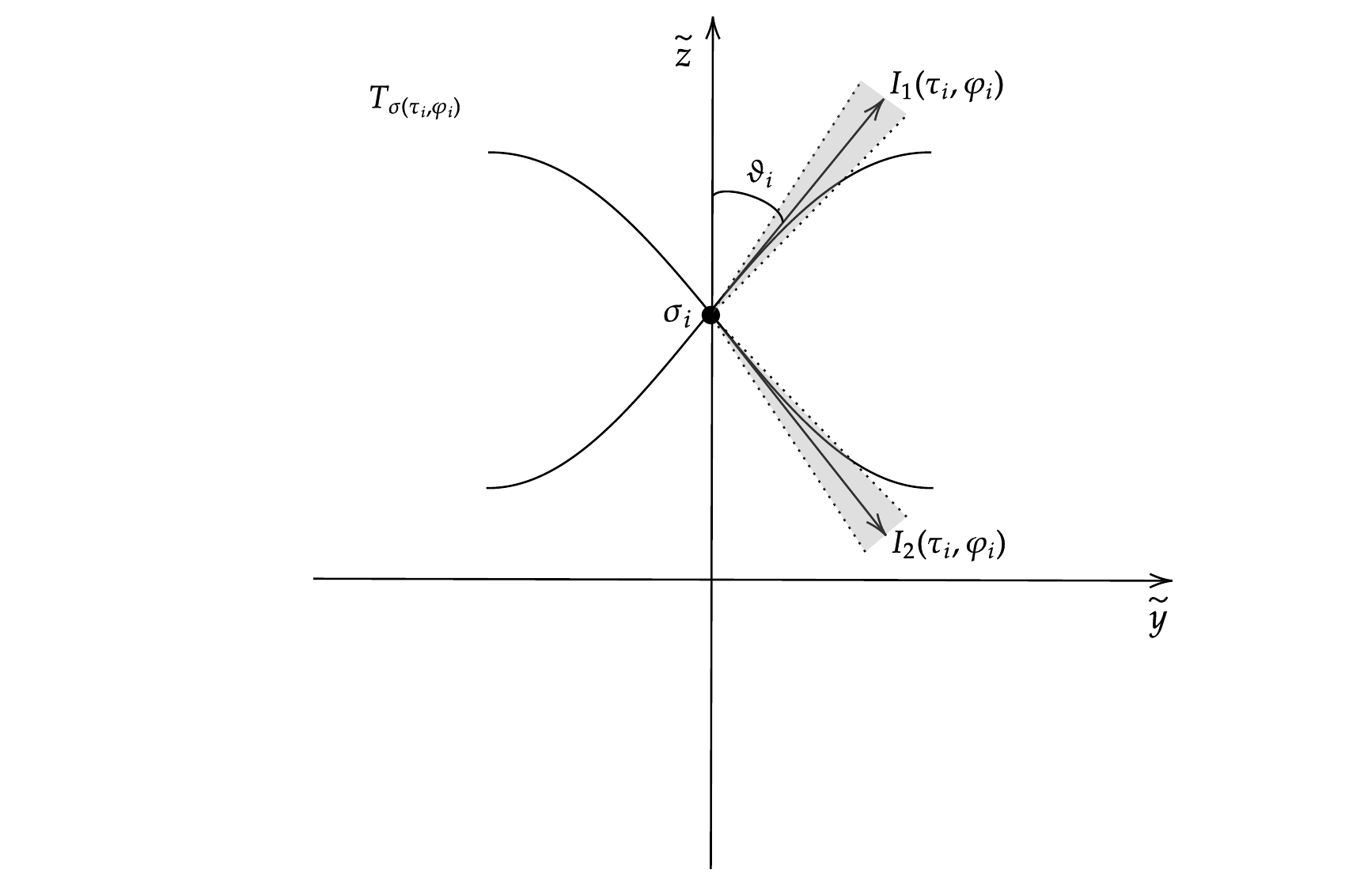}
	\end{subfigure}
	\hfill
	\begin{subfigure}[h]{0.495\textwidth}
	\includegraphics[width=\linewidth]{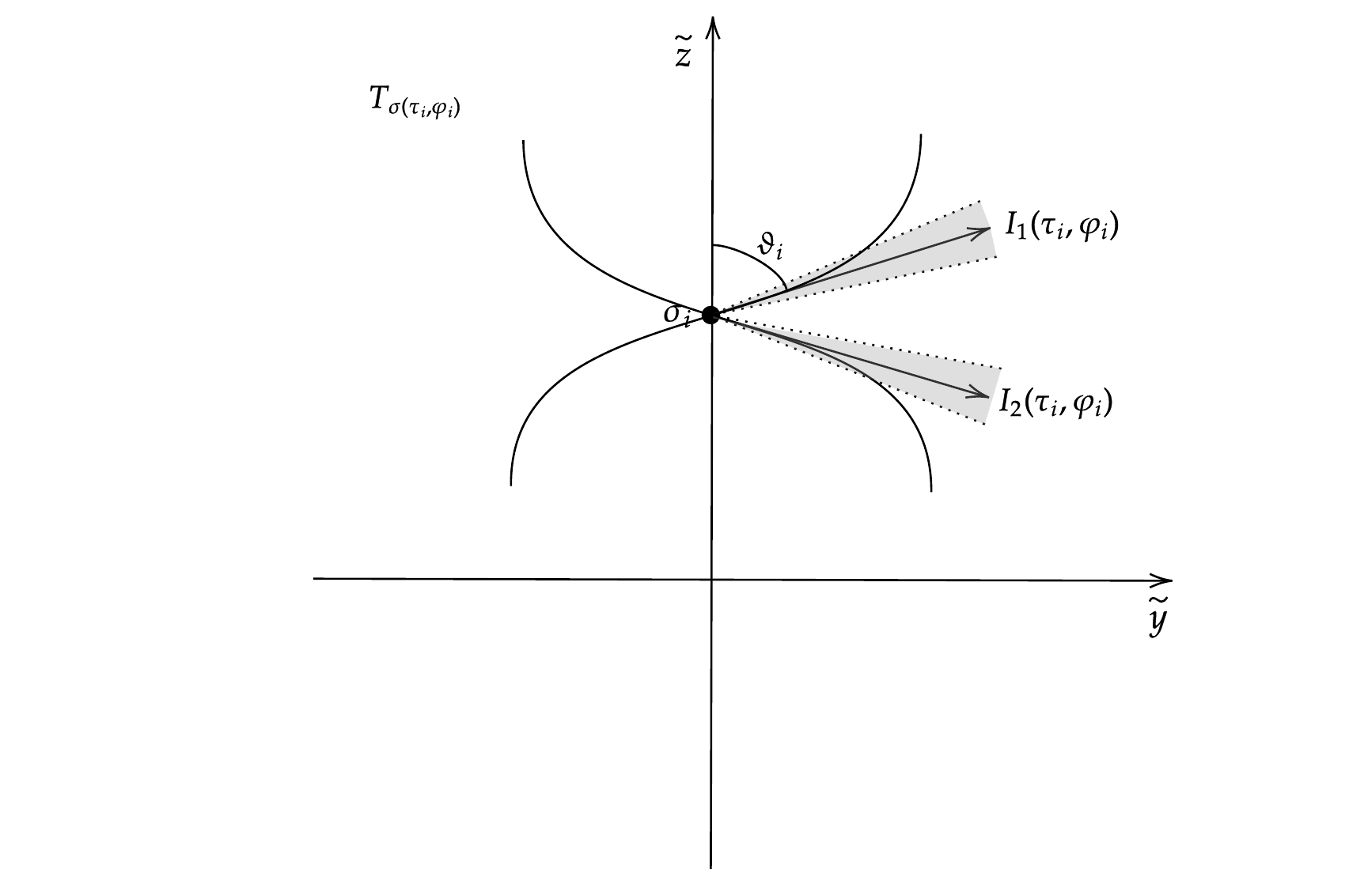}
	\end{subfigure}
	\caption{Assumption \eqref{assump 3}}
	\label{remark 2}
\end{figure}
\begin{proof}
\textbf{Step 1.} We first claim that each bounce is nongrazing. This is similar to the argument that we made in the proof of Lemma \ref{compare}. Since we assume $|\Delta \varphi_{i}|, |\Delta \tau_{i}|< \delta \ll 1$. we consider sufficiently small circle $B_i$ near $\sigma_i$ with radius $\delta$ so that both $\sigma_{i-1}$ and $\sigma_{i+1}$ locate inside of $B_i$. Note that both $(\sigma_i,  \overrightarrow{\sigma_{i-1}\sigma_{i}})$ and $(\sigma_i,  \overrightarrow{\sigma_{i+1}\sigma_{i}})$ must belong to $\gamma_0$ or $\gamma_{+}$. (Note that direction of $n(\sigma_{i})$ is inward to the paper in Figure \ref{proof 2}.) \\
		\indent Now we exclude $\gamma_0$ case using assumption \eqref{assump 3}. First, let us assume $\tau_{i}\notin (Z_h)_{\varepsilon}$ so that $T_{\sigma_{i}}\cap \p\O \cap B_{i}$ is concave or convex locally near $\sigma_{i}$. WLOG, let us assume it is locally concave as in Figure \eqref{fig5:a}. (For convex case, the argument is nearly similar.) If $(\sigma_i,  \overrightarrow{\sigma_{i-1}\sigma_{i}})\in\gamma_0$, $\sigma_{i-1}$ must be in $T_{\sigma_{i}}\cap \p\O \cap B_{i}$ as in Figure \eqref{fig5:a}. By concavity of the intersection curve in \eqref{fig5:a} of Figure \ref{proof 2}, $\sigma_{i+1}$ cannot hit $\tau_{i}\notin (Z_h)_{\varepsilon}$ for sufficiently small $\delta \ll 1$. See Figure \eqref{fig5:b} to check that $\gamma_0$ is possible if $\sigma_{i}$ is inflection point of $T_{\sigma_{i}}\cap \p\O \cap B_{i}$, for example $\tau_{i}\in Z_h$.   \\
		\indent Next, we consider $\tau_{i}\in (Z_h)_{\varepsilon}$ with extra assumption\eqref{assump 3}. We can choose sufficiently $\delta \ll 1$ so that excluded region via \eqref{assump 3} also excludes $T_{\sigma_{i}}\cap \p\O \cap B_{i}$ as in Figure \ref{remark 2}. So $\sigma_{i-1}, \sigma_{i+1}\in B_{i}$ is possible only when both $(\sigma_i,  \overrightarrow{\sigma_{i-1}\sigma_{i}})$ and $(\sigma_i,  \overrightarrow{\sigma_{i+1}\sigma_{i}})$ are $\gamma_{+}$.  \\
\begin{figure}[h]
	\begin{subfigure}[h]{0.495\textwidth}
	\includegraphics[width=\linewidth]{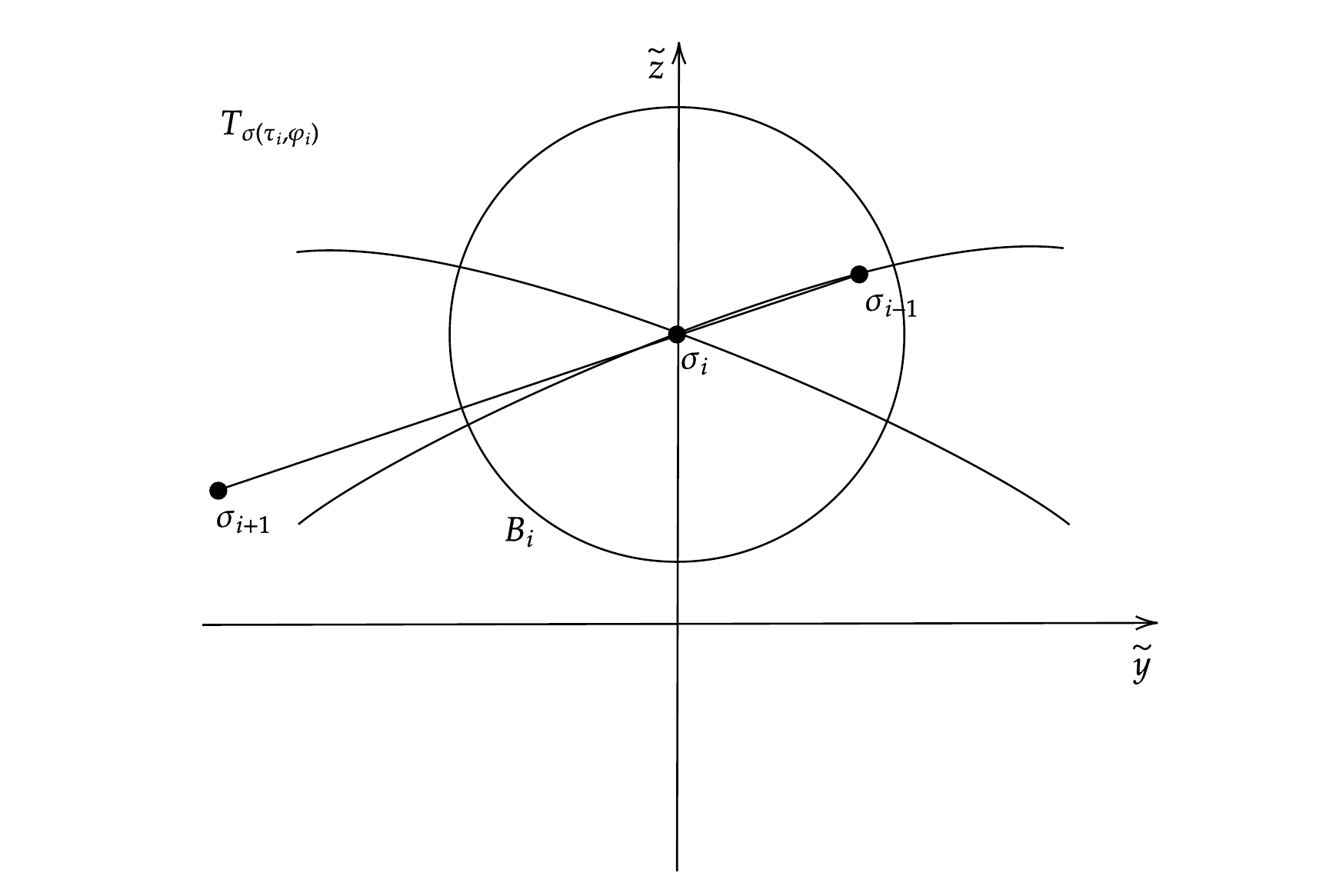}
	\caption{$\tau_i \notin (Z_h)_{\varepsilon}$}
	\label{fig5:a}
	\end{subfigure}
	\begin{subfigure}[h]{0.495\textwidth}
	\includegraphics[width=\linewidth]{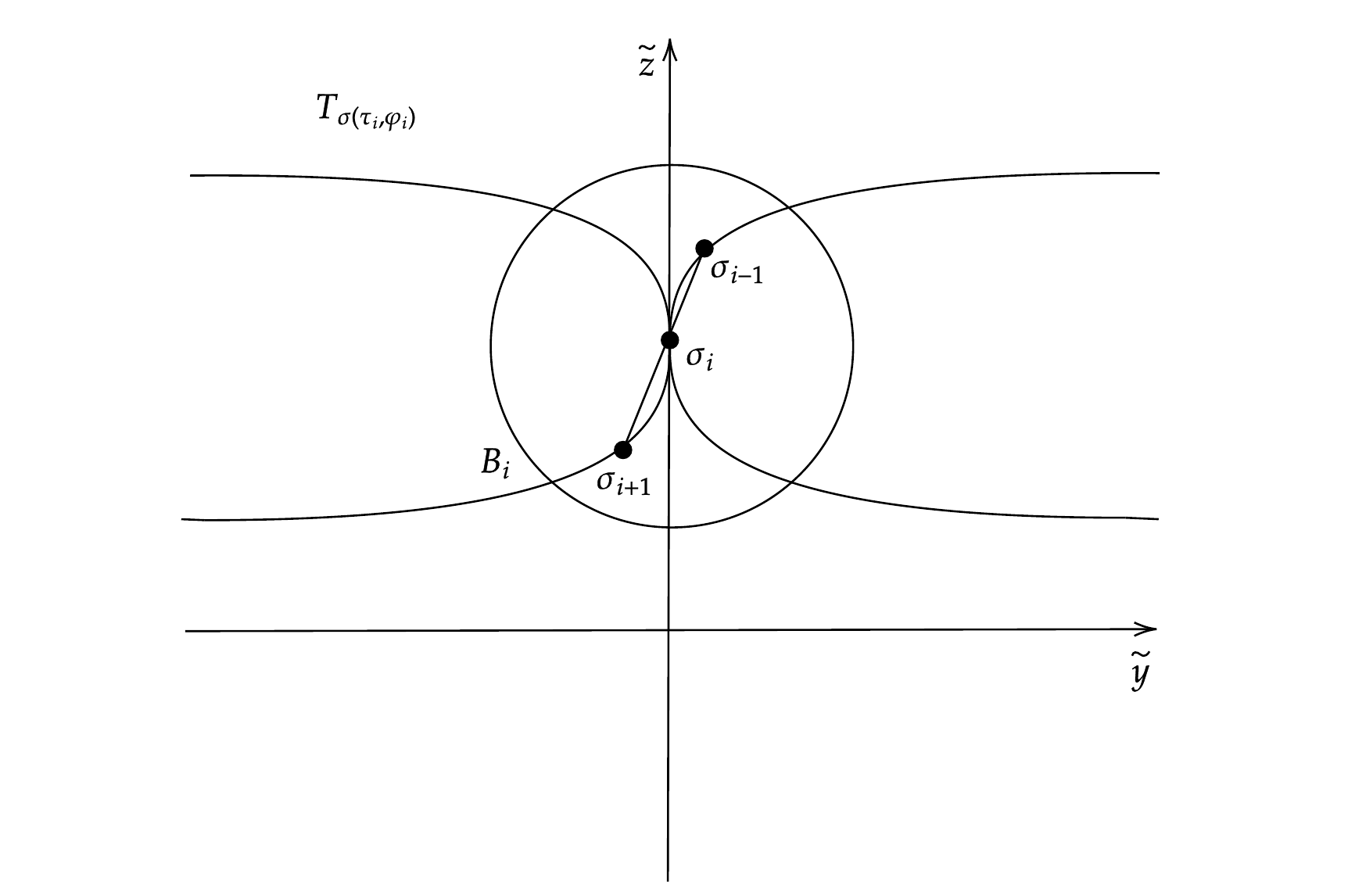}
	\caption{$\tau_i \in (Z_h)_{\varepsilon}$}
	\label{fig5:b}	
	\end{subfigure}
	\caption{The trajectory between $\sigma_{i-1}$ and $\sigma_{i+1}$ on the tangent plane $T_{\sigma(\tau_i,\varphi_i)}$}
	\label{proof 2}
\end{figure}

\textbf{Step 2.} For convenience of notation, we set $I=1$ in assumption \eqref{assump 2} of Lemma \ref{compare} and consider three boundary points 
\begin{align*}
	\sigma_i:=\sigma(\tau_i,\varphi_i)=(\gamma_1(\tau_i)\cos\varphi_i,\gamma_1(\tau_i)\sin\varphi_i,\gamma_2(\tau_i)), \quad 0\leq i \leq 2,
\end{align*}
which are in sufficiently small ball $B_{i}$. Without loss of generality, we set $\varphi_1=0$ due to axis-symmetry. Suppose that  
\begin{align*}
	\overrightarrow{\sigma_0\sigma_1} \parallel v \quad \textrm{ and } \quad \overrightarrow{\sigma_1\sigma_2} \parallel R_{\sigma(\tau_1,0)} v,
\end{align*}
where $R_{\sigma(\tau_1,0)}$ is reflection operator at the point $\sigma_1=\sigma(\tau_1,0)$. From the above, we have
\begin{align} \label{relation}
	\begin{bmatrix}
	\gamma_1(\gamma_2)\cos \varphi_2 -\gamma_1(\tau_1) \\ 
	\gamma_1(\tau_2)\sin\varphi_2\\
	\gamma_2(\tau_2)-\gamma_2(\tau_1)
	\end{bmatrix}=l \begin{bmatrix}
	1-2(\gamma_2'(\tau_1))^2 & 0 & 2\gamma_1'(\tau_1)\gamma_2'(\tau_1)\\
	0 &1 &0 \\
	2\gamma_1'(\tau_1)\gamma_2'(\tau_1) & 0 & 2(\gamma_2'(\tau_1))^2-1 
	\end{bmatrix}\begin{bmatrix}
	\gamma_1(\tau_1)-\gamma_1(\tau_0)\cos \varphi_0 \\
	-\gamma_1(\tau_0)\sin\varphi_0\\
	\gamma_2(\tau_1)-\gamma_2(\tau_0)
	\end{bmatrix},
\end{align}
for some constant $l$. From the first and third components, we have 
\Be
\begin{split}
	&\gamma_1(\tau_2)-\gamma_1(\tau_1) 
	\\
	&=l\Big((1-2(\gamma_2'(\tau_1))^2)(\gamma_1(\tau_1)-\gamma_1(\tau_0)\cos\varphi_0)+2\gamma_1'(\tau_1)\gamma_2'(\tau_1)(\gamma_2(\tau_1)-\gamma_2(\tau_0))\Big)\\
	&  \ \ 
+(1-\cos\varphi_2)\gamma_1(\tau_2), \label{gamma1}
\end{split}\Ee
\Be \begin{split}
&	\gamma_2(\tau_2)-\gamma_2(\tau_1)  \\
	&=l\Big(2\gamma_1'(\tau_1)\gamma_2'(\tau_1)(\gamma_1(\tau_1)-\gamma_1(\tau_0)\cos\varphi_0)+(2(\gamma_2'(\tau_1))^2-1)(\gamma_2(\tau_1)-\gamma_2(\tau_0))\Big)\label{gamma2}.
\end{split}\Ee
Using the second component of \eqref{relation}, we obtain that 
\begin{align*}
	\gamma_1(\tau_2)=-\frac{l \gamma_1(\tau_0)\sin\varphi_0}{\sin\varphi_2}.
\end{align*}
Then, the condition \eqref{gamma1} becomes 
\begin{align} \label{gamma1new}
	\begin{split}
&	\gamma_1(\tau_2)-\gamma_1(\tau_1)\\&=l\Big((1-2(\gamma_2'(\tau_1))^2)(\gamma_1(\tau_1)-\gamma_1(\tau_0)\cos\varphi_0)\\
&  \ \ \ +2\gamma_1'(\tau_1)\gamma_2'(\tau_1)(\gamma_2(\tau_1)-\gamma_2(\tau_0))-\gamma_1(\tau_0)\frac{1-\cos\varphi_2}{\sin\varphi_2}\sin\varphi_0\Big)\\
	&=l\Big((1-2(\gamma_2'(\tau_1))^2)(\gamma_1(\tau_1)-\gamma_1(\tau_0))+2\gamma_1'(\tau_1)\gamma_2'(\tau_1)(\gamma_2(\tau_1)-\gamma_2(\tau_0))\\
	&\quad +(1-2(\gamma_2'(\tau_1))^2)(1-\cos\varphi_0)\gamma_1(\tau_0)-\gamma_1(\tau_0)\frac{1-\cos\varphi_2}{\sin\varphi_2}\sin\varphi_0\Big).
	\end{split}
\end{align}
Multiply (LHS) in \eqref{gamma1new} by (RHS) in \eqref{gamma2}:
\begin{align*} 
	&l(\gamma_1(\tau_2)-\gamma_1(\tau_1))\\
	& \times \Big(2\gamma_1'(\tau_1)\gamma_2'(\tau_1)(\gamma_1(\tau_1)-\gamma_1(\tau_0)\cos\varphi_0)+(2(\gamma_2'(\tau_1))^2-1)(\gamma_2(\tau_1)-\gamma_2(\tau_0))\Big)\\
	&=l(\gamma_1(\tau_2)-\gamma_1(\tau_1))\\
	& \times \Big(2\gamma_1'(\tau_1)\gamma_2'(\tau_1)(\gamma_1(\tau_1)-\gamma_1(\tau_0))\\
	& \ \ \ \ +(2(\gamma_2'(\tau_1))^2-1)(\gamma_2(\tau_1)-\gamma_2(\tau_0))+2\gamma_1'(\tau_1)\gamma_2'(\tau_1)(1-\cos\varphi_0)\gamma_1(\tau_0)\Big).
\end{align*}
And, multiply (LHS) in \eqref{gamma2} by (RHS) in \eqref{gamma1new}:
\begin{align*}
&l(\gamma_2(\tau_2)-\gamma_2(\tau_1))\Big((1-2(\gamma_2'(\tau_1))^2)(\gamma_1(\tau_1)-\gamma_1(\tau_0))+2\gamma_1'(\tau_1)\gamma_2'(\tau_1)(\gamma_2(\tau_1)-\gamma_2(\tau_0))\\
	&\qquad \qquad \qquad \qquad  +(1-2(\gamma_2'(\tau_1))^2)(1-\cos\varphi_0)\gamma_1(\tau_0)-\gamma_1(\tau_0)\frac{1-\cos\varphi_2}{\sin\varphi_2}\sin\varphi_0\Big).
\end{align*}
 Since two forms above are the same, we get 
 \begin{align}\label{mult}
 	\begin{split}
 	&(\gamma_1(\tau_2)-\gamma_1(\tau_1))\Big(2\gamma_1'(\tau_1)\gamma_2'(\tau_1)(\gamma_1(\tau_1)-\gamma_1(\tau_0))+(2(\gamma_2'(\tau_1))^2-1)(\gamma_2(\tau_1)-\gamma_2(\tau_0)) \\
 	& \ \ \ \ \ \ \ \ \ \ \ \ \ \ \ \ \ \ \ \ \ \ \  +2\gamma_1'(\tau_1)\gamma_2'(\tau_1)(1-\cos\varphi_0)\gamma_1(\tau_0)\Big)\\
	&  = (\gamma_2(\tau_2)-\gamma_2(\tau_1))\Big((1-2(\gamma_2'(\tau_1))^2)(\gamma_1(\tau_1)-\gamma_1(\tau_0))+2\gamma_1'(\tau_1)\gamma_2'(\tau_1)(\gamma_2(\tau_1)-\gamma_2(\tau_0))\\
	&\qquad \qquad \qquad \qquad \qquad  +(1-2(\gamma_2'(\tau_1))^2)(1-\cos\varphi_0)\gamma_1(\tau_0)-\gamma_1(\tau_0)\frac{1-\cos\varphi_2}{\sin\varphi_2}\sin\varphi_0\Big).
	\end{split}
 \end{align}
 We expand some terms in \eqref{mult} by using Taylor's expansion at $\tau=\tau_1$:
 \begin{align*}
 	\gamma_1(\tau_2)-\gamma_1(\tau_1)&=\gamma'_1(\tau_1)\Delta \tau_2+\frac{\gamma_1''(\tau_1)}{2}(\Delta \tau_2)^2+\mathcal{O}((\Delta\tau_2)^3)\\
 	& =\gamma'_1(\tau_1)\Delta \tau_2-\frac{\kappa(\tau_1)\gamma_2'(\tau_1)}{2}(\Delta \tau_2)^2+\mathcal{O}((\Delta \tau_2)^3),\\
	\gamma_1(\tau_1)-\gamma_1(\tau_0)&=\gamma_1'(\tau_1)\Delta \tau_1-\frac{\gamma_1''(\tau_1)}{2}(\Delta \tau_1)^2+\mathcal{O}((\Delta \tau_1)^3)
	\\
	& =\gamma_1'(\tau_1)\Delta \tau_1+\frac{\kappa(\tau_1)\gamma_2'(\tau_1)}{2}(\Delta \tau_1)^2 +\mathcal{O}((\Delta \tau_1)^3),\\
	\gamma_2(\tau_2)-\gamma_2(\tau_1)&=\gamma'_2(\tau_1)\Delta \tau_2+\frac{\gamma_2''(\tau_1)}{2}(\Delta \tau_2)^2+\mathcal{O}((\Delta\tau_2)^3)\\
	& =\gamma'_2(\tau_1)\Delta \tau_2+\frac{\kappa(\tau_1)\gamma_1'(\tau_1)}{2}(\Delta \tau_2)^2+\mathcal{O}((\Delta \tau_2)^3),\\
	\gamma_2(\tau_1)-\gamma_2(\tau_0)&=\gamma_2'(\tau_1)\Delta \tau_1-\frac{\gamma_2''(\tau_1)}{2}(\Delta \tau_1)^2+\mathcal{O}((\Delta \tau_1)^3)\\
	& =\gamma_2'(\tau_1)\Delta \tau_1 -\frac{\kappa(\tau_1)\gamma_1'(\tau_1)}{2}(\Delta \tau_1)^2 +\mathcal{O}((\Delta \tau_1)^3),\\
	(1-\cos \varphi_0)\gamma_1(\tau_0)&=(\frac{1}{2}\varphi_0^2 +\mathcal{O}(\varphi_0^4))\times (\gamma_1(\tau_1)-\gamma_1'(\tau_1)\Delta \tau_1+\mathcal{O}(\Delta \tau_1^2))\\
	&=\frac{\gamma_1(\tau_1)}{2}\varphi_0^2  -\frac{\gamma_1'(\tau_1)}{2}\varphi_0^2\Delta \tau_1+\mathcal{O}(\Delta \tau_1^2)\varphi_0^2+\mathcal{O}(\varphi_0^4),\\
	\gamma_1(\tau_0)\frac{1-\cos\varphi_2}{\sin\varphi_2}\sin\varphi_0&=(\gamma_1(\tau_1)-\gamma_1'(\tau_1)\Delta \tau_1 + \mathcal{O}(\Delta \tau_1^2))\times (\frac{\varphi_2}{2}+\mathcal{O}(\varphi_2^3)) \times (\varphi_0+\mathcal{O}(\varphi_0^3))\\
	&=\frac{\gamma_1(\tau_1)}{2}\varphi_0\varphi_2-\frac{\gamma_1'(\tau_1)}{2}\varphi_0\varphi_2\Delta \tau_1\\
	&\quad+\mathcal{O}(\varphi_0^3)\varphi_2+\mathcal{O}(\varphi_2^3)\varphi_0+\mathcal{O}(\Delta \tau_1^2)\varphi_0\varphi_2,
 \end{align*}
 where $\Delta \tau_k=\tau_k-\tau_{k-1}$ and $\Delta \varphi_k=\varphi_k-\varphi_{k-1}$ for $k=1,2$. We compare coefficients of the left and right terms in \eqref{mult}:
 \begin{align*}
 	(\Delta \tau_1 \Delta \tau_2) \quad (LHS)&=\gamma_1'(\tau_1)(2(\gamma_1'(\tau_1))^2\gamma_2'(\tau_1)+\gamma_2'(\tau_1)(2(\gamma_2'(\tau_1))^2-1))\\
	&=\gamma_1'(\tau_1)\gamma_2'(\tau_1)(2(\gamma_1'(\tau_1))^2+2(\gamma_2'(\tau_1))^2-1)\\
	&=\gamma_1'(\tau_1)\gamma_2(\tau_1),\\
	(RHS)&=(1-2(\gamma_2'(\tau_1))^2)\gamma_1'(\tau_1)\gamma_2'(\tau_1)+2\gamma_1'(\tau_1)(\gamma_2'(\tau_1))^3\\
	&=\gamma_1'(\tau_1)\gamma_2'(\tau_1)(1-2(\gamma_2'(\tau_1))^2+2(\gamma_2'(\tau_1))^2)\\
	&=\gamma_1'(\tau_1)\gamma_2'(\tau_1),\\
	(\Delta \tau_1)^2(\Delta \tau_2)^2 \quad  (LHS)&=-\frac{k^2(\tau_1)}{2}\gamma_1'(\tau_1)(\gamma_2'(\tau_1))^3+\frac{k^2(\tau_1)}{4}\gamma_1'(\tau_1)\gamma_2'(\tau_1)(2(\gamma_2'(\tau_1))^2-1)\\
	&=\frac{k^2(\tau_1)}{4}\gamma_1'(\tau_1)\gamma_2'(\tau_1)(-2(\gamma_2'(\tau_1))^2+2(\gamma_2'(\tau_1))^2-1)\\
	&=-\frac{k^2(\tau_1)}{4}\gamma_1'(\tau_1)\gamma_2'(\tau_1),\\
	(RHS) &= \frac{k^2(\tau_1)}{4}\gamma_1'(\tau_1)\gamma_2'(\tau_1)(1-2(\gamma_2'(\tau_1))^2)-\frac{k^2(\tau_1)}{2}(\gamma_1'(\tau_1))^3\gamma_2'(\tau_1)\\
	&=\frac{k^2(\tau_1)}{4}\gamma_1'(\tau_1)\gamma_2'(\tau_1)(1-2(\gamma_2'(\tau_1))^2-2(\gamma_1'(\tau_1))^2)\\
	&=-\frac{k^2(\tau_1)}{4}\gamma_1'(\tau_1)\gamma_2'(\tau_1),\\
 	(\Delta \tau_2)(\Delta \tau_1)^2 \quad (LHS)&= \kappa(\tau_1)(\gamma_1'(\tau_1))^2(\gamma_2'(\tau_1))^2-\frac{\kappa(\tau_1)}{2}(\gamma_1'(\tau_1))^2(2(\gamma_2'(\tau_1))^2-1)\\
	&=\frac{\kappa(\tau_1)}{2}(\gamma_1'(\tau_1))^2(2(\gamma_2'(\tau_1))^2-2(\gamma_2'(\tau_1))^2+1)\\
	&=\frac{\kappa(\tau_1)}{2}(\gamma_1'(\tau_1))^2,\\
	(RHS)&=\frac{\kappa(\tau_1)}{2}(1-2(\gamma_2'(\tau_1))^2)(\gamma_2'(\tau_1))^2-\kappa(\tau_1)(\gamma_1'(\tau_1))^2(\gamma_2'(\tau_1))^2\\
	&=\frac{\kappa(\tau_1)}{2}(\gamma_2'(\tau_1))^2)(1-2(\gamma_2'(\tau_1))^2-2(\gamma_1'(\tau_1))^2)\\
	&=-\frac{\kappa(\tau_1)}{2}(\gamma_2'(\tau_1))^2, \\
	(\Delta \tau_1)(\Delta \tau_2)^2 \quad (LHS)&=-\kappa(\tau_1)(\gamma_1'(\tau_1))^2(\gamma_2'(\tau_1))^2-\frac{\kappa(\tau_1)}{2}(\gamma_1'(\tau_1))^2(2(\gamma_2'(\tau_1))^2-1)\\
	&=-\kappa(\tau_1)(\gamma_2'(\tau_1))^2((\gamma_1'(\tau_1))^2+(\gamma_2'(\tau_1))^2-1/2)\\
	&=-\frac{\kappa(\tau_1)}{2}(\gamma_2'(\tau_1))^2,\\
	(RHS)&=\frac{\kappa(\tau_1)}{2}(1-2(\gamma_2'(\tau_1))^2)(\gamma_1'(\tau_1))^2+\kappa(\tau_1)(\gamma_1'(\tau_1))^2(\gamma_2'(\tau_1))^2\\
	&=\frac{\kappa(\tau_1)}{2}(\gamma_1'(\tau_1))^2(1-2(\gamma_2'(\tau_1))^2+2(\gamma_2'(\tau_1))^2)\\
	&=\frac{\kappa(\tau_1)}{2}(\gamma_1'(\tau_1))^2,\\
 	(\varphi_0^2\Delta \tau_2) \quad (LHS)&=\gamma_1(\tau_1)(\gamma_1'(\tau_1))^2(\gamma_2'(\tau_1)),\\
	(RHS)&=\frac{\gamma_1(\tau_1)}{2}\gamma_2'(\tau_1)(1-2(\gamma_2'(\tau_1))^2),\\
	(\varphi_0\varphi_2\Delta \tau_2) \quad (LHS)&=0,\\
	(RHS)&=-\frac{\gamma_1(\tau_1)}{2}\gamma_2'(\tau_1),\\
	(Error \; term) \quad (LHS)&=\mathcal{O}(\Delta \tau_2^3)\mathcal{O}(\Delta \tau_1)+\mathcal{O}(\Delta \tau_1^3)\mathcal{O}(\Delta \tau_2) +\mathcal{O}(\varphi_0^2)\mathcal{O}(\Delta \tau_2^2)\\
	& +\mathcal{O}(\varphi_0^2)\mathcal{O}(\Delta \tau_1)\mathcal{O}(\Delta \tau_2)+\mathcal{O}(\varphi_0^4)\mathcal{O}(\Delta \tau_2),\\
	(RHS)&=\mathcal{O}(\Delta \tau_2^3)\mathcal{O}(\Delta \tau_1)+\mathcal{O}(\Delta \tau_1^3)\mathcal{O}(\Delta \tau_2) \\
	& +\mathcal{O}(\varphi_0^2)\mathcal{O}(\Delta \tau_2^2)+\mathcal{O}(\varphi_0^2)\mathcal{O}(\Delta \tau_1)\mathcal{O}(\Delta \tau_2)+\mathcal{O}(\varphi_0^4)\mathcal{O}(\Delta \tau_2)\\
	&\quad +\mathcal{O}(\varphi_0)\mathcal{O}(\varphi_2)\mathcal{O}(\Delta \tau_2^2)+\mathcal{O}(\varphi_0)\mathcal{O}(\varphi_2)\mathcal{O}(\Delta \tau_1)\mathcal{O}(\Delta \tau_2)\\
	&\quad +\mathcal{O}(\varphi_0^3)\mathcal{O}(\varphi_2)\mathcal{O}(\Delta \tau_2) +\mathcal{O}(\varphi_2^3)\mathcal{O}(\varphi_0)\mathcal{O}( \Delta \tau_2). 
 \end{align*}
 Therefore, summing the above yields that 
 \begin{align*}
 	&\frac{\kappa(\tau_1)}{2}(\Delta\tau_2)^2(\Delta\tau_1)-\frac{\kappa(\tau_1)}{2}(\Delta \tau_1)^2(\Delta \tau_2) -\frac{\gamma_1(\tau_1)}{2}\gamma_2'(\tau_1)\varphi_0^2\Delta \tau_2-\frac{\gamma_1(\tau_1)}{2}\gamma_2'(\tau_1)\varphi_0\varphi_2\Delta \tau_2 \\
	&\quad =\mathcal{O}((\Delta \tau_1)^3)\mathcal{O}(\Delta \tau_2)+\mathcal{O}((\Delta \tau_2)^3)\mathcal{O}(\Delta \tau_1)+\mathcal{O}(\varphi_0^2)\mathcal{O}(\Delta \tau_1)\mathcal{O}(\Delta \tau_2)\\
	&\quad \quad +\mathcal{O}(\varphi_0)\mathcal{O}(\varphi_2)\mathcal{O}(\Delta \tau_1)\mathcal{O}(\Delta \tau_2)  +\mathcal{O}(\varphi_0^2)\mathcal{O}(\Delta \tau_2^2)+\mathcal{O}(\varphi_0)\mathcal{O}(\varphi_2)\mathcal{O}(\Delta \tau_2^2)\\
		&\quad \quad+\mathcal{O}(\varphi_0^3)\mathcal{O}(\varphi_2)\mathcal{O}(\Delta \tau_2)+\mathcal{O}(\varphi_2^3)\mathcal{O}(\varphi_0)\mathcal{O}(\Delta \tau_2) +\mathcal{O}(\varphi_0^4)\mathcal{O}(\Delta \tau_2),
 \end{align*}
 where we have used $(\gamma'_1(\tau_1))^2+(\gamma_2'(\tau_1))^2=1$. Divide both sides above by $\frac{\kappa(\tau_1)}{2}\Delta \tau_1 \Delta \tau_2$: 
 \begin{align} \label{diff_comp_1}	
 \begin{split}
 \Delta \tau_2-\Delta \tau_1 &=\frac{\gamma_1(\tau_1)\gamma_2'(\tau_1)}{\kappa(\tau_1)}(\varphi_0+\varphi_2)\frac{\varphi_0}{\Delta \tau_1}\\
 &\quad+\mathcal{O}(\Delta \tau_1^2)+\mathcal{O}(\Delta \tau_2^2)+\mathcal{O}(\varphi_0^2)+\mathcal{O}(\varphi_0)\mathcal{O}(\varphi_2)+\frac{1}{\Delta \tau_1}\mathcal{O}(\varphi_0^2)\mathcal{O}(\Delta \tau_2)\\
 &\quad +\frac{1}{\Delta \tau_1}\mathcal{O}(\varphi_0)\mathcal{O}(\varphi_2)\mathcal{O}(\Delta \tau_2)\\
 &\quad+ \frac{1}{\Delta \tau_1} \mathcal{O}(\varphi_0^3)\mathcal{O}(\varphi_2) +\frac{1}{\Delta \tau_1} \mathcal{O}(\varphi_2^3)\mathcal{O}(\varphi_0)+\frac{1}{\Delta \tau_1} \mathcal{O}(\varphi_0^4).  
 \end{split}
 \end{align}
 Since we set $\varphi_1=0$, we have $\varphi_2=\varphi_2-\varphi_1=\Delta \varphi_2$,  $\varphi_0=-(\varphi_1-\varphi_0)=-\Delta \varphi_1$, and 
 \begin{align*}
 	\varphi_0+\varphi_2=\Delta \varphi_2 -\Delta \varphi_1.
 \end{align*}
 Thus, we have to derive that $\varphi_0+\varphi_2=\Delta \varphi_2-\Delta \varphi_1$ has the second-order truncation error. \\
  
 \noindent \textbf{Step 3.} From the second component in \eqref{relation}, we have 
 \begin{align} \label{second comp}
 	\gamma_1(\tau_2)\sin\varphi_2 = -l \gamma_1(\tau_0)\sin\varphi_0. 
 \end{align}
 By using the Taylor expansion, 
 \begin{align*}
 	\gamma_1(\tau_2)\sin\varphi_2&=(\gamma_1(\tau_1)+\mathcal{O}(\Delta \tau_2))(\varphi_2 +\mathcal{O}(\varphi_2^3)) = \gamma_1(\tau_1)\varphi_2 +\mathcal{O}(\Delta \tau_2)\varphi_2 + \mathcal{O}(\varphi_2^3),\\
	\gamma_1(\tau_0)\sin\varphi_0&=(\gamma_1(\tau_1)+\mathcal{O}(\Delta \tau_1))(\varphi_0+\mathcal{O}(\varphi_0^3))=\gamma_1(\tau_1)\varphi_0+\mathcal{O}(\Delta \tau_1)\varphi_0 +\mathcal{O}(\varphi_0^3),
 \end{align*}
 which implies that 
 \begin{align} \label{relation2}
 	\gamma_1(\tau_1)(\varphi_2+l\varphi_0)=\mathcal{O}(\Delta \tau_2)\varphi_2+\mathcal{O}(\Delta \tau_1)\varphi_0.
 \end{align}
 Since $\tau_2$ and $\varphi_2$ are determined by $\tau_0, \varphi_0, \tau_1,$ and $\varphi_1$,  $l$ in \eqref{second comp} can be seen as a function of $\tau_0,\varphi_0,\tau_1$ 
 \[
 	l=l(\tau_0,\tau_1,\varphi_0).
 \] 
 (Note that we have chosen $\varphi_{1}=0$ WLOG.)  \\
 
 We will show that $\tau_{2}$, $\varphi_{2}$ is locally smooth function of $\tau_{0}, \tau_{1}, \varphi_{0}, \varphi_{2}$ when $(\sigma_{1}-\sigma_{0})\cdot N_{\sigma(\tau_1,0)} \neq 0$. In fact, by \textbf{Step 1} of this proof, it suffices to consider non-grazing cases only. By definition parametrization, map 
 \[
 	\sigma : (\tau, \varphi)\mapsto \sigma(\tau, \varphi) \in \p\O
 \]
 is smooth and injective locally. Hence $\det \frac{\p\sigma(\tau, \varphi)}{ \p (\tau, \varphi)} \neq 0$ and by inverse function theorem, $(\tau_{2}, \varphi_{2})(\sigma)$ is locally smooth function of boundary point $\sigma\in\p\O$.  Now, let us read \eqref{relation} as ($\sigma_{2}:=\sigma(\tau_{2}, \varphi_{2})$ satisfies)
 \[
 	[\sigma_{2} - \sigma_{1}]\times R_{\sigma_{1}}(\sigma_{1}-\sigma_{0}) = 0,
 \] 
 which implies
 \Be \label{sigma2 cond}
 	\sigma_{2} = \sigma_{1} + k R_{\sigma_{1}}(\sigma_{1}-\sigma_{0}) \quad \text{for some $k\in \R$.}
 \Ee
 If $\sigma\in \p\O$ is locally parametrized by analytic function $F(\sigma)=0$, condition $F(\sigma_{2})=0$ yields
 \[
 	\tilde{F}(\sigma_{0}, \sigma_{1}, k) := F(\sigma_{1} + k R_{\sigma_{1}}(\sigma_{1}-\sigma_{0})) = 0.
 \]
 Observe that
 \[
	\p_{k}\tilde{F}(\sigma_{0}, \sigma_{1}, k) = \nabla F\vert_{\sigma_{2}} \cdot R_{\sigma_{1}}(\sigma_{1}-\sigma_{0}) \neq 0,
 \]
 because we assume that the trajectory hits $\sigma_{2}$ non-grazingly. Note that $\nabla F$ is normal to the boundary at the boundary point $\sigma_{2}$. Therefore, by Lemma \ref{analytic IFT}, $k$ is (local) smooth function of $\sigma_{0}, \sigma_{1}$ (or equivalently $\tau_{0}, \tau_{1}, \varphi_{0}, \varphi_{1}$) . From \eqref{sigma2 cond}, $\sigma_{2}$ is smooth function of $\tau_{0}, \tau_{1}, \varphi_{0}, \varphi_{1}$ and by inverse function theorem as explained above, $(\tau_{2}, \varphi_{2})(\tau_{0}, \tau_{1}, \varphi_{0}, \varphi_{1})$ is smooth. Finally, we obtain that $l$ is smooth function of $(\tau_{0}, \tau_{1}, \varphi_{0}, \varphi_{1})$.  \\

 \hide
 so $\sigma_{2}$ is analytic funtion of $\tau_{0}, \tau_{1}, \varphi_{0}, \varphi_{1}$. Hence $(\tau_{2}, \varphi_{2})(\tau_{0}, \tau_{1}, \varphi_{0}, \varphi_{1})$ is smooth. 
 \[
 	l(\tau_0,\tau_1,\varphi_0) = \frac{|\sigma(\tau_{2}, \varphi_{2}) - \sigma(\tau_{1}, \varphi_{1})| }{ |\sigma(\tau_{1}, \varphi_{1}) - \sigma(\tau_{0}, \varphi_{0})| }.
 \]
 Since $\O$ is smooth, $l$ is also locally smooth. \\
 \unhide
 
 In particular, if there is no $\varphi_{0}$ perturbation, trajectory lies on the $\varphi_{1}=0$ cross-section and $\varphi_2=0$ obviously. In this case, we consider a regularized parametrization $r(\tau) = (r_{1}(\tau), r_{2}(\tau))$ for a closed uniformly convex curve in a plane. If we write the normal vector of the curve $r$ at $\tau_1$ as $N(\tau_1)$, from 
 \Be \label{taylor}
 	r(\tau_{2}) - r(\tau_{1}) = r^{\prime}(\tau_{1})(\tau_{2}-\tau_{1}) + \frac{\kappa(\tau_{1})}{2} N(\tau_1)(\tau_{2}-\tau_{1})^{2} + \mathcal{O}(\Delta\tau_{2}^{3}),
 \Ee
 we have
 \begin{equation*}
 \begin{split}
	| N(\tau_{1})\cdot \big( r(\tau_{2}) - r(\tau_{1}) \big) | &= \frac{|\kappa(\tau_{1})|}{2} (\tau_{2}-\tau_{1})^{2} + \mathcal{O}(\Delta\tau_{2}^{3}), \\
	| r^{\prime}(\tau_{1})\cdot \big( r(\tau_{2}) - r(\tau_{1}) \big) | &= \Delta\tau_{2} + \mathcal{O}(\Delta\tau_{2}^{3}).
 \end{split}
 \end{equation*}
The specular reflection condition
 \[
 	\frac{ \frac{|\kappa(\tau_{1})|}{2} (\Delta\tau_{2})^{2} + \mathcal{O}(\Delta\tau_{2}^{3}), }{\Delta\tau_{2} + \mathcal{O}(\Delta\tau_{2}^{3})} 
 	= 
 	\frac{ \frac{|\kappa(\tau_{1})|}{2} (\Delta\tau_{1})^{2} + \mathcal{O}(\Delta\tau_{1}^{3})}{\Delta\tau_{1} + \mathcal{O}(\Delta\tau_{1}^{3})},
 \]
 yields
 \Be \label{unperturbed}
 	(\Delta\tau_{2}-\Delta\tau_{1}) = \mathcal{O}(\Delta\tau_{2}^{2}) +\mathcal{O}(\Delta\tau_{1}^{2}).
 \Ee
From \eqref{unperturbed}, we have 
 \begin{align*}
 	\frac{\Delta \tau_2}{\Delta \tau_1} -1&= \mathcal{O}(\Delta \tau_1) + \frac{1}{\Delta \tau_1} \mathcal{O}(\Delta \tau_2^2)\\
	&= \mathcal{O}(\Delta \tau_1) +\frac{\Delta \tau_2}{\Delta \tau_1} \mathcal{O}(\Delta \tau_2),
 \end{align*}
 which implies 
 \begin{align*}
 	\frac{\Delta \tau_2}{\Delta \tau_1} = \frac{1+\mathcal{O}(\Delta \tau_1)}{1+\mathcal{O}(\Delta \tau_2)}=(1+\mathcal{O}(\Delta \tau_1))(1+\mathcal{O}(\Delta \tau_2)=1+\mathcal{O}(\Delta \tau_1)+\mathcal{O}(\Delta \tau_2),
 \end{align*}
 where we have used $\frac{1}{1+x}=1+\mathcal{O}(x)$ whenever $x\ll1$. 
Combining this fact with \eqref{taylor},
 \begin{equation*}
 \begin{split}
	 	l(\tau_0,\tau_1,0) &= \frac{|r(\tau_{2}) - r(\tau_{1}) |}{|r(\tau_{1}) - r(\tau_{0})  |} \\
	 	&= 1 + \mathcal{O}(\Delta \tau_{1}) +\mathcal{O}(\Delta \tau_{2}).
 \end{split}
 \end{equation*}

\hide
 To find the second order truncation error for $\varphi_0+\varphi_2$, we assume the followings: 
 \textcolor{red}{
 \begin{itemize}
 	\item $l=l(\tau_0,\tau_1,\varphi_0)$ is continuous at $\varphi_0$
	\item	$\varphi_2=\varphi_2(\tau_0,\tau_1,\varphi_0)$ is a smooth function. 
 \end{itemize}}
\unhide 

\noindent Divide both sides in \eqref{relation2} by $\gamma_1(\tau_1)\varphi_0$:
 \begin{align*}
 	\frac{\varphi_2}{\varphi_0}+l =\mathcal{O}(\Delta \tau_2)\frac{\varphi_2}{\varphi_0} + \mathcal{O}(\Delta \tau_1).
 \end{align*}
 Since $\lim_{\varphi_0\rightarrow 0}l(\tau_0,\tau_1,\varphi_0) = 1 + \mathcal{O}(\Delta \tau_{1}) +\mathcal{O}(\Delta \tau_{2})$, one obtains that  
 \begin{align*}
 	\lim_{\varphi_0\rightarrow 0} \frac{\varphi_2}{\varphi_0} = -1.
 \end{align*}
	Since $\varphi_2$ is a smooth function if $(\sigma_{1}-\sigma_{0})\cdot N_{\sigma(\tau_1,0)} \neq 0$ (non-grazing case), we get $\varphi_2=-\varphi_0+\mathcal{O}_{\tau_0, \tau_{1}, \tau_{2},\varepsilon}(\varphi_0^2)$. 
	Since $\mathcal{O}_{\tau_0, \tau_{1}, \tau_{2},\varepsilon}(\varphi_0^2)$ is smooth function that continuously depends also on $\tau_{0},\tau_{1},\varepsilon$ and we are assuming that $\tau_{i}$ converges, we can impose smallness condition $|\tau_{i}-\tau_{j}| \leq \delta \ll 1$ for all $i, j$ and hence 
	\[
	\lim_{\varphi_{0}\rightarrow 0} \big| \frac{\mathcal{O}_{\tau_0, \tau_{1},\tau_{2},\varepsilon}(\varphi_0^2)}{\varphi_0^2} \big| \leq C_{\varepsilon},
	\]
	for some constant $C_{\varepsilon}$ depending on $\varepsilon$. The same bound $C_{\varepsilon}$ can be used when we apply above argument for next bounces to get
	\[
	\lim_{\varphi_{i}\rightarrow 0} \big| \frac{\mathcal{O}_{\tau_i, \tau_{i+1}, \tau_{i+2},\varepsilon}(\varphi_i^2)}{\varphi_i^2} \big| \leq C_{\varepsilon},\quad \forall i.
	\]
	In other words, 
	\begin{align} \label{diff_rel_1}
	\varphi_0+\varphi_2 = \Delta \varphi_2 -\Delta \varphi_1 = \mathcal{O}_{\varepsilon}(\varphi_0^2).
	\end{align}
	We rewrite $\varphi_0$ and $\varphi_2$ as $-\Delta \varphi_1$ and $\Delta \varphi_2$, and then \eqref{diff_comp_1} becomes 
	\begin{align} \label{diff_comp_2}
	\begin{split}
	\Delta \tau_2-\Delta \tau_1 &=-\frac{\gamma_1(\tau_1)\gamma_2'(\tau_1)}{\kappa(\tau_1)}(\Delta \varphi_2-\Delta \varphi_1)\frac{\Delta\varphi_1}{\Delta \tau_1}\\
	&\quad+\mathcal{O}(\Delta \tau_1^2)+\mathcal{O}(\Delta \tau_2^2)+\mathcal{O}(\Delta \varphi_1^2)+\mathcal{O}(\Delta \varphi_1)\mathcal{O}(\Delta \varphi_2)+\frac{1}{\Delta \tau_1}\mathcal{O}(\Delta \varphi_1^2)\mathcal{O}(\Delta \tau_2) \\
	& \quad+\frac{1}{\Delta \tau_1}\mathcal{O}(\Delta \varphi_1)\mathcal{O}(\Delta \varphi_2)\mathcal{O}(\Delta \tau_2)\\
	&\quad+ \frac{1}{\Delta \tau_1} \mathcal{O}(\Delta \varphi_1^3)\mathcal{O}(\Delta \varphi_2) +\frac{1}{\Delta \tau_1} \mathcal{O}(\Delta \varphi_2^3)\mathcal{O}(\Delta \varphi_1)+\frac{1}{\Delta \tau_1} \mathcal{O}(\Delta \varphi_1^4). 
	\end{split}
	\end{align}
	Using $\left \vert \frac{\Delta \varphi_1}{\Delta \tau_1} \right \vert \leq C_{\varepsilon}$ from \eqref{tau and varphi} in Lemma \ref{compare} and \eqref{diff_rel_1}, we obtain the following relation through \eqref{diff_rel_1}: 
	\begin{align*}
	\vert \Delta \tau_2 -\Delta \tau_1\vert&=(1+C_{\varepsilon})\big(\mathcal{O}(\Delta \tau_2^2)+\mathcal{O}(\Delta \tau_1^2)+\mathcal{O}(\Delta \varphi_2^2)+\mathcal{O}(\Delta \varphi_1^2)\big)\\
	&=\mathcal{O}_{\varepsilon}(\Delta \tau_2^2)+\mathcal{O}_{\varepsilon}(\Delta \tau_1^2)+\mathcal{O}_{\varepsilon}(\Delta \varphi_2^2)+\mathcal{O}_{\varepsilon}(\Delta \varphi_1^2).
	\end{align*}
	When we apply above argument and \eqref{tau and varphi} in Lemma \ref{compare} for next bounces, we obtain 
	\begin{align*}
	\vert \Delta \tau_{i+1} -\Delta \tau_i \vert =\mathcal{O}_{\varepsilon}(\Delta \tau_i^2)+\mathcal{O}_{\varepsilon}(\Delta \tau_{i+1}^2)+\mathcal{O}_{\varepsilon}(\Delta \varphi_i^2) + \mathcal{O}_{\varepsilon}(\Delta \varphi_{i+1}^2), \quad \forall i \geq I.  
	\end{align*}
\end{proof}

From Lemma \ref{difference}, we can prove that accumulation of bouncing points never happens away from inflection grazing. 
\begin{lemma} \label{no infty bounce}
Under the same assumption as in Lemma \ref{difference},  if $[x^i(x,v),v^i(x,v)] \notin \gamma_0^I$ for all $i\in \mathbb{N}$,  then 
\begin{equation*}
	\sum_{i=1}^\infty \vert x^i(x,v)-x^{i-1}(x,v)\vert = \infty.  
\end{equation*}
\end{lemma}
\begin{proof}
For a contradiction argument, we suppose that 
\begin{align*}
	\sum_{i=1}^\infty \vert x^i(x,v)-x^{i-1} (x,v) \vert < \infty. 
\end{align*}
And then, $x^i(x,v)(=\sigma(\tau_i, \varphi_i)) \rightarrow x^\infty(=\sigma(\tau_\infty, \varphi_\infty))\in \partial \Omega$. 
If  $x^{\infty} = \sigma(\tau_{\infty}, \varphi_{\infty})$ with $\tau_{\infty}\notin [\tau_{1,*}-\varepsilon, \tau_{2,*}+\varepsilon]$, the $x^{\infty}$ is an uniformly convex point of $\p\O$. Then by velocity lemma (see Lemma 1 of \cite{Guo10} or Lemma 2.7 of \cite{KimLee}), the trajectory cannot converge. \\
\indent Therefore, let us assume $\tau_{\infty}\in [\tau_{1,*}-\varepsilon, \tau_{2,*}+\varepsilon]$ for $x^{\infty} = \sigma(\tau_{\infty}, \varphi_{\infty})$. By Lemma \ref{difference}, there exist constants $C_{1,\varepsilon},C_{2,\varepsilon}>0$ depending on $\varepsilon>0$ such that 
	\begin{align}
	C_{1,\varepsilon}(\vert \Delta  \tau_{i+1}\vert^2 +\vert \Delta \tau_i\vert^2+\vert \Delta \varphi_{i+1}\vert^2 +\vert\Delta \varphi_i\vert^2) \leq \vert \Delta \tau_{i+1}-\Delta \tau_i\vert \notag \\
	 \leq C_{2,\varepsilon}(\vert \Delta  \tau_{i+1}\vert^2 +\vert \Delta \tau_i\vert^2+\vert \Delta \varphi_{i+1}\vert^2 +\vert\Delta \varphi_i\vert^2),   \label{equiv delta}
	\end{align} 
	for all $i\geq I$. Under assumptions of Lemma \ref{difference}, Lemma \ref{compare} holds obviously.  For notational convenience, we assume that \eqref{equiv delta} and \eqref{tau and varphi} in Lemma \ref{compare} hold for all $i\geq 1\;(I=1)$, WLOG. Then, for any $N>0$, we have 
	\begin{align}\label{sum}
	&\sum_{i=1}^N \vert \Delta \tau_{i+1}\vert
	\notag
	\\
	&\geq\sum_{i=1}^N (\vert \Delta \tau_{i}\vert -C_{2,\varepsilon}(\vert \Delta  \tau_{i+1}\vert^2 +\vert \Delta \tau_i\vert^2+\vert \Delta \varphi_{i+1}\vert^2 +\vert\Delta \varphi_i\vert^2))\notag\\
	&\geq  \sum_{i=1}^N (\vert \Delta \tau_i \vert-C_{2,\varepsilon}(1+C_{\varepsilon}^{-2})\vert\Delta \tau_{i+1}\vert^2-C_{2,\varepsilon}(1+C_{\varepsilon}^{-2})\vert \Delta \tau_{i}\vert^2) \notag\\
	&\geq\sum_{i=1}^N (\vert \Delta \tau_{i-1}\vert  -C_{2,\varepsilon}(\vert \Delta  \tau_{i}\vert^2 +\vert \Delta \tau_{i-1}\vert^2+\vert \Delta \varphi_{i}\vert^2 +\vert\Delta \varphi_{i-1}\vert^2) \notag \\
	&  \ \ \ \  \ \ \ -C_{2,\varepsilon}(1+C_{\varepsilon}^{-2})\vert\Delta \tau_{i+1}\vert^2-C_{2,\varepsilon}(1+C_{\varepsilon}^{-2})\vert \Delta \tau_{i}\vert^2)\notag \\
	&\geq \sum_{i=1}^N(\vert \Delta \tau_{i-1}\vert-C_{2,\varepsilon}(1+C_{\varepsilon}^{-2})\vert\Delta \tau_{i+1}\vert^2-2C_{2,\varepsilon}(1+C_{\varepsilon}^{-2})\vert \Delta \tau_i\vert^2 -C_{2,\varepsilon}(1+C_{\varepsilon}^{-2})\vert \Delta \tau_{i-1}\vert^2)\notag \\ 
	&\geq \sum_{i=1}^N\left(\vert\Delta \tau_1\vert -C_{2,\varepsilon}(1+C_{\varepsilon}^{-2})\vert\Delta \tau_{i+1}\vert^2-2C_{2,\varepsilon}(1+C_{\varepsilon}^{-2})\sum_{j=1}^{i-1} \vert\Delta \tau_{j+1}\vert^2-C_{2,\varepsilon}(1+C_{\varepsilon}^{-2})\vert \Delta \tau_{1}\vert ^2\right).
	\end{align}
Notice that constants $C_{\varepsilon}$ and $C_{2,\varepsilon}$ depend on $\varepsilon$ wth fixed $\varepsilon$. Since $\Delta \tau_i\rightarrow 0$ as $i\rightarrow \infty$, for any $\varepsilon>0$ there exists $N_1>0$ such that $\vert \Delta \tau_i \vert<\varepsilon$ for $i>N_1$. Let $m= \arg\max_{j\leq N_1} \vert \Delta \tau_j\vert$.  If we take $\varepsilon$ small enough to satisfy $\vert \Delta \tau_m\vert>\varepsilon$, then $\vert\Delta \tau_m\vert \geq \vert \Delta \tau_i \vert$ for all $i\in \mathbb{N}$. We change an index such as $ \Delta \tau_m\rightarrow \Delta \tau_1$, so that $\vert\Delta \tau_1 \vert \geq \vert\Delta \tau_i \vert$ for all $i \in \mathbb{N}$. Since  
\begin{align*}
	\sum_{i=1}^\infty \vert \Delta \tau_i \vert <\infty,
\end{align*}
for any $N>0$, there exists $N_2=N_2(N)>0$ such that 
\begin{align*}
	\vert \Delta \tau_i \vert< 1/N, \quad   \forall i\geq N_2,
\end{align*}
and 
	\begin{align*}
	\vert \Delta \tau_1\vert+\cdots+\vert \Delta \tau_{N_2}\vert<\frac{1}{4C_{2,\varepsilon}(1+C_{\varepsilon}^{-2})}.
	\end{align*}
Thus, from our assumption of $\Delta \tau_i$ above, we have
\begin{align*}
	 \sum_{i=1}^N \sum_{j=1}^{i+1} \vert \Delta \tau_j\vert^2 &= \sum_{i=1}^{N_2-1}\sum_{j=1}^{i+1} \vert \Delta \tau_j\vert^2 +\sum_{i=N_2}^N \sum_{j=1}^{N_2-1} \vert \Delta \tau_j\vert ^2 +\sum_{i=N_2}^N \sum_{j=N_2}^i \vert \Delta \tau_j\vert^2\\
	 &\leq  \sum_{i=1}^{N_2-1}\sum_{j=1}^{i+1} \vert \Delta \tau_j\vert^2 + \sum_{i=N_2}^N \sum_{j=1}^{N_2-1} \vert \Delta \tau_j\vert^2 +\sum_{i=N_2}^N \frac{1}{N}\sum_{j=N_2}^i \vert \Delta \tau_j \vert \\
	 &\leq \frac{C}{N}(N-N_2) \\
	 &\quad +\sum_{i=1}^{N_2-1}(\vert \Delta \tau_1\vert^2+\cdots+\vert\Delta \tau_{i+1}\vert^2)+ \sum_{i=N_2}^N (\vert\Delta \tau_1\vert^2+\cdots+\vert\Delta \tau_{N_2-1}\vert^2)   \\
	  &\leq C + \sum_{i=1}^N (\vert \Delta\tau_1\vert^2+\cdots+ \vert \Delta \tau_{N_2}\vert^2)\\
	  &\leq C+\sum_{i=1}^N \vert \Delta \tau_1\vert  (\vert \Delta \tau_1\vert+\cdots+\vert \Delta \tau_{N_2}\vert)\\
	 &\leq C+ \frac{1}{4C_{2,\varepsilon}(1+C_{\varepsilon}^{-2})} \vert \Delta \tau_1\vert N,
\end{align*}
where we have used Lemma \ref{compare}. Thus, the term in \eqref{sum} 
\begin{align*}
C_{2,\varepsilon}(1+C_{\varepsilon}^{-2})\sum_{i=1}^N (\vert \Delta \tau_{i+1}\vert^2+2\sum_{j=1}^{i-1} \vert \Delta \tau_{j+1}\vert^2+\vert \Delta \tau_1\vert^2)
\end{align*}
does not dominate $\sum_{i=1}^N \vert \Delta \tau_1 \vert$ since $\tau_i$ does not converge to some points near $\tau_{1,*}$ and $\tau_{2,*}$. Thus, as $N\rightarrow \infty$, $\sum_{i=1}^N \vert \Delta \tau_{i+1}\vert$ diverges which contradicts the hypothesis $\tau_i\rightarrow \tau_\infty$.  
\end{proof}

\subsection{Uniform number of bounce for $\gamma_0^I$}
The trajectory with total length $NT$ in $\O$ cannot wind the $z-$axis of the general toroidal domain more than $\mathbf{w} = \left(\left[\frac{NT}{2\pi r}\right]+1\right)$ times. So, we consider domain of $\sigma(\tau, \varphi) \in \p\O$ :
 \[
 	\tau \in [a,b),\quad \varphi \in (-2\pi\mathbf{w},0].
 \]
Note that grazing of trajectory cannot happen for  $\tau\in[\tau_{2,*}+\varepsilon, \tau_{1,*}-\varepsilon]$  (uniformly convex part) by Lemma \ref{lem:nograzing_V}. 
If $\tau\in [\tau_{1,*}-\varepsilon, \tau_{1,*}+\varepsilon]\cup [\tau_{2,*}-\varepsilon, \tau_{2,*}+\varepsilon]$, we will exclude near inflection sets $(\sigma(\tau, \varphi), I_{1,2}(\tau, \varphi))$ as mentioned in \eqref{assump 1} in Lemma \ref{compare}. Similarly, if $\tau\in (Z_h)_{\varepsilon}$, we also exclude near inflection sets in \eqref{assump 3} of Lemma \ref{difference}. \\
 \indent Therefore, from now on, we consider the domain $(\tau,\varphi)$ 
\Be \label{Dom:thetaphi}
	\tau \in [\tau_{1,*}+\varepsilon, \tau_{2,*}-\varepsilon]\backslash (Z_h)_{\varepsilon} ,\quad \varphi \in (-2\pi\mathbf{w},0].  \\
\Ee
where $(Z_{h})_{\varepsilon}$ was defined in \eqref{Z_h nbh}. Recall that we use $S_{\varphi}$ to denote $\varphi$ cross-section. Note that we carefully distinguish $S_{\varphi}$ and $S_{\varphi+2\pi}$, because $\mathbf{w}$, the winding number of the trajectory with respect to $z$-axis during $[0,T_{0}]$ can be large in general. \\

First, we start with the lemma about the basic property arising from axis-symmetry of $\O$.
\begin{lemma} \label{no consecutive}
	(There is no consecutive grazing in $\Omega$.) Let $(x,v) \in \overline\Omega \times \mathbb{V}^N$. We assume that the angular momentum $w(x,v)$ defined in \eqref{def_ang moment} is positive. Then, $n(x^i(x,v))\nparallel n(x^{i+1}(x,v))$ where $n(y)$ is an outward unit normal vector at $y\in \p\O$. Hence, if $[x^i(x,v),v^{i-1}(x,v)]\in \gamma_0$, then $[x^{i+1}(x,v),v^{i}(x,v)] \notin \gamma_0$. 
\end{lemma}
\begin{proof}
	Note that trajectory is not located on only one cross-section due to positive angular momentum. If there are two consecutive concave grazing, two tangential planes at $x^{i}$ and $x^{i+1}$ must coincide, or equivalently, $n(x^{i}(x,v)) \parallel n(x^{i+1}(x,v))$ which is impossible in $\O$. If we set $x^j(x,v)=\sigma(\tau_j,\varphi_j)\in \p\O$ for $j=i,i+1$, then $\varphi_{i+1}=\varphi_i+\pi$ to satisfy $n(x^{i}(x,v)) \parallel  n(x^{i+1}(x,v))$. However, it cannot happen in domain $\O$ which rotates the uniformly convex region.    
\end{proof}
Let us consider a trajectory segment between two points. By Lemma \ref{no consecutive}, we can assume at least one point between them is non-grazing. 
\begin{lemma} \label{basic}
		Assume $(x,v)\notin \gamma_{0}^{I}$ and $(\xf^{1}(x,v), \vf^{1}(x,v))\notin \gamma_{0}^{I}$.  \\ 
		(a) Let us assume $(\xf^{1}(x,v), \vf^{1}(x,v))\notin \gamma_{0}$. 
		For any $\varepsilon > 0$, there exists corresponding $\delta>0$ which satisfies : if $(y,u)\in \{\p\O\times \mathbb{S}^{2}\}$ and $|(x,v) - (y,u)| < \delta$, then $(y,u) \notin \gamma_{0}^{I}$ and satisfies
		\[
			| (\xf^{1}(x,v), \vf^{1}(x,v)) - (\xf^{1}(y,u), \vf^{1}(y,u)) | < \varepsilon.
		\]
		(b) Let us assume $(x, v)\notin \gamma_{0}$ and $(\xf^{1}(x,v), \vf^{1}(x,v))\in \gamma_0^{C}$.
		 In this case, we have the following dichotomy for $(y,u)$ close to $(x,v)$:  \\
		\indent (b-1) $ (\xf^{1}(y,u), \vf^{1}(y,u))$ is closer to $ (\xf^{1}(x,v), \vf^{1}(x,v))$ than $ (\xf^{2}(x,v), \vf^{2}(x,v))$. In the case of (b-1), for sufficiently small $\varepsilon \ll 1$, there exists $\delta > 0$ such that if   $|(x,v) - (y,u)| < \delta$,
		\Be \label{b1}
		| (\xf^{j}(x,v), \vf^{j}(x,v)) - (\xf^{j}(y,u), \vf^{j}(y,u)) | < \varepsilon,\quad j=1,2. \\ 
		\Ee \\
		\indent (b-2) $ (\xf^{1}(y,u), \vf^{1}(y,u))$ is closer to $ (\xf^{2}(x,v), \vf^{2}(x,v))$ than $ (\xf^{1}(x,v), \vf^{1}(x,v))$. In the case of (b-2), for sufficiently small $\varepsilon \ll 1$, there exists $\delta > 0$ such that if   $|(x,v) - (y,u)| < \delta$,
		\Be \label{b2}
		| (\xf^{2}(x,v), \vf^{2}(x,v)) - (\xf^{1}(y,u), \vf^{1}(y,u)) | < \varepsilon.  
		\Ee  
		Moreover, we choose sufficiently small $\delta_*>0$ so that $(\xf^{1}(y,u), \vf^{1}(y,u)) \notin \gamma_0$ whenever $|(x,v) - (y,u)| < \delta_*$.
		
\end{lemma}
\begin{proof}
	\textbf{Proof of (a) :} First note that $\gamma_0^{I}$ is parametrized by $(\sigma(\tau, \varphi), I^{1,2}(\tau, \varphi))$ which is smooth in $(\tau, \varphi)$. Since $I^{1,2}(\tau, \varphi)$ is explicit two directions, $\gamma_0^{I}$ is closed and we can choose sufficiently small $\delta_{*} \ll 1$ so that $(y,u)\notin\gamma_0^{I}$ whenever $|(y,u)-(x,v)| < \delta_{*}$. Now, we have only two cases : (i) $(y,u)\in \gamma_0^{C}$ or (ii) $(y,u)\notin \gamma_0$. For fixed $(x,v)$, we consider $\xi(x+v\ell)$ where $\xi(x+v\tf(x,v))=\xi(\xf(x,v))=0$. Since we assume no grazing bounce,
	\Be \label{forward implicit}
		\frac{\p}{\p\ell}\xi(x+v\ell)\vert_{\ell=\tf(x,v)} = \nabla\xi(\xf(x,v))\cdot v \neq 0,
	\Ee
	and so by Lemma \ref{analytic IFT}, both $\tf(x,v)$ and $\xf(x,v) = x+v\tf(x,v)$ are locally analytic function of $(x,v)$. Also, $\vf^1(x,v) = R_{\xf(x,v)}v$ is locally analytic. For both (i) and (ii) cases, we have continuity of $(\xf^{1}(y,u), \vf^{1}(y,u))$ near $(x,v)$ by choosing $\delta$ smaller than $\delta_{*}$.  \\
	
	
	\textbf{Proof of (b) :}  We assume $(\xf^{1}(x,v), \vf^{1}(x,v)) \in \gamma_{0}^{C}$. By dichotomy, (b-1) or (b-2) must happen. If (b-1) happens, there exists $\delta_{*} \ll 1$ such that if $(x^{\prime}, v^{\prime})\in \p\O\times \mathbb{S}^{2}$ satisfies $|(x^{\prime}, v^{\prime}) - (\xf^{1}(x,v), \vf^{1}(x,v)) | < \delta_{*}$, then $(x^{\prime}, v^{\prime})\notin \gamma_0^{I}$ since $\gamma_0^{I}$ is closed. Now let us consider $\overline{B_{\delta}(\xf^{1}(x,v), \vf^{1}(x,v))  } \subset B_{\delta_{*}}(\xf^{1}(x,v), \vf^{1}(x,v))$. Considering invertible (and locally continuous for sufficiently small $\delta > 0$ by (a)) map,
	\[
		(\xb, \vb) : (x,v)\in \{\p\O\times\mathbb{S}^{2}\}\backslash \gamma_{0}^{I} \mapsto (\xb, \vb)(x,v),
	\]   
	$(\xb, \vb)$ is continuously onto map from $\overline{B_{\delta}(\xf^{1}(x,v), \vf^{1}(x,v))  }$ to $(\xb, \vb)(\overline{B_{\delta}(\xf^{1}(x,v),}$\\
	$\overline{ \vf^{1}(x,v))  })$ where the range includes $(x,v)$. Therefore, the inverse map 
	\[
		(\xf^{1}, \vf^{1}) : (\xb, \vb)(\overline{B_{\delta}(\xf^{1}(x,v), \vf^{1}(x,v))  }) \mapsto \overline{B_{\delta}(\xf^{1}(x,v), \vf^{1}(x,v))  } \notin \gamma_0^{I}
	\]
	is also continuous. Now using (a) again, the map
	\[
		(\xf^{1}, \vf^{1}) : \overline{B_{\delta}(\xf^{1}(x,v), \vf^{1}(x,v))  } \mapsto U((\xf^{2}(x,v), \vf^{2}(x,v))) 
	\]
	is also continuous where $U((\xf^{2}(x,v), \vf^{2}(x,v)))\in \p\O\times \mathbb{S}^{2}$ is a neighborhood of $(\xf^{2}(x,v),$\\$ \vf^{2}(x,v))$. This proves \eqref{b1} for (b-1) case.   \\
	 
	 \hide
	we get 
	 for sufficiently small $\varepsilon \ll 1 $ which is smaller than $\varepsilon_{*}$ WLOG, $(\xf^{1}(y,u), \vf^{1}(y,u)) \notin \gamma_0^{I}$. Again, we split into two cases :    \\
	\indent (b-1-i) $(\xf^{1}(y,u), \vf^{1}(y,u))\notin \gamma_0$,  \\
	\indent (b-1-ii) $(\xf^{1}(y,u), \vf^{1}(y,u))\in \gamma_0^{C}$.  \\
	
	For (b-1-i), map $(x,v)\mapsto (\xf^{1}(x,v), \vf^{1}(y,u))$. \\
	
	For (b-1-ii), is easier. \\
	\unhide
	
	If (b-2) happens, we compare $(\xf^{1}(y,u), \vf^{1}(y,u))$ and $(\xf^{2}(x,v), \vf^{2}(x,v))$. From Lemma \ref{no consecutive}, $(\xf^{2}(x,v), \vf^{2}(x,v))\notin \gamma_0$ because we are assuming $(\xf^{1}(x,v), \vf^{1}(x,v)) \in \gamma_0^{C}$. Since we assume (b-2), if we ignore $(\xf^{1}(x,v), \vf^{1}(x,v))$, map $(x,v)\mapsto (\xf^{2}(x,v), \vf^{2}(x,v))$ must be continuous for sufficiently small $\varepsilon \ll 1$. So we proved \eqref{b2}. Since $\gamma_0$ is closed set, by choosing sufficiently small $\delta_* \ll 1$ WLOG, we can guarantee $ (\xf^{1}(y,u), \vf^{1}(y,u))$\\
	$ \notin \gamma_0$.   

	\hide
	{\color{blue} {\bf Ignore} Nearly for case (ii) \\
	Since $(x^{c_{1}}, v^{c_{1}})(x,v) \in \gamma_0^{C}$, $x^{c_{1}}(x,v) = \sigma(\tau_{c_{1}}(x,v), \varphi_{c_{1}}(x,v))$, for some $(\tau_{c_{1}}(x,v), \varphi_{c_{1}}(x,v))$ with $\tau_{1,z} < \tau_{c_{1}}(x,v) < \tau_{2,z}$ ({\color{red} notation change later}) and
	\Be \label{cond}
	v^{c_{1}-1}(x,v) = v^{c_{1}}(x,v) = \alpha I^{1}(\tau_{c_{1}}(x,v), \varphi_{c_{1}}(x,v))  + \beta I^2(\tau_{c_{1}}(x,v), \varphi_{c_{1}}(x,v)),\quad \text{for some}\quad \alpha, \beta \in (0,1).
	\Ee
	
	Similalry, we write 
	\[
	x^{c_{1}}(y, u) = \sigma(\tau_{c_{1}}(y,u), \varphi_{c_{1}}(y,u)),
	\] 	
	and then from \eqref{close}	
	\[
	|(\tau_{c_{1}}(x,v), \varphi_{c_{1}}(x,v)) - (\tau_{c_{1}}(y,u), \varphi_{c_{1}}(y,u)) | < \varepsilon
	\]
	also. Since $\p\O$ is smooth, $I^{1,2}(\tau, \varphi) \in \mathbb{S}^{2}$ is also continuous and we can claim
	\[
	|I^{1,2}(\tau_{c_{1}}(x,v), \varphi_{c_{1}}(x,v)) - I^{1,2}(\tau_{c_{1}}(y,u), \varphi_{c_{1}}(y,u))| < \varepsilon.
	\]
	for sufficiently (even smaller if  necessary) $\delta \ll 1$. \\
	
	Now let us project $v^{c_{1}-1}(y,u)$ onto the tangent plane of $x^{c_{1}}(y, u)$. We define
	\Be
	Proj_{c_{1}}v^{c_{1}-1}(y,u) := \big(I - n(x^{c_{1}}(y,u))\otimes n(x^{c_{1}}(y,u)) \big) v^{c_{1}-1}(y,u).
	\Ee 
	(Claim 1) By continuity, for sufficiently small perturbation, the (normalized) projected velocity satisfies
	\[
	\frac{Proj_{c_{1}}v^{c_{1}-1}(y,u) }{ |Proj_{c_{1}}v^{c_{1}-1}(y,u) | }
	= \alpha^{\prime} I^{1}(\tau_{c_{1}}(y,u), \varphi_{c_{1}}(y,u))  + \beta^{\prime} I^2(\tau_{c_{1}}(y,u), \varphi_{c_{1}}(u,u)), \\
	\]
	for some $\alpha^{\prime}, \beta^{\prime} \in (0,1)$, i.e., still concave grazing.  \\
	Use orthonormal transform $Q$ such that
	\[
	Q \begin{pmatrix}
	\vert & \vert & \vert  \\
	e_{1} & e_{2} & e_{3}  \\
	\vert & \vert & \vert  
	\end{pmatrix}_{x,v}
	=
	\begin{pmatrix}
	\vert & \vert & \vert  \\
	e_{1} & e_{2} & e_{3}  \\
	\vert & \vert & \vert  
	\end{pmatrix}_{y,u}
	\]
	and $\alpha, \beta \in (0,1)$ argument etc... \\
	
	In other words, the projected velocity $\frac{Proj_{c_{1}}v^{c_{1}-1}(y,u) }{ |Proj_{c_{1}}v^{c_{1}-1}(y,u) | }$ is concave grazing at $x^{c_{1}}(y, u)$, not $\gamma_{0}^{I}$.   \\
	} 
	\unhide
\end{proof}
\begin{definition} Let $L>0$. We define the number of bounce $\mathcal{N}(x,v,L)$ as follows,
	{\footnotesize
	\begin{align} \label{nongrazing NB}
	\mathcal{N}(x,v,L) := \sup \left \{ k \in \mathbb{N} : (x^j(x,v),v^{j-1}(x,v))\notin \gamma_0^{I_-}, \; \;  \forall 1\leq j \leq k \; \text{and} \; 	 \  \sum_{j=1}^k \vert x^j(x,v)-x^{j-1}(x,v) \vert \leq L 	\right \}. 
	\end{align}
	}
\end{definition}
Above lemma helps us to count bounce number of perturbed trajectory.
\begin{lemma} \label{analyticity} (Analyticity of  non-grazing trajectory) 
	Assume that $(x,v) \in \overline\Omega \times \mathbb{V}^N$ and 
	\[
		\mathcal{N}(x, v, N(t-s)) := M_{1} < \infty.
	\]
	i.e., there are finite $M_{1}<\infty$ bounces while backward in time trajectory travels total length $L$ from $(x,v)$ under specular reflection. If 
	\begin{align}
	(x^i(x,v),v^i(x,v))\notin \gamma_0 \text{ for all } 1\leq i \leq M_{1},
\label{no_grazing}\\
	\text{resp, } (x^i(x,v),v^i(x,v))\notin \gamma_0^{I}
	 \text{ for all } 1\leq i \leq M_{1},
\label{no_Igrazing}
	\end{align}
	then \\
	(a) $(t^{i}(t,x,v), x^{i}(x,v), v^{i}(x,v))$ is locally analytic function of $(x,v)$ (resp, locally continuous function of $(x,v)$ up to some proper renumbered bounce number). \\
	(b) Under assumption of (a), if $s\notin t^{i}(t,x,v)$ for any $1\leq i \leq M_{1}$, then $(X(s;t,x,v),$\\
	$ V(s;t,x,v))$ is locally analytic  (resp, locally continuous) function of $(x,v)$ for fixed $s$.   \\
	(c) There existx $\delta_{x,v} \ll 1$ such that if $|(y,u)-(x,v)| < \delta_{x,v}$ then $\mathcal{N}(y, u, N(t-s)) \leq M_{1}$  (resp, $\mathcal{N}(y, u, N(t-s)) \leq M_{1}$).   
\end{lemma} 	

	\hide	
	Then, for any $\varepsilon>0$, there exists $\delta_{x,v,\varepsilon}>0$ such that 
	\begin{equation*}
	\vert (x^i, v^i)(x,v) - (x^i, v^i)(y,u) \vert <\varepsilon,
	\end{equation*}
	whenever $\vert (x,v)-(y,u)\vert <\delta_{x,v,\varepsilon}$. In addition, $(x^i(y,u),v^i(y,u)) \notin \gamma_0$ for all $i\in \mathbb{N}$. \\
	\unhide

\begin{proof}
	\textit{Step 1.} Let us prove the case of \eqref{no_grazing}, first. Similar to the proof of (a) in Lemma \ref{basic}, for fixed $(x,v)$, we consider $\xi(x-v\ell)$ where $\xi(x-v\tb(x,v))=\xi(\xb(x,v))=0$ and take derivative
	\Be \label{implicit}
		\frac{\p}{\p\ell}\xi(x-v\ell)\vert_{\ell=\tb(x,v)} = \nabla\xi(\xb(x,v))\cdot v \neq 0,
	\Ee
	because there is no grazing bounce. By Lemma \ref{analytic IFT}, $\tb(x,v),\xb(x,v) = x-v\tb(x,v)$ and $v^{1}(x,v) = R_{\xb(x,v)}v$ are locally analytic function of $(x,v)$. By the same argument,
	\[
		(x^{i}, v^{i})\mapsto 	(x^{i+1}, v^{i+1}), 
	\]
	 is also analytic and by chain rule, we prove (a). Now let us prove (b). By \eqref{implicit} and chain rule, it is $t^{i}(t,x,v)$ is also locally analytic in $(x,v)$ for fixed $t$. Therefore, mapping
	 \[
	 \begin{split}
	 	(x,v) &\mapsto (t^{i}(t,x,v), x^{i}(x,v), v^{i}(x,v)) \mapsto (t^{M_{1}}(t,x,v), x^{M_{1}}(x,v), v^{M_{1}}(x,v)) \\
		&\mapsto (X(s;t,x,v), V(s;t,x,v))
	\end{split}
	 \]
	 is also analytic because we have local formula $X(s;t,x,v) = x^{M_{1}}(x,v) - v^{M_{1}}(x,v)(t^{M_{1}}(t,$\\
	 $x,v)-s)$ and $V(s;t,x,v)=v^{M_{1}}(x,v)$ when $s\notin t^{M_{1}}(t, x, v)$.  \\ 
	 Now let us prove (c). We first note that
	 $
	 	v^{i}(x,v)\cdot n(x^{i}(x,v))
	 $
	 is also locally analytic by (a). Therefore, there exists $\delta_{1}\ll 1$ such that if $|(y,u)-(x,v)|<\delta_1$, then 
	 \Be \label{normal angle}
	 	|v^{M_{1}}(y,u)\cdot n(x^{M_{1}}(y,u))| \geq \frac{1}{2}| v^{M_{1}}(x,v)\cdot n(x^{M_{1}}(x,v))|.
	 \Ee
	 By finite curvature of $\p\O$, we can estimate an upper bound of $|x^{M_{1}}(x,v)-x^{M_{1}+1}(x,v)|=|t^{M_{1}}(t,x,v)-t^{M_{1}+1}(t,x,v)|$. Since it suffices to control for small $|x^{M_{1}}(x,v)-x^{M_{1}+1}(x,v)|$
	 $ \ll 1$,
	 \begin{align*}
	 \begin{split}
		 	0 &= \xi(x^{M_{1}}(x,v)) - \xi(x^{M_{1}+1}(x,v)) = \int_{t^{M_{1}+1}}^{t^{M_{1}}} \frac{\p}{\p s}\xi(X(s;t,x,v)) ds \\
		 	&= \int_{t^{M_{1}+1}}^{t^{M_{1}}} \nabla\xi(X(s;t,x,v)) \cdot V(s;t,x,v) ds \\
		 	&= \int_{t^{M_{1}+1}}^{t^{M_{1}}} \Big[ \nabla\xi(x^{M_{1}}(x,v)) + \nabla^{2}\xi(x^{M_{1}}(x,v)) v^{M_{1}}|s-t^{M_{1}+1}(t,x,v)| \\
			&\qquad \qquad+ \mathcal{O}(|s-t^{M_{1}}(t,x,v)|^{2}) \Big] \cdot v^{M_{1}}(x,v) ds. \\
	 \end{split}
	 \end{align*}
	 Using $|\nabla^{2}\xi|\lesssim 1$ and $\min_{x\in \partial \O}|\nabla\xi(x)| \gtrsim 1$, we derive
	 \Be \label{min length}
			C| v^{M_{1}}(x,v)\cdot n(x^{M_{1}}(x,v))| \leq |t^{M_{1}}(t,x,v)-t^{M_{1}+1}(t,x,v)|,\quad \text{for some generic $C>0$.}
	 \Ee
	 If $t^{M_{1}}(t, y, u) \leq s $, $\mathcal{N}(y,u,N(t-s)) \leq M_{1}$ by definition. So we only consider $t^{M_{1}}(t, y, u) > s $ case. Choose $\delta_{2}\ll 1$ (smaller than $\delta_{1}$ WLOG) so that
	 \[
	 		|X(s;t,y,u) - x^{M_{1}}(t,y,u)| \leq \frac{C}{2}| v^{M_{1}}(x,v)\cdot n(x^{M_{1}}(x,v))|,
	 \] 
	 whenever $|(x,v)-(y,u)| < \delta_{2}$. By \eqref{normal angle} and \eqref{min length}, 
	 \begin{equation}\begin{split}\notag
	 	|X(s;t,y,u) - x^{M_{1}}(t,y,u)| &\leq \frac{C}{2}| v^{M_{1}}(x,v)\cdot n(x^{M_{1}}(x,v)) \leq C|v^{M_{1}}(y,u)\cdot n(x^{M_{1}}(y,u))| \\
	 	& \leq |t^{M_{1}}(t,y,u)- t^{M_{1}+1}(t,y,u)|.
	 \end{split}\end{equation}
	 Hence, there cannot be extra bounce between time $s$ and $t^{M_{1}}(t,y,u)$. We finish the proof of (c) for the case of \eqref{no_grazing}.  \\
	 
	\noindent \textit{Step 2.} Now, we prove $\notin \gamma_0^{I}$ case. Note that $\gamma_{0}^{I+}$ cannot  happen in finite bounce for backward in time trajectory. We exclude $\gamma_0^{I}$ and only consider concave grazing $\gamma_{0}^{C}$ in addition.  \\	

	\noindent To prove (a), note that from (b) of Lemma \ref{basic}, we have continuity of backward in time trajectory. For each concave grazing, if (b-2) case happens, bounce number is reduced by $1$ and we get continuity of bounce number after some renumbering.   
	
	\hide
	(up to renumbering bounce number) Consider small perturbation $|(x,v)-(y,u)| < \delta \ll 1$. Wlog, let $(x^{c_{1}}(x,v), v^{c_{1}}(x,v))$ be first backward in time concave grazing. By (a) of \textit{Step 1} (non grazing case), the trajectory is locally analytic until $(c_{1}-1)$-th bounce and $|(x^{c_{1}-1}(y, u), v^{c_{1}-1}(y, u)) - (x^{c_{1}-1}(x, v), x^{c_{1}-1}(x, v))| < \varepsilon \ll 1$ for arbitrary small $\varepsilon$  we choose. For any small perturbation, there are three scenarios for perturbed $(x^{c_{1}}(y,u), v^{c_{1}}(y,u))$: {\color{blue}Question: Why we can exclude $\gamma_0^{I}$ case?} \\
	
	\noindent {\bf (i)} $c_{1}$-th bounce of trajectory from $(y,u)$ hits near $(x^{c_{1}}(x,v), v^{c_{1}}(x,v))$ grazingly, \\
	{\bf (ii)} $c_{1}$-th bounce of trajectory from $(y,u)$ hits near $(x^{c_{1}}(x,v), v^{c_{1}}(x,v))$ nongrazingly, \\
	{\bf (iii)} $c_{1}$-th bounce of trajectory from $(y,u)$ does not hit near $(x^{c_{1}}(x,v), v^{c_{1}}(x,v))$, but hit very close to $(x^{c_{1}+1}(x,v), v^{c_{1}+1}(x,v))$ nongrazingly. (From Lemma \ref{no consecutive}, $(x^{c_{1}+1}(x,v), v^{c_{1}+1}(x,v))\notin \gamma_0$.) \\ 
	
	{\color{blue} Quote result of Lemma \ref{basic}} \\
	Case (i) and (iii) are rather simple. For case (iii), $(x^{c_{1}}(y, u), v^{c_{i}}(y, u))$ is sufficiently close to $(x^{c_{1}+1}(x, v), v^{c_{i}+1}(x, v))$ obviously by continuity of nongrazing trajectory ignoring $(x^{c_{1}}(x, v), v^{c_{i}}(x, v))$. For case (i), it is obvious that $|x^{c_{1}}(y,u) - x^{c_{1}}(x,v)|$ is sufficiently small by smoothness of $\p\O$. We also easily get continuity of velocity from
	\[
		|v^{c_{1}}(y,u) - v^{c_{1}}(x,v)| = |v^{c_{1}-1}(y,u) - v^{c_{1}-1}(x,v)| < \varepsilon.
	\]
	$(x^{c_{1}+1}(y, u), v^{c_{i}+1}(y, u))$ is also sufficiently close to $(x^{c_{1}+1}(x, v), v^{c_{i}+1}(x, v))$ by continuity ignoring both $(x^{c_{1}}(x, v), v^{c_{i}}(x, v))$ and $(x^{c_{1}}(y, u), v^{c_{1}}(y, u))$. \\

	{\color{blue} Quote result of Lemma \ref{basic}} \\
	{\bf The most tricky part is case (ii).} We claim that for sufficiently small $\delta>0$, $(x^{c_{1}}(y, u), v^{c_{1}}(y, u))$ {\bf hits near $(x^{c_{1}}(x, v), v^{c_{1}}(x, v))$ only once} in the following sense: \\
	 Let us choose a small $\varepsilon > 0$. {\bf In particular, to distinguish small neighborhood of each $x^{i}(x,v)$, we always choose $\varepsilon$ sufficiently small so that
	 \[
	 	\varepsilon \ll \min_{1\leq i \leq M_{1}}|x^{i}(x,v) - x^{i+1}(x,v)|.
	 \]
	} 
	 We claim that there exists $\delta \ll 1$ such that if $|(x,v)-(y,u)| < \delta$, then
	\Be \label{close}
		|(x^{c_{1}}(x, v), v^{c_{1}}(x, v)) - (x^{c_{1}}(y, u), v^{c_{1}}(y, u))| < \varepsilon.
	\Ee
	{\color{red} Above claim will be proved in the proof of Lemma \ref{basic}} 
	
	So for all possible three cases, we proved (a) until its first bounce. We repeat above argument for second concave grazing $(x^{c_{2}}(x,v), v^{c_{2}}(x,v))$. Whenever case (iii) happen for each $\gamma_{0}^{C}$, bounce number will be reduced by $1$. This proves (a). \\
	\unhide
	
	Proof for (b) is nearly the same as \textit{Step 1}. From Lemma \ref{basic}, we use continuity of the following mapping
	\Be
	\begin{split}\notag
		(x,v) &\mapsto (t^{i}(t,x,v), x^{i}(x,v), v^{i}(x,v)) \mapsto (t^{M_{1}}(t,x,v), x^{M_{1}}(x,v), v^{M_{1}}(x,v))\\
		&  \mapsto (X(s;t,x,v), V(s;t,x,v)).
		\end{split}
	\Ee
	We omit the detail. Proof for (c) is also the same as \textit{Step 1} case. Note that whenever (b-2) of dichotomy in Lemma \ref{basic} happens, bounce number is reduced by $1$ and hence $\mathcal{N}(y,u,N(t-s))$ is maximal when there is no concave grazing. 
\end{proof}

\hide
By chain rule and finite number of bounce, it suffices to claim only for first bounce. In fact, these are the direct outcome of the  basic computation (\cite{CKL1, GKTT1})
\Be\label{nabla_x_bv_b}
\begin{split}
	&	\nabla_{x} \tb = \frac{ \nabla \xi(\xb)}{ \nabla \xi (\xb) \cdot v}
	,
	\  \ \
	\nabla_{v} \tb =
	- \tb 	\nabla_{x} \tb,
	\\
	&		\nabla_{x} x_{\mathbf{b}} =  I - \frac{v\otimes n(\xb)}{v\cdot n(\xb)},
	\  \ \ \nabla_{v} x_{\mathbf{b}} = - \tb  \nabla_{x} x_{\mathbf{b}},  \\
	&		\nabla_{x} x_{\mathbf{b}} =  I - \frac{v\otimes n(\xb)}{v\cdot n(\xb)},
	\  \ \ \nabla_{v} x_{\mathbf{b}} = - \tb  \nabla_{x} x_{\mathbf{b}},  \\
	& \nabla_{x} n(\xb) = \frac{1}{|\nabla\xi(\xb)|} \Big( I - n(\xb)\otimes n(\xb) \Big)\nabla^{2}\xi(\xb).  \\
\end{split}
\Ee
Since the trajectory does not graze for each bounce, above computations are well-defined and continuity of $x^{i}(x,v)$ follows. For continuity of $v^{i}(x,v)$ we note that $\vb(x,v) = (I - 2n(\xb)\otimes n(\xb))v$ which is locally differentiable by \eqref{nabla_x_bv_b}.
\unhide

Next, we define arrival time function $s(x, \varphi, v)$. Before, we define the function we should specify $\varphi$-coordinate of $x$ since winding number $\mathbf{w} > 1$ in general. \\

\begin{definition} \label{bfX}
	For given $(x,\varphi)\in \overline{\O}\times (-2\pi\mathbf{w},0]$, we use $\mathbf{x}=\mathbf{x}(x,\varphi)$  to denote position $x$ contained in $S_{\varphi}$. 
\end{definition}

\begin{lemma} \label{lem time s}
	 (Arrival time function $s(x,\varphi,v)$) Assume positive angular momentum (with respect to the $z$-axis) $\omega(x,v) > 0$ defined in \eqref{def_ang moment} and the forward in time trajectory from $(x,v)\in \overline{\O}\times \mathbb{S}^{2}$ is well-defined until it passes $S_{0}$ defined in \eqref{cross section}(i.e., forward in time trajectory does not belong $\gamma_0^{I+}$ until it passes $S_{0}$ cross-section). \\
	 \indent When $\phi(x) = \varphi$ in $\mod 2\pi$ (where $\phi$ is defined in Definition \ref{cross section}),  we define arrival time function $s(x,\varphi,v)$ as time (or distance since $|v|=1$) to reach $S_0$ from $(x, v) \in S_\varphi \times \S^{2}$ via forward in time trajectory. Then, the arrival time function $s(x,\varphi,v)$ which satisfies
	\Be \label{def time s}
		X(s(x,\varphi,v); 0, x, v) \in S_{0}
	\Ee
	is well-defined and continuous in $(x,\varphi,v)$. \\
	Moreover, if the trajectory does not belong to $\gamma_0$ (resp, $\gamma_0^{I}$) and 
	\Be \label{notin bdry}
		X(s(x,\varphi,v); 0, x, v) \notin \p S_{0},
	\Ee
	in addition, then both $s(x,\varphi,v)$ and $(X, V)(s(\cdot,\cdot, \cdot); 0, \cdot, \cdot)$ are also locally analytic (resp, locally continuous) in $(x,\varphi,v)$.
\end{lemma}

Before we give proof for above lemma, let us define some notation. When position and velocity(direction) are parametrized by $(\tau, \varphi)$ for grazing phase $(x,v) = (\sigma(\tau, \varphi), I^{2}(\tau, \varphi)) \in \gamma_0^{I-}$, we abbreviate
\Be \label{s_theta_phi}
s(\sigma(\tau, \varphi),\varphi, I^{2}(\tau, \varphi)) = s(\tau, \varphi),
\Ee
for notational convenience. 
Also, we use the following abbreviation for concave grazing case :
\begin{equation} \label{s_tau_varphi_eta} 
X(s(\tau,\varphi,\eta);0,\tau,\varphi,\eta):= X(s(\tau,\varphi,\eta);0,\sigma(\tau,\varphi),v_c(\eta)),  
\end{equation}
where we used $s(\tau,\varphi,\eta)=s(\sigma(\tau,\varphi),\varphi,v_c(\eta))$ with definition \eqref{concave para}. 
\begin{proof} [Proof of Lemma \ref{lem time s}]
	If forward in time trajectory does not belong to $\gamma_{0}^{I+}$, the trajectory is well-defined. Now, we consider $\varphi$-directional angular velocity. Since the distance from a point $x\in \p\O$ to the $z$-axis is given by $|x\times \hat{z}| =|(x_2,-x_1,0)|\geq C > 0$ for some uniform $C$, the angular velocity of the trajectory $(X(s;0,x,v), V(s;0,x,v))$ which is given by
	\Be \notag
	\begin{split}
		\omega_{\text{ang}}(X(s;0,x,v), V(s;0,x,v)) &:= \frac{1}{|X(s;0,x,v)\times \hat{z}|}\omega(X(s;0,x,v), V(s;0,x,v)) \\
		& = \frac{1}{|X(s;0,x,v)\times \hat{z}|}\omega(x,v)\\
		& = \frac{1}{| (X_2 (s;0,x,v), - X_1 (s;0,x,v),0)|}\omega(x,v)	\\
& \geq \frac{1}{C} \o (x,v),
	\end{split}
	\Ee
	is bounded from below. Hence $\varphi$-coordinate of $X(s;0,x,v)$ increases monotonically as time $s$ increases. Note that $\varphi(x)$ is not well-defined if we denote $\varphi(x)$ as (negative) $\varphi$-coordinate of $x$. 
	Then,
	\Be \label{ang vel}
		\Phi(s_{*}, x,\varphi,v) := \int_{0}^{s_{*}} \frac{\omega(\mathbf{x}(x,\varphi),v)}{|X(s;0,\mathbf{x}(x,\varphi),v)\times \hat{z}|} ds = -\varphi > 0,
	\Ee
	where $\mathbf{x}(x,\varphi)$ is defined in Definition \ref{bfX}. Since the integrand is uniformly positive and finite, $\Phi(s_{*}, x, \varphi, v)$ is monotone increasing function of $s_{*}$ and there exists a unique $s_{*}$ as a function of $(x,\varphi,v)$ which solves above equation. Defining $s_{*}=s(x,\varphi,v)$, $s(x,\varphi,v)$ is well-defined. Continuity of $s(x,\varphi,v)$ comes from \eqref{ang vel} because the integrand is continuous in $(x,\varphi,v)$ for all $0 \leq s \leq s_{*}$. \\
	
	Now, we assume forward in time trajectory does not belong $\gamma_0$ and $X(s(x,\varphi,v); 0, x,$\\
	$ v) \notin \p\O$, in addition. $s_{*} = s(x,\varphi,v)$ solves
	\[
		\tilde{\varphi}(s_{*}, x,\varphi,v) := X(s_{*}; 0, \mathbf{x}(x,\varphi), v) \cdot e_{2} = 0,
	\]
	where $e_{2}$ is $y$-direction vector, and $S_0$ belongs $xz$ plane. \\
	Note that $\tilde{\varphi}$ is locally analytic since $X(s; 0, x, v)$ is locally analytic by (b) of Lemma \ref{analyticity}. Since $(\p_{s}\tilde{\varphi})(s(x,\varphi,v), x, \varphi,v) = V(s(x,\varphi,v);0, \mathbf{x}(x,\varphi), v)\cdot e_{2} \neq 0$ by nonzero angular momentum, $s(x,\varphi,v)$ is locally analytic by implicit function theorem. 	
	\hide
	until it passes $S_0$. Obviously, $\Phi$ is analytic function of $(s_{*}, x, v)$. Taking $\frac{\p}{\p s_{*}}$,
	\[
		\frac{\p}{\p s_{*}} \Phi(s_{*}, x, v) = \frac{\omega(x,v)}{|X(s_{*};0,x,v)\times \hat{z}|} > 0.
	\]
	By implicit function theorem, $s_{*} = s(x,v)$ is local analytic function of $(x,v)$. Moreover, $X(s(x,v); 0, x, v) \notin \p S_{0}$ if and only if $s(x,v)\notin t^{i}(x,v)$ for any $i\in\mathbb{N}$. From analyticity of $s(x,v)$ and statement (b) of Lemma \ref{analyticity},
	$
		(X, V)(s(x,v); 0, x, v)		
	$
	is also locally analytic in $(x,v)$.  \\
	\unhide
\end{proof}

When arrival time function $s(x, \varphi, v)$ is well-defined, we use the following notations,
\Be \label{map XV}
	(X,V)(s)(E) := \{ (X,V)(s(x, \varphi, v); 0, x, v) \in \overline{S_0}\times \S^{2} : (x,v)\in E \}.
\Ee
In addition, if the trajectory is reversible, we also define
\Be \label{map XV -1}
	(X,V)^{-1}(s)(F) := \{ (y,u)\in \overline{\O}\times \S^{2} : (X,V)(s(y,u); 0, y, u)\in F  \}. 
\Ee

\hide
\begin{lemma} \label{lem conti S0} (Continuity from inflection to $S_{0}$) 
	Let us assume $(\theta, \varphi) \in U$ belongs to $\gamma_0^{I-}$ in the sense of
	\[
		(\sigma(\theta, \varphi), I^{2}(\theta, \varphi) ) \in \gamma_0^{I-}. \\
	\]
	If $(X, V)(s(\theta, \varphi); 0, \theta, \varphi)\in S_{0}^{\varepsilon}$ is well-defined, i.e., if there is no inflection grazing until arriving $S_0$, then mapping
	\[
		(X, V)(s(\theta, \varphi); 0, \theta, \varphi) \in S_0^{\varepsilon}\times \S^{2}
	\]
	is continuous at the point $(\theta, \varphi)$ where function $s$ is defined in Lemma \ref{def time s} ({\color{blue} Or make separate def for $s$}). \\
\end{lemma} 
\begin{proof}
	Maybe from Lemma \ref{continuity}? \\
\end{proof}

{\color{blue}
We may need another verison of continuity lemma as the following.
\begin{lemma} \label{continuity nI} (Continuity for non-inflection) 
	Assume that $(x,v) \in cl(\Omega) \times \V^N$ satisfying (during traveling $L$ in backward, Let $
	|v|=1$)
	\[
	(X,V)(s; L, x,v) \notin \gamma_0^I \quad\text{for all}\quad 0\leq s \leq L
	\]
	($\gamma_0^C$ is possible) Then, for any $\epsilon>0$, there exists $\delta_{x,v,\epsilon}>0$ such that 
	\[
	(X,V)(s; L, y,u) \notin \gamma_0^I \quad\text{for all}\quad 0\leq s \leq L
	\]
	whenever $\vert (x,v)-(y,u)\vert <\delta_{x,v,\epsilon}$. 
\end{lemma} 
} 
\unhide

\hide{\color{blue}
We control collection of bad phase sets those are nearly grazing set for each open covers contaning boundary $\p\O$.
\begin{lemma} \label{near_boundary}
	Let $\Omega\subset\mathbb{R}^2$ be an analytic non-convex domain, defined in Definition~\ref{AND}. For $\varepsilon \ll 1, N \gg 1$, there exist finite points 
	\[
	\{{x}_{1}^{nB},\cdots, {x}_{l_{nB}}^{nB}\} \subset cl({\O}),
	\]
	and their open neighborhoods 
	\[
	B({x}_1^{nB}, r_1^{nB}), \cdots, B({x}_{l_{nB}}^{nB}, r_{l_{nB}}^{nB}) \subset \mathbb{R}^{2}, 
	\]
	as well as corresponding open sets 
	\[
	\mathcal{O}_1^{nB}, \cdots, \mathcal{O}_{l_{nB}}^{nB} \subset \mathbb{V}^N, 
	\]
	with $\mathfrak{m}_{2}(\mathcal{O}_i^{nB}) \leq \varepsilon$ for all $i=1,\cdots,l_{nB}$ such that for every $x\in cl({\O})$ there exists $i\in \{1,\cdots, l_{nB}\}$ with $x\in B({x}_i^{nB}, r_i^{nB})$ and satisfies either 
	\[
	B(x_i^{nB}, r_i^{nB}) \cap \partial\Omega = \emptyset \quad \text{or} \quad |v^{\prime}\cdot \mathbf{n}(x^{\prime})| > \varepsilon/ N^4,
	\] 
	for all $x^{\prime}\in B({x}^{nB}_i, r^{nB}_i)\cap\partial\Omega$ and $v^{\prime}\in \mathbb{V}^N  \backslash \mathcal{O}_i^{nB}$ .
\end{lemma}
\begin{proof}
	By Definition~\ref{AND}, $\partial{\Omega} \in \mathbb{R}^2$ is a compact set in $\mathbb{R}^2$ and a union of the images of finite curves. For ${x}\in{\Omega}$, we define $r_{{x}}>0$ such that $B(x,r_{{x}})\cap \partial{\Omega}=\emptyset$. For each ${x}\in\partial{\Omega}$, we can define the outward unit normal direction $\mathbf{n}(x)$ and the outward normal angle $\theta_n({x})\in [0,2\pi)$ specified uniquely by $\mathbf{n}({x}) = (\cos\theta_n({x}), \ \sin\theta_n({x}) )$. Using the smoothness and uniform boundedness of curvature of the boundary $\partial{\Omega}$, there exist uniform $r_{\varepsilon, N} > 0$ such that for $r_{x} \leq r_{\varepsilon, N}$, 
	\begin{equation} \label{near-paral}
	| \ \theta_{{n}}({x}^{\prime}) -  \theta_{{n}}({x}) \ | <  \varepsilon / 2N^2 \ \  \ \text{for all }  \ {x}^{\prime} \in {B}({x},r_{{x}}) \cap \partial{\Omega}, 
	\end{equation}
	and ${B}({x},r_{{x}}) \cap \partial{\Omega}$ is a simply connected curve. \\
	\indent By compactness, we have finite integer $l_{nB} > 0$,  points $\{x_{i}^{nB}\}_{i=1}^{l_{nB}}$, and positive numbers $\{r_{i}^{nB}\}_{i=1}^{l_{nB}}$ such that 
	\begin{equation*}
	cl({\O}) \subset \bigcup_{i=1}^{l_{nB}} B(x^{nB}_{i}, r^{nB}_{i}), \quad r_{i}^{nB} \leq r_{\varepsilon,N}.
	\end{equation*}
	By above construction, for each $1\leq i \leq l_{nB}$, we have either
	\begin{equation} \label{int}
	B(x_{i}^{nB}, r_{i}^{nB}) \cap \p\O = \emptyset,
	\end{equation}
	or 
	\begin{equation} \label{bdry}
	x_{i}^{nB} \in \p\O \quad \text{and} \quad r_{i}^{nB} < r_{\varepsilon, N} \quad \text{so that} \quad (\ref{near-paral})\quad \text{holds}.
	\end{equation}
	For $i$ with case (\ref{int}), we set $\mathcal{O}^{nB}_{i} = \emptyset$. For $i$ with case (\ref{bdry}), we define
	\begin{equation*} \label{def O}
	\begin{split}
	\mathcal{O}^{nB}_i &:= \Big\{ v\in \mathbb{V}^N :  {v}= \big(|{v}|\cos\theta, |{v}|\sin\theta\big)  \quad \text{where} \quad \theta \in \Big( \big(\theta_i \pm \frac{\pi}{2}\big) -\frac{\varepsilon}{N^3}, \ \big(\theta_i \pm \frac{\pi}{2}\big) + \frac{\varepsilon}{N^3} \Big) \Big\},  \\
	\end{split}
	\end{equation*}
	where we abbreviated $\theta_{n}(x_{i}^{nB}) = \theta_{i}$. Obviously, $\mathfrak{m}_{2}(\mathcal{O}_{i}^{nB}) \leq \pi \frac{N^{2}}{4}\frac{\varepsilon/N^{2}}{\pi} \leq \varepsilon$ and
	\begin{eqnarray*}
		|v^{\prime} \cdot \mathbf{n}(x^{\prime})| &\geq& |{v}^{\prime}| \times \Big|(\cos\theta^{\prime},\sin\theta^{\prime}) \cdot  (\cos \theta_{{n}}({x}^{\prime}), \sin \theta_{{n}}({x}^{\prime})) \Big|\\
		&\geq& \frac{2}{N} \times \Big|\cos\big(\frac{\pi}{2} + \frac{\varepsilon}{N^{3}} \big) \Big| = \frac{2}{N} \Big|\sin\big(\frac{\varepsilon}{N^{3}} \big) \Big| ,\quad \varepsilon/N^{3} \ll 1,  \\
		&\geq& \frac{\varepsilon}{N^{4}} ,
	\end{eqnarray*}
	for $x^{\prime}\in B({x}^{nB}_i, r^{nB}_i)$ and $v^{\prime} = |v^{\prime}|(\cos\theta^{\prime}, \sin\theta^{\prime})\in \mathbb{V}^N  \backslash \mathcal{O}^{nB}_i$.  \\
\end{proof}
}\unhide 

We control the collection of bad phase sets that are nearly grazing sets for each open cover containing boundary $\p\O$.
\begin{lemma} \label{near_boundary}
	Let $S_0$ be the $\varphi=0$ cross-section defined in \eqref{cross section}. For $\varepsilon \ll 1$, there exist finite points 
	\[
	\{{x}_{1}^{nB},\cdots, {x}_{l_{nB}}^{nB}\} \subset \overline{S_0},
	\]
	and their spatial relatively open neighborhoods 
	\[
	B({x}_1^{nB}, r_1^{nB})\cap \overline{S_0}, \cdots, B({x}_{l_{nB}}^{nB}, r_{l_{nB}}^{nB})\cap \overline{S_0} , 
	\]
	as well as corresponding open sets in velocity
	\[
	\mathcal{O}_1^{nB}, \cdots, \mathcal{O}_{l_{nB}}^{nB} \subset \mathbb{S}^2, 
	\]
	with $\mathfrak{m}_2(\mathcal{O}_i^{nB}) \lesssim \varepsilon$ for all $i=1,\cdots,l_{nB}$ such that for every $x\in \overline{S_0}$, there exists $i\in \{1,\cdots, l_{nB}\}$ with $x\in B({x}_i^{nB}, r_i^{nB})\cap \overline{S_0}$ and satisfies either 
	\begin{eqnarray}
&	B(x_i^{nB}, r_i^{nB}) \cap \partial S_0 = \emptyset, &\notag  \\
	& \text{or}&\notag \\
	& |v^{\prime}\cdot n(x^{\prime})| > \varepsilon/2 \ \text{for all} \
	 x^{\prime}\in B({x}^{nB}_i, r^{nB}_i)\cap\partial S_0 \ \text{and } \  v^{\prime}\in \S^2  \backslash \mathcal{O}_i^{nB}.&\notag
	\end{eqnarray}
\end{lemma}
\begin{proof}
	 For ${x}\in S_0$, we choose $r_{{x}}>0$ such that $B(x,r_{{x}})\cap \partial S_0=\emptyset$. For each ${x}\in\partial{S_0}$, we can define the outward unit normal direction $n(x)$ and the outward normal angle $\theta_n({x})\in [0,2\pi)$ specified uniquely by $n({x}) = (\cos\theta_n({x}), 0,\sin\theta_n({x}) )$. Using the smoothness and uniform boundedness of curvature of the boundary $\partial{\Omega}$, there exists uniform $r_{\varepsilon} > 0$ such that for $r_{x} \leq r_{\varepsilon}$, 
	\begin{equation} \label{near-paral}
	| \ \theta_{{n}}({x}^{\prime}) -  \theta_{{n}}({x}) \ | <  \varepsilon/2 \ \  \ \text{for all }  \ {x}^{\prime} \in {B}({x},r_{{x}}) \cap \partial S_0.
	\end{equation}
	\indent By compactness, we have finite integer $l_{nB} > 0$,  points $\{x_{i}^{nB}\}_{i=1}^{l_{nB}}$, and positive numbers $\{r_{i}^{nB}\}_{i=1}^{l_{nB}}$ such that 
	\begin{equation*}
	\overline{S_0} \subset \bigcup_{i=1}^{l_{nB}} B(x^{nB}_{i}, r^{nB}_{i}), \quad r_{i}^{nB} \leq r_{\varepsilon}.
	\end{equation*}
	By above construction, for each $1\leq i \leq l_{nB}$, we have either
	\begin{equation} \label{int}
	B(x_{i}^{nB}, r_{i}^{nB}) \cap \partial S_0 = \emptyset,
	\end{equation}
	or 
	\begin{equation} \label{bdry}
	x_{i}^{nB} \in \partial S_0 \quad \text{and} \quad r_{i}^{nB} < r_{\varepsilon} \quad \text{so that} \quad (\ref{near-paral})\quad \text{holds}.
	\end{equation}
	For $i$ with case (\ref{int}), we set $\mathcal{O}^{nB}_{i} = \emptyset$. For $i$ with case (\ref{bdry}), we define
	{\small
	\begin{equation*} \label{def O}
	\begin{split}
	\mathcal{O}^{nB}_i := &\Big\{ v\in \mathbb{S}^2 :  {v}= \big(\sin \phi \cos\theta, \sin \phi \sin \theta , \cos \phi\big) \quad \text{where} \\
	& \ \phi \in \Big(2\pi-\theta_i (mod \; \pi) , 2\pi-\theta_i + \varepsilon (mod \; \pi) \Big)\cup \Big(3\pi-\theta_i-\varepsilon (mod \; \pi),3\pi-\theta_i (mod \; \pi) \Big),\\
& \ \theta \in \Big(\sqrt{\varepsilon},2\pi-\sqrt{\varepsilon}\Big)  \Big\},  \\
	\end{split}
	\end{equation*}
	}
	where we abbreviated $\theta_{n}(x_{i}^{nB}) = \theta_{i}$. Obviously, $\mathfrak{m}_2(\mathcal{O}_{i}^{nB}) \leq \frac{1}{3}\times\varepsilon\times 2\pi \lesssim \varepsilon$ and
	\begin{eqnarray*}
		|v^{\prime} \cdot n(x^{\prime})| &\geq& |{v}^{\prime}| \times \Big|(\sin \phi^{\prime}\cos\theta^{\prime},\sin \phi^{\prime}\sin\theta^{\prime},\cos \phi^{\prime}) \cdot  (\cos \theta_{{n}}({x}^{\prime}),0, \sin \theta_{{n}}({x}^{\prime})) \Big|\\
		&=&   \Big| \sin \phi^{\prime}\cos \theta_n(x^{\prime})+\cos \phi^{\prime}\sin \theta_n (x^{\prime})-(1-\cos \theta^{\prime})\sin\phi^{\prime}\cos \theta_n(x^{\prime})  \Big|\\
		&\geq&  |\sin\varepsilon | -\frac{\varepsilon}{2},\quad \varepsilon \ll 1,  \\
		&\geq& \frac{\varepsilon}{2} ,
	\end{eqnarray*}
	for $x^{\prime}\in B({x}^{nB}_i, r^{nB}_i)$ and $v^{\prime} = (\sin\phi^{\prime}\cos\theta^{\prime}, \sin\phi^{\prime}\sin\theta^{\prime},\cos \phi^{\prime})\in \mathbb{S}^2  \backslash \mathcal{O}^{nB}_i$.   
\end{proof}

Motivated by angular momentum conservation, we define $\tau$-ring and small angular momentum direction on $\mathbb{S}^{2}$.  Recall that we assume angular momentum $\omega > 0$, WLOG.  
\begin{definition} \label{def ring}
	We define $\tau$-angular momentum ring on $\mathbb{S}^{2}$ centered at $x\in \overline{\O}$. For $\tau_{1,*}\leq\tau\leq\tau_{2,*}$, 
	\Be
	\begin{split}\label{small_ang}
		R_{x, \tau} &:= \{ v\in\mathbb{S}^{2} :  \omega(\sigma(\tau,\varphi), I^{2}(\tau, \varphi)) = \omega(x,v)  \},  \\
		R_{x,\tau}^{\varepsilon} &:= \{ v\in\mathbb{S}^{2} :  |\omega(\sigma(\tau,\varphi), I^{2}(\tau,\varphi)) - \omega(x,v) | < \varepsilon  \},  \\
	\end{split}
	\Ee
	(it is equivalent to use $I_{1}$ instead of $I_{2}$, of course) and other small regimes in $\S^{2}$,
	\Be \label{small_direc}
	\begin{split}
		R_{\perp} &:= \{ v\in\mathbb{S}^{2} :  |v_{\varphi} | =0  \}, \quad
		 R_{\perp}^{\varepsilon} := \{ v\in\mathbb{S}^{2} :  |v_{\varphi} | < \varepsilon  \},  \\
		R_{\hat{\varphi}} &:= \{ v\in\mathbb{S}^{2} : |v_{\varphi}| = 1 \}, 
		\quad R_{\hat{\varphi}}^{\varepsilon} := \{ v\in\mathbb{S}^{2} :  1-\varepsilon < |v_{\varphi}| \},  \\
		R_{sym} &:= \{ v\in\mathbb{S}^{2} :  |v_{x}| = |v_{y}| \},  \quad R_{sym}^{\varepsilon} := \{ v\in\mathbb{S}^{2} :  ||v_{x}| - |v_{y}|| < \varepsilon \},  \\
	\end{split}
	\Ee
	where we used the coordinate $v=(v_x,v_{\varphi},v_y)$ defined in \eqref{coord}. Note that it is easy to check  $\mathfrak{m}_{2}(R_{x,\tau}^{\varepsilon}) \leq C\varepsilon$ and $\mathfrak{m}_{2}(R_{i}^{\varepsilon})\leq C\varepsilon $ for both all $i\in \{ \perp, \hat{\varphi}, sym\}$ with some generic constant $C>0$. Here, $\mathfrak{m}_{2}$ is Liouville measure on $\mathbb{S}^{2}$. Note that $R_{x,\tau}^{\varepsilon}$ and $R_{i}^{\varepsilon}$ are open on $\S^{2}$. \\
\end{definition}
\hide
\begin{center}

\tikzset{every picture/.style={line width=0.75pt}} 

\begin{tikzpicture}[x=0.75pt,y=0.75pt,yscale=-1,xscale=1]

\draw   (251,1362.5) .. controls (251,1321.91) and (266.67,1289) .. (286,1289) .. controls (305.33,1289) and (321,1321.91) .. (321,1362.5) .. controls (321,1403.09) and (305.33,1436) .. (286,1436) .. controls (266.67,1436) and (251,1403.09) .. (251,1362.5) -- cycle ;
\draw    (286,1362.5) -- (372.88,1333.33) ;
\draw    (286,1363.67) -- (375.32,1391.67) ;
\draw  [dash pattern={on 4.5pt off 4.5pt}] (355.12,1363.08) .. controls (354.45,1347.3) and (362.4,1334.5) .. (372.88,1334.5) .. controls (383.36,1334.5) and (392.4,1347.3) .. (393.08,1363.08) .. controls (393.75,1378.87) and (385.8,1391.67) .. (375.32,1391.67) .. controls (364.84,1391.67) and (355.8,1378.87) .. (355.12,1363.08) -- cycle ;
\draw    (286,1362.5) -- (387,1343.83) ;
\draw    (286,1363.67) -- (391,1380) ;
\draw   (431,1357.5) .. controls (431,1313.59) and (466.59,1278) .. (510.5,1278) .. controls (554.41,1278) and (590,1313.59) .. (590,1357.5) .. controls (590,1401.41) and (554.41,1437) .. (510.5,1437) .. controls (466.59,1437) and (431,1401.41) .. (431,1357.5) -- cycle ;
\draw    (459,1417.5) -- (560.71,1296.53) ;
\draw [shift={(562,1295)}, rotate = 130.06] [color={rgb, 255:red, 0; green, 0; blue, 0 }  ][line width=0.75]    (10.93,-3.29) .. controls (6.95,-1.4) and (3.31,-0.3) .. (0,0) .. controls (3.31,0.3) and (6.95,1.4) .. (10.93,3.29)   ;
\draw    (567,1412.25) -- (451.46,1304.36) ;
\draw [shift={(450,1303)}, rotate = 43.04] [color={rgb, 255:red, 0; green, 0; blue, 0 }  ][line width=0.75]    (10.93,-3.29) .. controls (6.95,-1.4) and (3.31,-0.3) .. (0,0) .. controls (3.31,0.3) and (6.95,1.4) .. (10.93,3.29)   ;
\draw    (464,1290) -- (552.5,1423.75) ;
\draw    (440,1319) -- (577,1398.5) ;
\draw    (450,1406) -- (573,1305) ;
\draw  [dash pattern={on 4.5pt off 4.5pt}] (444.41,1319.54) .. controls (437.88,1321.36) and (436.13,1316.51) .. (440.48,1308.71) .. controls (444.84,1300.91) and (453.66,1293.12) .. (460.19,1291.3) .. controls (466.71,1289.48) and (468.47,1294.32) .. (464.11,1302.12) .. controls (459.75,1309.92) and (450.93,1317.72) .. (444.41,1319.54) -- cycle ;
\draw  [dash pattern={on 4.5pt off 4.5pt}] (549.67,1292.71) .. controls (547.09,1286.79) and (550.52,1284.14) .. (557.33,1286.79) .. controls (564.14,1289.44) and (571.75,1296.38) .. (574.33,1302.29) .. controls (576.91,1308.21) and (573.48,1310.86) .. (566.67,1308.21) .. controls (559.86,1305.56) and (552.25,1298.62) .. (549.67,1292.71) -- cycle ;
\draw  [dash pattern={on 4.5pt off 4.5pt}] (558.51,1425.51) .. controls (551.6,1427.44) and (549.13,1423.4) .. (552.99,1416.49) .. controls (556.85,1409.58) and (565.58,1402.42) .. (572.49,1400.49) .. controls (579.4,1398.56) and (581.87,1402.6) .. (578.01,1409.51) .. controls (574.15,1416.42) and (565.42,1423.58) .. (558.51,1425.51) -- cycle ;
\draw  [dash pattern={on 4.5pt off 4.5pt}] (446.67,1412.71) .. controls (444.09,1406.79) and (447.52,1404.14) .. (454.33,1406.79) .. controls (461.14,1409.44) and (468.75,1416.38) .. (471.33,1422.29) .. controls (473.91,1428.21) and (470.48,1430.86) .. (463.67,1428.21) .. controls (456.86,1425.56) and (449.25,1418.62) .. (446.67,1412.71) -- cycle ;
\draw   (141.35,1288.45) .. controls (158.48,1295.29) and (170.95,1333.64) .. (169.2,1374.11) .. controls (167.45,1414.58) and (152.14,1441.84) .. (135,1435) .. controls (117.86,1428.16) and (105.39,1389.8) .. (107.15,1349.33) .. controls (108.9,1308.86) and (124.21,1281.6) .. (141.35,1288.45) -- cycle ;
\draw    (137,1365) -- (137,1432) ;
\draw [shift={(137,1434)}, rotate = 270] [color={rgb, 255:red, 0; green, 0; blue, 0 }  ][line width=0.75]    (10.93,-3.29) .. controls (6.95,-1.4) and (3.31,-0.3) .. (0,0) .. controls (3.31,0.3) and (6.95,1.4) .. (10.93,3.29)   ;
\draw    (137,1365) -- (162,1433) ;
\draw    (137,1365) -- (113,1433) ;
\draw  [dash pattern={on 4.5pt off 4.5pt}] (125.66,1425.34) .. controls (137.81,1423.8) and (152.68,1425.15) .. (158.87,1428.37) .. controls (165.06,1431.59) and (160.22,1435.45) .. (148.07,1437) .. controls (135.92,1438.54) and (121.05,1437.19) .. (114.86,1433.97) .. controls (108.67,1430.75) and (113.51,1426.89) .. (125.66,1425.34) -- cycle ;
\draw    (136.86,1365.97) -- (125,1438) ;
\draw    (136.86,1365.97) -- (148.07,1437) ;
\draw    (286,1363.67) -- (391.07,1363.09) ;
\draw [shift={(393.07,1363.08)}, rotate = 179.69] [color={rgb, 255:red, 0; green, 0; blue, 0 }  ][line width=0.75]    (10.93,-3.29) .. controls (6.95,-1.4) and (3.31,-0.3) .. (0,0) .. controls (3.31,0.3) and (6.95,1.4) .. (10.93,3.29)   ;
\draw    (471.5,1425) -- (549.5,1286) ;

\draw (128,1330) node [anchor=north west][inner sep=0.75pt]    {$S_{\varphi }$};
\draw (124,1443) node [anchor=north west][inner sep=0.75pt]    {$R_{\perp }^{\varepsilon }$};
\draw (313.66,1443) node [anchor=north west][inner sep=0.75pt]    {$R_{\hat{\varphi }}^{\varepsilon }$};
\draw (498,1330) node [anchor=north west][inner sep=0.75pt]    {$S_{\varphi }$};
\draw (491,1443) node [anchor=north west][inner sep=0.75pt]    {$R_{sym}^{\varepsilon }$};
\draw (278,1330) node [anchor=north west][inner sep=0.75pt]    {$S_{\varphi }$};
\draw (395,1349) node [anchor=north west][inner sep=0.75pt]    {$v_{\varphi }$};

\end{tikzpicture}
\end{center}
\unhide

\begin{figure}[h]
\centering
\includegraphics[width=15cm]{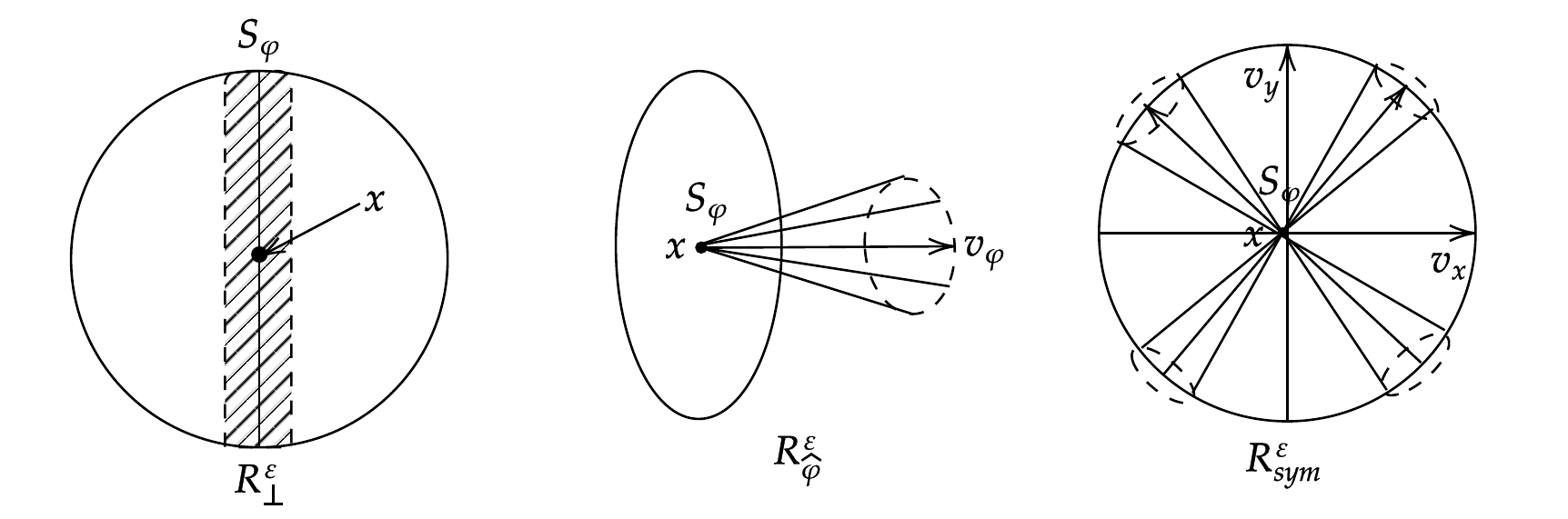}
\caption{Sets $R_{\perp}^{\varepsilon},R_{\hat{\varphi}}^{\varepsilon}$, and $R_{sym}^{\varepsilon}$}
\label{sets R}
\end{figure}

In Figure \ref{sets R}, $x$ is located at the center of cross-section $S_{\varphi}$. In $R_{\perp}^{\varepsilon}$, $S_{\varphi}$ containing $x$ was projected to be shown as above. Hence, the underlined part of the figure is the same with projection of $R_{\perp}^{\varepsilon}$. In $R_{\hat{\varphi}}^{\varepsilon}$, velocity $v_{\varphi} \in \S^2$ is perpendicular to cross-section $S_{\varphi}$. $R_{\hat{\varphi}}^{\varepsilon}$ is the set of velocity which perturbs $v_{\hat{\varphi}}$ as much as $\varepsilon$. Lastly, the set $R_{sym}^{\varepsilon}$ consists of velocity which perturbs direction satisfying $v_x=v_y$ in \eqref{coord}.  \\

\hide We define $S_{0}^{\delta}$
\Be
	S_{0}^{\delta} = \{ x\in S_{0} : \text{dist}(\p S_{0}, x) > \delta \},\quad \delta > 0.
\Ee
Small $\delta \ll 1$ is TBD later. \unhide 
\hide
\textcolor{red}{
Now, we remove some small parts of velocity from $\S^{2}$ and define open set
\Be \label{def X int}
X^{\ephalf, \circ} := \Big\{ (x,v) \in {S_{0}} \times \mathbb{S}^{2} : v \in \mathbb{S}^{2}\backslash \ cl\{ R_{x,\frac{\pi}{2}}^{\ephalf} \cup R_{x,\pi}^{\ephalf} \cup R_{x,\frac{3\pi}{2}}^{\ephalf} \cup R_{x, \perp}^{\ephalf} \cup R_{x, \hat{\varphi}}^{\ephalf} \cup R_{x, sym}^{\ephalf} \cup \mathcal{O}_i^{nB}\} \textrm{ for some } i \textrm{ satisfying } x\in B(x_i,r_i)\Big\}. \\
\Ee 
}
\unhide
	\begin{lemma} \label{lem nbd R}
		Let $x\in \overline{S_0}$, $\varepsilon \ll 1$, and $\tau \in[\tau_{1,*},\tau_{2,*}]$. Recall definition of $R_{x,\tau}^{\varepsilon}$ in \eqref{small_ang}. For each $x\in \overline{S_{0}}$, there exists $r^{\prime}=r^{\prime}(x,\varepsilon)$ such that 
		\begin{equation*}
			{R_{y, \tau}^{\varepsilon}} \subset R_{x, \tau}^{2\varepsilon}, 
		\end{equation*}
		whenever  $y\in B(x,r^{\prime}(x,\varepsilon)) \cap \overline{S_{0}}$. 
	\end{lemma}
	\begin{proof}
		Suppose that $u \in {R_{y,\tau}^{\varepsilon}}$. Then, the velocity $u$ satisfies $\vert \omega(\sigma(\tau,\varphi),I^2(\tau,\varphi))-\omega(y,u)\vert \leq\varepsilon$. By definition of angular momentum \eqref{def_ang moment}, we choose $r^{\prime}=r^{\prime}(x,\varepsilon)$ such that 
		\begin{align*}
		\vert \omega(x,u)-\omega (y,u) \vert <\varepsilon,
		\end{align*}
		if $y\in B(x,r^{\prime}(x,\varepsilon))\cap \overline{S_0}$. Thus, if $(y,u)\in (B(x,r^{\prime}(x,\varepsilon))\cap \overline {S_0})\times R_{y,\tau}^\varepsilon$, we have 
		\begin{align*}
		\vert \omega(\sigma(\tau,\varphi), I^2(\tau,\varphi))-\omega (x,u)\vert &= \vert \omega(\sigma(\tau,\varphi),I^2(\tau,\varphi))-\omega(y,u)+\omega(y,u)-\omega (x,u)\vert \\
		&\leq \vert\omega(\sigma(\tau,\varphi),I^2(\tau,\varphi))-\omega(y,u) \vert + \vert \omega(y,u)-\omega (x,u)\vert \\
		&< 2\varepsilon,
		\end{align*}
		which means $u \in R_{x,\tau}^{2\varepsilon}$.  
	\end{proof}


\begin{lemma} \label{lem Xe}
	For given $\varepsilon \ll 1$, there exists a compact set $X^{\varepsilon}\subset\overline{S_0}\times \S^2$ such that the followings hold: if $(x,v)\in X^{\varepsilon}$, \\
	(i) 
	\begin{align*}
	& |v\cdot n(x)| > \varepsilon/2 \ \text{for} \
	x \in \partial S_0 ,
	\end{align*}
	(ii) 
	\[
		 v \notin \bigcup_{j=1}^{\ell} {R_{x,\tau_j^{\mathbf{z}}}^{\varepsilon}} \cup {R_{x,\tau_{1,*}}^{\varepsilon}} \cup {R_{x,\tau_{2,*}}^{\varepsilon}} \cup {R_{\perp}^{\varepsilon}} \cup {R_{ \hat{\varphi}}^{\varepsilon}} \cup {R_{sym}^{\varepsilon}},
	\]
	(iii) 
	\begin{align*}
	\mathfrak{m}_{2}((\{\overline{S_0}\times \S^2\}\backslash X^{\e})_x) \lesssim \varepsilon, 
	\end{align*}
	where we used notation
	\begin{align} \label{proj_v} 
	A_x:=\{y \in Y: (x,y) \in A\},
	\end{align}
	for $A \subset X\times Y$. 
\end{lemma}
\begin{proof}
	Recall that we constructed  the relatively open cover $\bigcup_{i=1}^{l_{nB}}\{B(x_i^{nB},r_i^{nB})\cap \overline{S_0}\}$ for $\overline S_0$ and corresponding open sets in velocity $\mathcal{O}_j^{nB}$ for $j=1,\cdots,l_{nB}$ in Lemma \ref{near_boundary}. For given $x\in \overline{S_0}$, we define an index set 
\begin{align} \label{index_nB}
	I_x^{nB}:=\{i\in\mathbb{N}: x \in B(x_i^{nB},r_i^{nB})\cap \overline{S_0}\},
\end{align}
to collect relatively open sets containing $x$. Notice that 
\begin{align} \label{near_size} 
	\mathfrak{m}_2(\bigcap_{i \in I_{x}^{nB}}\mathcal{O}_i^{nB})\leq \mathfrak{m}_2(\mathcal{O}_j^{nB})\lesssim \varepsilon,
\end{align}
for all $j\in I_{x}^{nB}$. Moreover, if $v \notin \bigcap_{i \in I_x^{nB}} \mathcal{O}_{i}^{nB}$ for fixed $x\in \overline{S_0}$, then it holds that  
\begin{align*}
	  \vert v \cdot n(x)\vert> \varepsilon /2, \quad  \text{if } x \in \p S_0,
\end{align*}
from Lemma \ref{near_boundary}. 
Hence, the following  set 
\begin{align} \label{bad phase set1}
	\bigcup_{x\in \overline{S_0}}\Big\{(x,v):v \in \bigcap_{i \in I_x^{nB}}\mathcal{O}_i^{nB}\Big\}
\end{align}
contains all near grazing points in phase space. To cover above \eqref{bad phase set1}, we construct an open cover
\begin{align} \label{NB}
	\mathcal{NB}:=\bigcup_{i=1}^{l_{nB}} \Big \{ (x,v)\in \overline{S_0}\times \S^2: x \in B(x_i^{nB},r_i^{nB})\cap \overline{S_0},\; v \in \bigcap_{j\in I_x^{nB}}\mathcal{O}_j^{nB}\Big\}_{\varepsilon},
\end{align}
where we used the following general notation $(\cdot)_{\varepsilon}$, a $\varepsilon$-neighborhood for a set in a metric space,
\begin{align} \label{extension domain}
A_{\varepsilon}:=\{ x\in O : dist(x, A)<\varepsilon\},\quad \text{$O$ is a metric space and $A\subset O$}.
\end{align}
It follows from \eqref{near_size} and \eqref{NB} that 
\Be \label{NB e}
	\mathfrak{m}_{2}(\mathcal{NB}_{x}) \lesssim \mathfrak{m}_{2}((\mathcal{O}_{j}^{nB})_{\varepsilon}) \lesssim \varepsilon,\quad \forall j\in I_{x}^{nB},
\Ee
where we used the notation $A_x$ in \eqref{proj_v}. \\ 
\indent Meanwhile, by \eqref{near_size} and definition of $\mathcal{NB}$ in \eqref{NB}, for any $x \in \overline{S_0}$, there exists a velocity $v \in \S^2 \backslash (\mathcal{NB})_x$ where the notation $(\mathcal{NB})_x$ is defined in \eqref{proj_v}. 
Hence, if we define the projection operator $\mathcal{P}_1$ for $A\subset X\times Y$ as
\begin{align} \label{proj operator}
	\mathcal{P}_1(A):=\{x \in X: (x,y)\in A \text{ for some } y\in Y\}, 
\end{align}
then we obtain $\mathcal{P}_1((\mathcal{NB})^c)=\overline{S_0}$ where $(\mathcal{NB})^{c}$ is a compact set
\Be \label{def NBc}
	(\mathcal{NB})^{c} := \{\overline{S_0}\times\S^{2}\}\backslash \mathcal{NB}. 
\Ee

	Now, recall definition \eqref{def NBc} and we define 
	\Be \label{Y}
		Y := \Big\{ (x,v) \in (\mathcal{NB})^c : v \in  \bigcup_{j=1}^{\ell} {R_{x,\tau_j^{\mathbf{z}}}^{\varepsilon}}\cup {R_{x,\tau_{1,*}}^{\varepsilon}} \cup {R_{x,\tau_{2,*}}^{\varepsilon}}\cup {R_{\perp}^{\varepsilon}} \cup {R_{ \hat{\varphi}}^{\varepsilon}} \cup {R_{sym}^{\varepsilon}}   \Big\},
	\Ee
	where $R_{x,\tau}$, $R_{\perp,\hat{\varphi},sym}$, and $\ell$ are defined in \eqref{small_ang}, \eqref{small_direc}, and \eqref{Zh_ell}, respectively. 
	Then, using compactness of $\mathcal{P}_1((\mathcal{NB})^c)=\overline{S_0}$ and the argument in the proof of Lemma \ref{lem nbd R}, we obtain an open cover for $\mathcal{P}_1((\mathcal{NB})^c)=\overline{S_0}$, 
	 \Be \label{Y epsilon}
		\Big\{B(x_i^R,r_i^R)\cap \overline{S_0}\Big \}_{i=1}^{l_R},
	\Ee
	and corresponding each open set $\mathcal{O}_i^{R} \subset \S^{2}$ which is defined by 
	\begin{align} \label{O_i^R}
		 \mathcal{O}_i^{R}:=\bigcup_{j=1}^{\ell} {R_{x_i^R,\tau_j^{\mathbf{z}}}^{2\varepsilon}}\cup {R_{x_i^R,\tau_{1,*}}^{2\varepsilon}} \cup {R_{x_i^R,\tau_{2,*}}^{2\varepsilon}}\cup {R_{\perp}^{2\varepsilon}} \cup {R_{\hat{\varphi}}^{2\varepsilon}} \cup {R_{sym}^{2\varepsilon}}, \quad  1\leq i \leq l_R,
	\end{align}
	satisfying $\mathfrak{m}_2(\mathcal{O}_i^{R})\lesssim \varepsilon$ for all $1\leq i \leq l_R$. Also, from Lemma \ref{lem nbd R}, we have 
	\begin{align} \label{O_i^R 2}
		\bigcup_{j=1}^{\ell} {R^{\varepsilon}_{x,\tau_j^{\mathbf{z}}}}\cup {R^{\varepsilon}_{x,\tau_{1,*}}} \cup {R^{\varepsilon}_{x,\tau_{2,*}}}\cup {R^{\varepsilon}_{\perp}} \cup {R^{\varepsilon}_{ \hat{\varphi}}} \cup {R^{\varepsilon}_{sym}} \subset \mathcal{O}_i^R, 
	\end{align}
	for all $x \in B(x_i^R,r_i^R)$ and sufficiently small $r_{i}^{R}$. 
	 Similar as \eqref{index_nB}, for given $x\in \mathcal{P}_1((\mathcal{NB})^c)=\overline{S_0}$, we define an index set 
\begin{align} \label{index_R}
	I_x^{R}:=\{i\in\mathbb{N}: x\in B(x_i^R,r_i^R)\cap \overline{S_0}\}. 
\end{align}
Next, let us consider
\Be \label{bad phase set2}
	\bigcup_{x\in \overline{S_0}} \Big \{(x,v): v \in  \bigcap_{i\in I_x^R} \mathcal{O}_i^R\Big \},
\Ee
which contains  $Y$ of \eqref{Y}. Similar to what we did in \eqref{NB}, we also construct an open cover for \eqref{bad phase set2} as
\begin{align} \label{def R}
	\mathcal{R}:=\bigcup_{i=1}^{l_{R}} \Big \{ (x,v)\in (\mathcal{NB})^c: x \in B(x_i^{R},r_i^{R})\cap \overline{S_0},\; v \in \bigcap_{j\in I_x^{R}}\mathcal{O}_j^{R}\Big\}_{\varepsilon},
\end{align}
where we have used the notation $A_{\varepsilon}$ in \eqref{extension domain}. Let us define compact set $X^{\e}$,
\begin{align} \label{def Xe}
	X^{\varepsilon}:=(\mathcal{NB})^c\backslash\mathcal{R}. 
\end{align}
By the same reason for $\mathcal{P}_1((\mathcal{NB})^c)=\overline{S_0}$, we also get 
\begin{align} \label{Xe proj}
	\mathcal{P}_1(X^{\varepsilon})=\overline{S_0},
\end{align}
where we have used the notation $\mathcal{P}_1$ in \eqref{proj operator}. Now we check that $X^{\varepsilon}$ in \eqref{def Xe} satisfies statements (i)-(iii). Assume $(x,v)\in X^{\varepsilon}$. First, since $X^{\varepsilon}\subset (\mathcal{NB})^c$, we have 
\begin{align*}
	\vert v \cdot n(x) \vert > \varepsilon/2, \quad \text{for  } x\in \p S_0. 
\end{align*}

\noindent Moreover, since we exclude set $\mathcal{R}$ in \eqref{def Xe} and the definition \eqref{O_i^R} of $\mathcal{O}_i^R$, it satisfies that 
\begin{align*}
	v \notin \bigcup_{j=1}^{\ell} {R^{\varepsilon}_{x,\tau_j^{\mathbf{z}}}}\cup {R^{\varepsilon}_{x,\tau_{1,*}}} \cup {R^{\varepsilon}_{x,\tau_{2,*}}}\cup {R^{\varepsilon}_{\perp}} \cup {R^{\varepsilon}_{ \hat{\varphi}}} \cup {R^{\varepsilon}_{sym}} ,
\end{align*}
by \eqref{Y} and \eqref{O_i^R 2}. Lastly, for fixed $x \in \overline{S_0}$, 
\begin{align} \label{Xe c small}
\begin{split}
	\mathfrak{m}_{2}(( \{\overline{S_0}\times \S^2\}\backslash X^{\e})_x) &= \mathfrak{m}_{2}((\mathcal{NB}\cupdot\mathcal{R})_x) \\
	&\leq \mathfrak{m}_{2}(\mathcal{R}_x) + \mathfrak{m}_{2}(\mathcal{NB}_x) \\
	&\leq \mathfrak{m}_{2}((\mathcal{O}_{j}^{R})_{\varepsilon}) + \mathfrak{m}_{2}(\mathcal{NB}_x) \\
	&\lesssim \varepsilon,\quad \forall j\in I_{x}^{R},
\end{split}
\end{align}
by \eqref{NB e},\eqref{def R}, and the fact $\mathfrak{m}_2(\mathcal{O}_j^R) \lesssim \e$ for all $j=1,2,\dots,l_R$. Here, we used notation $A_x$ in \eqref{proj_v}. $\varepsilon$-neighborhood notation \eqref{extension domain} comes from the definition \eqref{def R}. 
\end{proof}

We decompose above set depending on whether its backward in time trajectory belongs $\gamma_{0}^{I}$ before it travels given length $L > 0$. We decompose $X^{\varepsilon}= G^{\varepsilon}_{L} \cupdot B^{\varepsilon}_{L}$ where $\cupdot$ means a disjoint union. 

\Be \label{B,G sets}
\begin{split}
	B^{\varepsilon}_{L} &:= \left \{ (x,v)\in X^{\varepsilon} :
	\begin{array}{cc}
	 \exists k \in \mathbb{N} \text{ such that } (x^{k}(x,v), v^{k-1}(x,v))\in \gamma_{0}^{I-}, \\ \text{and} \  \sum_{j=1}^k \vert x^j(x,v)-x^{j-1}(x,v) \vert \leq L
	 \end{array}
	  \right \}, \\
	G^{\varepsilon}_{L} &:= \left \{ (x,v)\in X^{\varepsilon} :
	\begin{array}{cc}
	 \nexists k \in \mathbb{N} \text{ such that } (x^{k}(x,v), v^{k-1}(x,v))\in \gamma_{0}^{I-}, \\ \text{as long as } \  \sum_{j=1}^k \vert x^j(x,v)-x^{j-1}(x,v) \vert \leq L 
	 \end{array}\right \}  .
\end{split}	
\Ee
\hide 
{\color{red} {\bf NOT USED, Remove later}	
We further disjointly split $G_{L}^{\varepsilon}$ into $G_{L}^{\varepsilon} = G_{L}^{\varepsilon, C} \cup G_{L}^{\varepsilon, N} $ where
\Be
\begin{split}
	G^{\varepsilon, C}_{L} &:= \left \{ (x,v)\in G_{L}^{\varepsilon} \text{ such that } \exists k \in \mathbb{N} : (x^{k}(x,v), v^{k-1}(x,v))\in \gamma_{0}^{C}, \ \text{and} \  \sum_{j=1}^k \vert x^j(x,v)-x^{j-1}(x,v) \vert \leq L \right \} \\
	G^{\varepsilon, N}_{L} &:= \left \{ (x,v)\in G_{L}^{\varepsilon} \text{ such that } \nexists k \in \mathbb{N} : (x^{k}(x,v), v^{k-1}(x,v))\in \gamma_{0}^{C}, \ \text{if} \  \sum_{j=1}^k \vert x^j(x,v)-x^{j-1}(x,v) \vert \leq L \right \},
\end{split}	
\Ee
so in the case of $G_{L}^{\varepsilon, N}$ all bounces are nongrazing up to traveling $L$. Note that we have disjoint union,
\[
	X^{\varepsilon} = G^{\varepsilon}_{L} \cupdot B^{\varepsilon}_{L} = \big( G^{\varepsilon, C}_{L} \cupdot G^{\varepsilon, N}_{L} \big) \cupdot B^{\varepsilon}_{L}. 
\]
} 
\unhide

\begin{lemma} \label{trajectory invertible}
For $(x,v) \in X^{\varepsilon}$, if 
\begin{align} \label{no grazing cond}
	(x^k(x,v),v^{k-1}(x,v)) \notin \gamma_0^I, \; \forall k \quad \text{during finite travel length}\; L,
\end{align}
then the number of bounces in the trajectory within finite travel length $L$ is finite. In particular, $\mathcal{N}(x,v,L) < \infty$ when $(x,v)\in G_{L}^{\varepsilon}$, i.e., the number of bounces up to travel length $L$ is finite. Under the assumption of finite bounce, trajectory \eqref{XV} is reversible in time. 
\end{lemma}
\begin{proof}
In $X^\varepsilon$, constructed in Lemma \ref{lem Xe}, we exclude sets $R_{x,\tau_{1,*}}^{\varepsilon}$ and $R_{x,\tau_{2,*}}^{\varepsilon}$ where $R_{x,\tau}^{\varepsilon}$ in \eqref{small_ang}. Consider the backward in time trajectory starting from $x$ with $v$ for $(x,v)\in X^{\varepsilon}$. To apply Lemma \ref{no infty bounce}, we first check whether \eqref{assump 1} in Lemma \ref{compare} holds. We rewrite $x^i(x,v)$ using $\sigma$: 
\begin{align*}
	x^i(x,v)=\sigma(\tau_i,\varphi_i)\in \p\O, \quad \tau_i\in[a,b), \; \varphi_i\in(-2\pi\mathbf{w},0],
\end{align*}
for all $1\leq i \leq \mathcal{N}(x,v,L)$ where $\mathcal{N}(x,v,L)$ in \eqref{nongrazing NB}. If $\tau_i \notin [\tau_{1,*}-\varepsilon,\tau_{1,*}+\varepsilon]\cup [\tau_{2,*}-\varepsilon,\tau_{2,*}+\varepsilon]$ for all $1 \leq i \leq \mathcal{N}(x,v,L)$, it does not matter to apply Lemma \ref{compare}. Hence, we suppose that there exists $j$ such that  
\begin{align*}
	x^j(x,v)=\sigma(\tau_j,\varphi_j), \quad \tau_j\in[\tau_{1,*}-\varepsilon,\tau_{1,*}+\varepsilon]\cup [\tau_{2,*}-\varepsilon,\tau_{2,*}+\varepsilon].
\end{align*}
Recall definition of $\eta_j$ in Lemma \ref{compare}. Since an angle $\vartheta_j$ between $\hat{\tilde{z}}$ and $I_2(\tau_j,\varphi_j)$ satisfies
\begin{align*}
	\tan \vartheta_j = \sqrt{\frac{\kappa(\tau_j)\gamma_1(\tau_j)}{\vert \gamma_2'(\tau_j)\vert}},
\end{align*}
and sets $R_{x,\tau_{1,*}}^{\varepsilon},R_{x,\tau_{2,*}}^{\varepsilon}$ are excluded in the definition $X^{\varepsilon}$, we have 
\begin{align*}
	\tan \eta_j \leq \sqrt{\frac{\kappa(\tau_j)\gamma_1(\tau_j)}{\vert \gamma_2'(\tau_j)\vert}}-\varepsilon,
\end{align*}
by angular momentum conservation. This implies that $\eta_j \leq \frac{\pi}{2}-C\varepsilon$ for some positive constant $C$ due to $\tau_j \in[\tau_{1,*}-\varepsilon,\tau_{1,*}+\varepsilon]\cup [\tau_{2,*}-\varepsilon,\tau_{2,*}+\varepsilon]$. It remains to check whether \eqref{assump 3} in Lemma \ref{difference} holds. Similarly, we suppose that there exists $j$ such that
\begin{align*}
	x^j(x,v)=\sigma(\tau_j,\varphi_j), \quad \tau_j\in (Z_h)_{\varepsilon},
\end{align*}
where $(Z_h)_{\varepsilon}$ was defined in \eqref{Z_h nbh}. Since we also exclude $\bigcup_{i=1}^{\ell} R_{x,\tau_i^{\mathbf{z}}}^{\varepsilon}$ in $X^{\varepsilon}$, it holds that 
\begin{align*}
	\vert \tan \eta_j - \tan \vartheta_j \vert \geq \varepsilon, 
\end{align*}
which is identical to \eqref{assump 3}. Therefore, we can apply Lemma \ref{no infty bounce} 
and obtain finite number of bounce during travel length $L$ whenever $(x,v)$ satisfies \eqref{no grazing cond}. $G_{L}^{\varepsilon}$ satisfies condition of this Lemma by definition \eqref{B,G sets}. Reversibility of trajectory with under finite bounce is also obvious. 
\end{proof}
\hide
{\color{red}

Similar as above definitions, we can also define slightly larger corresponding sets
\[	
	B_{L}^{\ephalf, \circ}, \ G_{L}^{\ephalf, \circ}, \ G_{L}^{\ephalf, C, \circ}, \ G_{L}^{\ephalf, N, \circ}
\]
by changing $X^{\varepsilon}$ into $X^{\varepsilon, \circ}$ in above definitions. We have disjoint union
\[
	X^{\ephalf, \circ} = B_{L}^{\ephalf, \circ} \cup G_{L}^{\ephalf, C, \circ} \cup G_{L}^{\ephalf, N, \circ}.
\]
}
\unhide
\hide
\begin{definition}
	We define $N(\cdot, \cdot, L) :  B_{L}^{\varepsilon} \rightarrow \R_{+}$
	\Be
	\begin{split}
		N(x,v,L) &:= \inf \left \{ k \in \mathbb{N} :	(x^{k}, v^{k-1}(x,v))\in \gamma_{0}^{I-},  \  \sum_{j=1}^k \vert \xb^j(x,v)-\xb^{j-1}(x,v) \vert \leq L 	\right \}, \\
	\end{split}
	\Ee
	where $v^{0}=v$ and $x^{0}=x$. \\
\end{definition}
 \unhide
	
	When backward in time trajectory belongs to $\gamma_0^{I}$, the trajectory cannot be defined anymore. Hence, using reversibility of trajectory, we can construct corresponding backward in time image of $B_{L}^{\varepsilon}$  in $(\tau, \varphi)$ domain as follows.
{\small
	\Be \label{def cal B}
	\begin{split}
		\mathcal{B}_{L}^{\varepsilon}&:= \Big\{ (\tau, \varphi)\in [a,b)\times (-2\pi\mathbf{w}, 0]  : (X,V)(s( \sigma(\tau, \varphi), \varphi, I^{2}(\tau, \varphi)); 0, \sigma(\tau, \varphi), I^{2}(\tau, \varphi)) \in B_{L}^{\varepsilon} \Big\}  \\
		&= (X,V)^{-1}(s)(B_{L}^{\varepsilon}),
	\end{split}
	\Ee}
by definition \eqref{map XV -1}. In particular, we have one-to-one correspondence between $B_{L}^{\varepsilon}$ and $\mathcal{B}_{L}^{\varepsilon}$.  \\
	
	\begin{lemma} \label{lem:open} Recall definition $G^{\varepsilon}_L$ and $B^{\varepsilon}_L$ in \eqref{B,G sets}. The sets have the following properties: \\
		 (a) $G_{L}^{\varepsilon}$ is open in $X^{\varepsilon}$. \\
		 (b) $B_{L}^{\varepsilon}$ closed in $X^{\varepsilon}$ and $\mathcal{B}_{L}^{\varepsilon}$ is also closed. Both are compact. 
	\end{lemma}	
	\begin{proof} 
		\textbf{proof of (a) :} Let $(x,v) \in G^{\varepsilon}_{L}$ and $(y,u) \in X^{\varepsilon}$. By  Lemma \ref{trajectory invertible}, $\mathcal{N}(x,v,L) < \infty$. Also, from (c) in Lemma \ref{analyticity}, we can choose small $\delta_1>0$ such that $\mathcal{N}(y,u,L)\leq \mathcal{N}(x,v,L)$ if $\vert (x,v)-(y,u)\vert < \delta_1$. Let us write $\mathcal{N}_{x,v} = \mathcal{N}(x,v,L)$ to simplify notation here. Since $(x,v) \in G^{\varepsilon}_L$, for all $\varepsilon>0$, there exists $\delta_2>0$ such that 
	\begin{align*}
		\vert (x^j(x,v),v^{j-1}(x,v))-(x^j(y,u),v^{j-1}(y,u))\vert <\varepsilon, \quad \forall 1\leq j \leq \mathcal{N}_{x,v},
	\end{align*}
whenever $\vert (x,v)-(y,u) \vert <\delta$ by using (a) of Lemma \ref{analyticity}. Moreover, $(x^j(y,u),v^{j-1}(y,u))$\\
$ \notin \gamma_0^{I_-}$ for all $1\leq j \leq \mathcal{N}_{y,u}(\leq \mathcal{N}_{x,v})$. Hence, if we take $\delta=\min\{\delta_1,\delta_2\}$, then $(y,u)\in G^{\varepsilon}_{L}$ whenever $\vert (x,v)-(y,u)\vert <\delta$.   \\
		\textbf{proof of (b) :} Note that $X^{\varepsilon} = G_{L}^{\varepsilon}\cupdot B_{L}^{\varepsilon}$. By (a) of Lemma \ref{lem:open}, we directly deduce $B^{\varepsilon}_L$ is closed in $X^\varepsilon$. For $(x,v)\in B_{L}^{\varepsilon}$, there exists $k\in \mathbb{N}$ such that $(x^k(x,v),v^{k-1}(x,v))\in \gamma_0^{I_-}$ and $(x^j(x,v),v^{j-1}(x,v))\notin \gamma_0^{I_-}$ for all $1\leq j \leq k-1$. By definition of $B_{L}^{\varepsilon}\subset X^{\varepsilon}$, the number of bounce in trajectory during travel length $L$ is finite. Again, by Lemma \ref{trajectory invertible}, forward in time trajectory from $\mathcal{B}_{L}^{\varepsilon}$ is well-defined until it arrives $\overline{S_0}$. By Lemma \ref{lem time s}, $(X(s(\cdot); 0, \cdot), V(s(\cdot); 0, \cdot))$ is continuous on $\mathcal{B}_{L}^{\varepsilon}$. Using local diffeomorphism $(\sigma, I^2) : (\tau, \varphi)\mapsto (\sigma(\tau, \varphi), I^2(\tau, \varphi))$ and the fact that the inverse image of closed set under continuous function is closed, we finish the proof.   
		\end{proof}
	\hide
	{\color{red} {\bf NOT USED, Remove later}	
	We can similalry construct backward in time image of $G_{L}^{\varepsilon,C}$  in $(\tau, \varphi, \eta)$ 	
	\Be  \label{def cal G}
	\begin{split}
		\mathcal{G}_{L}^{\varepsilon,C} &:= \Big\{ (\tau, \varphi,\eta) \in U\times (0,1) : (\sigma(\tau, \varphi), v_{c}(\eta))\in(X,V)^{-1}(s)( G_{L}^{\varepsilon,C})  \Big\}  \\
	\end{split}
	\Ee		
	\hide
	\Be  \label{def cal G int}
	\begin{split}
		{\mathcal{G}_{L}^{\ephalf, C, \circ} &:= \Big\{ (\theta, \varphi,\tau) \in U\times (0,1) : (\sigma(\theta, \varphi), v_{c}(\tau))\in(X,V)^{-1}(s)( G_{L}^{\ephalf, C, \circ})  \Big\} \\
	\end{split}
	\Ee		
	\unhide
	where $v_{c}(\eta)$ for $0 < \eta< 1$ means direction for concave grazing $\gamma_0^{C}$.	Note that there is no one-to-one correspondence between $G_{L}^{\varepsilon,C}$ and $\mathcal{G}_{L}^{\varepsilon,C}$ since there might be several concave grazing $\gamma_0^C$  in the trajectory. \\

Using the forward trajectory, we define the number of bounces from an inflection point $\sigma(\tau,\varphi)$ with the direction $I^2(\tau,\varphi)$. Here, $I^2(\tau,\varphi)$ are directions defined in Proposition \ref{inflection}. There is no forward in time trajectory propagation with $I^{1}$ direction.  \\ }
\unhide

\begin{proposition} \label{unif bounce}
Recall the definition \eqref{nongrazing NB} of $\mathcal{N}(x,v,L)$. We have the following uniform finite number of bounce
\begin{align*}
	\sup_{(x,v)\in B_{L}^{\varepsilon}}\mathcal{N}(x,v,L) < K=K(\varepsilon,L),
\end{align*}
where the constant $K$ depends on $\varepsilon$ and $L$. \\
i.e., Up to travel length $L > 0$ and for $(x,v)\in B_{L}^{\varepsilon}$, backward in time trajectory which experiences inflection grazing $\gamma_{0}^{I-}$ within travel length $L$ cannot enjoy arbitrarily large number of bounce.
\end{proposition} 
\begin{proof} 
\hide
 {\color{blue} Basically comes from continuity, compactness of $\mathcal{B}^{\varepsilon}_{L}$, and Lemma \ref{no consecutive}. }   \\
	{\color{blue} (Change using Lemmas \ref{analyticity}, \ref{lem time s} etc,) We use forward trajectory and use Lem \ref{lem time s}.} \\
\unhide
	Let us recall $\mathcal{B}_{L}^{\varepsilon}$ is compact by (b) in Lemma \ref{lem:open}. 
	\hide
	and
	\Be
	\begin{split}
		\mathcal{B}_{L}^{\varepsilon} \subset \mathcal{B}_{L}^{\ephalf, \circ} &:= (X, V)^{-1}(s)(B_{L}^{\ephalf, \circ}) 
	\end{split}
	\Ee
	\unhide
\hide
Note that if this set
\begin{equation*}
	 \mathcal{A} = \left \{ k \in \mathbb{N} : \sum_{j=1}^k \vert \xf^j(\sigma(\theta,\varphi),v^I(\theta,\varphi))-\xf^{j-1}(\sigma(\theta,\varphi),v^I(\theta,\varphi)) \vert >L \right \}
\end{equation*}
is empty, then there exists $j$ such that $[\xf^j(\sigma(\theta,\varphi),v^I(\theta,\varphi)),\vf^{j-1}((\sigma(\theta,\varphi),v^I(\theta,\varphi))]\in \gamma_0^I$. In this case, we regard $\mathcal{N}(\theta,\varphi,NT_0)$ as 
\begin{align*}
	\inf \left\{ j\in \mathbb{N}: [\xf^j(\sigma(\theta,\varphi),v^I(\theta,\varphi)),\vf^{j-1}((\sigma(\theta,\varphi),v^I(\theta,\varphi))]\in \gamma_0^I\right\}.
\end{align*}
{\color{blue} (Move above into definition part)} \\
 
{\color{blue} Step 1 [$(\theta,\varphi)$ such that $\mathcal{A}\neq\emptyset$]} 
\unhide
Let us choose $(\tau, \varphi) \in \mathcal{B}_{L}^{\varepsilon}$. Using the forward trajectory, we define the number of bounce $\mathcal{N}(\tau,\varphi,L)$ during whole travel length $L$: 
{\small
\begin{align} \label{forward n}
	\begin{split}
	\mathcal{N}(\tau,\varphi,L):=\sup \left \{ k\in \mathbb{N} : (\xf^j,\vf^{j-1}(\tau,\varphi))\notin \gamma_0^{I_-} \quad \forall 1\leq j\leq k, \; \sum_{j=1}^k \vert \xf^j(\tau,\varphi)-\xf^{j-1}(\tau,\varphi)\vert \leq L \right \},
	\end{split}
\end{align}
}
where we used abbreviation $(\xf^j,\vf^{j-1})(\tau,\varphi)=(\xf^j,\vf^{j-1})(\sigma(\tau,\varphi),I^2(\tau,\varphi))$. We suppose that $[\xf^i(\tau,\varphi),\vf^{i-1}(\tau,\varphi)] \notin \gamma_0^I$ for all $1\leq i \leq \mathcal{N}(\tau,\varphi,L)$. Next, we split two cases:

\textit{(Case 1)} There is no concave grazing in a finite travel length. i.e., $[\xf^i (\tau,\varphi),\vf^{i-1}(\tau,\varphi)]$\\
$ \notin \gamma_0^C$ for all $1\leq i \leq \mathcal{N}(\tau,\varphi,L)$. \\

\textit{(Case 2)} There is a concave grazing at least once. i.e., There exists $j\in \mathbb{N}$ such that $[\xf^j(\tau,\varphi),\vf^{j-1}(\tau,\varphi)]\in \gamma_0^C$. \\
In \textit{(Case 1)}, $(x^i(\tau,\varphi),v^i(\tau,\varphi))$ and $(X(s;0,\sigma(\tau,\varphi),I^2(\tau,\varphi)),V(s;0,\sigma(\tau,\varphi),I^2(\tau,\varphi)))$ are continuous in $(\tau,\varphi)$ by (a) and (b) of Lemma \ref{analyticity}. Hence, there exists $r^1=r^1_{\tau,\varphi,\epsilon}>0$ such that 
\begin{align*}
	\vert \mathcal{N}(\tau,\varphi,L)-\mathcal{N}(\tau^*,\varphi^*,L) \vert \leq 1, 
\end{align*}
whenever $\vert (\tau,\varphi)-(\tau^*,\varphi^*) \vert \leq r^1_{\tau,\varphi,\varepsilon}$. \\
In \textit{(Case 2)}, we use the fact that there is no consecutive concave grazing in $\Omega$. Suppose that $[\xf^k(\tau,\varphi),\vf^{k-1}(\tau,\varphi)] \notin \gamma_0^C$ for all $1\leq k<j$ and $[\xf^j(\tau,\varphi),\vf^{j-1}(\tau,\varphi)] \in \gamma_0^C$. If we choose sufficiently small $r^2=r^2_{j,\tau,\varphi,\varepsilon}>0$ such that $\vert(\tau,\varphi)-(\tau^*,\varphi^*)\vert<r^2_{j,\tau,\varphi,\varepsilon}$, then $[\xf^k(\tau^*,\varphi^*),\vf^{k-1}(\tau^*,\varphi^*)]\notin \gamma_0^C$ for all $1\leq k <j$ by continuity argument. Similar to (Case 1), we have 
\begin{align*}
	\vert \mathcal{N}(\tau,\varphi,\tf^{j-1}(\tau,\varphi))-\mathcal{N}(\tau^*,\varphi^*,\tf^{j-1}(\tau,\varphi))\vert \leq 1,
\end{align*}
where $\tf^{j-1}(\tau,\varphi):=\tf^{j-1}(\sigma(\tau,\varphi),I^2(\tau,\varphi))$. Moreover, one obtains that 
\begin{align} \label{before grazing}
	\vert (\xf^{j-1}(\tau,\varphi),\vf^{j-1}(\tau,\varphi))-(\xf^{j-1}(\tau^*,\varphi^*),\vf^{j-2}(\tau^*,\varphi^*))\vert <\delta,
\end{align}
where $\delta>0$ was defined in Lemma \ref{basic}. Note that $ (\xf^{j+1}(\tau,\varphi),\vf^{j}(\tau,\varphi)) \notin \gamma_0^C$ since there is no consecutive concave grazing from Lemma \ref{no consecutive}. 
We split $[\xf^j(\tau^*,\varphi^*),\vf^{j-1}(\tau^*,\varphi^*)]$ into three cases:\\

\textit{(i)} $\overline{\xf^{j-1}(\tau^*,\varphi^*)\xf^j(\tau^*,\varphi^*)}$ does not bounce near $\xf^j(\tau,\varphi)$: This case coincides with (b-2) in Lemma \ref{basic} for $\xf^j$. Hence, from (b-2) in Lemma \ref{basic} and \eqref{before grazing}, we obtain 
\begin{align*}
	\vert (\xf^{j+1}(\tau,\varphi),\vf^{j+1}(\tau,\varphi))-(\xf^j(\tau^*,\varphi^*),\vf^j(\tau^*,\varphi^*))\vert < \varepsilon,
\end{align*}
if $\vert (\tau,\varphi)-(\tau^*,\varphi^*)\vert <r^2_{j,\tau,\varphi,\varepsilon}$. Also, (a) in Lemma \ref{analyticity} implies that $\tf^j$ is locally continuous function of $(\tau,\varphi)$ up to some renumbered bounce number, that is, $\vert \tf^{j+1}(\tau,\varphi)-\tf^j(\tau^*,\varphi^*)\vert <\varepsilon$. Hence, $\mathcal{N}(\tau,\varphi,\tf^{j+1}(\tau,\varphi))=\mathcal{N}(\tau^*,\varphi^*,\tf^{j}(\tau^*,\varphi^*))+1$, and
{\footnotesize
\begin{align*}
	\vert \left(\mathcal{N}(\tau,\varphi,\tf^{j+1}(\tau,\varphi))-\mathcal{N}(\tau,\varphi,\tf^{j-1}(\tau,\varphi))\right)-\left(\mathcal{N}(\tau^*,\varphi^*,\tf^{j+1}(\tau,\varphi)-\mathcal{N}(\tau^*,\varphi^*,\tf^{j-1}(\tau,\varphi))\right)\vert \leq 2.
\end{align*}
}
\textit{(ii)} $\xf^j(\tau^*,\varphi^*)$ is located near $\xf^j(\tau,\varphi)$ and $[(\xf^j(\tau^*,\varphi^*),\vf^{j-1}(\tau^*,\varphi^*)]\notin \gamma_0^C:$  (b-1) in Lemma \ref{basic} includes this case. Applying (b-1) in Lemma \ref{basic} and \eqref{before grazing} yields that 
\begin{align*}
	\vert (\xf^{k}(\tau,\varphi),\vf^{k}(\tau,\varphi))-(\xf^k(\tau^*,\varphi^*),\vf^k(\tau^*,\varphi^*))\vert < \varepsilon, \quad k=j,j+1,
\end{align*}
if $\vert (\tau,\varphi)-(\tau^*,\varphi^*)\vert <r^2_{j,\tau,\varphi,\varepsilon}$. Similarly, in this case, it holds that $\vert \tf^k(\tau,\varphi)-\tf^k(\tau,\varphi) \vert<\varepsilon$ for $k=j,j+1$ due to continuity of $\tf^j$. Hence, 
{\footnotesize
\begin{align*}
	\vert \left(\mathcal{N}(\tau,\varphi,\tf^{j+1}(\tau,\varphi))-\mathcal{N}(\tau,\varphi,\tf^{j-1}(\tau,\varphi))\right)-\left(\mathcal{N}(\tau^*,\varphi^*,\tf^{j+1}(\tau,\varphi)-\mathcal{N}(\tau^*,\varphi^*,\tf^{j-1}(\tau,\varphi))\right)\vert \leq 1. 
\end{align*}
 }
\textit{(iii)} $\xf^j(\tau^*,\varphi^*)$ is located near $\xf^j(\tau,\varphi)$ and $[(\xf^j(\tau^*,\varphi^*),\vf^{j-1}(\tau^*,\varphi^*)]\in \gamma_0^C:$ This case is also included in (b-1) in Lemma \ref{basic}. Thus, similar to (2-2), we obtain the same result above. \\

By Lemma \ref{no consecutive}, the maximum number of bounces where concave grazing occurs is $[ \frac{\mathcal{N}(\tau,\varphi,L)}{2}]$ during a finite travel length $L$. We set $r^2_{\tau,\varphi,\varepsilon}:= \min \{ r^1_{\tau,\varphi,\varepsilon},\bigcup_{j\in \mathcal{C}_{\tau,\varphi}}r^2_{j,\tau,\varphi,\varepsilon}\}$ where $\mathcal{C}_{\tau,\varphi}=\{l \in \mathbb{N}: (\xf^l(\tau,\varphi),\vf^{l-1}(\tau,\varphi))\in \gamma_0^C\}\;(\max \mathcal{C}_{\tau,\varphi}\leq\mathcal{N}(\tau,\varphi,L)).$ Considering three cases above, we derive that 
\begin{align*}
	\vert \mathcal{N}(\tau,\varphi,L)- \mathcal{N}(\tau^*,\varphi^*,L) \vert \lesssim \mathcal{N}(\tau,\varphi,L),
\end{align*}
if $\vert (\tau,\varphi)-(\tau^*,\varphi^*) \vert < r^2_{\tau,\varphi,\varepsilon}$. \\


To get a uniform number of bounces, we use a compactness argument. 
Since the range $\mathcal{B}_{L}^{\varepsilon}$ of $(\tau,\varphi)$  is compact, using a compactness argument gives uniform boundedness for the number of bounces. For each $(\tau,\varphi) \in \mathcal{B}_{L}^{\varepsilon}$, we construct a small ball $B_{\delta_{\tau,\varphi,\varepsilon}}(\tau,\varphi)\subset \R^2$, where $\delta_{\tau,\varphi,\varepsilon}:= \min\{r^1_{\tau,\varphi,\varepsilon}, r^2_{\tau,\varphi,\varepsilon}\}$. Then, for $(\tau^*,\varphi^*)\in B_{\delta_{\tau,\varphi,\varepsilon}}(\tau,\varphi)$, it satisfies that 
\begin{equation*}
	\vert \mathcal{N}(\tau,\varphi,L)- \mathcal{N}(\tau^*,\varphi^*,L) \vert \lesssim \mathcal{N}(\tau,\varphi,L).
\end{equation*}
Thus, by compactness of $\mathcal{B}_L^{\varepsilon}$, there exists a finite open covering $\cup_{j=1}^l B_{\delta^j_{\tau_j,\varphi_j,\varepsilon}}(\tau_j,\varphi_j)$. Since we exclude $R_{x,\tau_{1,*}}^{\varepsilon},R_{x,\tau_{2,*}}^{\varepsilon}$ in $X^{\varepsilon}$ and $B_L^{\varepsilon}$ is a subset of $X^{\varepsilon}$, for any $(\tau,\varphi)\in \mathcal{B}_L^\varepsilon$, we have 
\begin{equation*}
	\mathcal{N}(\tau,\varphi,L) \lesssim \max_{1\leq j \leq l} \mathcal{N}(\tau_j,\varphi_j,L)<+\infty,
\end{equation*}
where the finiteness can be obtained by Lemma \ref{no infty bounce}. Then, trajectory is deterministic and reversible, which gives one-to-one correspondence between $B_{L}^{\varepsilon}$ and $\mathcal{B}_{L}^{\varepsilon}$. Hence, we finish the proof.   
\end{proof}

\subsection{Uniform number of bounce away from inflection grazing $\gamma_0^{I}$}
Through Proposition \ref{unif bounce}, we obtain that an inflection grazing could only happen within the maximum $K$ bounce in backward in time trajectory starting at $(x,v)\in B_{L}^{\varepsilon}$ defined in \eqref{B,G sets}. To remove all inflection grazing, for fixed $x\in \overline{S_0}$, we will construct bad sets $\{ (B_j)_x\}_{j=1}^K\subset \mathbb{S}^2$ where inflection grazing does happen on the $j$-th bounce within the backward in time trajectory starting from $x$. Moreover, the bad sets also contain concave grazing to use analyticity such as Lemma \ref{analyticity} by excluding them. \\
Let $K$ be an upper bound for the number of bounces in Proposition \ref{unif bounce}. Now we make slightly different decomposition of $X^{\varepsilon}$ as follows for fixed $x\in \overline{S_{0}}$,
\Be \label{def GB}
\begin{split}
	(G_1)_x&:= \{v \in \S^2 : (x,v)\in X^{\varepsilon}, \ (x^1(x,v),v^0(x,v))\notin \gamma_0^C \cup \gamma_0^I\},\\ 
	(B_1)_x&:=\{v \in \S^2: (x,v)\in X^{\varepsilon}, \ (x^1(x,v),v^0(x,v))\in \gamma_0^C \cup \gamma_0^I\},\\ 
	(G_2)_x&:=\{v \in (G_1)_x: (x^2(x,v),v^1(x,v))\notin \gamma_0^C \cup \gamma_0^I\},\\ 
	(B_2)_x&:=\{v \in (G_1)_x: (x^2(x,v),v^1(x,v))\in \gamma_0^C \cup \gamma_0^I\},\\ 
	\vdots\\
	(G_{K})_x&:= \{v \in (G_{K-1})_x: (x^{K}(x,v),v^{K-1}(x,v))\notin \gamma_0^C \cup \gamma_0^I\},\\ 
	(B_{K})_x&:=\{v \in (G_{K-1})_x: (x^{K}(x,v),v^{K-1}(x,v))\in \gamma_0^C \cup \gamma_0^I\}.
\end{split}
\Ee
\hide
{\color{red} (remove?)Using Definition \ref{decomp BL}, we can simply write (check)
	\Be
	\begin{split}
		(G_{j})_{x} &:= \{v \in (G_{j-1})_{x} : (x,v)\notin B_{L}^{\varepsilon, j}\cup G_{L}^{\varepsilon, C, j}  \}, \\ 
		(B_{j})_{x} &:= \{v \in (G_{j-1})_{x} : (x,v)\in B_{L}^{\varepsilon, j}\cup G_{L}^{\varepsilon, C, j} \},\\ 
	\end{split}
	\Ee
} 
\unhide
We note that for fixed $x\in \overline{S_{0}}$,
\Be \label{BG decomp}
\begin{split}
	(X^{\varepsilon})_{x} &:= \{v\in \S^{2} : (x,v)\in X^{\varepsilon} \} \\
	&= (B_1)_{x} \cupdot (G_1)_{x} = (B_1)_{x} \cupdot \big(  (B_2)_{x}  \cupdot (G_2)_{x} \big)\\
	&=\cdots \\
	&= \Big( \bigcupdot_{j=1}^{K} (B_j)_{x} \Big) \cupdot (G_K)_{x}.
\end{split}
\Ee 
For fixed $x\in \overline{S_{0}}$, we assume that $(B_j)_x$ in \eqref{def GB} has measure zero for each $1\leq j\leq K$. (this will be proved in Proposition \ref{mB zero}.) If we define (set of bad directions) 
\begin{align} \label{def Ki}
\mathcal{BK}_x := \bigcup _{j=1}^{K} (B_j)_x \subset \S^{2},
\end{align}
there exists an open set $\mathcal{BK}_x^{\varepsilon} \subset \S^{2}$ such that
\Be \label{Ki epsilon}
	\mathcal{BK}_{x} \subset \mathcal{BK}_x^{\varepsilon},\quad \mathfrak{m}_{2}\big( \mathcal{BK}_x^{\varepsilon}\big) < \varepsilon.
\Ee 


\begin{lemma}\label{continuity 2} For $x\in \overline{S_0}$, $\varepsilon>0$, recall definition $\mathcal{BK}_x$ in \eqref{def Ki}. For fixed $0< \varepsilon \ll1$. For each $x\in \overline{S_{0}}$, there exists $r=r(x,\varepsilon)$ such that 
	\begin{equation*}
	\mathcal{BK}_y \subset \mathcal{BK}_{x}^{\varepsilon},
	\end{equation*}
	if $y\in B(x,r(x,\varepsilon)) \cap \overline{S_{0}}$.    \\
\end{lemma}
\begin{proof}
	We use a contradiction argument to prove Lemma. Suppose that there exists $v\in \S^2$ satisfying 
	\begin{equation*}
	v\in (\mathcal{BK}_y) \cap (\mathcal{BK}_{x}^{\varepsilon})^c. 
	\end{equation*}
	Then, one obtains that  $(x^j(y,v),v^j(y,v))\in \gamma_0$ for some $1\leq j \leq K$ and $(x^k(x,v),v^k(x,$\\
	$v))\notin \gamma_0$ for all $1\leq k \leq K$. By (a) in Lemma \ref{analyticity}, one obtains that $(x^i(x,v),v^i(x,v))$ is locally continuous of $(x,v)$, which leads to a contradiction.  
\end{proof} 


\hide
\begin{definition}
For each cover $B(x_i,r(x_i,\varepsilon))\cap \overline{S_0}$ in above argument, we define corresponding sets
\Be \label{def bad O}
\mathcal{O}^{\mathcal{K}}_{i} := \bigcup_{j=1}^l R_{x_i,\tau_{j}^{\mathbf{z}}}^{2\varepsilon}\cup R_{x_i,\tau_{1,*}}^{2\varepsilon} \cup R_{x_i, \tau_{2,*}}^{2\varepsilon} \cup R_{\perp}^{2\varepsilon} \cup R_{\hat{\varphi}}^{2\varepsilon} \cup R_{sym}^{2\varepsilon} \cup \mathcal{O}_i^{nB} \cup (\mathcal{K}_{x_{i}})_{ \varepsilon},\quad 1\leq i\leq l_{K}, 
\Ee
where $\mathfrak{m}_{2}(R^{2\varepsilon}_{x_{i}, *}) \lesssim \varepsilon$ for all $1\leq i \leq l_{K}$ and $*\in \bigcup_{j=1}^l \{ \tau_{j}^{\mathbf{z}}\}\cup\{\tau_{1,*},\tau_{2,*}\}$. Also, by Lemma \ref{near_boundary}, we have $\mathfrak{m}_{2}(\mathcal{O}_i^{nB})\leq \varepsilon$ for all $1\leq i \leq l_K$. Note that each $\mathcal{O}^{\mathcal{K}}_{i}$ is open in $\mathbb{S}^{2}$ so that  $(\mathcal{O}^{\mathcal{K}}_{i})^{c}$ is compact on $\S^{2}$.  
\end{definition}
\unhide

\begin{lemma} \label{construct K}
	For fixed $x\in \overline{S_0}$, assume that $\mathfrak{m}_2((B_j)_x)=0$ for all $1\leq j \leq K$ where we defined $(B_j)_x$ in \eqref{BG decomp}. Recall the definition \eqref{def Xe} of $X^{\varepsilon}$ in Lemma \ref{lem Xe}. Then, there exists a compact set $\mathcal{K}^c\subset \overline{S_0}\times \S^2$ such that if $(x,v) \in \mathcal{K}^c$, 
	\[
	(x^{i}(x,v), v^{i-1}(x,v)) \notin \gamma_{0}^{I}\cup \gamma_0^{C},\quad \text{for first $K$ bounce},
	\]
	and
	\[
	(x^{i}(x,v), v^{i-1}(x,v)) \notin \gamma_{0}^{I},\quad \text{up to whole travel length $L$}.
	\]
	Moreover, for $x\in \overline{S_0}$, 
	\begin{align*}
	\mathfrak{m}_2((\{\overline{S_0}\times\S^2\}\backslash \mathcal{K}^c)_x) \lesssim \varepsilon,
	\end{align*}
	 where we have used the notation $A_x$ in \eqref{proj_v}. 
\end{lemma}
\begin{proof}
	Let us define 
	\Be \label{E}
		E := \Big\{ (x,v) \in X^{\varepsilon} : v \in  \mathcal{BK}_x   \Big\},
	\Ee
	where $X^{\varepsilon}$ is defined in \eqref{def Xe} and $\mathcal{BK}_x$ is defined in \eqref{def Ki}. Recall that $\mathcal{P}_1(X^{\varepsilon})=\overline{S_0}$ from \eqref{Xe proj} in proof of Lemma \ref{lem Xe}. Using compactness of $\mathcal{P}_1(X^{\varepsilon})=\overline{S_0}$, Lemma \ref{continuity 2} and \eqref{Ki epsilon} from our assumption, we obtain an open cover for $\mathcal{P}_1(X^{\varepsilon})=\overline{S_0}$ 
	 \Be \label{open cover_K}
		\Big\{B(x_i^{\mathcal{K}},r_i^{\mathcal{K}})\cap \overline{S_0}\Big \}_{i=1}^{l_K},
	\Ee
	and corresponding each open set $\mathcal{O}_i^{\mathcal{K}} \subset \S^{2}$ which is defined by 
	\begin{align} \label{def bad O}
		 \mathcal{O}_i^{\mathcal{K}}:=\mathcal{BK}_{x_{i}^{\mathcal{K}}}^{\varepsilon}, \quad  1\leq i \leq l_K,
	\end{align}
	satisfying $\mathbf{m}_2(\mathcal{O}_i^{\mathcal{K}})< \varepsilon$ for all $1\leq i \leq l_K$ and  
	\begin{align} \label{O_i^K 2}
		\mathcal{BK}_x \subset \mathcal{O}_i^{\mathcal{K}}, 
	\end{align}
	for all $x \in B(x_i^{\mathcal{K}},r_i^{\mathcal{K}})$ and sufficiently small $r_{i}^{\mathcal{K}}$. 
	 Similar to \eqref{index_nB} and \eqref{index_R}, for given $x\in \mathcal{P}_1(X^{\varepsilon})=\overline{S_0}$, we define an index set 
\begin{align} \label{index_K}
	I_x^{\mathcal{K}}:=\Big\{i\in\mathbb{N}: x\in B(x_i^{\mathcal{K}},r_i^{\mathcal{K}})\cap \overline{S_0}\Big\}. 
\end{align}
We consider the following set containing set $E$ defined in \eqref{E}:
\begin{align} \label{bad phase set 3} 
 \bigcup_{x\in \overline{S_0}} \Big \{(x,v)\in X^{\varepsilon}: v \in  \bigcap_{i\in I_x^{\mathcal{K}}}\mathcal{O}_i^{\mathcal{K}}\Big \}.
\end{align}
Similar to \eqref{def R} in proof of Lemma \ref{lem Xe}, we construct an open cover for \eqref{bad phase set 3} as
\begin{align} \label{K}
	\mathcal{K}:=\bigcup_{i=1}^{l_{K}} \Big \{ (x,v)\in X^{\varepsilon}: x \in B(x_i^{\mathcal{K}},r_i^{\mathcal{K}})\cap \overline{S_0},\; v \in \bigcap_{j\in I_x^{\mathcal{K}}}\mathcal{O}_j^{\mathcal{K}}\Big\}_{\varepsilon},
\end{align}
where we have used the notation $A_{\varepsilon}$ in \eqref{extension domain}. Throughout this paper, we denote a compact set $X^{\varepsilon}\backslash \mathcal{K}$ as  
\begin{align} \label{Kc}
	\mathcal{K}^c := X^{\varepsilon} \backslash \mathcal{K}. 
\end{align}
Assume that $(x,v) \in \mathcal{K}^c=X^{\varepsilon}\backslash\mathcal{K}$. 
Since definition \eqref{K} of $\mathcal{K}$ and \eqref{O_i^K 2}, $v\notin \mathcal{BK}_x= \bigcup_{j=1}^K (B_j)_x$. This implies that 
\begin{align*}
	(x^i(x,v),v^{i-1}(x,v)) \notin \gamma_0^I \cup \gamma_0^C, \quad \forall 1\leq i \leq K. 
\end{align*}
From Proposition \ref{unif bounce}, during the travel length $L$, the maximum number of bounce that inflection grazing can occur is $K$. For $(x,v)\in \mathcal{K}^c$, there is no inflection grazing for first $K$ bounce. Hence,  $(x^i(x,v),v^{i-1}(x,v))\notin \gamma_0^I$ during the travel length $L$. For fixed $x \in \overline{S_0}$, 
\begin{align} \label{def Kx}
	\mathcal{K}_x=\Big\{v \in (X^{\varepsilon})_x: v \in \bigcap_{j\in I_x^{\mathcal{K}}}\mathcal{O}_j^{\mathcal{K}}\Big\}_{\varepsilon},
\end{align}
where we used notation $A_x$ in \eqref{proj_v}. Notice that the notation $A_\varepsilon$ in \eqref{extension domain} is caused by the definition \eqref{K}. 
From the definition \eqref{Kc} of $\mathcal{K}^c$, one obtains that 
\begin{align*}
	\{\overline{S_0}\times \S^2\} \backslash \mathcal{K}^c =(\{\overline{S_0}\times \S^2\} \backslash X^{\e}) \cupdot \mathcal{K}. 
\end{align*}
Thus, 
\begin{align}\label{Kc small} 
	\begin{split}
	\mathfrak{m}_{2}(( \{\overline{S_0}\times \S^2\}\backslash \mathcal{K}^c)_x) 
	&\leq \mathfrak{m}_{2}((\{\overline{S_0}\times \S^2\} \backslash X^\e)_x) + \mathfrak{m}_{2}(\mathcal{K}_x) \\
	&\leq \mathfrak{m}_{2}((\{\overline{S_0}\times \S^2\} \backslash X^\e)_x) + \mathfrak{m}_{2}((\mathcal{O}_j^{\mathcal{K}})_{\varepsilon}) \\
	&\lesssim \varepsilon, \quad \forall j \in I_x^{\mathcal{K}},
	\end{split}
\end{align}
by using \eqref{Xe c small},\eqref{def Kx} and $\mathfrak{m}_2(\mathcal{O}_j^{\mathcal{K}})\lesssim \e$ for all $j=1,2,\dots,l_K$. 
\end{proof}
\hide
From above lemma, we define $\mathcal{K}\in \overline{S_{0}}\times \S^{2}$
\Be \label{K}
\begin{split}
	\mathcal{K} &:=  \bigcup_{i=1}^{l_{\mathcal{K}}} \Big\{ (B(x_{i}, r_{i}) \cap \overline{S_0})\times \mathcal{O}_{i}^{\mathcal{K}}  \Big\},
\end{split}
\Ee
Hence $\mathcal{K}$ is an open $\varepsilon$-neighborhood. By taking complement,
\Be \label{Kc}
\begin{split}
	\mathcal{K}^{c} &:=  \Big\{ (x,v)\in \overline{S_{0}}\times \S^{2} \ : \ x\in \{ B(x_i,r(x_i,\varepsilon)) \cap \overline{S_{0}} \}, \ v\in (\mathcal{O}_{i}^{\mathcal{K}})^{c} \ \text{for some $1\leq i \leq l_{K}$} \Big\}. 
\end{split}
\Ee
\unhide

\hide
\begin{remark}
Note that above $\mathcal{K}^{c}$ is good set in which trajectory does not undergo $\gamma_0^{I}$. So $\mathcal{K}$ corresponds to $\mathcal{IB}^{c}$ in \cite{cylinder}. Also for fixed $x\in \overline{S_{0}}$, measure $|\{v \in \S^2 : \ (x,v)\in \mathcal{K} \}| \lesssim \varepsilon$ since such $v$ must belong to every other $\mathcal{O}_{j}^{\mathcal{K}}$ for all $j$ such that $x$ belongs. So, using definiton $\mathcal{K}$, we can briefly rewrite above lemma as following Lemma \ref{mB zero}. If $(x,v)\in \mathcal{K}^{c}$, then its trajectory does not graze up to $K$ bounce and there is no inflection grazing for whole length $L$.
\end{remark}
\unhide

Now we obtain uniform number of bounce on $\mathcal{K}^{c}=X^{\varepsilon} \backslash \mathcal{K}$.

\begin{lemma} \label{bounce K}
	$\mathcal{N}(x,v, L)$ in \eqref{nongrazing NB} is pointwisely finite for each $(x,v) \in \mathcal{K}^{c}$ by Proposition \ref{no infty bounce}. We have a uniform number of bounces on $\mathcal{K}^{c}$ set, i.e.,
	\[
	\sup_{(x,v)\in \mathcal{K}^{c} } \mathcal{N}(x,v,L) \leq M=M(\varepsilon,L),
	\]
	where the constant $M$ depends on $\varepsilon$ and $L$. 
\end{lemma}
\begin{proof}
	 Since we have excluded $B_{L}^{\varepsilon} \subset \cupdot_{j=1}^{K} (B_j)$ uniformly in the definition of $\mathcal{K}^{c}$,  there is no more inflection grazing bounce. Now, using similar argument which was used in the proof of Proposition \ref{unif bounce} (for counting) and compactness of $\mathcal{K}^{c}$ (because $\mathcal{K}$ is open in $X^\varepsilon\subset \overline{S_0}\times \S^2$), we get the result.  
\end{proof}

Using above number of bounces away from $\mathcal{K}$, let us assume that $M$ of Lemma \ref{bounce K} is bigger than $K$ of Proposition \ref{unif bounce} WLOG. Since $(x^j(x,v),v^{j-1}(x,v))\notin\gamma_0^{I}$ for $(x,v)\in \mathcal{K}^c$ and $K+1\leq j \leq M$, we define the following (similar to \eqref{def GB})
\Be \label{def GB2}
\begin{split}
	(G_{K+1})_x&:= \{v \in (G_{K})_x: (x^{K+1}(x,v),v^{K}(x,v))\notin \gamma_0^C  \},\\ 
	(B_{K+1})_x&:=\{v \in (G_{K})_x:  (x^{K+1}(x,v), v^{K}(x,v))\in \gamma_0^C  \}, \\
	\vdots \\
	(G_{M})_x&:= \{v \in (G_{M-1})_x: (x^{M}(x,v),v^{M-1}(x,v))\notin \gamma_0^C  \},\\ 
	(B_{M})_x&:=\{v \in (G_{M-1})_x: (x^{M}(x,v),v^{M-1}(x,v))\in \gamma_0^C  \}, \\
\end{split}
\Ee
with decomposition 
\Be \label{BG decomp M}
\begin{split}
	(X^{\varepsilon})_{x} := \{v\in \S^{2} : (x,v)\in X^{\varepsilon} \}  
	= \Big( \cupdot_{j=1}^{M} (B_j)_{x} \Big) \cupdot (G_M)_{x}. \\
\end{split}
\Ee \\

\section{Construction of Bad sets}

\subsection{Construction of $\mathcal{IB}$ set} 
In this section, we prove that $\mathfrak{m}_{2} ((B_{j})_{x}) = 0$ for all $1\leq j \leq M$ which we have assumed from \eqref{def Ki}.
\hide
\begin{remark} \label{Remark2}
	Now we should classify small neighborhood which cause trouble. To argue small size of such neighborhood, we need analyticity. So we will consider $S_0^{\delta}$ to exclude $\delta$ layer on $S_0$. Note that this looks essential since we need analyticity of $X(s(\cdot, \cdot); 0, \cdot, \cdot)$ which is guaranteed when $X(s(x,v);0,x,v)\notin \p S_0$ by the second statement of Lemma \ref{lem time s}.  \\
\end{remark}
\unhide
$\newline$
\hide
\noindent \textbf{5-1. Bad set is measure zero when $x\in \overline{S_{0}}$} \\
To obtain analyticity of trajectories and $X(s(\cdot, \cdot);0, \cdot, \cdot)$, we define $S_{0}^{\delta}$ ($\delta$-layer removed domain)
\Be \label{S_0 delta}
S_{0}^{\delta} = \{ x\in S_{0} : \text{dist}(\p S_{0}, x) > \delta \},\quad \delta > 0.
\Ee 
\unhide

\begin{lemma}\label{Lusin}(Lusin's property)
	Let $f:\R^n\rightarrow \R^n$ be a Lipschitz continuous function. Then, $f(E)$ has measure zero in $\R^n$ for a measure zero set $E\subset \R^n$. 
\end{lemma}
\begin{proof}
	Let $E$ be a set of measure zero in $\R^n$. Then, for any $\varepsilon>0$, there exist countably many open balls $E\subset \cup_{j=1}^\infty B(x_j,r_j)$ with 
\begin{align*}
	\sum_{j=1}^\infty \mathfrak{m}_{n} (B(x_j,r_j)) \leq \varepsilon, 
\end{align*}
where $\mathfrak{m}_{n}$ is Lebesgue measure on $\R^n$ and $B(x,r)$ be an open ball centered at $x$ with radius $r$. By Lipschitz continuity of $f$, for each $j$, $f(B(x_j,r_j))$ is contained in a ball of radius $L\cdot r_j$ centered $f(0)$ where $L$ is the Lipschitz constant of $f$. Thus, 
\begin{align*}
	\mathfrak{m}_{n}^* (f(B(x_j,r_j)))\leq L^n \mathfrak{m}_{n} (B(x_j,r_j)),
\end{align*}
where $\mathfrak{m}_{n}^*$ be Lebesgue outer measure on $\R^n$. Hence, we get 
\begin{align*}
	\mathfrak{m}_{n}^* (f(E)) \leq L^n \sum_{j=1}^\infty \mathfrak{m}_{n} (B(x_j,r_j)) <L^n \varepsilon. 
\end{align*}
Since $\varepsilon>0$ was arbitrary, $f(E)$ has zero outer measure, which means $f(E)$ is a null set in $\R^n$. 
\end{proof}

\begin{proposition} \label{mB zero}
	 Let $M$ be the constant defined in Lemma \ref{bounce K} and $S_0$ be $\varphi=0$ cross-section in \eqref{cross section}. Here, the constant $M=M(\varepsilon,L)$ depends on $\varepsilon$ and travel length $L$. We denote $\overline{S_0}$ in Definition \ref{varphi cross section} as closure of $S_0$. Recall that we constructed bad direction sets $\{B_j\}_{j=1}^M$ in \eqref{def GB} and \eqref{def GB2}. Then, for fixed $x\in \overline{S_{0}}$,
	\[
		\mathfrak{m}_{2} ((B_{j})_{x} )= 0, \quad 1\leq j \leq M.
	\]
\end{proposition}
\begin{proof} For fixed $x\in \overline{S_0}$ and $1\leq j \leq M$, we defined the set $(B_j)_x$ as 
\begin{align*}
	(B_j)_x= \{v \in (G_{j-1})_x : (x^{j}(x,v),v^{j-1}(x,v))\in \gamma_0^C \cup \gamma_0^I\},
\end{align*}
where $(G_{j-1})_x$ was constructed inductively in \eqref{def GB} and \eqref{def GB2}. To get $\mathfrak{m}_{2} ((B_{j})_{x}) = 0$, 
	We separate into concave grazing ($\gamma_0^C$) and inflection grazing ($\gamma_0^I$). \\
\noindent {\bf Step 1} In this step, we deal with concave grazing case in $(B_j)_x$. We claim 
\begin{align*}
\mathfrak{m}_{2}(\{v \in (G_{j-1})_{x} : (x^j(x,v),v^{j-1}(x,v))\in \gamma_0^C\} )= 0.
\end{align*}
\noindent{(\bf{Interior part $S_0$})} Firstly, we consider the interior of $\overline{S_0}$. Let $x\in {S_0}$. Note that we exclude $R_{x,\tau_{1,*}}^{\varepsilon}$ and $R_{x,\tau_{2,*}}^{\varepsilon}$ from $(X^{\varepsilon})_x$ defined in \eqref{BG decomp M}. Moreover, inflection grazing does not happen in trajectory between $(x,v)$ and $(x^j(x,v),v^{j-1}(x,v))$. From Lemma \ref{no infty bounce}, one obtains that trajectory is invertible. Using an invertible map $(X,V)(s(\cdot))$, $(x,v)$ can be expressed by
\begin{align*}
\begin{split}
	(x,v) = (X, V)(s(\sigma(\tau, \varphi),\varphi, v_c(\eta)); 0, \sigma(\tau, \varphi), v_{c}(\eta)),
\end{split}
\end{align*}
where $s(\cdot,\cdot,\cdot)$ be the arrival time defined in \eqref{def time s}, $\sigma(\cdot,\cdot)$ be the parametrization of boundary $\p\O$ in \eqref{sigma para}, and $v_c(\cdot)$ be parametrization for concave grazing in \eqref{concave para}. 
We reparametrize into
\Be \label{Q_L^e}
\begin{split} 
	Q_{L}^{\varepsilon, j}&:= \left\{ (\tau, \varphi, \eta) : 
	\begin{array}{cc}
	(\sigma(\tau, \varphi), v_{c}(\eta)) = (x^{j}(y,u), v^{j-1}(y,u))\in \gamma_{0}^{C}\\
	 \ \text{for some} \ (y,u)\in {S_0}\times \S^2, \ u\in (G_{j-1})_{y}
	\end{array}
	 \right\}. 
\end{split}
\Ee
We consider 
\begin{align*}
	\tilde{F}_j(\tau,\varphi,\eta):= X(s(\sigma(\tau,\varphi),\varphi,v_c(\eta));0, \sigma(\tau,\varphi), v_{c}(\eta)) - x, \quad (\tau, \varphi, \eta) \in Q_{L}^{\varepsilon,  j}.   \\  
\end{align*}

First, for analyticity of the function $\tilde{F}_j(\tau,\varphi,\eta)$ on $Q_{L}^{\varepsilon,j}$ later, we should claim that $Q_{L}^{\varepsilon,j}$ is open in $(\tau,\varphi,\eta)$ space. 
Note that $\sigma$ in \eqref{sigma para} and $v_{c}$ in \eqref{concave para} are all locally smooth. Now, let us consider forward in time trajectory from $(\sigma(\tau, \varphi), v_{c}(\eta)) = (x^{j}(y,u), v^{j-1}(y,u))$. By definition of $(G_{j-1})_{y}$ and Lemma \ref{trajectory invertible}, the trajectory is non-grazing until it arrives at $(y,u)$ and hence the map
\[	
	(\tau, \varphi, \eta) \mapsto (\sigma(\tau, \varphi), v_{c}(\eta))  \mapsto (X,V)(s(\sigma(\tau, \varphi),
	\varphi,v_c(\eta)); 0, \sigma(\tau, \varphi), v_{c}(\eta)) \in S_0 \times \S^{2}
\]
is continuous with finite number of bounce by Lemma \ref{lem time s}. Now, for sufficiently small $\varepsilon > 0$, there exists $\delta  \ll 1$ such that if $| (\tau, \varphi, \eta) - (\tau^{\prime}, \varphi^{\prime}, \eta^{\prime}) | < \delta$, then
\[
	|(y,u) - (y^{\prime}, u^{\prime})| < \varepsilon,
\] 
where 
\[
	(y^{\prime}, u^{\prime}) := (X,V)(s(\sigma(\tau^{\prime}, \varphi^{\prime}),\varphi^{\prime}, v_c(\eta^{\prime})); 0, \sigma(\tau^{\prime}, \varphi^{\prime}), v_{c}(\eta^{\prime}))  \in S_0\times \S^{2},
\]
or equivalently,
\[
	(\sigma(\tau^{\prime}, \varphi^{\prime}), v_{c}(\eta^{\prime})) = (x^{j}(y^{\prime}, u^{\prime}), v^{j-1}(y^{\prime}, u^{\prime})) \in \gamma_0^{C}.
\] 
Since $(y,u)\in G_{L}^{\varepsilon}$ and $G_{L}^{\varepsilon}$ is open by (a) in Lemma \ref{lem:open}, choosing $\varepsilon$ sufficiently small, $(y^{\prime}, u^{\prime}) \in G_{L}^{\varepsilon}$ either. Moreover,  since $x\in S_0$, if we choose sufficiently small $\varepsilon$ so that $\varepsilon < \frac{1}{2}\text{dist}(x, \p S_0)$ WLOG, applying Lemma \ref{analyticity} for forward in time trajectory from $(\sigma(\tau, \varphi), v_{c}(\eta))$, we obtain $u^{\prime} \in (G_{j-1})_{x^{\prime}}$.  
And, by definition \eqref{Q_L^e} of $Q_{L}^{\varepsilon,j}$, we obtain that $(\tau^{\prime}, \varphi^{\prime}, \eta^{\prime}) \in Q_{L}^{\varepsilon, j}$ which means that $Q_{L}^{\varepsilon, j}$ is open. Moreover, $\tilde{F}_j$ is locally analytic on $Q_{L}^{\varepsilon,j}$ by Lemma \ref{lem time s}. \\  

Now, let us use coordinate vectors $\{\hat{x},\hat{y},\hat{\varphi}\}$ where $\hat{x},\hat{y}$ are perpendicular coordinate in $S_0$ and $\hat{\varphi}$ is perpendicular to $S_0$ such as \eqref{coord}. Specifically, since cross section $S_0$ in \eqref{cross section} is on $xz$-plane, $\hat{x}=(1,0,0),\hat{y}=(0,0,1),$ and $\hat{\varphi}=(0,1,0)$. Using this coordinate, we write $\tilde{F}_{j}$ as 
\begin{align} \label{F_j coord}
	\begin{bmatrix}  \tilde{F}_j^{x} \\ \tilde{F}_j^y \\  \tilde{F}_j^{\varphi} \end{bmatrix}
	= \begin{bmatrix} X_1(\tau,\varphi,\eta)-x_1 \\ X_3(\tau,\varphi,\eta)-x_3 \\ X_2(\tau,\varphi,\eta)-x_2\end{bmatrix}=\begin{bmatrix} X_1(\tau,\varphi,\eta)-x_1 \\ X_3(\tau,\varphi,\eta)-x_3 \\ 0 \end{bmatrix},
\end{align}
where $X(s(\sigma(\tau,\varphi),\varphi,v_c(\eta));0, \sigma(\tau,\varphi), v_{c}(\eta))=(X_1(\tau,\varphi,\eta),X_2(\tau,\varphi,\eta),X_3(\tau,\varphi,\eta)),$\\$ x=(x_1,x_2,x_3) \in S_0$. 
By definition of $s(\sigma(\tau,\varphi),\varphi,\eta)$ in \eqref{def time s}, we know that $\tilde{F}_j^\varphi\equiv 0$. Note that we abbreviate notation:
\begin{align} \label{notation abbrev}
	(X,V)(s(\tau,\varphi,\eta);0,\tau,\varphi,\eta):= (X,V)(s(\sigma(\tau,\varphi),\varphi,v_c(\eta));0,\sigma(\tau,\varphi),v_c(\eta)). 
\end{align}
We split things into several cases: \\

\textit{(Case 1)} $X(s(\tau,\varphi,\eta);0,\tau,\varphi,\eta)$ is not a constant vector with respect to $(\tau,\varphi,\eta)\in Q_{L}^{\varepsilon,j}$. \\
Assume that $\tilde{F}_{j}=0$ at some point $(\tau_0,\varphi_0,\eta_0) \in Q_{L}^{\varepsilon,  j}$. Motivated by the $z$-axial symmetry structure, we further split cases:
\begin{align*}
\begin{cases} \textit{(Case 1-a)}\;\left. \frac{\partial \tilde{F}_j}{\partial \varphi}\right \vert_{(\tau_0,\varphi_0,\eta_0)}  \neq 0, \\ \\ \textit{(Case 1-b)}\;\left. \frac{\partial \tilde{F}_j}{\partial \varphi}\right \vert_{(\tau_0,\varphi_0,\eta_0)}  = 0. \end{cases}
\end{align*}
In \textit{(Case 1-a)}, at $(\tau_0,\varphi_0,\eta_0)$, we have either
\begin{align*}
	\left.\frac{\partial \tilde{F}_j^{x}}{\partial \varphi}\right \vert_{(\tau,\varphi,\eta)=(\tau_0,\varphi_0,\eta_0)} \neq 0 \quad \textrm{ or } \quad \left.\frac{\partial \tilde{F}_j^{y}}{\partial \varphi}\right \vert_{(\tau,\varphi,\eta)=(\tau_0,\varphi_0,\eta_0)} \neq 0,
\end{align*}
where $\tilde{F}_j=(\tilde{F}_j^x,\tilde{F}_j^y,\tilde{F}_j^\varphi)$ in \eqref{F_j coord}. WLOG, we assume that $\left. \frac{\partial \tilde{F}_j^{x}}{\partial \varphi}\right \vert_{(\tau_0,\varphi_0,\eta_0)}\neq0$. Applying Lemma \ref{analytic IFT}, there exists an open set $U_{\tau_0,\eta_0}\subset \R^2$ containing $(\tau_0,\eta_0)$ such that there exists an analytic function $\varphi=\varphi(\tau,\eta)$ such that 
\begin{align*}
\varphi(\tau_0,\eta_0)=\varphi_0 \quad \text{and}\quad \tilde{F}_j^x(\tau,\varphi(\tau,\eta),\eta)=0 \; \text{for all }(\tau,\eta)\in \tilde{U}_{\tau_0,\eta_0},
\end{align*}
where $ \tilde{U}_{\tau_0,\eta_0}:=\Big\{(\tau,\eta)\in U_{\tau_0,\eta_0}: (\sigma(\tau,\varphi(\tau,\eta)),v_c(\eta))=(x^j(y,u),v^{j-1}(y,u))\in \gamma_0^C \; \text{for}$\\
$\text{some } (y,u)\in S_0 \times \S^2, u\in (G_{j-1})_y\Big\}$. 
 If $\tilde{F}_j^y(\tau,\varphi(\tau,\eta),\eta)$ is not identically zero on $\tilde{U}_{\tau_0,\eta_0}$, then the zero set $\mathcal{Z}_{x}^{F,j,C}$ of $\tilde{F}_j^{y}$ 
\begin{align} \label{zero set}
	\mathcal{Z}_{x}^{F,j,C}:=\{(\tau,\eta) \in \tilde{U}_{\tau_0,\eta_0} : \tilde{F}_j^{y}(\tau,\varphi(\tau,\eta),\eta)=0\}
\end{align}
has measure zero in $\R^2$ by Lemma \ref{analytic_zeros}. Note that the map
\begin{align*}
	(\tau,\eta)\mapsto V(s(\tau,\varphi(\tau,\eta),\eta);0,\sigma(\tau,\varphi(\tau,\eta)),v_c(\eta))
\end{align*}
is locally analytic on $\tilde{U}_{\tau_0,\eta_0}$ because $\varphi(\tau,\eta)$ is locally analytic on $U_{\tau_0,\eta_0}$ and Lemma \ref{lem time s}. Also, the zero set $\mathcal{Z}_x^{F,j,C}$ of the analytic function $\tilde{F}_j^y$ is compact. Due to the fact that analytic function on compact set is Lipschitz function, we can apply Lemma \ref{Lusin}, and then
\begin{align} \label{zero image}
\left \{v\in \S^2 : v=V(s(\tau,\varphi(\tau,\eta),\eta);0,\sigma(\tau,\varphi(\tau,\eta)),v_c(\eta)) \quad \textrm{for } (\tau,\eta) \in \mathcal{Z}_x^{F,j,C}\right \}
 \end{align}
 has measure zero in $\S^2$. \\
 
 Otherwise, if $\tilde{F}_j^y(\tau,\varphi(\tau,\eta),\eta)$ is identically zero on $\tilde{U}_{\tau_0,\eta_0}$, then $\tilde{F}_j^{x}(\tau,\varphi(\tau,\eta),\eta) = \tilde{F}_j^{y}(\tau,\varphi(\tau,\eta),\eta)$ for all $(\tau,\eta)\in \tilde{U}_{\tau_0,\eta_0}$. By differentiating $\tilde{F}_j$ with respect to $\varphi$, then from $z$-axial symmetry of our domain $\O$, 
 \begin{equation} \label{dvarphi F}
 \begin{split}
 	\left. \frac{\partial \tilde{F}_j}{\partial \varphi} \right \vert _{(\tau,\varphi,\eta)=(\tau_0,\varphi_0,\eta_0)}&= \left. \frac{\partial X(s;0,\tau,\varphi,\eta)}{\partial s} \frac{\partial s}{\partial \varphi} \right \vert_{s=s(\tau_0,\varphi_0,\eta_0), (\tau,\varphi,\eta)=(\tau_0,\varphi_0,\eta_0)} \\
 	& \quad + \left. \frac{\partial X(s;0,\tau,\varphi,\eta)}{\partial \varphi} \right \vert_{s=s(\tau_0,\varphi_0,\eta_0), (\tau,\varphi,\eta)=(\tau_0,\varphi_0,\eta_0)}.
  \end{split}
 \end{equation}
We claim 
 \[
 	\left. \frac{\partial X(s;0,\tau,\varphi,\eta)}{\partial \varphi} \right \vert_{s=s(\tau_0,\varphi_0,\eta_0),(\tau,\varphi,\eta)=(\tau_0,\varphi_0,\eta_0)}= k_{\tau_0, \varphi_0, \eta_0}\hat{\varphi}(0), 
 \]
 for some constant $k_{\tau_0, \varphi_0, \eta_0}$, where $\hat{\varphi}(0)$ in Definition \ref{cross section}. First, let us define the $z$-axial rotation operator 
 \begin{align*}
 	R_{\Delta \varphi} = \begin{pmatrix}
 		\cos\Delta\varphi & -\sin\Delta\varphi & 0 \\
 		 \sin\Delta\varphi & \cos\Delta\varphi & 0 \\
 		 0 & 0 & 1 \\
 	\end{pmatrix},
 \end{align*}
 for some $\Delta \varphi \in \R$.  If we write $(y,u) = R_{\Delta \varphi}(x,v)$, by axial symmetric property of $\O$, we have $R_{\Delta \varphi}\overline{\O} = \overline{\O}$ for any $\Delta\varphi \in \R$. Hence, $R_{-\Delta \varphi} X(s;0, y,u) = X(s; 0, x, v)$ which is equivalent to 
 \Be\label{axisym}
 	X(s; 0, R_{\Delta\varphi}x, R_{\Delta\varphi}v) = R_{\Delta\varphi}X(s; 0, x, v).
 \Ee
 Therefore, for fixed $s\in \R$, using \eqref{axisym} gives 
 \begin{equation} \label{deriv X,1}
 \begin{split}
 	&\left.\frac{\partial X(s;0,\tau,\varphi,\eta)}{\partial \varphi}\right\vert_{(\tau,\varphi,\eta)=(\tau_0,\varphi_0,\eta_0)}\\
	& = \left.\frac{\partial X(s;0,\sigma(\tau,\varphi),v_c(\eta))}{\partial \varphi}\right\vert_{(\tau,\varphi,\eta)=(\tau_0,\varphi_0,\eta_0)}\\
	 &=\lim_{\Delta\varphi\rightarrow 0} \frac{1}{\Delta\varphi}\big( X(s; 0, \sigma(\tau_0, \varphi_0+\Delta\varphi),v_c(\eta_0)) - X(s; 0, \sigma(\tau_0, \varphi_0),v_c(\eta_0)) \big)  \\
 	&= \lim_{\Delta\varphi\rightarrow 0} \frac{1}{\Delta\varphi} (R_{\Delta\varphi}-I) X(s; 0, \sigma(\tau_0, \varphi_0), v_c(\eta_0))  \\
 	&= \begin{pmatrix}
 	0 & -1 & 0 \\
 	1 & 0 & 0 \\
 	0 & 0 & 0 \\
 	\end{pmatrix}
 	X(s; 0, \tau_0, \varphi_0, \eta_0).  \\
 	 \end{split}
 \end{equation}
At $s=s(\tau_0,\varphi_0,\eta_0)=s(\sigma(\tau_0,\varphi_0),\varphi_0,v_c(\eta_0))$ in Lemma \ref{lem time s}, from the calculation above, we have 
 \begin{align} \label{deriv X,2} 
 	\begin{split}
 	\left. \frac{\partial X(s;0,\tau,\varphi,\eta)}{\partial \varphi} \right \vert_{s=s(\tau_0,\varphi_0,\eta_0),(\tau,\varphi,\eta)=(\tau_0,\varphi_0,\eta_0)}&=\begin{pmatrix}
 	0 & -1 & 0 \\
 	1 & 0 & 0 \\
 	0 & 0 & 0 \\
 	\end{pmatrix}
 	X(s(\tau_0,\varphi_0,\eta_0); 0, \tau_0, \varphi_0, \eta_0)\\
	&=
	\begin{pmatrix}
 	0 & -1 & 0 \\
 	1 & 0 & 0 \\
 	0 & 0 & 0 \\
 	\end{pmatrix}
 	\begin{pmatrix}
	X_1(\tau_0,\varphi_0,\eta_0) \\ 0 \\ X_3(\tau_0,\varphi_0,\eta_0)
	\end{pmatrix}\\
	&=\begin{pmatrix}
	0 \\ X_1(\tau_0,\varphi_0,\eta_0) \\ 0 
	\end{pmatrix}\\
	&=k_{\tau_0,\varphi_0,\eta_0}\hat{\varphi}(0),
	\end{split}
 \end{align}
 where $X(s(\tau_0,\varphi_0,\eta_0);0,\tau_0,\varphi_0,\eta_0)=(X_1(\tau_0,\varphi_0,\eta_0),0,X_3(\tau_0,\varphi_0,\eta_0))\in S_0$ and $\hat{\varphi}(0)$ was defined in Definition \ref{cross section}. When we write vectors $V(s(\tau_0,\varphi_0,\eta_0);0,\sigma(\tau_0,\varphi_0),v_c$\\
 $(\eta_0))$ and $\hat{\varphi}(0)$ using $\{\hat{x},\hat{y},\hat{\varphi}\}$ coordinates as 
 \begin{align*}
 	V(s(\tau_0,\varphi_0,\eta_0);0,\sigma(\tau_0,\varphi_0),v_c(\eta_0))=\begin{pmatrix}
	V_1(\tau_0,\varphi_0,\eta_0) \\ V_3(\tau_0,\varphi_0,\eta_0) \\ V_2(\tau_0,\varphi_0,\eta_0)
	\end{pmatrix}, \quad \hat{\varphi}(0)=\begin{pmatrix}
	0 \\ 0 \\ 1
	\end{pmatrix}, 
 \end{align*}
 using \eqref{F_j coord},\eqref{dvarphi F}, and \eqref{deriv X,2} gives that 
 \begin{align*}
 	\left. \frac{\partial \tilde{F}_j}{\partial \varphi} \right \vert _{(\tau,\varphi,\eta)=(\tau_0,\varphi_0,\eta_0)}&=V(s(\tau_0,\varphi_0,\eta_0);0,\sigma(\tau_0,\varphi_0),v_c(\eta_0)) \frac{\partial s}{\partial \varphi}(\tau_0,\varphi_0,\eta_0) + k_{\tau_0,\varphi_0,\eta_0}\hat{\varphi}(0)\\
	&=\begin{pmatrix}
	V_1(\tau_0,\varphi_0,\eta_0) \frac{\partial s}{\partial \varphi}(\tau_0,\varphi_0,\eta_0) \\ V_3(\tau_0,\varphi_0,\eta_0) \frac{\partial s}{\partial \varphi}(\tau_0,\varphi_0,\eta_0) \\ V_2(\tau_0,\varphi_0,\eta_0) \frac{\partial s}{\partial \varphi}(\tau_0,\varphi_0,\eta_0) +k_{\tau_0,\varphi_0,\eta_0}
	\end{pmatrix}\\
	&=\left. \begin{bmatrix}
	 \frac{\partial \tilde{F}_j^x}{\partial \varphi} \\ \frac{\partial \tilde{F}_j^y}{\partial \varphi} \\ 0 
	\end{bmatrix}\right\vert_{(\tau,\varphi,\eta)=(\tau_0,\varphi_0,\eta_0)}.
 \end{align*}
 
Now, let us project the last two lines onto the cross-section $S_{0}$. If  $\tilde{F}_j^{x}(\tau,\varphi(\tau,\eta),\eta) = \tilde{F}_j^{y}(\tau,\varphi(\tau,\eta),\eta)$ hold as our assumption, above equation yields that 
\begin{align*}
	V_1(\tau_0,\varphi_0,\eta_0)=V_3(\tau_0,\varphi_0,\eta_0).
\end{align*}
 However, this is not possible, since we have excluded $R_{sym}^{\varepsilon}$ from $X^{\varepsilon}$ in Lemma \ref{lem Xe}. Therefore, $\tilde{F}_j^y(\tau,\varphi(\tau,\eta),\eta)$ cannot be identically zero. \\


Lastly, we investigate \textit{(Case 1-b)}. We use the calculation $\frac{\partial \tilde{F}_j}{\partial \varphi}$ in \textit{(Case 1-a)}. From conservation of angular momentum and \eqref{ang vel}, we have 
\begin{align*}
	\int_0^{s(\sigma(\tau,\varphi),\varphi,\eta)} \frac{\omega(\mathbf{x}(x,\varphi),v)}{|X(s;0,\mathbf{x}(x,\varphi),v)\times \hat{z}|} ds =\int_0^{s(\sigma(\tau,\varphi),\varphi,\eta)} \frac{\omega(x,v)}{|X(s;0,x,v)\times \hat{z}|}ds= -\varphi.
\end{align*}
By taking $\varphi$-derivative from the equation above, we have
 $\frac{\partial s}{\partial \varphi}<0$ (remind that we are assuming positive orientation, i.e., trajectory with positive angular momentum only) and 
\begin{equation*}
	\left. \frac{\partial \tilde{F}_j}{\partial \varphi} \right \vert _{(\tau_0,\varphi_0,\eta_0)}=V(s(\tau_0,\varphi_0,\eta_0);0,\sigma(\tau_0,\varphi_0),v_c(\eta_0)) \frac{\partial s}{\partial \varphi}(\tau_0,\varphi_0,\eta_0) +k_{\tau_0,\varphi_0,\eta_0}\hat{\varphi}(0) = 0,
\end{equation*}
the vector $V(s(\tau_0,\varphi_0,\eta_0);0,\sigma(\tau_0,\varphi_0),v_c(\eta_0))$ has $\hat{\varphi}$ direction. This is not possible since we have excluded $R_{\hat{\varphi}}^{\varepsilon}$ from $X^{\varepsilon}$ in Lemma \ref{lem Xe}. Hence, \textit{(Case 1-b)} cannot happen.  \\

\textit{(Case 2)} $X(s(\tau,\varphi,\eta);0,\tau,\varphi,\eta)$ is a constant vector for all $(\tau, \varphi, \eta)\in Q_{L}^{\varepsilon,  j}$ where we defined the set $Q_L^{\varepsilon,j}$ in \eqref{Q_L^e}.  \\
In this case, it holds that 
\begin{align*}
		\frac{\partial \tilde{F}_j}{\partial \varphi} = \frac{\partial}{\partial \varphi}( X(s(\tau,\varphi,\eta);0,\tau,\varphi,\eta))=0, \quad \forall (\tau,\varphi,\eta) \in Q_L^{\varepsilon,j}.
\end{align*}
From the same argument in \textit{(Case 1-b)}, this case cannot happen. \\

Now we use compactness. Since $\tilde{F}_{j}$ is continuous in its domain $Q_{L}^{\varepsilon, j}$ by Lemma \ref{lem time s}, the zero set $\mathcal{T}_{x}^{F,j,C}$ of the function $\tilde{F}_j$ 
\begin{align} \label{zero set T}
\mathcal{T}_{x}^{F,j,C} := \{ (\tau, \varphi, \eta) : \tilde{F}_{j} (\tau,\varphi,\eta)= 0 \}
\end{align}
is compact. For each point $(\tau,\varphi,\eta)$ in the set $\mathcal{T}_{x}^{F,j,C}$, we construct an open set $U_{\tau,\eta}$ from Lemma \ref{analytic IFT} and zero set such as \eqref{zero set} by repeating above argument. If we consider projection $\Pi_{(\tau, \eta)}$ onto $\tau\times \eta$ plane, $\Pi_{(\tau, \eta)}\mathcal{T}_{x}^{F,j,C}:=\{(\tau,\eta):(\tau,\varphi,\eta)\in \mathcal{T}_x^{F,j,C}\}$ is still compact. So, we have finite subcover for $\Pi_{(\tau, \eta)}\mathcal{T}_{x}^{F,j,C}$, and hence image under $V(s)$ map as in \eqref{zero image} is still zero.  \\

\noindent{\bf(Boundary $\p S_0$)} Now, we consider the boundary $\partial S_0$. Let $x \in \partial S_0$. Using the map $(\xf^j,\vf^{j-1})$, we have 
\begin{equation*}
	(x,v)=(\xf^j,\vf^{j-1})(\tau,\varphi,\eta),
\end{equation*}
where $\xf^j(\tau,\varphi,\eta):=X(\tf^j(\sigma(\tau,\varphi),v^c(\eta));0,\sigma(\tau,\varphi),v^c(\eta)), \vf^{j-1}(\tau,\varphi,\eta)$
$:=V(\tf^j(\sigma(\tau,\varphi),$\\
$v^c(\eta));0,\sigma(\tau,\varphi),v^c(\eta))$. Define 
\Be \label{B_j,C}
	\tilde{B}^{j,C}:=\{(\tau,\varphi,\eta)\in [a,b)\times (-2\pi \mathbf{w},0]\times (0,1):(\xf^i,\vf^{i-1})(\tau,\varphi,\eta)\notin \gamma_0^C \cup \gamma_0^I, \quad \forall 1\leq i \leq j\},
\Ee
where $\mathbf{w}$ be the winding number in \eqref{winding number}. Similar to $Q_{L}^{\varepsilon,j}$, we should check that $\tilde{B}^{j,C}$ is open for analyticity of $\xf^j$. Note that $\sigma$ and $v_{c}$ are all locally smooth. Now, let us consider forward in time trajectory from $(\sigma(\tau,\varphi),v_c(\eta))$ for $(\tau,\varphi,\eta)\in \tilde{B}^{j,C}$. By definition of $\tilde{B}^{j,C}$, 
\begin{align*}
	(\xf^i,\vf^{i-1})(\tau,\varphi,\eta) \notin \gamma_0^C \cup \gamma_0^I, \quad \forall 1\leq i \leq j. 
\end{align*}
Hence, by (a) of Lemma \ref{analyticity}, for sufficiently small $\varepsilon>0$, there exists $\delta \ll 1$ such that if $\vert (\tau,\varphi,\eta)-(\tau^{\prime},\varphi^{\prime},\eta^{\prime}) \vert <\delta$, then $\xf^i$ and $\vf^{i-1}$ are analytic at $(\tau^{\prime},\varphi^{\prime},\eta^{\prime})$ for all $1\leq i \leq j$. Moreover, 
\begin{align*}
	(\xf^i,\vf^{i-1})(\tau^{\prime},\varphi^{\prime},\eta^{\prime}) \notin \gamma_0^I \cup \gamma_0^C, \quad \forall 1\leq i \leq j, 
\end{align*}
which implies $(\tau^{\prime},\varphi^{\prime},\eta^{\prime}) \in \tilde{B}^{j,C}$. 
We consider 
\begin{align}\label{def tilde H}
	\tilde{H}_j(\tau,\varphi,\eta):=\xf^j(\tau,\varphi,\eta)-x, \quad (\tau,\varphi,\eta)\in \tilde{B}^{j,C},
\end{align}
which is locally analytic due to (a) in Lemma \ref{analyticity}. Similar to the interior $S_0$ case, we split:

\textit{(Case 1)} $\xf^j(\tau,\varphi,\eta)$ is a constant point in $\p\O$ for all $(\tau,\varphi,\eta)\in \tilde{B}^{j,C}$.\\
If we put $\xf^j$ into the calculation \eqref{deriv X,1} and \eqref{deriv X,2} instead of $X(s;0,\tau,\varphi,\eta)$, then
\begin{align}\label{deriv xf} 
\frac{\partial \xf^j}{\partial \varphi}=k_{\tau,\varphi,\eta}\hat {\varphi}(0),
\end{align}
for some constant $k_{\tau,\varphi,\eta}\in \R$, where we defined $\hat\varphi(0)$ in Definition \ref{varphi cross section}. To ensure that $\xf^j(\tau,\varphi,\eta)$ is a constant point for all $(\tau,\varphi,\eta)\in \tilde{B}^{j,C}$, $\frac{\partial \xf^j}{\partial \varphi}$ must be a zero vector for all $(\tau,\varphi,\eta)\in \tilde{B}^{j,C}$. Thus, this case does not happen. \\

\textit{(Case 2)} $\xf^j(\tau,\varphi,\eta)$ is not a constant point in $\p\O$ for all $(\tau,\varphi,\eta)\in \tilde{B}^{j,C}$. \\
Recall the function $\tilde{H}_j$ defined in \eqref{def tilde H}. We assume that $\tilde{H}_j=0$ at some point $(\tau_1,\varphi_1,\eta_1)\in \tilde{B}^{j,C}$. Through the coordinate $\hat{x}=(1,0,0),\hat{y}=(0,0,1)$, and $\hat{\varphi}=(0,1,0)$ such as \eqref{coord} and the calculation \eqref{deriv xf} above, we have 
\begin{align}\label{deriv tilde H}
	\tilde{H}_j(\tau,\varphi,\eta)=\begin{bmatrix} \tilde{H}_j^{x} \\ \tilde{H}_j^y\\ \tilde{H}_j^{\varphi}  \end{bmatrix}=
	\begin{bmatrix}
	{\xf^j}_{,1}(\tau,\varphi,\eta)-x_1 \\ {\xf^{j}}_{,3}(\tau,\varphi,\eta)-x_3 \\ {\xf^j}_{,2}(\tau,\varphi,\eta)-x_2
	\end{bmatrix}
	 \quad \textrm{ and } \quad 
	\frac{\partial \tilde{H}_j}{\partial \varphi}=\begin{bmatrix}   \frac{ \partial \tilde{H}_j^{x}}{\partial \varphi} \\ \frac{\partial \tilde{H}_j^y}{\partial \varphi} \\ \frac{\partial \tilde{H}_j^{\varphi}}{\partial \varphi}\end{bmatrix}
	= \frac{\partial \xf^j}{\partial \varphi}=k_{\tau,\varphi,\eta}\begin{bmatrix} 0 \\ 0 \\ 1 \end{bmatrix},
\end{align}
where $\xf^j(\tau,\varphi,\eta)=({\xf^j}_{,1}(\tau,\varphi,\eta),{\xf^j}_{,2}(\tau,\varphi,\eta),{\xf^j}_{,3}(\tau,\varphi,\eta))$ and $x=(x_1,x_2,x_3)$. Hence, $\left. \frac{\partial \tilde{H}_j^{\varphi}}{\partial \varphi}\right \vert _{(\tau_1,\varphi_1,\eta_1)}\neq 0$. Using Lemma \ref{analytic IFT}, there exists an open set $V_{\tau_1,\eta_1}\subset \R^2$ containing $(\tau_1,\eta_1)$ such that there exists an analytic function $\varphi=\varphi(\tau,\eta)$ such that 
\begin{align} \label{implicit result}
\varphi(\tau_1,\eta_1)=\varphi_1 \quad \text{ and } \quad \tilde{H}_j^{\varphi}(\tau,\varphi(\tau,\eta),\eta)=0 \text{ for all } (\tau,\eta)\in \tilde{V}_{\tau_1,\eta_1},
\end{align}
where $\tilde{V}_{\tau_1,\eta_1}:=\Big \{ (\tau,\eta)\in V_{\tau_1,\eta_1}: (\xf^i,\vf^{i-1})(\tau,\varphi(\tau,\eta),\eta)\notin \gamma_0^C\cup \gamma_0^I, \quad \forall 1\leq i \leq j\Big\}$. If $\tilde{H}_j^x$ or $\tilde{H}_j^y$ is identically zero on $\tilde{V}_{\tau_1,\eta_1}$, then $\tilde{H}_j^\varphi(\tau,\varphi(\tau,\eta),\eta)=\tilde{H}_j^x(\tau,\varphi(\tau,\eta),\eta)$ or $\tilde{H}_j^\varphi(\tau,\varphi(\tau,\eta),\eta)=\tilde{H}_j^y(\tau,\varphi(\tau,\eta),\eta)$ for all $(\tau,\eta)\in \tilde{V}_{\tau_1,\eta_1}$ due to \eqref{implicit result}. Since  $\frac{\partial \tilde{H}_j^{\varphi}}{\partial \varphi}\neq  \frac{\partial \tilde{H}_j^{x}}{\partial \varphi}$ and $ \frac{\partial \tilde{H}_j^{\varphi}}{\partial \varphi}\neq  \frac{\partial \tilde{H}_j^{y}}{\partial \varphi}$ in \eqref{deriv tilde H}, $\tilde{H}_j^x$ and $\tilde{H}_j^y$ cannot be identically zero. Thus, the zero set $\mathcal{Z}_{x}^{H,j,C}$ of $\tilde{H}_j$
\begin{align} \label{zero set 2}
	\mathcal{Z}_{x}^{H,j,C}:=\{(\tau,\eta)\in \tilde{V}_{\tau_1,\eta_1} : \tilde{H}_j^x(\tau,\varphi(\tau,\eta),\eta)=\tilde{H}_j^y(\tau,\varphi(\tau,\eta),\eta)=0\}
\end{align}
has measure zero in $\R^2$ by Lemma \ref{analytic_zeros}. Applying Lemma \ref{Lusin} yields that 
\Be \label{vf map}
	\{v\in \S^2: v=\vf^{j-1}(\tau,\varphi(\tau,\eta),\eta) \quad \textrm{for } (\tau,\eta)\in \mathcal{Z}_x^{H,j,C}\}
\Ee
has measure zero in $\S^2$. Similar to the zero set $\mathcal{T}_x^{F,j,C}$ in \eqref{zero set T}, we construct the zero set ${T}_x^{H,j,C}$ of $\tilde{H}_j$ as:
\begin{align*}
{T}_{x}^{H,j,C} := \{ (\tau, \varphi, \eta) : \tilde{H}_{j} (\tau,\varphi,\eta)= 0 \}.
\end{align*}
The zero set ${T}_x^{H,j,C}$ is also compact. For each point $(\tau,\varphi,\eta)\in {T}_x^{H,j,C}$, we construct an open set $V_{\tau,\eta}$ and zero set such as \eqref{zero set 2} by repeating above argument. If we consider projection $\Pi_{(\tau, \eta)}$ onto $\tau\times \eta$ plane, $\Pi_{(\tau, \eta)}{T}_{x}^{H,j,C}:=\{(\tau,\eta):(\tau,\varphi,\eta)\in {T}_x^{H,j,C}\}$ is still compact. So, we have finite subcover for $\Pi_{(\tau, \eta)}{T}_{x}^{H,j,C}$, and hence image under $\vf^{j-1}$ map as in \eqref{vf map} is still zero.

{\bf Step 2} In this step, we deal with inflection grazing $\gamma_0^I$ in $(B_j)_x$. We claim 
\[
\mathfrak{m}_{2} \{ v \in (G_{j-1})_{x} : (x^{j}(x,v), v^{j-1}(x,v))\in \gamma_0^{I} \} = 0, \\
\] 		
where $(G_{j-1})_x$ was constructed inductively in \eqref{def GB} and \eqref{def GB2}. Since $v\in (G_{j-1})_{x}$, trajectory does not graze between its first bounce and $(j-1)$-th bounce. \\

\noindent{\bf (Interior case $S_0$)} We firstly consider $x\in S_0$. Since $(G_{j-1})_x\subset (X^\varepsilon)_x$ and inflection grazing does not happen in trajectory between $(x,v)$ and $(x^j(x,v),v^{j-1}(x,v))$, the map $(X,V)(s(\cdot))$ is invertible by Lemma \ref{no infty bounce}. Using invertible map $(X , V )(s(\cdot))$, 
\begin{align*}
\begin{split}
	(x,v) = (X , V)(s(\sigma(\tau, \varphi),\varphi,I^2(\tau,\varphi)); 0, \sigma(\tau, \varphi), I^{2}(\tau, \varphi)),
\end{split}
\end{align*}
where $s(\cdot,\cdot,\cdot)$ be the arrival time defined in \eqref{def time s}, $\sigma(\cdot,\cdot)$ be the parametrization of boundary $\p\O$ in \eqref{sigma para}, and $I^2(\cdot,\cdot)$ be parametrization for inflection grazing in Lemma \ref{inflection}. 
Let us define 
{\footnotesize
\Be \label{P_L^e}
\begin{split}
	P_{L}^{\varepsilon, j}:= \Big\{ (\tau, \varphi) : (\sigma(\tau, \varphi), I^{2}(\tau, \varphi)) = (x^{j}(y,u), v^{j-1}(y,u))\in \gamma_0^{I} \ \text{for some} \ (y,u)\in X^{\varepsilon}, \ u\in (G_{j-1})_{y} \Big\}.  \\
\end{split}
\Ee
}
Note that $P_{L}^{\varepsilon, j}$ is also disjoint for all $1\leq  j \leq K$. By applying (b) of Lemma \ref{analyticity} to forward in time trajectory and Lemma \ref{lem time s}, We consider
\begin{equation} \label{def F_j}
F_j(\tau,\varphi):=X(s(\sigma(\tau,\varphi),\varphi,I^2(\tau,\varphi));0,\sigma(\tau,\varphi),I^{2}(\tau,\varphi)) - x , \quad (\tau, \varphi)\in P_{L}^{\varepsilon, j}. \\
\end{equation}

In \textbf{Step 1}, we showed that $Q_L^{\varepsilon,j}$ is open. Similarly, we can prove that the set $P_L^{\varepsilon,j}$ is open to obtain analyticity of the function $F_j$. Note that $\sigma$ in \eqref{sigma para} and $I^2$ in Lemma \ref{inflection} are all locally smooth. Now, let us consider forward in time trajectory from $(\sigma(\tau, \varphi), I^2(\tau,\varphi)) = (x^{j}(y,u), v^{j-1}(y,u))$. By definition of $(G_{j-1})_{y}$ and Lemma \ref{trajectory invertible}, the trajectory is non-grazing until it arrives at $(y,u)$, and hence the map
\[	
	(\tau, \varphi) \mapsto (\sigma(\tau, \varphi), I^2(\tau,\varphi))  \mapsto (X,V)(s(\sigma(\tau, \varphi),\varphi,
	I^2(\tau,\varphi)); 0, \sigma(\tau, \varphi), I^2(\tau,\varphi)) \in S_0 \times \S^{2}
\]
is continuous with finite number of bounce by Lemma \ref{lem time s}. Now, for sufficiently small $\varepsilon > 0$, there exists $\delta  \ll 1$ such that if $| (\tau, \varphi) - (\tau^{\prime}, \varphi^{\prime}) | < \delta$, then
\[
	|(y,u) - (y^{\prime}, u^{\prime})| < \varepsilon,
\] 
where 
\[
 (y^{\prime}, u^{\prime}) := (X,V)(s(\sigma(\tau^{\prime}, \varphi^{\prime}),\varphi^{\prime}, I^2(\tau^{\prime},\varphi^{\prime})); 0, \sigma(\tau^{\prime}, \varphi^{\prime}), I^2(\tau^{\prime},\varphi^{\prime})) \in S_0\times \S^{2},
\]
or equivalently,
\[
	(\sigma(\tau^{\prime}, \varphi^{\prime}), I^2(\tau^{\prime},\varphi^{\prime})) = (x^{j}(y^{\prime}, u^{\prime}), v^{j-1}(y^{\prime}, u^{\prime})) \in \gamma_0^{I}.
\] 
Proving open property of $P_{L}^{\varepsilon, j}$ and analyticity of $F_{j}$ is quite similar to what we did in {\bf Step 1}. We know that $(y,u)\in G_{L}^{\varepsilon}$ and $G_{L}^{\varepsilon}$ is open by (a) in Lemma \ref{lem:open}. Thus, for $\varepsilon$ sufficiently small, $(y^{\prime}, u^{\prime}) \in G_{L}^{\varepsilon}$. Using the fact that $x$ is not on the boundary, choosing sufficiently small $\varepsilon$ and applying Lemma \ref{analyticity} for forward in time trajectory from $(\sigma(\tau, \varphi), v_{c}(\eta))$, yields $u^{\prime} \in (G_{j-1})_{x^{\prime}}$.  
And, by definition \eqref{P_L^e} of $P_{L}^{\varepsilon,j}$, we obtain that $(\tau^{\prime}, \varphi^{\prime}, \eta^{\prime}) \in P_{L}^{\varepsilon, j}$ which means that $P_{L}^{\varepsilon, j}$ is open. Moreover, ${F}_j$ is locally analytic on $P_{L}^{\varepsilon, j}$ by Lemma \ref{lem time s}. \\  

Repeat \eqref{dvarphi F} of \textbf{Step 1} to claim the case of $F_j$ being a constant vector in $P_L^{\varepsilon,j}$ cannot happen by excluding $R_{sym}^{\varepsilon}$ in $X^{\varepsilon}$. Then, the zero set $\mathcal{Z}_x^{F,j,I}$ of the analytic function $F_j$
\begin{align*}
	\mathcal{Z}_{x}^{F,j,I}:=\{(\tau,\varphi)\in P_L^{\varepsilon,j}:F_j(\tau,\varphi)=0\}
\end{align*}
has measure zero in $\R^2$ due to Lemma \ref{analytic_zeros}. We can use Lemma \ref{Lusin} directly since $\mathcal{Z}_x^{F,j,I}$ has the same dimension with $\S^2$. Thus, from Lemma \ref{Lusin}, we get 
\begin{align*}
	\{v\in \S^2: v=V(s(\tau,\varphi);0,\sigma(\tau,\varphi),I^2(\tau,\varphi)) \; \textrm{for } (\tau,\varphi) \in \mathcal{Z}_x^{F,j,I}\}
\end{align*}
has measure zero in $\S^2$. \\ 

\noindent{(\bf Boundary case $\p S_0$)} Next, let $x\in \partial S_0$. Using the map $(\xf^j,\vf^{j-1})$, 
\begin{align*}
	(x,v)=(\xf^j,\vf^{j-1})(\tau,\varphi),
\end{align*}
where $\xf^j(\tau,\varphi):=X(\tf^j(\sigma(\tau,\varphi),I^2(\tau,\varphi));0,\sigma(\tau,\varphi),I^2(\tau,\varphi)), \vf^{j-1}(\tau,\varphi):=V(\tf^j(\sigma(\tau,$\\
$\varphi),I^2(\tau,\varphi));0,\sigma(\tau,\varphi),I^2(\tau,\varphi))$. Define 
\begin{align*}
	\tilde{B}^{j,I}:=\{(\tau,\varphi)\in [a,b)\times (-2\pi \mathbf{w},0]: (\xf^i,\vf^{i-1})(\tau,\varphi)\notin \gamma_0^C \cup \gamma_0^I, \quad \forall 1\leq i \leq j\},
\end{align*}
where $\mathbf{w}$ be the winding number in \eqref{winding number}. By the same reason for openness of $\tilde{B}^{j,C}$ defined in \eqref{B_j,C}, the set $\tilde{B}^{j,I}$ is also open. From Lemma \ref{analyticity}, the function 
\begin{align*}
	H_j(\tau,\varphi):=\xf^j(\tau,\varphi)-x, \quad (\tau,\varphi) \in \tilde{B}^{j,I},
\end{align*}
is locally analytic on $\tilde{B}^{j,I}$. Using the same arguments in \textbf{Step 1} and Lemma \ref{analytic_zeros}, $H_j$ cannot be a constant vector and the zero set $\mathcal{Z}_{x}^{H,j,I}$ of $H_j$ has measure zero in $\R^2$. Then, the image $\vf^{j-1}(\mathcal{Z}_{x}^{H,j,I})$ has measure zero in $\S^2$ by applying Lemma \ref{Lusin} directly. 

\hide
Therefore, the zero set of $F_{j}$ is of measure zero in $\R^d$[08,ARMA], i.e., for the zero set $Z_{x,j}^I$ of $F_j$ 
\begin{align*}
Z_{x,j}^I:= \{(\theta,\varphi):F_{j}(\theta,\varphi)=0\}
\end{align*} 
has measure zero in $\R^2$. Considering $V_{\mathbf{f}}(s(\cdot); 0, \cdot)$ mapping from $Z_{x,j}^I$ to $\S^2$ and definition of $(*)$,
\begin{align*}
(*) = \Big\{ V_{\mathbf{f}}(s(\theta,\varphi):0,\sigma(\theta,\varphi),I^{2}(\theta,\varphi))  : (\theta, \varphi) \in Z_{x,j}^{I} \Big\}
\end{align*}
Applying Lemma \ref{Lusin} to the map $V(\theta,\varphi)$, we finish Step 2.  
\unhide

\end{proof}
\hide
{\color{red} (Remove?)
Let $x \in cl(\Omega) $. Let us fix $x\in \overline{S_{0}^{\delta}}$ and travel length $L>0$. We construct the following sets,  
\Be \label{def X delta}
X^{\varepsilon, \delta} := \Big\{ (x,v) \in \overline{S_{0}^{\delta}}\times \mathbb{S}^{2} : v \in \mathbb{S}^{2}\backslash \{ R_{x,\frac{\pi}{2}}^{\varepsilon} \cup R_{x,\pi}^{\varepsilon} \cup R_{x,\frac{3\pi}{2}}^{\varepsilon} \cup R_{x, \perp}^{\varepsilon} \cup R_{x, \hat{\varphi}}^{\varepsilon} \cup R_{x, sym}^{\varepsilon} \} \Big\} \\
\Ee
{\color{red}(NOte: Since we exclude near boundary points, we don't have to consider nB set.)} Similar as \eqref{BG}, we also define
\Be \label{BG delta}
\begin{split}
	B^{\varepsilon, \delta}_{L} &:= \left \{ (x,v)\in X^{\varepsilon, \delta} : \exists k \in \mathbb{N} : (x^{k}, v^{k-1}(x,v))\in \gamma_{0}^{I-}, \ \text{and} \  \sum_{j=1}^k \vert \xb^j(x,v)-\xb^{j-1}(x,v) \vert \leq L \right \}, \\
	G^{\varepsilon, \delta}_{L} &:= \left \{ (x,v)\in X^{\varepsilon, \delta} : \nexists k \in \mathbb{N} : (x^{k}, v^{k-1}(x,v))\in \gamma_{0}^{I-}, \ \text{and} \  \sum_{j=1}^k \vert \xb^j(x,v)-\xb^{j-1}(x,v) \vert \leq L \right \}  \\
\end{split}	
\Ee
From the above definitions \eqref{def X delta} and \eqref{BG delta}, we have 
\[
X^{\varepsilon, \delta}\subset X^{\varepsilon}, \quad X^{\varepsilon, \delta} = B_{L}^{\varepsilon, \delta}\cupdot G_{L}^{\varepsilon, \delta}, \quad B_{L}^{\varepsilon, \delta} \subset B_{L}^{\varepsilon},\quad G_{L}^{\varepsilon, \delta} \subset G_{L}^{\varepsilon} 
\]
Therefore, from Proposition \ref{unif bounce}, we also obtain
\Be \label{unif bounce delta}
\sup_{(x,v)\in B_{L}^{\varepsilon, \delta}}N(x,v,L) < C_{\varepsilon, L},\quad \forall \delta>0.
\Ee
} 

\noindent {\bf 5-2. Construction of $\mathcal{IB}$ set } \\
\unhide
\begin{definition}
From Lemma \ref{bounce K}, we define $\mathcal{BM}_{x}$ similar to $\mathcal{BK}_{x}$ replacing $K$ into $M$ in \eqref{def Ki},
	\begin{align} \label{def IBi}
	\mathcal{BM}_x := \bigcup _{j=1}^{M} (B_j)_x \subset \S^{2}.
	\end{align}
	Similar to \eqref{Ki epsilon}, by Lemma \ref{mB zero}, there exists an open set $\mathcal{BM}_x^{\varepsilon}\subset \S^2$ such that 
	\begin{align} \label{IB epsilon}
		\mathcal{BM}_x \subset \mathcal{BM}_x^{\varepsilon}, \quad \mathfrak{m}_2\left( \mathcal{BM}_x^{\varepsilon}\right) <\varepsilon. 
	\end{align} 
\end{definition}
		
	\begin{lemma} \label{S2 version}($\overline{S_0}\times\S^2$ version) Let $S_0$ be $\varphi=0$ cross-section in \eqref{cross section} and $\varepsilon>0$. Recall the definition \eqref{Kc} of $\mathcal{K}^c$. There exists a compact set $\mathcal{Y}^{\varepsilon}\subset \overline{S_0}\times \S^2$ such that  
		\Be \label{IB number}
			\sup_{(x,v)\in \mathcal{Y}^{\e} } \mathcal{N}(x,v,L) \leq M, 
		\Ee
		where we defined $\mathcal{N}(x,v,L)$ in \eqref{nongrazing NB}, the constant $M$ in Lemma \ref{bounce K}
		 Moreover, if $(x,v)\in \mathcal{Y}^\e$, then 
		 
		
		\Be \label{nongrazing}
		[x^{j}(x,v), v^{j-1}(x,v)] \notin \gamma_{0}^{I}\cup \gamma_0^{C} ,\quad \forall 1\leq j \leq M,
		\Ee
		and 
		\Be \label{unif nongrazing}
	\inf_{1\leq i \leq M}|v^{i-1}(x,v)\cdot n(x^{i}(x,v))| \geq C^*_{\varepsilon, L} > 0,
		\Ee
	where the constant $C_{\varepsilon,L}^{*}$ does not depend on $x$. In addition, for $x\in \overline{S_0}$, 
	\Be \label{IB x small}
		\mathfrak{m}_2((\{\overline{S_0}\times \S^2\} \backslash \mathcal{Y}^\e)_x) \lesssim \varepsilon,
	\Ee 
	where we used the notation $A_x$ in \eqref{proj_v}. 

	\end{lemma}
	\begin{proof}
		Consider the following set  
	\begin{align} \label{F}
		F=\Big\{ (x,v) \in \mathcal{K}^c : v \in  \mathcal{BM}_x   \Big\},
	\end{align}
	where $\mathcal{K}^c$ in \eqref{Kc} is compact and $\mathcal{BM}_x$ is defined in \eqref{def IBi}. Recall that the projection operator $\mathcal{P}_1$ in \eqref{proj operator}. Note that $\mathcal{P}_1(\mathcal{K}^c)=\overline{S_0}$ by the same argument to deduce $\mathcal{P}_1((\mathcal{NB})^c)=\overline{S_0}$ in proof of Lemma \ref{lem Xe}. For each $x\in \mathcal{P}_1(\mathcal{K}^c)=\overline{S_0}$, by Lemma \ref{bounce K} and definition of $\mathcal{BM}_x $ in \eqref{def IBi},  
\begin{align*}
		(x^i(x,v),v^{i-1}(x,v))\notin \gamma_0^I\cup \gamma_0^C, \quad \text{up to whole travel length $L$},
\end{align*}
for $v \notin \mathcal{BM}_x$. Notice that the trajectory with a finite number of bounces is reversible and deterministic. If convex grazing occurs during whole travel length $L$,  the forward in time trajectory cannot be defined. Hence, there is no convex grazing during travel length $L$. Thus, there exists a positive constant $C_{x,\varepsilon,L}^{*}>0$ depending on $x,\varepsilon,L$ satisfying
\begin{align}\label{vn lb}
	\vert v^{i-1}(x,v)\cdot n(x^{i}(x,v))| \geq C^*_{x,\varepsilon, L}>0, \quad \text{up to whole travel length $L$}.
\end{align}
By compactness of $\mathcal{P}_1(\mathcal{K}^c)=\overline{S_0}$, Lemma \ref{continuity 2}, Lemma \ref{mB zero} and \eqref{IB epsilon}, we have an open cover for $\mathcal{P}_1(\mathcal{K}^c)=\overline{S_0}$ 
	 \Be \label{open cover_IB}
		\Big\{B(x_i^{\mathcal{IB}},r_i^{\mathcal{IB}})\cap \overline{S_0}\Big \}_{i=1}^{l_{IB}},
	\Ee
	and corresponding each open set $\mathcal{O}_i^{\mathcal{IB}} \subset \S^{2}$ which is defined by 
	\begin{align} \label{def bad IB}
		 \mathcal{O}_i^{\mathcal{IB}}:=\mathcal{BM}_{x_{i}^{\mathcal{IB}}}^{\varepsilon}, \quad  1\leq i \leq l_{IB},
	\end{align}
	satisfying $\mathfrak{m}_2(\mathcal{O}_i^{\mathcal{IB}})< \varepsilon$ for all $1\leq i \leq l_{IB}$ and  
	\begin{align} \label{O_i^IB prop}
		\mathcal{BM}_x \subset \mathcal{O}_i^{\mathcal{IB}}, 
	\end{align}
	for all $x \in B(x_i^{\mathcal{IB}},r_i^{\mathcal{IB}})$ and sufficiently small $r_{i}^{\mathcal{IB}}$. 
	 Similar to \eqref{index_K} in Lemma \ref{construct K}, for given $x\in \mathcal{P}_1(\mathcal{K}^c)=\overline{S_0}$, we define an index set 
\begin{align} \label{index_IB}
	I_x^{\mathcal{IB}}:=\Big\{i\in\mathbb{N}: x\in B(x_i^{\mathcal{IB}},r_i^{\mathcal{iB}})\cap \overline{S_0}\Big\}. 
\end{align}
Now, we consider the following set 
\begin{align}\label{bad phase set 4} 
	\bigcup_{x\in \overline{S_0}} \Big \{(x,v)\in \mathcal{K}^c: v \in  \bigcap_{i\in I_x^{\mathcal{IB}}}\mathcal{O}_i^{\mathcal{IB}}\Big \},
\end{align} 
which contains set $F$ in \eqref{F}. Next, we construct an open cover for \eqref{bad phase set 4} as
\begin{align} \label{IB}
	\mathcal{IB}:=\bigcup_{i=1}^{l_{IB}} \Big \{ (x,v)\in \mathcal{K}^c: x \in B(x_i^{\mathcal{IB}},r_i^{\mathcal{IB}})\cap \overline{S_0},\; v \in \bigcap_{j\in I_x^{\mathcal{IB}}}\mathcal{O}_j^{\mathcal{IB}}\Big\}_{\varepsilon},
\end{align}
where we have used the notation $A_{\varepsilon}$ in \eqref{extension domain}. In this paper, we define 
\begin{align} \label{IBc}
	\mathcal{Y}^\e:= \mathcal{K}^c \backslash \mathcal{IB}, 
\end{align}
which is compact since $\mathcal{K}^c$ is compact and $\mathcal{IB}$ is open. 
\hide
From Lemma \ref{bounce K} and definition of $\mathcal{BM}_x$ in \eqref{def IBi}, 
		\begin{align*}
			(x^i(x,v),v^{i-1}(x,v))\notin \gamma_0^I\cup \gamma_0^C, \quad \text{up to whole travel length $L$}.
		\end{align*}
		for $(x,v) \in \mathcal{IB}^c$ in \eqref{IBc}. Notice that the trajectory with a finite number of bounce is reversible and deterministic. If convex grazing occurs during whole travel length $L$,  the forward in time trajectory cannot be defined. Hence, there is no convex grazing during travel length $L$. Thus, there exist some positive constant $C_{x,\varepsilon,L}^{*}>0$ depending on $x,\varepsilon,L$ satisfying
\begin{align*}
	\vert v^{i-1}(x,v)\cdot \mathbf{n}(x^{i}(x,v))| \geq C^*_{x,\varepsilon, L}>0, \quad \text{up to whole travel length $L$}
\end{align*}
for $(x,v)\in \mathcal{IB}^c$.  By compactness, there exist a relatively open cover $\bigcup_{i=1}^{l_{\mathcal{IB}}} B(x_i,r(x_i,\varepsilon))$ for $\overline{S_0}$ and corresponding $\mathcal{O}_i^{\mathcal{IB}}$ defined in \eqref{def bad IB}.  
 Using Lemma \ref{continuity 2}, we have 
		\begin{align*}
			\mathcal{O}_x^{\mathcal{IB}} \subset \mathcal{O}_i^{\mathcal{IB}}.
		\end{align*}
		 In \eqref{IB epsilon}, we have $\mathfrak{m}_2(\mathcal{IB}_{x_i})_{\varepsilon} < \varepsilon$, which implies that $\mathfrak{m}_2(\mathcal{O}_i^{\mathcal{IB}}) \lesssim \varepsilon$. 
\unhide	
		By definition \eqref{IBc} of the set $\mathcal{Y}^\e$, we directly deduce the relation $\mathcal{Y}^\e\subset \mathcal{K}^c$. Hence, from \eqref{bounce K}, we obtain \eqref{IB number}. Assume that $(x,v) \in \mathcal{Y}^\e$. By definition \eqref{IB} of $\mathcal{IB}$ and \eqref{O_i^IB prop}, $v\notin \mathcal{BM}_x= \bigcup_{j=1}^M (B_j)_x$. This implies that 
\begin{align*}
	(x^j(x,v),v^{j-1}(x,v)) \notin \gamma_0^I \cup \gamma_0^C, \quad \forall 1\leq j \leq M. 
\end{align*}
Also, \eqref{unif nongrazing} can be obtained if we choose $C^*_{\varepsilon,L}:=\min_{1\leq i \leq l_{IB}} C_{x_i^{\mathcal{IB}},\varepsilon,L}^*$ in \eqref{vn lb}.  For fixed $x \in \overline{S_0}$, 
\begin{align} \label{def IBx}
	\mathcal{IB}_x=\Big\{v \in (\mathcal{K}^c)_x: v \in \bigcap_{j\in I_x^{\mathcal{IB}}}\mathcal{O}_j^{\mathcal{IB}}\Big\}_{\varepsilon},
\end{align}
where the notation $A_x$ is defined in \eqref{proj_v} and the notation $A_\varepsilon$ in \eqref{extension domain} comes from the definition \eqref{K}. Since 
\begin{align*}
	\{\overline{S_0} \times \S^2\} \backslash \mathcal{Y}^\e=(\{\overline{S_0}\times \S^2\}\backslash\mathcal{K}^c) \cupdot \mathcal{IB},
\end{align*}
we obtain
\begin{align*}
	\begin{split}
	\mathfrak{m}_2((\{\overline{S_0}\times \S^2\}\backslash \mathcal{Y}^\e)_x) &\leq \mathfrak{m}_2((\{\overline{S_0}\times \S^2\}\backslash \mathcal{K}^c)_x)+\mathfrak{m}_2(\mathcal{IB}_x)\\
	&\leq \mathfrak{m}_2((\{\overline{S_0}\times \S^2\}\backslash \mathcal{K}^c)_x)+\mathfrak{m}_2((\mathcal{O}_j^{\mathcal{IB}})_\e)\\
	&\lesssim\varepsilon, \quad \forall j \in I_x^{\mathcal{IB}}, 
	\end{split}
\end{align*}
from \eqref{Kc small},\eqref{def IBx}, and $\mathfrak{m}_2(\mathcal{O}_j^{\mathcal{IB}})\lesssim \e$ for all $j=1,2,\dots,l_{IB}$. 
	\end{proof}
	
	\hide
	i.e.,
	\[
	\text{Lemma 22. If $(x,v)\in \mathcal{IB}^{c}$, then its trajectory does not graze for whole length $L$.}
	\]  \\ 
	\unhide 
	\hide
	Among finite open balls $\{ B(x_i,r(x_i,\varepsilon)) \cap \overline{S_{0}} \}_{i=1}^{l_{IB}}$, we choose all $x_{i}\in \overline{S_{0}^{\delta}}$ for fixed $\delta > 0$ and intersect with $\overline{S_{0}^{\delta}}$. By rename those $x_{i}$'s again, we get 
	\[
		\{ B(x_i,r(x_i,\varepsilon)) \cap \overline{S_{0}^{\delta}} \}_{i=1}^{l_{IB}^{\prime}},\quad l_{IB}^{\prime} \leq l_{IB}.
	\]
	{\color{blue}(To prevent possible issue later, note that $l_{\mathcal{IB}}$ is independent to $\delta$)}
	We cannot claim that $\cup_{i=1}^{l_{IB}^{\prime}}\{ B(x_i,r(x_i,\varepsilon)) \cap \overline{S_{0}^{\delta}} \}$ covers  $\overline{S_0^{\delta}}$ yet. We consider $\overline{S_{0}^{\delta}} \backslash \Big\{ \cup_{i=1}^{l_{IB}^{\prime}}\{ B(x_i,r(x_i,\varepsilon)) \cap \overline{S_{0}^{\delta}} \} \Big\}$ which is closed set, hence compact. By similarly as we did before, we gain finite open cover $ \cup_{i=1}^{l_{IB}^{\prime\prime}}\{ B(x_i,r(x_i,\varepsilon)) \cap \overline{S_{0}^{\delta}} \} $ for $\overline{S_{0}^{\delta}} \backslash \Big\{ \cup_{i=1}^{l_{IB}^{\prime}}\{ B(x_i,r(x_i,\varepsilon)) \cap \overline{S_{0}^{\delta}} \} \Big\}$ where $x_{i}\in \overline{S_{0}^{\delta}}$. 	{\color{blue}(To prevent possible issue later, note that $l^{\prime\prime}_{\mathcal{IB}}$ is independent to $\delta$??)} We reindex and rewriting $l_{\mathcal{IB}}$ instead of $l_{\mathcal{IB}}^{\prime} +l_{\mathcal{IB}}^{\prime\prime}$, we finally gain open cover of $\overline{S_{0}^{\delta}}$ :
	\[
		\cup_{i=1}^{l_{IB}} \{ B(x_i,r(x_i,\varepsilon)) \cap \overline{S_{0}^{\delta}} \},\quad x_{i}\in \overline{S_{0}^{\delta}} \\
	\]
	\unhide

\hide
	\newpage

\begin{lemma} \label{Unif nongrazing}
	{\color{blue}(Uniform nongrazing indepedent to $(x,v)$)} Let a trajectory bounce $M$ times nongrazingly. Then taking mimimum, 
	\[
	|v^{i}(x,v)\cdot \mathbf{n}(x^{i}(x,v))| \geq \bar{\delta}_{M} > 0.
	\]
	Then, there exists $\delta(\varepsilon, M) \ll 1$ so that the trajectory stays in $(\p S_0)^{\delta}$ sufficiently short time. 
\end{lemma}
\begin{proof}
	For each bounce, the trajectory with $|v|=1$ can stay in $(\p S_0)^{\delta}$ no longer than {\color{blue} $D(\delta)$(function which is decreasing in $\delta$)}. \\
	Hence, we choose $\delta \ll 1$ so that $D(\delta)\times M \ll 1$ as small as we want(?)  \\
\end{proof}

\unhide 
\begin{definition}
Recall that the compact set $\mathcal{K}^c\subset \overline{S_0}\times \S^2$ in proof of Lemma \ref{construct K}. Let us consider cross section $S_{\varphi}$ defined in \eqref{cross section} for fixed $\varphi \in [0,2\pi)$. Similar to $\mathcal{K}^c$ in \eqref{Kc}, using $\varphi$-symmetry in our domain $\O$, there exists a compact set $\mathcal{J}^c \subset \overline{S_{\varphi}} \times \S^2$ such that if $(x,v)\in \mathcal{J}^c$,
\begin{align*}
	(x^i(x,v),v^{i-1}(x,v))\notin \gamma_0^I \cup \gamma_0^C, \quad \text{for first $K$ bounce},
\end{align*}
and 
\begin{align*}
	(x^i(x,v),v^{i-1}(x,v))\notin \gamma_0^I, \quad \text{up to whole travel length $L$.}
\end{align*}
To extend the compact subset $\mathcal{J}^c\subset\overline{S_{\varphi}}\times \S^2$ to $\overline{S_{\varphi}} \times \mathbb{V}^N$, we define the following compact set
 \begin{align} \label{Jc_N}
 	\mathcal{V}^c := \Big \{(x,v)\in \mathcal{P}_1(\mathcal{J}^c) \times \mathbb{V}^N: \frac{v}{\vert v \vert} \in (\mathcal{J}^c)_x\Big \},
 \end{align}
where we used the projection operator $\mathcal{P}_1$ in \eqref{proj operator} and $A_x$ in \eqref{proj_v}. Note that  $\mathcal{P}_1(\mathcal{J}^c)=\mathcal{P}_1(\mathcal{V}^c)=\overline{S_\varphi}$. 
\end{definition}
\begin{remark}
Recall the notation $A_x$ in \eqref{proj_v}. Similar to \eqref{Kc small} in proof of Lemma \ref{construct K}, it holds that 
\begin{align*}
	\mathfrak{m}_2((\{\overline{S_\varphi}\times \S^2\} \backslash \mathcal{J}^c)_x)\lesssim \e, 
\end{align*}
for fixed $x \in \overline{S_\varphi}$. Since $\mathcal{V}^c$ can be obtained by scaling velocities in $\mathcal{J}^c$ from $1/N$ to $N$, we also deduce that for fixed $x \in \overline{S_\varphi}$ 
\begin{align} \label{Vc small} 
	\mathfrak{m}_3((\{\overline{S_\varphi}\times \mathbb{V}^N \} \backslash \mathcal{V}^c)_x) \leq C_N \e, 
\end{align}
where the constant $C_N$ depends on $N$. 
\end{remark}

\subsection{Phase Decomposition theorem}

\begin{theorem} \label{bad set thm}(Generalized version)
Let us consider cross section $S_{\varphi}$ defined in \eqref{cross section} for fixed $\varphi\in [0,2\pi)$. Then, for any $\varepsilon>0$, there exists a compact set $\mathcal{Z}^{\e}\subset \overline{S_\varphi}\times \mathbb{V}^N$ such that the following hold: if $(x,v)\in \mathcal{Z}^{\e}$, then  
\begin{equation*}
	[x^j(x,v),v^{j-1}(x,v)]\notin \gamma_0^C\cup\gamma_0^I,  \quad 1\leq j \leq M,
\end{equation*}
where $M$ was defined in Lemma \ref{bounce K}. Hence its backward in time trajectory up to a travel length $L=NT$ is always uniformly non-grazing. i.e.,
\begin{align*}
	\inf_{1\leq i \leq M}|v^{i-1}(x,v)\cdot n(x^{i}(x,v))| \geq C^*_{\varepsilon,L} > 0.
\end{align*}
Here, the constant $C^*_{\varepsilon,L}$ does not depend on $x$. Moreover, for fixed $x \in \overline{S_\varphi}$, we obtain 
\begin{align} \label{Ze c small} 
	\mathfrak{m}_3((\{\overline{S_{\varphi}}\times \mathbb{V}^N\}\backslash \mathcal{Z}^{\e})_x) \leq C_N \e,
\end{align}
where the constant $C_N$ depends only on $N$ and the notation $A_x$ was defined in \eqref{proj_v}. 
\end{theorem}
\begin{proof}
	For any $x\in \overline\Omega$, there exists $\varphi \in [0,2\pi)$ such that $x \in \overline{S_\varphi}$. 
	Since our domain is symmetric with respect to $\varphi$,  
	using compactness argument in proof of Lemma \ref{S2 version}, we can find a relatively open cover for $\mathcal{P}_1(\mathcal{V}^c)=\overline{S_{\varphi}}$ 
	\begin{align*}
	\bigcup_{i=1}^{l_{JB}} \{B(x_i^{\mathcal{JB}},r_i^{\mathcal{JB}}) \cap \overline{S_\varphi}\},
	\end{align*}
	and corresponding velocity set 
	\begin{align*}
	\mathcal{J}_i^{\mathcal{JB}} \subset \S^2,
	\end{align*}
	with $\mathfrak{m}_2(\mathcal{J}_i^{\mathcal{JB}})\lesssim\varepsilon$ for all $i=1,2,\dots,l_{JB}$ such as \eqref{open cover_IB} and \eqref{def bad IB}. We define the following velocity subset of $\mathbb{V}^N$ 
	\begin{align*}
		\mathcal{V}_i^{\mathcal{JB}} := \left \{ v\in \mathbb{V}^N: \frac{v}{\vert v \vert} \in \mathcal{J}_i^{\mathcal{JB}}\right\},
	\end{align*}
	and then $\mathfrak{m}_3(\mathcal{V}_i^{\mathcal{JB}})\leq C_N\varepsilon$ for all $i=1,2,\dots,l_{JB}$ because the set $\mathcal{V}_i^{\mathcal{JB}}$ can be obtained by scaling vectors in $\mathcal{J}_i^{\mathcal{JB}}$ from $1/N$ to $N$. Here, the constant $C_N$ increases as $N$ increases. Similar as \eqref{IB} in Lemma \ref{S2 version}, we construct an open cover $\mathcal{JB}$ such as 
\begin{align} \label{JB}
\mathcal{JB} := \bigcup_{i=1}^{l_{JB}} \Big \{ (x,v)\in \mathcal{V}^c: x \in B(x_i^{\mathcal{JB}},r_i^{\mathcal{JB}})\cap \overline{S_\varphi},\; v \in \bigcap_{j\in I_x^{\mathcal{JB}}}\mathcal{V}_j^{\mathcal{JB}}\Big\}_{\varepsilon},
\end{align}
where $A_{\e}$ was defined in \eqref{extension domain} and 
\begin{align} \label{index_JB}
I_x^{\mathcal{JB}}:=\Big\{i\in\mathbb{N}: x\in B(x_i^{\mathcal{JB}},r_i^{\mathcal{JB}})\cap \overline{S_\varphi}\Big\}.
\end{align}
Let us consider a compact set 
\begin{align} \label{def Ze} 
	\mathcal{Z}^\e:= \mathcal{V}^c \backslash \mathcal{JB}. 
\end{align} 
Since 
\begin{align*}
	\{\overline{S_\varphi}\times \mathbb{V}^N \} \backslash \mathcal{Z}^\e = (\{\overline{S_\varphi}\times \mathbb{V}^N \} \backslash \mathcal{V}^c) \cupdot \mathcal{JB},
\end{align*} 
it follows from \eqref{Vc small},\eqref{JB}, and $\mathfrak{m}_3(\mathcal{V}_j^{\mathcal{JB}})\leq C_N \e$ for all $j=1,2,\dots,l_{JB}$ that for fixed $x\in \overline{S_{\varphi}}$ 
\begin{align} \label{Ze c small}
	\begin{split}
		\mathfrak{m}_3((\{\overline{S_\varphi}\times \mathbb{V}^N\}\backslash \mathcal{Z}^\e)_x) &\leq \mathfrak{m}_3((\{\overline{S_\varphi}\times \mathbb{V}^N\}\backslash \mathcal{V}^c)_x)+\mathfrak{m}_3(\mathcal{JB}_x)\\
	&\leq \mathfrak{m}_3((\{\overline{S_\varphi}\times \mathbb{V}^N\}\backslash \mathcal{V}^c)_x)+\mathfrak{m}_3((\mathcal{V}_j^{\mathcal{JB}})_\e)\\
	&\leq C_N\varepsilon, \quad \forall j \in I_x^{\mathcal{JB}}, 
	\end{split}
\end{align}
where we have used the notation $A_x$ in \eqref{proj_v} and $A_\e$ in \eqref{extension domain}. \\

	For $v \in \mathbb{V}^N$, the size of velocity $v$ must be $N$ for the highest number of bounces during a fixed time travel length $T$. This case is the same as the size of velocity is $1$ and time travel length $NT$. Hence, in the two cases away from bad directions, the bouncing points $(x^i(x,v),v^{i-1}(x,v))$ are the same for all $i=1,2,\dots,M$. From Lemma \ref{S2 version}, if $(x,v) \in \mathcal{Z}^\e$ in \eqref{def Ze}, then $(x^j(x,v),v^{j-1}(x,v))$ is non-grazing for all $j=1,2,\dots,M$. Moreover, if we put a whole travel length as $L=NT$, then
\begin{align*}
	\inf_{1\leq i \leq M}|v^{i-1}(x,v)\cdot n(x^{i}(x,v))| \geq C^*_{\varepsilon, L} > 0,
\end{align*}  
where the constant $C^*_{\varepsilon, L}$ is independent of $x$. 
\end{proof}

\section{Transversality and Nondegeneracy condition}
In this section, we briefly explain nondegeneracy via triple Duhamel iteration technique which was constructed in \cite{KimLee}. In fact, Theorem \ref{bad set thm} plays role of velocity lemma (see Lemma 2.7 in \cite{KimLee}) and guarantees uniform non-grazing with uniformly finite number of bounces away from some sets with sufficiently small measure. Instead of repeating full detail of iteration scheme in \cite{KimLee} again, we provide a brief scheme of the triple iteration with some modification considering Theorem \ref{bad set thm}. \\

First, let us define local parametrization. For a given point ${p} \in \partial{\Omega}$, we choose sufficiently small $r_{1}>0, r_{2}>0$, and an bijective smooth mapping
\begin{equation}\label{eta}
	\begin{split}
		\eta_{{p}}:  \{\X_{{p}} \in \mathbb{R}^{3}:  \X_{{p},3}<0  \} \cap B(0; r_{1}) \ &\rightarrow  \ \Omega \cap B({p}; r_{2}),\\
		\X_{{p}}=(\X_{{p},1},\X_{{p},2},\X_{{p},3})	 \ &\mapsto \  (x_{1},x_{2},x_{3}) = \eta_{{p}}  (\X_{{p},1},\X_{{p},2},\X_{{p},3}),
	\end{split}
\end{equation} 
where $\eta_{ { {p}}}(\X_{ {{p}},1},\X_{ {{p}},2},\X_{ {{p}},3}) \in \p \Omega$ if and only if $\X_{ {{p}},3}=0$. We also define the transformed velocity field at $\eta_{{p}} (\X_{{p}})$ as  
\begin{equation}\label{v_p}
	\mathbf{v}_{ i}(\X_{{p}}) : = \frac{\p_{i} \eta_{{p}} (\X_{{p}})}{\sqrt{g_{{p},ii} (\X_{{p}})}} \cdot v,
\end{equation}	

\noindent where      
\begin{equation}\label{orthonormal_eta}
	\Big\{ \frac{\p_{1} \eta_{{p}}}{ \sqrt{g_{{p},11}}}, \frac{\p_{2} \eta_{{p}}}{ \sqrt{g_{{p},22}}},\frac{\p_{3} \eta_{{p}}}{ \sqrt{g_{{p},33}}}\Big\} \text{ is orthonormal at } \X_{{p},3} =0.
\end{equation}

\noindent For the second derivative $\p_i\p_j \eta_{p}$, we define Christoffel symbol $\Gamma{ij}^{k}$ by
\begin{equation} \label{Gamma}
	\p_{ij}\eta_p = \sum_{k} \Gamma_{p,ij}^{k} \p_{k}\eta_p.
\end{equation}

\noindent Without loss of generality, we set $g_{{p}, 33} (\X_{{p},1},\X_{{p},2},\X_{{p},3})=1$ whenever it is defined. The outward normal on the boundary is, for $x =\eta_{{p}} (\X_{{p},1}, \X_{{p},2},0)  \in \p\Omega$, 
\begin{equation}\label{normal_eta}
	n(x) = n(\eta_{{p}} (\X_{{p},1}, \X_{{p},2},0)) = 
	\p_{3} \eta_{{p}} (\X_{{p},1},\X_{{p},2},0)
	.%
\end{equation}


\noindent For each $k=0,1,2,3, \cdots,$ we assume that ${p}^{k} \in \p\Omega$ is chosen to be close to $x^{k}$ as in (\ref{eta}). Then, we define 
\begin{equation}\label{x^k}
	\begin{split}
		\X^{k}_{{p}^{k}}&:=(\X^{k}_{{p}^{k},1}, \X^{k}_{{p}^{k},2} , 0 )  \  \text{ 
			such that } \    x^{k}= \eta_{{p}^{k}} (\X^{k}_{{p}^{k}}), \\   \V^{k}_{p^{k},i} &:=\V_{i}^{k} (\X^{k}_{{p}^{k}}) = \frac{\p_{i} \eta_{{p}^{k}} (\X^{k}_{{p}^{k}})}{\sqrt{g_{{p}^{k},ii} (\X^{k}_{{p}^{k}})}} \cdot v^{k}, 
	\end{split}
\end{equation}
and speed-direction components
\begin{equation}\label{hat_v}
	|\V^{k}_{p^{k}}| = \sqrt{ ({\V}^{k}_{{p}^{k},1})^{2} + ({\V}^{k}_{{p}^{k},2})^{2} + ({\V}^{k}_{{p}^{k},3})^{2}},	
	\	\  \hat{\V}^{k}_{{p}^{k},1} = \frac{\V^{k}_{{p}^{k},1}}{|\V^{k}_{{p}^{k}}|},   \ \ 
	\hat{\V}^{k}_{{p}^{k},2} = \frac{\V^{k}_{{p}^{k},2}}{|\V^{k}_{{p}^{k}}|}. 
\end{equation}	

\noindent {\it\bf Step 1} To use advantage of $\frac{d X(s)}{d|v|}$, we first study change of variable between local coordinate and speed-direction coordinate. Let us assume $x\in B(x_{i}, r(x_{i}, \varepsilon))\cap \overline{S_{\varphi}}$ and $v\notin \mathcal{V}_{i}^{\mathcal{IB}}\cap \mathbb{V}_{N}$ where $B(x_{i}, r(x_{i}, \varepsilon))$ and $\mathcal{V}_{i}^{\mathcal{IB}}$ were defined in Theorem \ref{bad set thm}. Then by Theorem \ref{bad set thm} and \eqref{x^k}, $|\V_{p^{k},3}^{k}(t,x,v)| > \delta_{2}(\varepsilon, N) > 0$ for all $1\leq k \leq M$ for fixed $M$. (Note that we choose $M$ to be the uniform number of bounce $M_{\varepsilon, NT}$ what we obtained in Lemma \ref{bounce K}.) Applying Lemma 2.9 of \cite{KimLee} directly, we can obtain
\begin{equation}\label{Jac_hat}
	\bigg|\det\left[\begin{array}{cc|cc} 
		\frac{\p \X^{k}_{{p}^{k},1}}{\p \X^{1}_{{p}^{1},1}} & \frac{\p \X^{k}_{{p}^{k},1}}{\p \X^{1}_{{p}^{1},2}} & \frac{\p \X^{k}_{{p}^{k},1}}{\p \hat{\V}^{1}_{{p}^{1},1}}  & \frac{\p \X^{k}_{{p}^{k},1}}{\p  \hat{\V}^{1}_{{p}^{1},2}}  \\
		\frac{\p \X^{k}_{{p}^{k},2}}{\p \X^{1}_{{p}^{1},1}} & \frac{\p \X^{k}_{{p}^{k},2}}{\p \X^{1}_{{p}^{1},2}} & \frac{\p \X^{k}_{{p}^{k},2}}{\p \hat{\V}^{1}_{{p}^{1},1}}  & \frac{\p \X^{k}_{{p}^{k},2}}{\p \hat{\V}^{1}_{{p}^{1},2}}  \\ \hline
		\frac{\p \hat{\V}^{k}_{{p}^{k},1}}{\p \X^{1}_{{p}^{1},1}} & \frac{\p  \hat{\V}^{k}_{{p}^{k},1}}{\p \X^{1}_{{p}^{1},2}} & \frac{\p   \hat{\V}^{k}_{{p}^{k},1}}{\p \hat{\V}^{1}_{{p}^{1},1}}  & \frac{\p   \hat{\V}^{k}_{{p}^{k},1}}{\p \hat{\V}^{1}_{{p}^{1},2}}  \\
		\frac{\p   \hat{\V}^{k}_{{p}^{k},2}}{\p \X^{1}_{{p}^{1},1}} & \frac{\p   \hat{\V}^{k}_{{p}^{k},2}}{\p \X^{1}_{{p}^{1},2}} & \frac{\p   \hat{\V}^{k}_{{p}^{k},2}}{\p \hat{\V}^{1}_{{p}^{1},1}}  & \frac{\p   \hat{\V}^{k}_{{p}^{k},2}}{\p \hat{\V}^{1}_{{p}^{1},2}}  
	\end{array}\right] \bigg| \  >  \ \epsilon_{ \Omega,N, \delta_{2}} >0,
\end{equation}	
where $ t^{1}  =   t^{1} (t,x,v),$ $\X^{1}_{p^{1},i}  =  \X^{1}_{p^{1},i} (t,x,v),$ $\hat{\V}^{1}_{p^{1},i} = \hat{\V}^{1}_{p^{1},i} (t,x,v)$ and 
$$
\X^{k}_{{p}^{k},i} =   \X^{k}_{{p}^{k},i} (t^{1}, \X^{1}_{p^{1},1}, \X^{1}_{p^{1},2},  \hat{\V}^{1}_{p^{1},1},  \hat{\V}^{1}_{p^{1},2}, |\V^{1}_{p^{1}}|), \ \ \ 
\hat{\V}^{k}_{{p}^{k},i} =   \hat{\V}^{k}_{{p}^{k},i} (t^{1}, \X^{1}_{p^{1},1}, \X^{1}_{p^{1},2},  \hat{\V}^{1}_{p^{1},1},  \hat{\V}^{1}_{p^{1},2}, |\V^{1}_{p^{1}}|).$$
Here, the constant $\epsilon_{ \Omega, N, \delta_{2}} = \epsilon_{ \Omega, \varepsilon, N} > 0$ does not depend on $t$ and $x$.  \\

\noindent {\it \bf Step 2 } Now we study geometric decomposition for Jabocian of $X(s;t,x,v)$ with respect to speed-direction coordinate. We define specular basis first.
\begin{definition}[Specular Basis]
	Recall the specular cycles $(t^{k},x^{k}, v^{k})$ in section 1.2. Assume non-grazing
	\begin{equation}\label{no_grazing_0}
		n(x^{k}) \cdot v^{k} \neq 0,\quad 1\leq k \leq M,
	\end{equation}
	for each bounce we consider. Also, recall $\eta_{p^{k}}$ in (\ref{eta}). Specular basis is an orthonormal basis of $\R^{3}$ which is defined by 
	\begin{equation}\label{orthonormal_basis}
		\begin{split}
			&\mathbf{ {e}}^{k}_{0}  := \frac{v^{k}}{|v^{k}|}  
			, \\ 
			&\mathbf{  {e}}^{k}_{\perp,1}
			:=  \mathbf{e}^{k}_{0} \times \frac{\p_{2} \eta_{p^{k}}(x^{k})}{ \sqrt{g_{{p}^{k},22}(x^{k})}   } \bigg{/}  \Big| \mathbf{e}^{k}_{0} \times \frac{\p_{2} \eta_{p^{k}} (x^{k})}{ \sqrt{g_{{p}^{k},22}(x^{k})}   }   \Big|
			, \\  &\mathbf{  {e}}^{k}_{\perp,2}   :=  \mathbf{e}^{k}_{0} \times  \mathbf{  {e}}^{k}_{\perp,1},
		\end{split}
	\end{equation}
	for each $k$ and $(n(x^{k}), v^{k})$. 
\end{definition}	

From uniform nondegeneracy \eqref{Jac_hat}, we study the following Jacobian with respect to speed-direction coordinate,
\[
\det \frac{\p X(s;t,x,v)}{ \p(|\V_{p^{k}}^{k}|, \hat{\V}_{p^{k},1}^{k}, \hat{\V}_{p^{k},2}^{k}  )}.
\] 
Since perturbation of speed $|\V_{p^{k}}^{k}|=|v|$ does not change bouncing points. Using this advantage, we can obtain (see (3.32) of \cite{KimLee})
\begin{eqnarray*}
	&
	&
	\left[\begin{array}{ccc}
		\frac{\p X (s)}{\p |v|}\cdot\mathbf{e}_0^{k} &  \frac{\p X(s)}{\p \hat{v}_{1}}\cdot\mathbf{e}_0^{k} &  \frac{\p X(s)}{\p \hat{v}_{2}}\cdot\mathbf{e}_0^{k}   \\
		\frac{\p X(s)}{\p |v|}\cdot\mathbf{e}_{\perp,1}^{k} &  \frac{\p X(s)}{\p \hat{v}_{1}}\cdot\mathbf{e}_{\perp,1}^{k} &  \frac{\p X(s)}{\p \hat{v}_{2}}\cdot\mathbf{e}_{\perp,1} ^{k}  \\ 
		\frac{\p X(s)}{\p |v|}\cdot\mathbf{e}_{\perp,2} ^{k}&  \frac{\p X(s)}{\p \hat{v}_{1}}\cdot\mathbf{e}_{\perp,2} ^{k}&  \frac{\p X(s)}{\p \hat{v}_{2}}\cdot\mathbf{e}_{\perp,2}  ^{k}	
	\end{array}\right]
	=
	\left[\begin{array}{c|cc}
		-(t-s) & 
		{\bf (*)_{1,2}}
		\\ \hline
		\mathbf{0}_{2,1} &     
		{\bf (**)_{2,2}}
	\end{array}\right]_{3\times 3}.
	\notag \\
\end{eqnarray*} 
To check full rank property of above matrix, we study determinant of lower right ${\bf (**)_{2,2}}$ which is generated by directional derivative $\frac{\p X(s)}{\p \hat{\V}_{p^{k},i}^{k}}$. To study ${\bf (**)_{2,2}}$ efficiently, we define specular matrix (Definition 3.3 in \cite{KimLee}).

\begin{definition}[Specular Matrix]
	For fixed $k \in \mathbb{N}$ and a $C^{1}$-map $Y: (y_{1}, y_{2} ) \mapsto Y(y_{1}, y_{2} )\in \O$, assume (\ref{no_grazing_0})
	with 	${x}^{k} ={x}^{k} (t, Y( {y}_{1},  {y}_{2}), |v|, \hat{v}_{1}, \hat{v}_{2})$ and ${v}^{k} ={v}^{k} (t, Y( {y}_{1},  {y}_{2}), $\\
	$|v|, \hat{v}_{1}, \hat{v}_{2})$. We define the $4\times4$ specular transition matrix $\mathcal{S}^{k, p^{k}, Y}= \mathcal{S}^{k, p^{k}, Y}(t,y_{1},y_{2},\\ |v|, \hat{v}_{1}, \hat{v}_{2})$ as 
	\begin{equation}\label{specular_transition_matrix}
		\mathcal{S}^{k, p^{k}, Y}
		:= \left[\begin{array} {c|c}
			\mathcal{S}_{1}^{k, p^{k}, Y}	 &0_{2\times2}\\ \hline
			\mathcal{S}_{2}^{k, p^{k}, Y} & \mathcal{S}_{3}^{k, p^{k}, Y}
		\end{array} \right]_{4\times 4},
	\end{equation}
	where 
	\begin{equation}\begin{split}\notag
			\mathcal{S}_{1}^{k, p^{k}, Y} &: = 
			\left[\begin{array}{cc}	\p_{1} \eta_{ {p}^{k}} \cdot \mathbf{e}^{k}_{\perp,1} & 
				\p_{2} \eta_{{p}^{k}} \cdot \mathbf{e}^{k}_{\perp,1}\\
				\p_{1} \eta_{{p}^{k}} \cdot \mathbf{e}^{k}_{\perp,2} & 
				\p_{2} \eta_{{p}^{k}} \cdot \mathbf{e}^{k}_{\perp,2}
			\end{array}
			\right]_{2\times 2} ,\\
			\mathcal{S}_{2}^{k, p^{k}, Y} &: = 	
			\left[\begin{array}{cc}
				\Big(  \sum_{\ell=1}^{3} \p_{1} \big[ \frac{\p_{\ell} \eta_{{p}^{k}}}{\sqrt{g_{ {{p}^{k}},\ell\ell}}}  \big] \hat{\mathbf{v}}^{k}_{ {{p}^{k}},\ell}  \Big) \cdot \mathbf{e}^{k}_{\perp,1}   
				&   \Big(  \sum_{\ell=1}^{3} \p_{2} \big[ \frac{\p_{\ell} \eta_{{p}^{k}}}{\sqrt{g_{ {{p^{k}}},\ell\ell}}}  \big] \hat{\mathbf{v}}^{k}_{ { {p}},\ell}  \Big) \cdot \mathbf{e}_{\perp,1}  ^{k}
				\\
				\Big(  \sum_{\ell=1}^{3} \p_{1} \big[ \frac{\p_{\ell} \eta_{{p^{k}}}}{\sqrt{g_{ { {p^{k}}},\ell\ell}}}  \big] \hat{\mathbf{v}}^{k}_{ {{p^{k}}},\ell}  \Big) \cdot \mathbf{e}^{k}_{\perp,2}   & 
				\Big(  \sum_{\ell=1}^{3} \p_{2} \big[ \frac{\p_{\ell} \eta_{{p^{k}}}}{\sqrt{g_{ {{p^{k}}},\ell\ell}}}  \big] \hat{\mathbf{v}}^{k}_{ {{p^{k}}},\ell}  \Big) \cdot \mathbf{e}^{k}_{\perp,2}  
			\end{array} \right]_{2\times 2},\\
			\mathcal{S}_{3} ^{k, p^{k}, Y}& : = 
			\left[\begin{array}{cc}
				\Big[ \frac{\p_{1} \eta_{p^{k}}}{\sqrt{g_{p^{k},11}}} - \frac{\p_{3} \eta_{p^{k}}}{\sqrt{g_{p^{k},33}}}  \frac{\hat{\mathbf{v}}^{k}_{p^{k},1}}{\hat{\mathbf{v}}^{k}_{p^{k},3}} \Big] \cdot \mathbf{e}^{k}_{\perp,1}
				& \Big[ \frac{\p_{2} \eta_{p^{k}}}{\sqrt{g_{p^{k},22}}} - \frac{\p_{3} \eta_{p^{k}}}{\sqrt{g_{p^{k},33}}}  \frac{\hat{\mathbf{v}}^{k}_{p^{k},2}}{\hat{\mathbf{v}}^{k}_{p^{k},3}} \Big] \cdot \mathbf{e}^{k}_{\perp,1}\\
				\Big[ \frac{\p_{1} \eta_{p^{k}}}{\sqrt{g_{p^{k},11}}} - \frac{\p_{3} \eta_{p^{k}}}{\sqrt{g_{p^{k},33}}}  \frac{\hat{\mathbf{v}}^{k}_{p^{k},1}}{\hat{\mathbf{v}}^{k}_{p^{k},3}} \Big] \cdot \mathbf{e}^{k}_{\perp,2}
				& \Big[ \frac{\p_{2} \eta_{p^{k}}}{\sqrt{g_{p^{k},22}}} - \frac{\p_{3} \eta_{p^{k}}}{\sqrt{g_{p^{k},33}}}  \frac{\hat{\mathbf{v}}^{k}_{p^{k},2}}{\hat{\mathbf{v}}^{k}_{p^{k},3}} \Big] \cdot \mathbf{e}^{k}_{\perp,2}
			\end{array} \right]_{2\times 2},
	\end{split} \end{equation}
	where $\eta_{p^{k}}$ and $g_{p^{k}}$ are evaluated at $x^{k}(t, Y( {y}_{1},  {y}_{2}), |v|, \hat{v}_{1}, \hat{v}_{2})$. We also define the $4\times 4$ specular matrix $\mathcal{R}^{k,p^{k},Y} = \mathcal{R}^{k,p^{k},Y} (t,y_{1},y_{2}, |v|, \hat{v}_{1}, \hat{v}_{2})$ as
	\begin{equation}\label{specular_matrix}
		\mathcal{R}^{k, p^{k}, Y}: = 	\mathcal{S}^{k, p^{k}, Y}
		\frac{\p ( 
			\mathbf{x}^{k}_{ {p}^{k},1}, \mathbf{x}^{k}_{ {p}^{k},2},
			\hat{\mathbf{v}}^{k}_{ {p}^{k},1}, \hat{\mathbf{v}}^{k}_{ {p}^{k},1}
			)}{\p (y_{1},y_{2}, \hat{v}_{1}, \hat{v}_{2})}
		,
	\end{equation}	
	where $\mathbf{x}^{k}_{{p}^{k}}=\mathbf{x}^{k}_{{p}^{k}}(t, Y( {y}_{1},  {y}_{2}), |v|, \hat{v}_{1}, \hat{v}_{2})$, $\mathbf{v}^{k}_{{p}^{k}}=\mathbf{v}^{k}_{{p}^{k}}(t, Y( {y}_{1},  {y}_{2}), |v|, \hat{v}_{1}, \hat{v}_{2})$.  Here, $\hat{v}_{i}$ is directional component in global Eulerian coordinate $\hat{v}_{i} = \frac{v_{i}}{|v|}$.  \\
\end{definition}

Once we perform triple iteration, we can generate extra variables $y_{1}$ and $y_{2}$. Explicitly, we have 
\begin{eqnarray*}
	&&\left[\begin{array}{ccc}
		\mathbf{e}_{0}^{k} & \mathbf{e}_{\perp,1}^{k} & \mathbf{e}_{\perp,2}^{k}
	\end{array}\right]^{-1} 
	\frac{\p X (s; t, Y_{x}(y_{1},y_{2}), |v|, \hat{v}_{1}, \hat{v}_{2})}{ \p (|v |, \hat{v}_{1} , \hat{v}_{2} , y_{1}, y_{2})} \\
	&&
	= 
	\underbrace{\left[\begin{array}{ccccc}
			\frac{\p X}{\p |v |} \cdot \mathbf{e}_{0}^{k} & \frac{\p X}{\p \hat{v} _{1}} \cdot \mathbf{e}_{0}^{k} & \frac{\p X}{\p \hat{v} _{2}} \cdot \mathbf{e}_{0}^{k}& \frac{\p X}{\p y_{1}} \cdot \mathbf{e}_{0}^{k} & \frac{\p X}{\p y_{2}} \cdot \mathbf{e}_{0}^{k} \\
			\frac{\p X}{\p |v |} \cdot \mathbf{e}_{\perp,1}^{k} & \frac{\p X}{\p \hat{v} _{1}} \cdot \mathbf{e}_ {\perp,1}^{k} & \frac{\p X}{\p \hat{v} _{2}} \cdot \mathbf{e}_ {\perp,1}^{k}& \frac{\p X}{\p y_{1}} \cdot \mathbf{e}_ {\perp,1} ^{k}& \frac{\p X}{\p y_{2}} \cdot \mathbf{e}_ {\perp,1}^{k}\\
			\frac{\p X}{\p |v |} \cdot \mathbf{e}_{\perp,2}^{k} & \frac{\p X}{\p \hat{v} _{1}} \cdot \mathbf{e}_ {\perp,2} ^{k}& \frac{\p X}{\p \hat{v} _{2}} \cdot \mathbf{e}_ {\perp,2}^{k}& \frac{\p X}{\p y_{1}} \cdot \mathbf{e}_ {\perp,2} ^{k}& \frac{\p X}{\p y_{2}} \cdot \mathbf{e}_ {\perp,2}^{k}
		\end{array} \right]} .
\end{eqnarray*}
Using the \textit{specular basis} (\ref{orthonormal_basis}) and the \textit{specular matrix} (\ref{specular_matrix}), the underbraced term becomes
\begin{eqnarray*}
	{  \left[\begin{array}{c|c}
			- (t-s) & \substack{	
				-  | {v}^{k} |    \nabla_{\hat{v} _{1}, \hat{v} _{2},y_{1}, y_{2}}   t^{k}+
				\nabla_{\hat{v} _{1}, \hat{v} _{2}, y_{1}, y_{2}}   \mathbf{x}^{k}_{ {p}^{k}, \ell} \p_{\ell} \eta_{ {p}^{k}} \cdot \mathbf{e}_{0}^{k}
				\\ - (t^{k}-s) | {v}^{k} |  
				\sum_{j=1}^{2}
				\bigg( \sum_{\ell=1}^3  \frac{\p}{\p {\mathbf{x}^{k}_{ {p}^{k},j}}}
				\Big[  \frac{\p_\ell \eta_{ {p}^{k}}}{\sqrt{g_{ {p}^{k},\ell\ell}}} \Big] \hat{\mathbf{v}}^k_{{p}^{k},\ell} \bigg)
				\nabla_{\hat{v} _{1}, \hat{v} _{2}, y_{1}, y_{2}} \mathbf{x}^k_{ {p}^{k},j} 
			}\\ \hline
			\substack{ 0 \\ 
				0 }
			& *_{2 \times 4} 
		\end{array} \right]_{3\times 5}}, 
\end{eqnarray*}
where the lower right $2 \times 4$-submatrix equals
\begin{equation}\label{R_24}\begin{split}
		&\left[\begin{array}{cccc}
			\mathcal{R}_{1,1}^{k,{p}^{k},Y_{x}} & 	\mathcal{R}_{1,2}^{k,{p}^{k},Y_{x}}  & 	\mathcal{R}_{1,3}^{k,{p}^{k},Y_{x}} & 	\mathcal{R}_{1,4}^{k,{p}^{k},Y_{x}}\\
			\mathcal{R}_{2,1} ^{k,{p}^{k},Y_{x}}& 	\mathcal{R}_{2,2} ^{k,{p}^{k},Y_{x}} & 	\mathcal{R}_{2,3}^{k,{p}^{k},Y_{x}} & 	\mathcal{R}_{2,4}^{k,{p}^{k},Y_{x}}
		\end{array} \right]  \\
	&
		- (t^{k}-s) | {v}^{k} |
		\left[\begin{array}{cccc}
			\mathcal{R}_{3,1} ^{k,{p}^{k},Y_{x}}& 	\mathcal{R}_{3,2}^{k,{p}^{k},Y_{x}} &	\mathcal{R}_{3,3} ^{k,{p}^{k},Y_{x}}& 	\mathcal{R}_{3,4}^{k,{p}^{k},Y_{x}}\\
			\mathcal{R}_{4,1} ^{k,{p}^{k},Y_{x}}& 	\mathcal{R}_{4,2}^{k,{p}^{k},Y_{x}}	 & 	\mathcal{R}_{4,3} ^{k,{p}^{k},Y_{x}}& 	\mathcal{R}_{4,4}^{k,{p}^{k},Y_{x}}
		\end{array} \right].
\end{split}\end{equation}
Here $\mathcal{R}^{k,p^{k},Y_{x}}_{i,j}$ is defined in (\ref{specular_matrix}) with $x^{k} = x^{k} (t, Y_{x}(y_{1},y_{2}), |v|, \hat{v}_{1}, \hat{v}_{2})$ and $v^{k} = v^{k} (t,Y_{x}$\\
$(y_{1},y_{2}), |v|, \hat{v}_{1}, \hat{v}_{2})$. From the exact form of underbraced submatrix, choosing $2\times 2$ submatrix in the underbraced matrix is equivalent to choosing two independent variables $\{\zeta_{1}, \zeta_{2}\}$ among $\{ \hat{v}_{1}, \hat{v}_{2}, y_{1}, y_{2}\}$. (Here, $(y_{1}, y_{2})$ is introduced to represent a local spatial coordinate of a manifold which is constructed by $\p_{|v|}, \p_{\hat{v}_{i}}$ in its second Duhamel iteration step.) \\

\noindent {\it\bf Step 3} From the existence of $\{\zeta_{1}, \zeta_{2}\}\subset \{ \hat{v}_{1}, \hat{v}_{2}, y_{1}, y_{2}\}$ which gives nondegeneracy of $2\times 2$ submatrix of \eqref{R_24}, we can prove full rank property. 
We recall Lemm 3.7 of \cite{KimLee}. Away from arbitrary small time intervals near $C^{1}$ function $\psi_{i}^{k}:B_{\varepsilon}(t,x,v)\rightarrow \R$ (for $i=1,2,3$), we obtain nondegeneracy. The following Lemma is same as Lemma 3.7 in \cite{KimLee}. We obtain the following Lemma by  replacing $\mathbf{e}_{3}$ into $\hat{\varphi}(\varphi)$ WLOG. \\

\begin{lemma} \label{lemma rank 3} [$S_{\varphi}$ version for Lemma 3.7 in \cite{KimLee}]
	Let $x\in S_{\varphi}$. We also assume at most $M$ bounces during $[t-1,t]$ and
	\begin{equation}\begin{split}\label{con_v0}
			\frac{1}{N} \leq |v | \leq N,  \ 
			\frac{1}{N} \leq |v\cdot \hat{\varphi}(\varphi)|, \ \frac{1}{N} \leq| n(x^{1}(t, x  ,v )) \cdot \hat{\varphi}(\varphi)|, \\ \text{and} \  {| n(x^{k}(t, x  ,v )) \cdot v^{k}(t, x  ,v ) |} \geq \delta_{1} > 0
			,\quad 1\leq k \leq M.
		\end{split}
	\end{equation} 
	where $\hat{\varphi}(\varphi)$ is a unit vector that is perpendicular to the cross-section $S_{\varphi}$. Let a $C^{1}$-map $Y_{x}: (y_{1},y_{2}) \mapsto Y_{x}(y_{1},y_{2}) \in \bar{\O}$ with $Y_{x}(0, 0) = x$ and $\| Y\|_{C^{1}_{x,y_{1},y_{2}}}\lesssim1$. We assume that 
	\begin{equation} \label{Y times Y}
	\Big|\Big(\frac{\p Y_{x}(0,0)}{\p y_{1}} \times \frac{\p Y_{x}(0,0)}{\p y_{2}}\Big) \cdot 
	R_{ x^{1}(t, x ,v)  } v^{1}(t, x ,v)  \Big| > \delta_{3}>0.
	\end{equation} 
	%
	%
	
	For $k\in \mathbb{N}$ with $ t^{k}\geq t-1$, there exists $\e>0$ and finitely many $C^{0,\gamma}$-functions $\psi^{k}_{i}  : B_{\varepsilon}(t,x,v)\rightarrow \R$ with $\|\psi^{k}_{i}  \|_{C^{0,\gamma}_{t,x}}\lesssim 1$, and there exists a constant $\epsilon_{\delta_{1}, \delta_{3}, N, \O, \delta_{*} } >0$ and $\{\zeta_{1}, \zeta_{2}\} \subset \{  \hat{v} _{1},  \hat{v}_{2}, y_{1}, y_{2}\}$ such that 
	\begin{eqnarray*} 
	&& \text{if } \  \min_{i}|s- \psi_{i}^{k}(t, 
	Y_{x}(y_{1}, y_{2})
	,v)|>\delta_{*}\nonumber\\
	&& \ \ \ 
	\  \text{and} \ (s;t,Y_{x}(y_{1}, y_{2}),v) \in [ \max\{t-1, t^{k+1}\}, \min \{t- \frac{1}{N}, t^{k}\}] \times  B_{\e}(t, x, v), \nonumber \\
	&&\text{then}  \  \det \left(
	\frac{\p X (s; t, Y_{x}(y_{1}, y_{2}), |v|,\hat{v} _{1},  \hat{v}_{2}  )}{\p (|v |,\zeta_{1}, \zeta_{2} )}\right)> \epsilon_{ \delta_{1},  \delta_{3},N,  \O, \delta_{*}}
	>0.\nonumber  \\
	\end{eqnarray*}
\end{lemma}
Before we apply above Lemma to the following Proposition, we again note that the second and third conditions of \eqref{con_v0} are written as
\begin{equation*}
\begin{split} 
\frac{1}{N} \leq |v_{ 3}|, \quad \frac{1}{N} \leq| n(x^{1}(t, x  ,v )) \cdot e_{3}|, 
\end{split}
\end{equation*} 
 in \cite{KimLee}, where the choice of $e_{3}$ can be changed into any other fixed direction vector in $\S^{2}$ WLOG. Since we are dealing with $x\in S_{\varphi}$, we choose $\hat{\varphi}(\varphi)$ instead of $e_{3}$ which is more suitable in our domain. \\


 Modifying Theorem 3.9 of \cite{KimLee} into $\mathcal{Z}^{\varepsilon}$ version ($\mathcal{Z}^{\varepsilon}$ is obtained in Theorem \ref{bad set thm}), we obtain the following Proposition.  \\

\begin{proposition}\label{prop_full_rank} [$\mathcal{Z}^{\varepsilon}$ version for Theorem 3.9 in \cite{KimLee}]
	Fix any arbitrary $(t,x,v) \in [T,T+1] \times \overline{\O} \times \VN$. 
	There exists $\delta_{2}>0$ and $C^{1}$-function $\psi^{\ell_{0}, \vec{\ell}, k}$ for $k \leq M$ where $\psi^{\ell_{0}, \vec{\ell}, k}$ is defined locally around $(T+ \delta_{2} \ell_{0}, X(T+ \delta_{2} \ell_{0};t,x,v), \delta_{2} \vec{\ell})$ with $(\ell_{0},\vec{\ell}) = (\ell_{0}, \ell_{1}, \ell_{2},\ell_{3}) \in \{ 0,1, \cdots, \lfloor\frac{1}{\delta_{2}}\rfloor+1\} \times \{- \lfloor\frac{N}{\delta_{2}}\rfloor-1, \cdots, 0 , \cdots, \lfloor\frac{N}{\delta_{2}}\rfloor+1 \}^{3}$ and $ \| \psi^{\ell_{0}, \vec{\ell}, k}  \|_{C^{1}}   \leq C_{N,\Omega, \varepsilon, \delta_{2} }< \infty$. 
	
	Moreover, 
	if 
	\begin{equation}\label{Xs_OVi}
		(X(s;t,x,v),u) \in \mathcal{Z}^{\varepsilon},
	\end{equation}
	
	\begin{equation}\label{su_ell0}
		(s,u) \in [T+ (\ell_{0} -1) \delta_{2},T+ (\ell_{0} +1) \delta_{2} ] \times B(\delta_{2} \vec{\ell}; 2\delta_{2}),
	\end{equation}
	
	\begin{equation}\label{sprime_k}
		s^{\prime} \in \big[ t^{k+1} (T+ \delta_{2} \ell_{0}; X(T+ \delta_{2} \ell_{0} ;t,x,v), \delta_{2} \vec{\ell})+\frac{1}{N}, \ t^{k } (T+ \delta_{2} \ell_{0}; X(T+ \delta_{2} \ell_{0} ;t,x,v), \delta_{2} \vec{\ell})-\frac{1}{N}\big], 
	\end{equation}
	and 
	\begin{equation}\label{sprime_psi}
		|s^{\prime} -   \psi^{\ell_{0}, \vec{\ell}, k}  (T+ \delta_{2} \ell_{0} , X( T+ \delta_{2}\ell_{0} ; t,x,v), \delta_{2} \vec{\ell} ) | > N^{2} (1 + \| \psi^{\ell_{0}, \vec{\ell}, k} \|_{C^{1}}) \delta_{2} ,
	\end{equation}
	then 
	\begin{equation}\label{lower_1}
		\big| \p_{|u| } X(s^{\prime} ; s, X(s;t,x,v), u) \times   \p_{\hat{u}_{1}} X(s^{\prime} ; s, X(s;t,x,v), u)
		\big| > \epsilon_{\Omega, N, \varepsilon, \delta_{2}}.
	\end{equation}
	Here $\epsilon_{\Omega, N, \varepsilon, \delta_{2}}>0$ does not depend on $T, t,x,v$. \\

	There exists $\delta_{3}>0$ and $C^{1}$-functions 
	\begin{equation}\label{three_psi}
		\psi^{\ell_{0}, \vec{\ell}, k, m_{0}, \vec{m}, k^{\prime}}_{1},\psi^{\ell_{0}, \vec{\ell}, k, m_{0}, \vec{m}, k^{\prime}}_{2},\psi^{\ell_{0}, \vec{\ell}, k, m_{0}, \vec{m}, k^{\prime}}_{3},\end{equation}
	for $k^{\prime} \leq M$ where $\psi^{\ell_{0}, \vec{\ell}, k, m_{0}, \vec{m}, k^{\prime}}_{n}$ is defined locally around $(T+ \delta_{3}m_{0} ; X(T+ \delta_{3}m_{0}; T+ \delta_{2} \ell_{0}, X ( T+ \delta_{2} \ell_{0}; t,x,v   ), \delta_{2} \vec{\ell}  ), \delta_{3} \vec{m} )$ for some $(m_{0}, \vec{m}) = (m_{0}, m_{1}, m_{2},m_{3})  \in\{ 0,1, \cdots, \lfloor\frac{1}{\delta_{3}}\rfloor+1\} \times \{- \lfloor\frac{N}{\delta_{3}}\rfloor-1, \cdots, 0 , \cdots, \lfloor\frac{N}{\delta_{3}}\rfloor+1 \}^{3}$ with $0< \delta_{3} \ll1$.
	
	Moreover, if we assume (\ref{Xs_OVi}), (\ref{su_ell0}), (\ref{sprime_k}), (\ref{sprime_psi}),
	\begin{equation}\label{sprime_j}
	\begin{split}
		&(X(s^{\prime};s, X(s;t,x,v),u), u^{\prime}) \in \mathcal{Z}^{\varepsilon}, \ \text{and} \\
		&\quad\quad \Big|\Big( \frac{\p_{y_{1}} Y_{X(s^{\prime};s, X(s;t,x,v),u)}(0,0)  \times  \p_{y_{2}} Y_{X(s^{\prime};s, X(s;t,x,v),u)}(0,0) }{|\p_{y_{1}} Y_{X(s^{\prime};s, X(s;t,x,v),u)}(0,0)  \times  \p_{y_{2}} Y_{X(s^{\prime};s, X(s;t,x,v),u)}(0,0) |}\Big) \cdot 
		u^{\prime}  \Big| > \frac{1}{N} > 0,
	\end{split}
	\end{equation}
	\begin{equation}\label{sprimeprime_kprime}
		\begin{split}
			s^{\prime\prime}  \in& \  \Big[t^{k^{\prime}+1}  (T+ \delta_{3}m_{0} ; X(T+ \delta_{3}m_{0}; T+ \delta_{2} \ell_{0}, X ( T+ \delta_{2} \ell_{0}; t,x,v   ), \delta_{2} \vec{\ell}  ),  \delta_{3} \vec{m} ) + \frac{1}{N} \\ &   \ \ 
			\ \ , t^{k^{\prime} }(T+ \delta_{3}m_{0} ; X(T+ \delta_{3}m_{0}; T+ \delta_{2} \ell_{0}, X ( T+ \delta_{2} \ell_{0}; t,x,v   ), \delta_{2} \vec{\ell}  ), \delta_{3} \vec{m} ) - \frac{1}{N}\Big],
		\end{split}
	\end{equation} 
	and
	\begin{equation}\begin{split}\label{sprimeprime_psi}
			&\min_{n=1,2,3}|s^{\prime\prime} - \psi^{\ell_{0}, \vec{\ell},  k,  m_{0}, \vec{m}, k^{\prime}}_{n}   (T+ \delta_{3}m_{0} ; X(T+ \delta_{3}m_{0}; T+ \delta_{2} \ell_{0}, X ( T+ \delta_{2} \ell_{0}; t,x,v   ), \delta_{2} \vec{\ell}  ), \delta_{3} \vec{m} )  |\\
			&>N^{2} (1+ \max_{n=1,2,3} \| \psi^{ \ell_{0}, \vec{\ell},  k, m_{0}, \vec{m}, k^{\prime}}_{n}  \|_{C^{1}} ) (\delta_{3}),\end{split}\end{equation}
	then for each $\ell_{0}, \vec{\ell},  k, m_{0}, \vec{m}, k^{\prime}$ we can choose two distinct variables $\{\zeta_{1}, \zeta_{2}\} \subset \{|u|, \hat{u}_{1}, 
	\hat{u}_{1}^{\prime} , \hat{u}_{2}^{\prime}  \}$ such that $(|u^{\prime}|, \zeta_{1}, \zeta_{2} ) \mapsto X(s^{\prime\prime};s^{\prime},X( s^{\prime}; s, X(s;t,x,v),u), u^{\prime} ) $ is one-to-one locally and 
	\begin{equation}\label{lower_zeta12}
		\Big|\det \left( \frac{\p X(s^{\prime\prime};s^{\prime},X( s^{\prime}; s, X(s;t,x,v),u), u^{\prime} ) }{\p (|u^{\prime}|, \zeta_{1}, \zeta_{2})} \right)
		\Big| > \epsilon^{\prime}_{\Omega, N, \varepsilon, \delta_{2}, \delta_{3} }.
	\end{equation}
	Here $\epsilon^{\prime}_{\Omega, N, \varepsilon, \delta_{2}, \delta_{3} } > 0$ does not depend on $T,t,x,v$.  \\
\end{proposition}
\begin{proof}
	Main proof is nearly the same as the proof of Theorem 3.9 in \cite{KimLee}. Proof of  Theorem 3.9 in \cite{KimLee} comes from Lemma 3.7 in \cite{KimLee}. Similarly, Proposition \ref{prop_full_rank} can be proved using Lemma \ref{lemma rank 3}. Instead of providing full proof, we just explain why it is possible to apply Lemma \ref{lemma rank 3} and Theorem \ref{bad set thm}. \\
	\indent Let us combine Theorem \ref{bad set thm} and Lemma \ref{lemma rank 3}. First, when we assume Theorem \ref{bad set thm}, we should check that $(x,v)\in \mathcal{Z}^{\varepsilon}$
		satisfies \eqref{con_v0} in Lemma \ref{lemma rank 3}. The first and fourth conditions of \eqref{con_v0} hold obviously by the statement of Theorem \ref{bad set thm}, by the definition of $\VN$ and choosing $\delta_{1} = C^{*}_{\varepsilon, L}$ of Theorem \ref{bad set thm}. Let us check the second and third conditions. \\
		\indent For the second condition, we recall Lemma \ref{lem Xe} in which we excluded $R_{\perp}^{\varepsilon}$. If we consider  Lemma \ref{lem Xe} in general $S_{\varphi}$ cross-section, excluding $R_{\perp}^{\varepsilon}$ exactly corresponds to $\frac{1}{N} \leq |v\cdot \hat{\varphi}(\varphi)|$ for some proper $\varepsilon \ll 1$ and $N \gg 1$. The third condition of \eqref{con_v0} is also similar because it is easy to observe 
		\[
		|v\cdot \hat{\varphi}(\varphi)|=0  \Longleftrightarrow  | n(x^{1}(t, x  ,v )) \cdot \hat{\varphi}(\varphi)|=0,\quad x\in \overline{S_{\varphi}},
		\]
		by $z$-axisymmetric structure. In conclusion, $(x,v)\in \mathcal{Z}^{\varepsilon}$ satisfies all conditions in \eqref{con_v0}. \\
		\indent Now, above argument is applied to $(X(s;t,x,v),u) \in \mathcal{Z}^{\varepsilon}$ and $(X(s^{\prime};s, X(s;t,x,v),u)$\\
		$, u^{\prime}) \in \mathcal{Z}^{\varepsilon}$ by \eqref{Xs_OVi} and \eqref{sprime_j}. Moreover, condition \eqref{sprime_j} implies \eqref{Y times Y} putting $x^0 = X(s^{\prime};s, X(s;t,x,v),u)$. Therefore, the proof of Theorem 3.9 of \cite{KimLee} can be directly applied to prove Proposition \ref{prop_full_rank}. Note that $\delta_{1}$ of Lemma \ref{lemma rank 3} is chosen by $\delta_{1} = C^{*}_{\varepsilon, L}$ by Theorem \ref{bad set thm}. So choosing $L = NT_0 = N$ (we choose time interval $T_0=1$ here), we get the dependence of  $\epsilon^{\prime}_{\Omega, N, \varepsilon, \delta_{2}, \delta_{3} }$. We finish the proof. 
\end{proof}

\section{Triple Duhamel iterations}

Note that the kernel $\mathbf{k}(v,u)$ satisfies (see \cite{Guo10, Glassey})
\Be \notag
\mathbf{k}(v,u) \lesssim \big( \frac{1}{|v-u|} + |v-u| \big) e^{ -\frac{1}{8}|v-u|^{2} - \frac{1}{8}\frac{ ||v|^{2} - |u|^{2}|^{2} }{|v-u|^{2}}}.  
\Ee

Now we can write an equation for $f$ from \eqref{Boltzmann} and \eqref{pert}: 

Motivated by characteristic $(X(s;t,x,v), V(s;t,x,v))$, we define mild solution of \eqref{E_eqtn}: 
\Be \label{mild soln}
\begin{split}
	f(t,x,v) &= e^{-\int_{0}^{t}\nu(V(\tau;t,x,v)) d\tau}  f_0(X(0;t,x,v), V(0;t,x,v))  \\
	&\quad + \int_{0}^{t} e^{-\int_{s}^{t}\nu(V(\tau;t,x,v)) d\tau } \int_{u\in\R^{3}} k(V(s;t,x,v), u) f(s, X(s;t,x,v), u) du ds \\
	&\quad + \int_{0}^{t} e^{-\int_{s}^{t}\nu(V(\tau;t,x,v)) d\tau } \Gamma(f,f)(s, X(s; t,x,v), V(s;t,x,v)) ds.   \\
\end{split}	
\Ee

In this section, we also briefly explain triple Duhamel expansion which was explained in \cite{KimLee}.
First, let $f$ solves linearized Boltzmann equation (\ref{lin eq})
\begin{equation} \label{lin eq}
	\begin{split}
		&\p_t f  + v\cdot\nabla  f + \nu f   =
		Kf ,\quad f(t,x,v) = f(t, x, R_x v) \ \ \text{on $x\in \p\O$}.
	\end{split}
\end{equation}
Let $h :=  wf$ where $w=(1+|v|)^{\b}, \ \b > 5/2$. Then the linearized Boltzmann equation becomes
\begin{equation*} \label{equation_h}
	\begin{split}
		&  \p_t h + v\cdot\nabla_x h + \nu h =  K_w h ,\quad  K_{w}h := wK(\frac{h}{w}).
	\end{split}
\end{equation*}

\noindent Let us define,
\begin{equation*}\label{E_G}
	\begin{split}
		E(v,t,s) &:= \exp \Big\{ -\int_{s}^{t} \nu(V(\tau; t, x, v)) \Big\}.  \\
	\end{split}
\end{equation*}
Along the backward in time trajectory,
\begin{equation}\begin{split}\notag
		&\frac{d}{d s} \Big( E(v,t,s) h (s,X(s;t,x,v),V(s;t,x,v)) \Big) \\
		&=  {E(v,t,s) } \big[ K_w h  \big](s,X(s;t,x,v),V(s;t,x,v)).\end{split}
\end{equation}
By integrating from $0$ to $t$, we obtain the first Duhamel expansion,
\begin{equation} \begin{split}\label{Duhamel_once}
		h (t,x,v) =& E(v,t,0) h_0(X(0), V(0))\\  
		&+ \int^{t}_{0} E(v,t,s) \int_{\R^{3}} k_{w}(u,V(s)) h(s,  X(s;t,x,v),u)   \dd u  \dd s, \end{split}
\end{equation}
where we abbreviated $X(s) = X(s;t,x,v)$. Recall the standard estimates (see Lemma 4 and Lemma 5 in \cite{GKTT2017})
\begin{equation}\label{est_kw}
	\int_{\R^{3}} |k_{w} (v,u)| \dd u \leq C_{K}\langle v\rangle^{-1}. 
\end{equation}

\noindent From double and triple iteration, we get
\begin{eqnarray} 
	&&h (t,x,v) \notag \\ 
	&=& E(v,t,0) h_0(X(0), V(0)) \notag  \\
	&&+ \int^{t}_{0} E(v,t,s) \int_{u} k_{w}(u, V(s)) E(u,s,0) h_0(X(0), V(0)) duds  \notag  \\
	&& + \int^{t}_{0}  E(v,t,s) \int_{u} k_{w}(u, V(s)) 
	\int^{s}_{0} E(u,s,s^{\prime}) \notag \\
	&&\quad \times\int_{u^{\prime}}  k_{w}(u^\prime, V(s^{\prime})) h(s ^{\prime} , 
	X(s^{\prime}) ,u^{\prime} )   du^{\prime}  ds^{\prime}  du  d s   \notag  \\
	&=& E(v,t,0) h_0(X(0), V(0)) \notag \\
	&&+ \int^{t}_{0} E(v,t,s) \int_{u} k_{w}(u,V(s)) E(u,s,0) h_0 (X(0), V(0))  duds \notag \\
	&& + \int^{t}_{0}  E(v,t,s) \int_{u} k_{w}(u,V(s)) 
	\int^{s}_{0} E(u,s,s^{\prime}) \notag\\
&&\quad\times 	 \int_{u^{\prime}}
	k_{w}(u^\prime, V(s^{\prime})) E(u^{\prime}, s^{\prime}, 0) h_0(X(0), V(0)) du^{\prime}  ds^{\prime}  du  ds \notag  \\
	&& + \int^{t}_{0}  E(v,t,s) \int_{u} k_{w}(u, V(s)) 
	\int^{s}_{0} E(u,s,s^{\prime}) \int_{u^{\prime}}
	k_{w}(u^\prime, V(s^{\prime})) E(u^{\prime}, s^{\prime}, s^{\prime\prime}) \quad \quad  \label{triple}   \\
	&&\quad\times \int_{u^{\prime\prime}} k_{w}(u^{\prime\prime}, V(s^{\prime\prime})) h(s^{\prime\prime}, X(s^{\prime\prime}), u^{\prime\prime})  du^{\prime\prime}  ds^{\prime\prime} du^{\prime}  ds^{\prime}  du  ds, \notag 
\end{eqnarray}
where we abbreviated notations (similar for $V$),
\begin{equation}\label{XV prime}
	\begin{split}
		X(s) &:= X(s;t,x,v) ,\quad X(s^{\prime}) := X^{\prime}(s^{\prime};s,X(s;t,x,v),u),  \\
		X(s^{\prime\prime})  &:= X(s^{\prime\prime};s^{\prime}, X^{\prime}(s^{\prime};s,X(s;t,x,v),u), u^{\prime}).
	\end{split}
\end{equation}

Now, we choose $m({N})$ so that 
\begin{equation} \label{opeator k split}
	k_{w,m}(u,v) := \mathbf{1}_{\{ |u-v|\geq\frac{1}{m}, \ |u|\leq m \}} k_{w}(u,v) ,
\end{equation}
satisfies $\int_{\mathbb{R}^3} |k_{w,m}(u,v)-k_{w}(u,v)| \ \dd u \leq \frac{1}{{N}}$ for sufficiently large ${N}\geq 1$. From \eqref{opeator k split} and uniform lower bound of collision frequency, we can rewrite \eqref{triple} as
\begin{align}
	& h(t,x,v) \notag  \\
	&= E(v,t,0) h_0(X(0), V(0))  \notag \\
	&\quad + \int^{t}_{0} E(v,t,s) \int_{u\in \VN} k_{w,m}(u, V(s)) E(u,s,0) h_0 ( X(0), V(0))  duds \notag \\
	&\quad + \int^{t}_{0}  E(v,t,s) \int_{u\in \VN} k_{w,m}(u, V(s)) 
	\int^{s}_{0} E(u,s,s^{\prime}) \notag\\
	& \quad \quad \times  \int_{u^{\prime}}
	k_{w}(u^\prime, V(s^{\prime})) E(u^{\prime}, s^{\prime}, 0) h_0(X(0), V(0)) du^{\prime}  ds^{\prime}  du  ds \notag  \\
	&\quad + \int^{t}_{0}  E(v,t,s) \int_{u} k_{w,m}(u, V(s)) 
	\int^{s}_{0} E(u,s,s^{\prime}) \int_{u^{\prime}\in \VN}
	k_{w,m}(u^\prime, V(s^{\prime})) E(u^{\prime}, s^{\prime}, s^{\prime\prime}) \notag\\
	&\quad\quad \times  \int_{u^{\prime\prime}\in \VN} k_{w,m}(u^{\prime\prime}, V(s^{\prime\prime})) h(s^{\prime\prime}, X(s^{\prime\prime}), u^{\prime\prime})   \notag \\
	&\quad + O_{\O}(\frac{1}{N}) \sup_{0\leq s \leq t}\|h(s)\|_{L^{\infty}} \notag \\
	&\lesssim e^{-\frac{\nu_0}{2}t}\|h_0\|_{L^{\infty}} + O_{\O}(\frac{1}{N}) \sup_{0\leq s \leq t}\|h(s)\|_{L^{\infty}} \notag \\
	&\quad + \int^{t}_{0}  E(v,t,s) \int_{u} k_{w,m}(u, V(s)) 
	\int^{s}_{0} E(u,s,s^{\prime}) \int_{u^{\prime}\in \VN}
	k_{w,m}(u^\prime, V(s^{\prime})) E(u^{\prime}, s^{\prime}, s^{\prime\prime})  \notag \\
	&\quad\quad \times  \int_{u^{\prime\prime}\in \VN} k_{w,m}(u^{\prime\prime}, V(s^{\prime\prime})) h(s^{\prime\prime}, X(s^{\prime\prime}), u^{\prime\prime}).   \label{main}  
\end{align}

\noindent Let us analyze \eqref{main}. 
\hide
{\color{red}REMOVE(?) Using Theorem \ref{bad set thm}
\begin{equation}
	\begin{split}
		&\exists i_{s} \in \{ 1,2,\cdots, l_{IB}\} \quad\text{such that}\quad X(s) \in B(x_{i_{s}}, r(x_{i_{s}},\varepsilon))  \\
		&\exists j_{s,s^{\prime}} \in \{ 1,2,\cdots, l_{IB}\} \quad\text{such that}\quad X(s^{\prime}; s, X(s;t,x,v), u) \in \mathcal{V}_{j_{s,s^{\prime}}}^{\mathcal{IB}} B(x_{j_{s,s^{\prime}}}, r(x_{j_{s,s^{\prime}}},\varepsilon)) 
	\end{split}
\end{equation}
} 
\unhide
We define following sets for fixed $\delta, n,\vec{n},k,m,\vec{m}, k^{\prime}$, where Proposition \ref{prop_full_rank} does not hold,
{\small
\begin{equation*}
	\begin{split}
		R_1 &:= \{ u \ \vert \ 
		(X(s;t,x,v), u) \notin \mathcal{Z}^{\varepsilon} \ \text{and} \ u\notin  B( \vec{n}\delta ; 2\delta )
		\} ,    \\
		R_2 &:= \{ s^{\prime} \ \vert \ |s-s^{\prime}| \leq \delta \} , \\
		R_3 &:= \{ s^{\prime} \ \vert \ | s^{\prime} -  \psi_{1}^{n,\vec{n},k,m,\vec{m},k^{\prime}} ( n \delta, X(n \delta;t,x,v), \vec{n}\delta ) |\lesssim_{N} \delta  \|\psi_{1} \|_{C^{1}} \} ,  \\
		R_4 &:= \{ u^{\prime} \ \vert \ 
		(X(s'; s, X(s;t,x,v),u), u^{\prime}) \notin \mathcal{Z}^{\varepsilon} \ \text{and} \ u^{\prime}\notin  B( \vec{m}\delta ; 2\delta )
		\} ,   \\
		R_5 &:= \{ u^{\prime} \ \vert \ 
	\Big|\Big( \frac{\p_{y_{1}} Y_{X(s^{\prime};s, X(s;t,x,v),u)}(0,0)  \times  \p_{y_{2}} Y_{X(s^{\prime};s, X(s;t,x,v),u)}(0,0) }{|\p_{y_{1}} Y_{X(s^{\prime};s, X(s;t,x,v),u)}(0,0)  \times  \p_{y_{2}} Y_{X(s^{\prime};s, X(s;t,x,v),u)}(0,0) |}\Big) \cdot 
	u^{\prime}  \Big| \leq \delta
		\} ,   \\
		R_6 &:= \{ s^{\prime\prime} \ \vert \ |s^{\prime}-s^{\prime\prime}| \leq \delta \} ,  \\
		R_7 &:= \{ s^{\prime\prime} \ \vert \  \min_{r=1,2} | s^{\prime\prime} -  \psi_{r}^{n,\vec{n},k,m,\vec{m},k^{\prime}} ( m \delta, X( m \delta; n \delta, X(n \delta; t,x,v ), \vec{m}\delta  ), \vec{n}\delta  ) | \lesssim_{N} \delta \min_{r=1,2} \|\psi_{r} \|_{C^{1} } \}. \\
	\end{split}
\end{equation*}
}
Therefore, \eqref{main} can be controlled by  
\begin{align} 
	& \lesssim \sum_{n =0}^{[t/\delta]+1} \sum_{|\vec{n}| \leq N}  \sum_{m=0}^{[t/\delta]+1} \sum_{|\vec{m}| \leq N} \sum_{k=1}^{M} \sum_{k^{\prime}=1}^{M} \int^{(n+1) \delta}_{ (n-1) \delta} \ \int^{t^{k}-\delta}_{t^{k+1}+\delta} \int^{t^{k^{\prime}}-\delta}_{t^{k^{\prime}+1}+\delta} e^{-{\nu_0}(t-s^{\prime\prime})}   \notag \\ 
	& \quad\quad  \times \int_{ u, u^{\prime}, u^{\prime\prime} \in \VN} \ | h( s^{\prime\prime},X(s^{\prime\prime}), u^{\prime\prime} ) |  \ \mathbf{1}_{R_{1}^{c} \cap R_{2}^{c} \cap R_{3}^{c} \cap R_{4}^{c} \cap R_{5}^{c} \cap R_{6}^{c}} \label{main2} \\
	& \quad + B + R.  \notag
\end{align}
Here $B$ term corresponds to when the trajectory locates near bouncing points and $R$ corresponds to where $(u, s^{\prime}, u^{\prime}, s^{\prime\prime})$ belong to at least one of $R_1\sim R_7$. Exact definition and estimates for $B$ and $R$ are given as following
\begin{equation} \label{BR}
	\begin{split}
		B &\leq \int_0^t \int_0^s \int_0^{s^{\prime}} e^{- {\nu_0}(t-s^{\prime\prime})} \int_{u\in\VN} k_{w,m}(u,v) \int_{u^{\prime}\in\VN} k_{w,m}(u^{\prime},u)   \\
		& \quad\quad \times \int_{u^{\prime\prime}\in \VN} k_{w,m}(u^{\prime\prime},u^{\prime}) h (s^{\prime\prime},X^{\prime\prime}(s^{\prime\prime}),u^{\prime\prime})   \mathbf{1}_{|s^{\prime}-t^{k}|\leq \delta  \ \text{or} \ |s^{\prime\prime}-t^{k^{\prime}}|\leq \delta }  \\
		&\leq C_{N}\delta \sup_{0\leq s \leq t} \|h(s)\|_\infty,  \\
		R &\leq \int_0^t \int_0^s \int_0^{s^{\prime}} e^{- {\nu_0}(t-s^{\prime\prime})} \int_{u\in \VN} k_{w,m}(u,v) \int_{ u^{\prime}\in\VN} k_{w,m}(u^{\prime},u)   \\
		& \quad\quad \times \int_{ u^{\prime\prime}\in\VN} k_{w,m}(u^{\prime\prime},u^{\prime}) h (s^{\prime\prime},X^{\prime\prime}(s^{\prime\prime}),u^{\prime\prime}) \ \mathbf{1}_{ R_1 \cup R_2 \cup R_3 \cup R_4 \cup R_5 \cup R_6 \cup R_7 }  \\
		&\leq C_{N} \delta \sup_{0\leq s \leq t} \|h(s)\|_\infty.  \\
	\end{split}
\end{equation}
	In particular, for about $R$ estimate, deriving smallness $(\delta+\varepsilon)$ from $R_{2}, R_{3}, R_{5}, R_{6}, R_{7}$ is easy. To derive smallness from $R_{1}$ and $R_{4}$, we use \eqref{Ze c small} in Theorem \ref{bad set thm} to obtain
\[
	\int_{u} k_{w,m}(u,V(s)) \mathbf{1}_{R_{1}} \lesssim \delta,\quad 
	\int_{u^{\prime}} k_{w,m}(u^{\prime},V(s^{\prime})) \mathbf{1}_{R_{4}} \lesssim \delta,
\]
choosing $\varepsilon$ sufficiently small depending on $N \gg 1$ in \eqref{Ze c small}.  \\
For \eqref{main2}, we are away from two sets $B$ and $R$. Under the condition of $(u, s^{\prime}, u^{\prime}, s^{\prime\prime}) \in R_{1}^{c} \ \cap \ R_{2}^{c} \ \cap R_{3}^{c} \ \cap R_{4}^{c} \ \cap R_{5}^{c} \ \cap R_{6}^{c} \ \cap R_{7}^{c}$, indices $n, \vec{n}, k, m, \vec{m}, k^{\prime}$ are determined so that
\begin{equation} \label{XuXu}
\begin{split}
	t \ &\in \ [  (n-1)\delta,  (n+1)\delta ],   \\
	(X(s;t,x,v), u) \ &\in \ \mathcal{Z}^{\varepsilon},  \\
	u \ &\in \ B(\vec{n}\delta;2\delta), \\
	(X(s^{\prime};s,X(s;t,x,v),u), u^{\prime}) \ &\in \ \mathcal{Z}^{\varepsilon}, \\
	u^{\prime} \ &\in \ B(\vec{m}\delta;2\delta),  \\
	\Big|\Big( \frac{\p_{y_{1}} Y_{X(s^{\prime};s, X(s;t,x,v),u)}(0,0)  \times  \p_{y_{2}} Y_{X(s^{\prime};s, X(s;t,x,v),u)}(0,0) }{|\p_{y_{1}} Y_{X(s^{\prime};s, X(s;t,x,v),u)}(0,0)  \times  \p_{y_{2}} Y_{X(s^{\prime};s, X(s;t,x,v),u)}(0,0) |}\Big) \cdot 
	u^{\prime}  \Big| &\geq \delta.
\end{split}
\end{equation}
We can apply Proposition \ref{prop_full_rank} which gives local time-independent lower bound of 
\[
\Big| \det(\frac{\p(X(s^{\prime\prime}))}{\p(|u^{\prime}|, \zeta_1, \zeta_2)}) \Big| \ \geq \ \epsilon^{\prime}_{\delta } .
\]
Note that $\{ \zeta_{1}, \zeta_{2}\} \subset\{ |u|, \hat{u}_{1}, \hat{u}^{\prime}_{1},  \hat{u}^{\prime}_{2}  \}$ are chosen variables in Theorem \ref{prop_full_rank} and $\{ \zeta_{3}, \zeta_{4}\} \subset\{ |u|, \hat{u}_{1}, \hat{u}^{\prime}_{1},  \hat{u}^{\prime}_{2}  \}$ are unchosen variables. Let us use $\mathcal{P}$ to denote projection of 
\[
	B(\vec{n}\delta;2\delta) \cap (\mathcal{Z}^{\varepsilon})_{X(s;t,x,v)}  \times
	B(\vec{m}\delta;2\delta) \cap (\mathcal{Z}^{\varepsilon})_{X(s^{\prime};s,X(s;t,x,v),u)},
\]
 into $\R^3$ which corresponds to $(|u^{\prime}|, \zeta_1, \zeta_2)$ components. Note that both $(\mathcal{Z}^{\varepsilon})_{X(s;t,x,v)}$ and $(\mathcal{Z}^{\varepsilon})_{X(s^{\prime};s,X(s;t,x,v),u)} $ are defined using definition \eqref{proj_v}. If we choose sufficiently small $\delta$, there exist small $r_{\delta,n,\vec{n}, k,m,\vec{m}, k^{\prime}}$ such that there exist one-to-one map $\mathcal{M}$, 
\begin{eqnarray*}
	\mathcal{M} &:& \mathcal{P} \Big( B(\vec{n}\delta;2\delta) \cap (\mathcal{Z}^{\varepsilon})_{X(s;t,x,v)}  \times
	B(\vec{m}\delta;2\delta) \cap (\mathcal{Z}^{\varepsilon})_{X(s^{\prime};s,X(s;t,x,v),u)}  \Big)  \\
	&& \quad\quad \mapsto
	B( X(s^{\prime\prime}; s^{\prime}, X(s^{\prime};s, X(s;t,x,v), u), u^{\prime}), r_{\delta,n,\vec{n}, k,m,\vec{m}, k^{\prime}} ).
\end{eqnarray*} 
Now, we perform change of variable for  in \eqref{main2} to obtain

\begin{equation} \label{MAIN}
	\begin{split}	
		\text{\eqref{main2}} &\leq  \sum_{n =0}^{[t/\delta]+1} \sum_{|\vec{n}| \leq N}  \sum_{m=0}^{[t/\delta]+1} \sum_{|\vec{m}| \leq N} \sum_{k=1}^{M} \sum_{k^{\prime}=1}^{M} \int^{(n+1) \delta}_{ (n-1) \delta} \ \int^{t^{k}-\delta}_{t^{k+1}+\delta} \int^{t^{k^{\prime}}-\delta}_{t^{k^{\prime}+1}} e^{-{\nu_0}(t-s^{\prime\prime})}  \\ 
		& \quad\quad \times \int_{u^{\prime\prime}} \dd u^{\prime\prime} \int_{\hat{u}_2, \zeta_3, \zeta_4} \ \mathbf{1}_{ |u|\leq N, |u^{\prime}|\leq N, |u^{\prime\prime}|\leq N } \ \dd \hat{u}_2 \dd \zeta_{3} \dd \zeta_{4} \\ 
		& \quad\quad \times \int_{ |u^{\prime}|, \zeta_1, \zeta_2 } \ \dd |u^{\prime}| \dd \zeta_{1} \dd \zeta_{2} \ 
		| h( s^{\prime\prime},X( s^{\prime\prime}), u^{\prime\prime} ) | \  \dd s \dd s^{\prime} \dd s^{\prime\prime}  
		\\
		&\leq  \sum_{n =0}^{[t/\delta]+1} \sum_{|\vec{n}| \leq N}  \sum_{m=0}^{[t/\delta]+1} \sum_{|\vec{m}| \leq N} \sum_{k=1}^{M} \sum_{k^{\prime}=1}^{M} \int^{(n+1) \delta}_{ (n-1) \delta} \ \int^{t^{k}-\delta }_{t^{k+1}+\delta } \int^{t^{k^{\prime}}-\delta }_{t^{k^{\prime}+1}+\delta } e^{-{\nu_0}(t-s^{\prime\prime})}  \\ 
		& \quad\quad \times \int_{\hat{u}_2, \zeta_3, \zeta_4} \ \mathbf{1}_{ |u|\leq N, |u^{\prime}|\leq N, |u^{\prime\prime}|\leq N } \ \dd \hat{u}_2 \dd \zeta_{3} \dd \zeta_{4}    \\
		& \quad\quad \times \int_{u^{\prime\prime}} \int_{ 	B( X(s^{\prime\prime}), r_{\delta,n,\vec{n},i,k,m,\vec{m},j,k^{\prime}} ) } \ 
		| h ( s^{\prime\prime}, x, u^{\prime\prime} ) | \
		\frac{1}{ \epsilon^{\prime}_{\Omega, N, \varepsilon, \delta } } \dd x \dd u^{\prime\prime} \dd s \dd s^{\prime} \dd s^{\prime\prime}  
		\\
		&\leq \ C_{\O, N, \varepsilon, \delta} \int_{0}^{t} e^{-{\nu_0}(t-s^{\prime\prime})} \ \int_{\O}\int_{|u^{\prime\prime}|\leq N}  |h(s^{\prime\prime},x,u^{\prime\prime})| \  \dd u^{\prime\prime} \dd x \dd s^{\prime\prime}  \\
		&\leq \ C_{t, \O, N, \varepsilon, \delta} \int_{0}^{t}  \| h(s) \|_{L^{2}_{x,v}} ds.  \\
	\end{split}	
\end{equation}

We collect \eqref{MAIN}, \eqref{main2}, \eqref{BR}, and choose sufficiently large $N$ and (corresponding) sufficiently small $\varepsilon, \delta$ to obtain
\Be \label{h infty est}
\|h(t)\|_{L^{\infty}} \lesssim e^{-\frac{\nu_0}{2}t}\|h(0)\|_{L^{\infty}} + C_{t}\int_{0}^{t} \|h(s)\|_{L^{2}} ds.
\Ee

\section{Linear and nonlinear decay : Proof of main theorem}
This part is also standard and similar as \cite{cylinder} and \cite{Guo10}. Let us briefly explain key steps. We first start with linear $L^{2}_{x,v}$ decay.
\subsection{Linear $L^{2}$ decay}
We use a standard coercivity estimate (Proposition 1.4 in \cite{KimLee}):
\begin{proposition}\label{prop:coercivity}Let $f$ solves linearized Boltzmann equation \eqref{lin eq} with the specular BC. Furthermore, we assume \eqref{normalization}. Then there exists $C>0$ such that for all $N \in \mathbb{N}$, 
	\Be\label{coerciveN}
	\int^{N+1}_N \| \mathbf{P} f (t) \|_2^2 \dd t \leq  C \int^{N+1}_{N} 
	\| ( \mathbf I - \mathbf{P} ) f(t) \|_\nu^2 \dd t.
	\Ee
	\end{proposition} 
\begin{remark}
In \cite{Guo10}, Guo obtained the coercivity estimate \eqref{coerciveN} by using contradiction method. When we consider a periodic cylindrical domain with an annulus cross-section, we derive the same coercivity estimate through a constructive method in Appendix, Proposition \ref{test coercivity}.
\end{remark}
\begin{corollary}\label{decay_U}Assume the same conditions of Proposition \ref{prop:coercivity}. Then we have following estimate
	\begin{equation}\label{U_decay}
		\sup_{0 \leq t}e^{\lambda t} 	\|f(t)\|_{L^{2}_{x,v}}^{2} \lesssim \| f_{0} \|_{L^{2}_{x,v}}^{2},
	\end{equation}
	for some $\lambda>0$.
\end{corollary}
\begin{proof} 
	From Green's identity (or $L^{2}$ energy estimate) for (\ref{lin eq}) in a time interval $[0,N]$, 
	\begin{equation*}\label{L2_f} 
		\|f(N)\|_{2}^{2} + \int^{N}_{0} \iint_{ U\times\R^{3}}  f Lf 
		\leq \|f(0)\|_{2}^{2}  .
	\end{equation*}
	
	\noindent From \eqref{lin eq}, for any $\lambda>0$, we get
	\begin{equation*}\label{eqtn_lamda}
		\big[\p_{t}  + v\cdot \nabla_{x} \big] (e^{\lambda t} f)  + L (e^{ \lambda t}f )
		= \lambda e^{\lambda t} f,
	\end{equation*}
	Green's identity (or $L^{2}$ energy estimate) yields 
	\begin{equation}\label{energy_f}  
		\| e^{\lambda t}f(N)\|_{2}^{2} + \underbrace{\int^{N}_{0} \iint_{ U\times\R^{3}}  e^{2\lambda s} f Lf } _{(I)}
		-  
		{
			\lambda \int^{N}_{0} \iint_{ U\times\R^{3}}
			|e^{\lambda s} f(s)|^{2}}
		\leq \|f(0)\|_{2}^{2} 
		. 
	\end{equation}
	
	\noindent Let us consider $(I)$ in (\ref{energy_f}). From semi-positivity of linearized Boltzmann operator $L$, $(I)$ is bounded from below as 
	\begin{eqnarray*}
		%
		(I)\geq  
		\delta_{L} 
		\int^{N}_{0} \iint_{ U\times\R^{3}} \langle v\rangle | e^{ \lambda s}(\mathbf{I} - \mathbf{P})f|^{2} \geq  \delta_{L} \int^{N}_{0} \| e^{ \lambda s} (\mathbf{I} - \mathbf{P}) f  \|_{\nu}^{2}.
	\end{eqnarray*}
	By time translation, we apply coercivity estimate \eqref{coerciveN} to obtain
	\begin{equation*} \label{I}
		\begin{split} 
			(I)
			&\geq   \frac{\delta_{L} }{2} \int^{N}_{0} \| e^{ \lambda s}(\mathbf{I} - \mathbf{P}) f \|_{\nu}^{2}  + \frac{\delta_{L} }{2 C} \int^{N}_{0} \|  e^{ \lambda s} \mathbf{P} f \|_{2}^{2}  \geq  \frac{\delta_{L} }{2 C} \int^{N}_{0} \|  e^{ \lambda s}  f \|_{2}^{2},
		\end{split}
	\end{equation*}
	where $\|\cdot\|_{\nu} = \|\cdot \sqrt{\nu}\|_{L^{2}_{x,v}}$. Therefore, we derive 
	\begin{equation}\label{energy_N0} 
		e^{2\lambda N}\| f(N)\|_{2}^{2} 
		+ \Big(
		\frac{\delta_{L} }{2C} - \lambda
		\Big)
		\int^{N}_{0} \| e^{ \lambda s} f (s)\|_{2}^{2} ds 
		\leq \| f_0\|_{2}^{2}. 
	\end{equation}
	
	\noindent For a time interval $[N,t]$, it is not hard to get 
	\begin{equation}\label{energy_tN}
		\| f (t) \|_{2}^{2}  \leq \| f(N) \|_{2}^{2}.
	\end{equation}
	
	\noindent Choosing $\lambda \ll 1$, from (\ref{energy_N0}) and (\ref{energy_tN}), we conclude that 
	\begin{equation*}
		e^{\lambda t } \| f(t) \|_{2}^{2} = e^{\lambda (t-N)} e^{\lambda N }\| f(N) \|_{2}^{2}
		\leq 2 \| f_0 \|_{2}^{2},
	\end{equation*}
	and obtain (\ref{U_decay}).   
\end{proof}

\subsection{Nonlinear $L^{\infty}$ decay}
Now, we are ready to derive nonlinear $L^{\infty}_{x,v}$ decay from previous $L^{2}_{x,v}$ decay.
\begin{proof} [\textbf{Proof of Theorem \ref{main theorem}}] 
	From \eqref{h infty est}, for $T \leq t < T+1$,
	\begin{equation*} 
		\begin{split}
			\sup_{s\in[T,t]} \|h(s)\|_\infty & \lesssim \ e^{-\nu_{0} (t-T) } \|h (T)\|_{\infty} +  \int_{T}^{t} \|f (s)\|_2 \dd s  .
		\end{split}
	\end{equation*}
	
	We assume that $m \leq t < m+1$ and define $\lambda^{*}:=\min\{ \nu_{0}, \lambda\}$, where $\lambda$ is some constant from Corollary~\ref{decay_U}. We use (\ref{h infty est}) repeatedly for each time step, $[k,k+1), \ k\in\mathbb{N}$ and Corollary~\ref{decay_U} to perform $L^{2}-L^{\infty}$ bootstrap,
	\begin{equation*} \label{h infty decay}
		\begin{split}
			\|h(t)\|_{\infty} 
			&\lesssim e^{-m \nu_{0} } \|h(0)\|_{\infty} + \sum_{k=0}^{m-1} e^{-k \nu_{0} } \int_{m-1-k}^{m-k} \|f(s)\| \dd s  \\
			&\lesssim e^{-m \nu_{0} } \|h(0)\|_{\infty} + \sum_{k=0}^{m-1} e^{-k  \nu_{0} } \int_{m-1-k}^{m-k} e^{-\lambda(m-1-k)} \|f(0)\| \dd s  \lesssim  e^{- \lambda^{*}  t } \|h(0)\|_{\infty} . \\
		\end{split}
	\end{equation*} 
	
	For nonlinear problem from Duhamel principle,
	\begin{equation} \label{Duhamel}
		\begin{split}
			h  &:= U(t)h_0 + \int_{0}^{t} U(t-s) w  \Gamma(\frac{h }{w},\frac{h }{w}) (s)  \dd s ,  \\
			\|h (t) \|_{\infty} &\lesssim e^{- \lambda^{*} t } \|h(0)\|_{\infty} + \Big\| \int_{0}^{t}  U(t-s) w \Gamma\big( \frac{h }{w}, \frac{h }{w} \big)(s) \dd s \Big\|_{\infty} , 
		\end{split}
	\end{equation}
	where $U(t)$ is a semi-group for linearized Boltzmann equation. Using Duhamel's principle again, we have
	\begin{equation*} \label{double duhamel}
		U(t-s) = G(t-s) + \int_s^t G(t-s_1) K_{w}U(s_1-s) \dd s_1,
	\end{equation*}
	where $G(t)$ is semi-group for the system
	\begin{equation*} \label{step 2 eq}
		\begin{split}
			&\p_t h + v\cdot\nabla_x h  +  \nu h = 0,\quad \text{and}\quad |G(t)h_0| \leq e^{-\nu_{0}t} |h_0|.
		\end{split}
	\end{equation*}
	
	\noindent For the last term in (\ref{Duhamel}), using standard nonlinear estimate $|w\Gamma(f,f)| \lesssim \langle v \rangle \|wf\|_{\infty}^{2}$,
	we obtain (see \cite{Guo10, cylinder, KimLee} for detail)
	{\small
	\begin{equation*} \label{nu increase}
		\begin{split}
			& \Big\| \int_{0}^{t}  U(t-s) w \Gamma\big( \frac{h }{w}, \frac{h }{w} \big)(s) \dd s \Big\|_{\infty}   \\
			&\leq  \Big\| \int_{0}^{t}  G(t-s) w  \Gamma\big( \frac{h }{w}, \frac{h }{w} \big)(s) \dd s \Big\|_{\infty}   +   \Big\| \int_{0}^{t}  \int_{s}^{t} G(t-s_1) K_{w}  U(s_1-s) w \Gamma\big( \frac{h }{w}, \frac{h }{w} \big)(s) \dd s_1 \dd s  \Big\|_{\infty}   \\
			&\leq 
			C e^{- \lambda^{*} t} \Big(\sup_{0\leq s\leq \infty} e^{ \lambda^{*} s} \|h (s)\|_{\infty} \Big)^{2}.
		\end{split}
	\end{equation*}	
	}
	\noindent Therefore, for sufficiently small $\|h_0\|_{\infty} \ll 1$, we have uniform bound
	\begin{equation*} \label{uniform bound} 
		\sup_{0\leq t\leq \infty} e^{\lambda^{*} t } \|h(t)\|_\infty \ll 1,
	\end{equation*} 
	and hence global decay and uniqueness. Also note that positivity of $F$ is standard by sequence $F^{\ell}$ which solves
	\begin{equation}\notag
		\begin{split}
			&\p_{t } F^{\ell+1} + v\cdot \nabla F^{\ell+1} = Q_{+} (F^{\ell},F^{\ell}) - \nu(F^{\ell}) F^{\ell+1},  \ \ F|_{t=0} = F_{0}, \\
			&F^{\ell+1} (t,x,v) = F^{\ell+1} (t,x,R_{x}v) \ \ \ \text{on}  \ \  \p U.
		\end{split}
	\end{equation} 
	From $F_{0} \geq 0$ and $ F^{\ell} \geq 0$, we have $F^{\ell+1} \geq 0$.	
\end{proof}

\section{Appendix: $L^2$ coercivity in an axisymmetric domain}
In Appendix, we consider the linearized hard-sphere Boltzmann equation 
	\begin{equation} \label{linear BE}
		\p_{t}f + v\cdot\nabla_{y} f + Lf = g,
	\end{equation}
	with specular boundry condition 
	\begin{align} \label{specular}
	f(t,x,R_{y}v) = f(t,y,v) \text{ on } y\in\p U,
	\end{align}
	where $R_y:=(\mathbf{I}-2n(y)\otimes n(y))$. We consider a periodic cylindrical domain with annulus cross-section:
	\begin{align} \label{domain}
		U:= \O \times [0,H] \subset \R^3, \quad \p U: =\p \O \times (0,H), 
	\end{align}
	where $\O:=\{(x,y) \in \R^2: 0<R_1<x^2+y^2<R_2\}$ is an annulus and $[0,H]$ is a periodic interval. It can be parametrized by using a global triply orthogonal coordinate system: for $y\in U$,
	\begin{align} \label{ann domain}
		y = \eta(\theta,z, r) = (r\cos \theta, r \sin \theta, z)=(\eta_1(\theta,z,r),\eta_2(\theta,z,r),\eta_3(\theta,z,r)), 
	\end{align}
	where $0<R_{1}< r< R_{2},\; 0 \leq \theta \leq 2\pi,\; 0\leq z \leq H$.\\
	
	The main object of Appendix is to control the physical macroscopic quantities:
	\begin{align*}
		a(t,y)&= \int_{\R^3} f(t,y,v) \sqrt{\mu} dv,\\ 
		b(t,y)&= \int_{\R^3} vf(t,y,v) \sqrt{\mu} dv, \\ 
		c(t,y)&= \int_{\R^3} \frac{\vert v \vert^2 -3 }{2} f(t,y,v) \sqrt{\mu} dv. 
	\end{align*} 
	We provide a constructive proof of the coercivity estimate in the presence of the specular reflection boundary condition. Firstly, we introduce the mass and energy ($a$ and $c$)  estimates in Proposition \ref{prop_massenergy}. On the other hand, when it comes to the momentum component $b$, we use the global triply orthogonal coordinate system \eqref{ann domain} to transform the momentum part such as \eqref{b_xi}. Then, we divide it into the tangential momentum and normal momentum. Unlike mass and energy, momentum is not fully conserved, so it needs to be handled using the delicate geometric method. The proof consists of construction of correction term $\zeta$ to derive the Laplace-Beltrami operator \eqref{Beltrami}. 
	
	\subsection{Transformed equations and macroscopic parts} \label{transform}  
	In Appendix, we only consider the periodic cylindrical domain with an annulus cross-section. However, Lemma \ref{identities}, Lemma \ref{D bel difference}, and Lemma \ref{dv/dx} in section \ref{transform} hold for a domain with global triply orthogonal coordinate system of a single chart as well. \\
\indent Within Appendix, we use variables $(\theta,z,r)$ as $x_1,x_2$ and $x_3$ respectively. To distinguish derivatives, we use the following derivative symbols.
		\[
		\p_{y_{i}} := \frac{\p}{\p y_{i}},\quad \p_{i} := \frac{\p}{\p x_{i}},\quad \nabla_y := (\p_{y_{1}}, \p_{y_{2}}, \p_{y_{3}}),\quad \nabla := (\p_{1}, \p_{2}, \p_{3}).
		\]
		We also define
		\begin{equation} \label{diff operator}
			\begin{split}
				g_{ij} := \langle \p_{i}\eta, \p_{j}\eta \rangle, \quad D_{i} := \frac{1}{\sqrt{g_{ii}}} \p_{i},\quad D:= (D_{1}, D_{2}, D_{3}).
			\end{split}
		\end{equation}
		In the periodic cylindrical domain with annulus cross-section, it is directly computed that
		\begin{align} \label{g_ij}
			g_{11}=r^2, \quad g_{22}=g_{33}=1, \; \text{and} \quad g_{ij}=0, \quad \text{whenever } i \neq j.
		\end{align}
  		We define
		\begin{equation} \label{v_x}
			\V := 
			\begin{bmatrix}
				\V_{{1}}\\
				\V_{{2}}\\
				\V_{{3}}
			\end{bmatrix}
			=
			Q^{T}
			\begin{bmatrix}
				v_{{1}}\\
				v_{{2}}\\
				v_{{3}}
			\end{bmatrix},\quad
			Q := \begin{bmatrix}
				D_{1}\eta_{1} & D_{2}\eta_{1} & D_{3}\eta_{1}  \\
				D_{1}\eta_{2} & D_{2}\eta_{2} & D_{3}\eta_{2}  \\
				D_{1}\eta_{3} & D_{2}\eta_{3} & D_{3}\eta_{3}  \\
			\end{bmatrix},
		\end{equation}
		where $Q$ is orthonormal matrix, $QQ^{T} = Q^{T}Q = I$.  \\
		
		Given the global triply orthogonal coordinate system \eqref{ann domain}, we define transformed local mass $a_{x}$ and local energy $c_{x}$ by,
		\begin{equation} \label{ac_x}
		a_{x} := a\circ\eta \quad\text{and} \quad c_{x} := c\circ\eta. 
		\end{equation} 
		For the local momentum vector $b_{x}$, it is crucial to take account of geometry: we define
		\begin{equation}  \label{b_xi}
		\begin{split}
		b_{x_{i}} &:= \int_{\mathbb{R}^{3}}  \V_{i} f \sqrt{\mu(\V)} \dd \V  
		= \int_{\mathbb{R}^{3}}  D_{i}\eta\cdot v f \sqrt{\mu(\V)} \dd \V 
		= D_{i}\eta \cdot b.
		\end{split}
		\end{equation}
		Therefore, we get that, as \eqref{v_x},
		\begin{equation} \label{b_x}
		\begin{bmatrix}
		b_{x_{1}}\\
		b_{x_{2}}\\
		b_{x_{3}}
		\end{bmatrix}
		=
		Q^{T}
		\begin{bmatrix}
		b_{{1}}\\
		b_{{2}}\\
		b_{{3}}
		\end{bmatrix}.
		\end{equation}
		As the mapping in the velocity space is isometric (see $Q^{T}$ is an orthonormal matrix), we note that 
		\begin{equation} \label{isometry}
		|\V| = |v|,\quad |b_{x}| = |b|,\quad \mu = \frac{1}{\sqrt{(2\pi)^3}}e^{-\frac{|v|^{2}}{2}} =\frac{1}{\sqrt{(2\pi)^3}} e^{-\frac{|\V|^{2}}{2}}.
		\end{equation}

		In the coordinates, now we rewrite the transport operator in terms of $\V$ and local derivatives $\nabla$. Since
		\begin{equation}
		\begin{split}\notag
		\p_{i} &:= \frac{\p}{\p x_{i}} = \sum_{j} \frac{\p y_{j}}{\p x_{i}} \frac{\p}{\p y_{j}} ,
		\end{split}
		\end{equation}
		we have
		\begin{equation} \label{glo loc deriva}
		\begin{split}
		\begin{bmatrix}
		\p_{1} \\
		\p_{2} \\
		\p_{3}
		\end{bmatrix}
		=
		\begin{bmatrix}
		\sqrt{g_{11}} & 0 & 0 \\
		0 & \sqrt{g_{22}} & 0 \\
		0 & 0 & \sqrt{g_{33}} \\
		\end{bmatrix}
		Q^{T}
		\begin{bmatrix}
		\p_{y_{1}} \\
		\p_{y_{2}} \\
		\p_{y_{3}}
		\end{bmatrix}
		,\quad 
		\begin{bmatrix}
		D_{1} \\
		D_{2} \\
		D_{3}
		\end{bmatrix}
		=
		Q^{T}
		\begin{bmatrix}
		\p_{y_1} \\
		\p_{y_2} \\
		\p_{y_3}
		\end{bmatrix}.
		\end{split}
		\end{equation}
		Therefore, identities (\ref{v_x}) and (\ref{glo loc deriva}) yield that 
		\begin{equation} \label{loc transport}
		\begin{split}
		v\cdot \nabla_y \	
		&= 
		\begin{bmatrix}
		\V_{1} &
		\V_{2} &
		\V_{3}
		\end{bmatrix}
		Q^{T} Q
		\begin{bmatrix}
		D_{1} \\
		D_{2} \\
		D_{3}
		\end{bmatrix} 
		= \V \cdot D {\color{red}.}
		\end{split}
		\end{equation}
		
		We also note that, from (\ref{v_x}) and (\ref{b_x}), 
		\begin{equation} \label{bv}
		\begin{split}
		b\cdot v 
		&= 
		\begin{bmatrix}
		b_{x_{1}} &
		b_{x_{2}} &
		b_{x_{3}}
		\end{bmatrix}
		Q^{T}
		Q
		\begin{bmatrix}
		\V_{1}\\
		\V_{2}\\
		\V_{3}
		\end{bmatrix}
		=
		b_{x}\cdot \V.
		\end{split}
		\end{equation}
	Together with $a_x$ and $c_x$ defined in \eqref{ac_x}, 
		 we can rewrite the macroscopic part $\mathbf{P}f$ as
		\begin{equation}\label{Pf}
		\begin{split}
		\mathbf{P}f := \Big( a_{x} + b_{x}\cdot \V + \frac{|\V|^{2} - 3}{2} \Big)\sqrt{\mu}.
		\end{split}
		\end{equation}

		Now we consider some important properties of $D_{i}:=\frac{1}{\sqrt{g_{ii}}} \p_{i}$. By the normalization in \eqref{diff operator} and \eqref{g_ij},
		\Be
		\langle D_{i}\eta, D_{j}\eta \rangle = \delta_{ij}.{\color{red} \label{orth}}
		\Ee
		We define $D_{ij}$ and $\Gamma_{D,ij}^{k}$ by 
		\begin{equation} \label{D gamma def}
		\begin{split}
			D_{ij} &:= D_{i}D_{j}, \quad \Gamma_{D,ij}^{k} := \langle D_{ij} \eta , D_{k}\eta \rangle.
		\end{split}
		\end{equation}
		Notice that we calculated the Christoffel symbol in the domain $U$
		\begin{align} \label{christoffel}
			 \Gamma_{D,33}^{1} = \Gamma_{D,33}^{2} = 0 ,\quad \Gamma_{D,22}^{1} = \Gamma_{D,22}^{3} = 0, \quad \Gamma_{D,11}^{3} = -\frac{1}{r}, \quad \Gamma_{D,11}^{2} = 0. 
		\end{align}
		
				\begin{lemma} \label{identities}
			\noindent (1) We have
			\begin{equation} \label{swap index}
			\Gamma_{D,ij}^{k} = - \Gamma_{D,ik}^{j}, \quad \Gamma_{D,ij}^{j} = 0,\quad \text{for all} \ \ i,j,k\in\{1,2,3\}.  \\
			\end{equation}
			\noindent (2) We have the following commutator:
			\begin{equation} \label{noncommu}
			\begin{split}
			D_{i} D_{j} - D_{j} D_{i} = \Gamma_{D,jj}^{i} D_{j} - \Gamma_{D,ii}^{j} D_{i},\quad i\neq j.
			\end{split}
			\end{equation}
			\noindent (3) When all $i,j,k \in \{1,2,3\}$ are distinct,
			\begin{equation} \label{geo cond}
				\Gamma_{D,ij}^{k} = 0.
			\end{equation}
		\end{lemma}
		\begin{proof}
			(1) We get \eqref{swap index} directly from \eqref{orth} by taking derivatives. 
			
			\noindent (2) First we compute that, for $i \neq k$,
			\begin{equation} \label{1 over g}
			\begin{split}
				\p_{i}\big( \frac{1}{\sqrt{g_{kk}}}\big) &= -\frac{1}{\sqrt{g_{kk}}^{3}} \langle \p_{ik}\eta, \p_{k}\eta \rangle = -\sqrt{\frac{g_{ii}}{g_{kk}}} \langle \frac{1}{\sqrt{g_{kk}g_{ii}}} \p_{ki}\eta, \frac{1}{\sqrt{g_{kk}}}\p_{k}\eta \rangle  \\
				&= -\sqrt{\frac{g_{ii}}{g_{kk}}} \langle \frac{1}{\sqrt{g_{kk}}}\p_{k}\big( \frac{1}{\sqrt{g_{ii}}}\p_{i}\eta\big), \frac{1}{\sqrt{g_{kk}}}\p_{k}\eta \rangle \\
				&= -\sqrt{\frac{g_{ii}}{g_{kk}}} \langle D_{ki}\eta, D_{k}\eta \rangle = -\sqrt{\frac{g_{ii}}{g_{kk}}} \Gamma_{D,ki}^{k} = \sqrt{\frac{g_{ii}}{g_{kk}}} \Gamma_{D,kk}^{i},  \\
			\end{split}
			\end{equation}
			where we used (\ref{swap index}) at the last equality. Therefore we conclude \eqref{noncommu} as 
			\begin{equation}\notag
			\begin{split}
				D_{i} D_{j} - D_{j} D_{i} &= \frac{1}{\sqrt{g_{ii}}} \p_{i} \big( \frac{1}{\sqrt{g_{jj}}} \p_{j} \big) - \frac{1}{\sqrt{g_{jj}}} \p_{j} \big( \frac{1}{\sqrt{g_{ii}}} \p_{i} \big)  \\
				&= \frac{1}{\sqrt{g_{ii}}} \p_{i} \big( \frac{1}{\sqrt{g_{jj}}} \big)  \p_{j} - \frac{1}{\sqrt{g_{jj}}} \p_{j} \big( \frac{1}{\sqrt{g_{ii}}} \big) \p_{i}  \\
				&= \frac{1}{\sqrt{g_{jj}}} \Gamma_{D,jj}^{i} \p_{j} - \frac{1}{\sqrt{g_{ii}}} \Gamma_{D,ii}^{j} \p_{i}  \\
				&= \Gamma_{D,jj}^{i}D_{j} - \Gamma_{D,ii}^{j}D_{i}.  \\
			\end{split}
			\end{equation}			
			
			\noindent (3) When $i,j,k$ are all distinct, using (\ref{swap index}) and (\ref{noncommu}), we derive that 
			\begin{equation}\notag
			\begin{split}
				\Gamma_{D,ij}^{k} &= - \Gamma_{D,ik}^{j} = - \langle D_{j}\eta, D_{ik}\eta \rangle  \\
				&= - \langle D_{j}\eta, D_{ki}\eta + \Gamma_{D,kk}^{i} D_{k}\eta - \Gamma_{D,ii}^{k} D_{i}\eta \rangle = - \Gamma_{D,ki}^{j}.
			\end{split}
			\end{equation}		
		Now	using above relation successively we derive that 
			\[
				\Gamma_{D,ij}^{k} = - \Gamma_{D,ki}^{j} = \Gamma_{D,jk}^{i} = - \Gamma_{D,ij}^{k}.
			\]	
		Therefore we conclude (\ref{geo cond}).
		\end{proof}

					\begin{definition}
		Define the Laplacian operator in the normalized $D$-derivatives:
		\begin{equation} \label{def D laplacian}
		\Delta_{D} := \big( D_{1}^{2} + D_{2}^{2} + D_{3}^{2} \big).
		\end{equation}
	
We also define the Laplace-Beltrami operator in the coordinates:
	\Be
	\Delta_{bel} := \Delta_{y} = \sum_{i,j} \frac{1}{\sqrt{|g|}} \p_{i} \big( \sqrt{|g|} g_{ij} \p_{j} \big) .\label{Beltrami}
	\Ee

	\end{definition}

		\begin{lemma} \label{D bel difference}
		Two operator $\Delta_D$ and $\Delta_{bel}$ have the following relation:
			\begin{equation} \label{eqtn:D bel difference}
			\begin{split}
			\Delta_{bel} &= \Delta_{D} - \sum_{i} \sum_{k\neq i} \Gamma_{D,kk}^{i} D_{i}.	\\
			\end{split}
			\end{equation}
		\end{lemma}
		\begin{proof}
			Using \eqref{identities}, we derive that 
			\begin{equation}
			\begin{split}\notag
			\Delta_{bel} &= \frac{1}{\sqrt{|g|}} \sum_{i=1}^{3} \p_{{i}} \big( \sqrt{|g|}\frac{1}{g_{ii}} \p_{{i}} \big) = \frac{1}{\sqrt{|g|}} \sum_{i=1}^{3} \p_{{i}} \big( \sqrt{g_{jj}g_{kk}} D_{i} \big)  \\
			&= \sum_{i=1}^{3} D_{i}^{2} + \frac{1}{\sqrt{|g|}} \sum_{i=1}^{3} \Big( \sqrt{g_{kk}} \p_{x_{i}}\sqrt{g_{jj}} + \sqrt{g_{jj}} \p_{x_{i}}\sqrt{g_{kk}} \Big) D_{i} \\
			&= \sum_{i} D_{i}^{2} +  \sum_{i=1}^{3} \Big( \langle \frac{\p_{ij}\eta}{\sqrt{g_{ii}g_{jj}}}, \frac{\p_{j}\eta}{\sqrt{g_{jj}}} \rangle + \langle \frac{\p_{ik}\eta}{\sqrt{g_{ii}g_{kk}}}, \frac{\p_{k}\eta}{\sqrt{g_{kk}}} \rangle \Big) D_{i}  \\
			&= \sum_{i} D_{i}^{2} + \sum_{i=1}^{3} \Big( \langle \frac{1}{g_{jj}}\p_{j}\big( \frac{1}{\sqrt{g_{ii}}}\p_{i}\eta \big), \frac{\p_{j}\eta}{\sqrt{g_{jj}}} \rangle + \langle \frac{1}{g_{kk}}\p_{k}\big( \frac{1}{\sqrt{g_{ii}}}\p_{i}\eta \big), \frac{\p_{k}\eta}{\sqrt{g_{kk}}} \rangle \Big) D_{i}  \\
			&= \sum_{i} D_{i}^{2} + \sum_{i=1}^{3} \Big( \langle D_{ji}\eta, D_{j}\eta \rangle + \langle D_{ki}\eta, D_{k}\eta \rangle \Big)D_{i}  \\
			&= \Delta_{D} - \sum_{i=1}^{3} \Big( \langle D_{jj}\eta, D_{i}\eta \rangle + \langle D_{kk}\eta, D_{i}\eta \rangle \Big)D_{i}  = \Delta_{D} - \sum_{i=1}^{3} \big( \Gamma_{D,jj}^{i} + \Gamma_{D,kk}^{i} \big)D_{i}  \\
			&= \Delta_{D} - \sum_{i} \sum_{k\neq i} \Gamma_{D,kk}^{i} D_{i}.
			\end{split}
			\end{equation}
		\end{proof}
		
		\begin{lemma} \label{dv/dx}
			We have the following differentiation identity:
			\begin{equation*}
				D_{i} \V_{j} = \sum_{k} \Gamma_{D,ij}^{k} \V_{k}.
			\end{equation*}
		\end{lemma}
		\begin{proof}
		From definition (\ref{v_x}),
		\begin{equation*}
		\begin{split}
			D_{i} \V_{j} &= D_{i} \langle D_{j}\eta, v \rangle = \langle D_{ij}\eta, v \rangle  \\
			&=  \langle \sum_{k} \Gamma_{D,ij}^{k}D_{k}\eta, v \rangle = \sum_{k}\Gamma_{D,ij}^{k} \langle D_{k}\eta, v \rangle =  \sum_{k} \Gamma_{D,ij}^{k} \V_{k}.
		\end{split}
		\end{equation*}
		\end{proof}


		\subsection{Mass and Energy estimates}
	The main purpose of this section is to control macroscopic quantities:
	\begin{align*}
		a(t,y)= \int_{\R^3} f(t,y,v) \sqrt{\mu} dv, \quad c(t,y) = \int_{\R^3} \frac{\vert v \vert^2 -3}{2} f(t,y,v) \sqrt{\mu} dv. 
	\end{align*}
		\begin{proposition} \label{prop_massenergy}
		Let $U$ be a domain with global triply orthogonal coordinate system of a single chart. Assume that $f$ solves linearized Boltzmann equation \eqref{linear BE} with the specular boundary condition $f(t,y,v)= f(t,y,R_{y}v)$ on $ y\in\p U,$ where $R_y:=(I-2n(y) \otimes n(y))$. 
		\hide
		\begin{equation*} 
			\begin{split}
				f_{t} + (v\cdot\nabla_{y})f + Lf &= g, \quad y\in U,   \\
				f(t,y,v) &= f(t,y,R_{y}v), \ \ y\in\p U.		
			\end{split}
		\end{equation*} 
		\unhide
		Then we have the following mass and energy estimate:
		\begin{equation} 
			\begin{split}
				\|a\|_{2}^{2} &\lesssim \frac{d}{dt}G(t) + \|b\|_{2}^{2} + \|\sqrt{\nu}(\mathbf{I-P})f\|_{2}^{2} + \| g \|_{2}^{2} ,   \\
				\|c\|_{2}^{2} &\lesssim \frac{d}{dt}G(t) + \varepsilon\|b\|_{2}^{2} + \|\sqrt{\nu}(\mathbf{I-P})f\|_{2}^{2} + \| g \|_{2}^{2} ,\quad \vert G(t) \vert \lesssim \|f(t)\|_{2}^{2},\quad \varepsilon \ll 1. 
			\end{split}
		\end{equation}
	\end{proposition}
	\begin{proof}
	A starting point of the proof is the weak formulation:
		\begin{equation} \label{glo IBP re} 
		\begin{split}
		&\iint_{\p U\times\mathbb{R}^{3}} f^b \Psi^{b} (v\cdot n) dA dv - \iint_{ U\times\mathbb{R}^{3}} \mathbf{P}f (v\cdot \nabla_{y}) \Psi dy dv \\
		&= - \iint_{ U\times\mathbb{R}^{3}} \Psi \p_{t} f dy dv - \iint_{ U\times\mathbb{R}^{3}} \Psi L f dy dv + \iint_{ U\times\mathbb{R}^{3}} \Psi g dy dv \\
		&\quad+ \iint_{ U\times\mathbb{R}^{3}} (\mathbf{I-P})f (v\cdot \nabla_{y}) \Psi dy dv .
		\end{split}
		\end{equation}

		\noindent To estimate the mass part $a$, we pick $\Psi_{a} = \sum_{i=1}^{3} \phi_{i}(v) \psi_{i}(y)$ with
		\begin{equation} \label{mass test} 
		\begin{split}
		\phi_{i}(v) &= (|v|^2 - \b_a) v_{i} \sqrt{\mu},  \\
		\psi_{i}(y) &= \p_{y_{i}} u_{a}.
		\end{split}	
		\end{equation}
		Here, $u_{a}$ solves
		\begin{equation} \label{u_a eq} 
			- \Delta u_{a} = a(x),\quad \p_{n} u_{a} \vert_{\p U} = 0 ,
		\end{equation}
		and the constant $\b_a=10$ is chosen so that
		\begin{equation} \label{b_a cst} 
		\int_{\mathbb{R}^3} (|v|^2-\b_a)(\frac{|v|^2-3}{2}) v_i^2 \mu dv = 0,\quad i=1,2,3.
		\end{equation}   
Note that the Poisson equation with Neumann boundary condition is uniquely solvable as we have the total mass conservation $\int_{ U} a dy = 0$. When the solution uniquely well-defined, we use the notation 
\Be
u_a =- (-\Delta)_N^{-1} a.\label{Delta^-1_N}
\Ee

For the boundary contribution in \eqref{glo IBP re}, we rewrite the integral in the $x$-coordinate. Using $|\V|=|v|$ in \eqref{isometry} and $v\cdot\nabla_{y} = \V\cdot D$ by (\ref{loc transport}), we have that 
		\begin{equation} \label{mass bdry}
		\begin{split}
		\iint_{\p U\times\mathbb{R}^{3}} f^b \Psi_{a}^{b} (v\cdot n) = \sum_{i=1}^{3} \iint_{\p U\times\mathbb{R}^{3}} f^{b} \V_{3} (|\V|^2 - \b_a) \mu \V_{i} D_{i} u_{a} = 0.  
		\end{split}
		\end{equation}
	This term always vanishes because when $i=1,2$ velocity integration is odd; while when $i=3$ the Neumann boundary condition implies $D_{3}u_{a}\vert_{\p U}=0$.

	For the second term in LHS of (\ref{glo IBP re}), using (\ref{b_a cst}) and (\ref{u_a eq}), we derive that 
		\begin{equation} \label{mass int}
		\begin{split}
		& - \iint_{ U\times\mathbb{R}^{3}} \mathbf{P}f(v\cdot\nabla)\Psi_{a}  \\
		&= - \iint_{ U\times\mathbb{R}^{3}} \sqrt{\mu}\Big( a + b\cdot v + c\frac{|v|^{2}-3}{2}\Big) (v\cdot\nabla)\Psi_{a}  \\
		&= - \sum_{i,j=1}^{3} \iint_{ U\times\mathbb{R}^{3}} a {\mu} (|v|^2 - \b_a) v_{i} v_{j} \p_{y_{j}}^{2} u_{a} dy   \\
		&= - C \sum_{i=1}^{3} \int_{ U} a \p_{y_{i}}^{2} u_{a} dy =  C \|a\|^{2}_{2} . 
		\end{split}	
		\end{equation}
	Here, $C= \int_{\R^3} \mu (\vert v \vert^2 -\b_a) v_i^2 dv = \int_{\R^3} \mu (\vert v \vert^2 -10) v_i^2 dv=-5\neq 0$ for all $i=1,2,3.$

		For the RHS of (\ref{glo IBP re}), from the elliptic estimate to \eqref{u_a eq}, we use $\|u_{a}\|_{H^{2}}^2 \lesssim \|a\|_{2}^{2}$. Then we could conclude that 
		\begin{equation} \label{mass RHS}
		\begin{split}
		&\text{RHS of (\ref{glo IBP re})} +\frac{d}{dt}\iint_{ U\times\mathbb{R}^{3}} f \Psi_{a} dydv \\
		&\lesssim    \Big| \iint_{ U\times\mathbb{R}^{3}} f \p_{t}\Psi_{a} dydv \Big| + \varepsilon \|u_{a}\|_{H^{2}}^{2} + \|\sqrt{\nu}(\mathbf{I-P})f\|_{2}^{2} + \| g \|_{2}^{2} \\
		&\lesssim  \underbrace{ \Big| \iint_{ U\times\mathbb{R}^{3}} f \p_{t}\Psi_{a} dydv \Big| }_{(**)} + \varepsilon \|a\|_{2}^{2} + \|\sqrt{\nu}(\mathbf{I-P})f\|_{2}^{2} + \| g \|_{2}^{2}.  
		\end{split}
		\end{equation} 
Here, we set	$ G(t) = \iint_{ U\times\mathbb{R}^{3}} f \Psi_{a} dydv$ which satisfies that $ G(t)\lesssim \|f(t)\|_{2}^{2}$.
	
	To estimate the part $(**)$, we device another test function:
		\begin{equation} \label{a time test}
			\varphi_{a} = - (-\Delta)^{-1}_N (\p_t a)\sqrt{\mu}  . 
		\end{equation} 
		Note that $(-\Delta)^{-1}_N (\p_t a)$ is uniquely well-defined as the mass conservation $\int_{ U} a(t) = \int_{ U}\int_{\mathbb{R}^{3}} f_{0}\sqrt{\mu} = 0$ for all $t \geq 0$. When this test function combined with $\p_t f$, we get
		\begin{equation*}
		\begin{split}
		\iint_{ U \times \mathbb{R}^{3}} \p_t f \sqrt{\mu} \Delta^{-1}_N (\p_t a) &= \| \nabla\Delta_N^{-1} (\p_t a) \|^{2}_{2}.
		\end{split}
		\end{equation*}
		On the other hand, considering other terms, except $\p_t f$,
		\begin{equation} \label{c time 12 required}
		\begin{split}
		\|\nabla\Delta_N^{-1} (\p_{t} a) \|^{2}_{2} &\lesssim \Big| \iint_{ U\times\mathbb{R}^3} \sqrt{\mu} \Delta^{-1}_N (\p_t a) \Big( - v\cdot \nabla_{y} f + g - L f  \Big)  \Big| \\
		&\lesssim \Big| \iint_{ U\times\mathbb{R}^3} (\mathbf{P}f + (\mathbf{I}-\mathbf{P})f) \ \nabla\cdot\big( v\sqrt{\mu} \Delta^{-1}_N (\p_t a) \big) \Big| \\
		&\quad + \ \Big| \iint_{ U\times\mathbb{R}^3} \sqrt{\mu} \Delta^{-1}_N (\p_t a) \big( g - Lf \big) \Big| \\
		&\lesssim \varepsilon \|\nabla\Delta_N^{-1} (\p_{t} a) \|^{2}_{2} + \varepsilon\|a\|_{2}^{2} + \Big( \|\sqrt{\nu}(\mathbf{I}-\mathbf{P})f\|_{2}^{2} +  \|b\|^{2}_{2} + \|g\|_{2}^{2} \Big) , \\
		\end{split}	
		\end{equation}
		where we used Poincare inequality and the Young's inequality with sufficiently small $\varepsilon \ll 1$. Therefore we get 
		\begin{equation} \label{mid step for a}
		\|\nabla\Delta_N^{-1} (\p_{t} a) \|^{2}_{2} \lesssim  \varepsilon\|a\|_{2}^{2} + \Big( \|\sqrt{\nu}(\mathbf{I}-\mathbf{P})f\|_{2}^{2} +  \|b\|^{2}_{2} + \|g\|_{2}^{2} \Big).
		\end{equation}
		Using (\ref{mid step for a}) and $\Vert \cdot \Vert_{2} \leq \Vert \cdot \Vert_{\nu}$ , we can estimate $(**)$ of \eqref{mass RHS} by
		\begin{equation} \label{a time est}
		\begin{split}
			\Big| \iint_{ U\times\mathbb{R}^{3}} f \p_{t}\Psi_{a} dydv \Big|  
			&\lesssim \Big| \iint_{ U\times\mathbb{R}^{3}} \mathbf{P}f \p_{t}\Psi_{a} dydv \Big| + \Big| \iint_{ U\times\mathbb{R}^{3}} (\mathbf{I-P})f \p_{t}\Psi_{a} dydv \Big|  \\
			&\lesssim  \|b\|_{2}^{2} +  \|\nabla\Delta_N^{-1} (\p_{t} a) \|^{2}_{2} + \|(\mathbf{I-P})f\|_{2}^{2}  \\
			&\lesssim \varepsilon\|a\|_{2}^{2} + \|b\|_{2}^{2} + \|\sqrt{\nu}(\mathbf{I}-\mathbf{P})f\|_{2}^{2} + \|g\|_{2}^{2}.
		\end{split}
		\end{equation}
		
		Combining \eqref{glo IBP re}, (\ref{mass bdry}), (\ref{mass int}), (\ref{mass RHS}), and (\ref{a time est}), we gain
		\begin{equation} \label{mass est}
		\begin{split}
		\|a\|_{2}^{2} &\lesssim \frac{d}{dt}G(t) + \|b\|_{2}^{2} + \|\sqrt{\nu}(\mathbf{I-P})f\|_{2}^{2} + \| g \|_{2}^{2} ,\quad \vert G(t)\vert \lesssim \|f(t)\|_{2}^{2}. \\
		\end{split}
		\end{equation} 
	
		\noindent Similar as above mass estimate, in energy $c$ estimate, we pick $\Psi_{c} = \sum_{i=1}^{3} \phi_{i}(v) \psi_{i}(y)$ where,
		\begin{equation} \label{energy test} 
		\begin{split}
		\phi_{i}(v) &= (|v|^2 - \b_c) v_{i} \sqrt{\mu},  \\
		\psi_{i}(y) &= \p_{y_{i}} u_{c},	\\
		\end{split}	
		\end{equation}
		when each are read on $k$-chart and constant $\b_c=5$ is chosen so that
		\begin{equation} \label{b_c cst} 
		\int_{\mathbb{R}^3} (|v|^2-\b_c)v_i^2 \mu dv = 0,\quad i=1,2,3,
		\end{equation}  
		and $u_{c}$ solves,
		\begin{equation} \label{u_c eq}  
		- \Delta u_{c} = c(y),\quad \p_{n} u_{c} \vert_{\p U} = 0.
		\end{equation}
		Note that we have energy conservation $\int_{ U} c dy = 0$. Similar as (\ref{mass bdry}), we use oddness integral in $\V_{3}$ and Neumann condition of $u_{c}$ to get
		\begin{equation} \label{energy bdry}
		\begin{split}
		\iint_{\p U} f^b \Psi_{c}^{b} (v\cdot n) = \sum_{i=1}^{3} \iint f^{b} \V_{3} (|\V|^2 - \b_c) \mu \V_{i} D_{i} u_{c} = 0.  
		\end{split}
		\end{equation}
		For transport term of LHS, similar as (\ref{mass int}),
		\begin{equation} \label{energy int}
		\begin{split}
		- \iint_{ U\times\mathbb{R}^{3}} \mathbf{P}f(v\cdot\nabla)\Psi_{c} dydv
		&= - \sum_{i,j=1}^{3} \iint_{U\times \R^3} c {\mu} (|v|^2 - \b_c) \frac{\vert v \vert^2-3}{2} v_{i} v_{j} \p_{y_{j}}^{2} u_{c} dydv  \\
		&= - C \sum_{i=1}^{3} \int_{U} c \p_{y_{i}}^{2} u_{c} dy  =  C \|c\|^{2}_{2} . 
		\end{split}	
		\end{equation}
	Here, the constant $C=\int_{\R^3} \mu (\vert v \vert^2 -\b_c)\left(\frac{\vert v \vert^2 -3}{2}\right) v_i^2 dv = \int_{\R^3} \mu (\vert v \vert^2 -5)\left(\frac{\vert v \vert^2 -3}{2}\right) v_i^2 dv=5\neq 0$ for all $i=1,2,3.$

		For the RHS of (\ref{glo IBP re}), similar as (\ref{mass RHS}),
		\begin{equation} \label{energy RHS}
		\begin{split}
		\text{RHS of (\ref{glo IBP re})} &\lesssim \Big| \iint_{ U\times\mathbb{R}^{3}} \Psi_{c} \p_{t} f dydv \Big| + \varepsilon \|u_{c}\|_{H^{2}}^{2} + \|(\mathbf{I-P})f\|_{2}^{2} + \| g \|_{2}^{2} ,\quad \varepsilon \ll 1, \\
		&\lesssim \frac{d}{dt}G(t) + \underbrace{ \Big| \iint_{ U\times\mathbb{R}^{3}} f \p_{t}\Psi_{c} dydv \Big| }_{(**)} + \varepsilon \|c\|_{2}^{2} + \|\sqrt{\nu}(\mathbf{I-P})f\|_{2}^{2} + \| g \|_{2}^{2},\\
		\end{split}
		\end{equation} 
		where $\vert G(t)\vert \lesssim \|f(t)\|_{2}^{2}$.
		To estimate time part $(**)$, we apply another test function,
		\begin{equation} \label{c time test}
		\varphi_{c} = \big(\frac{|v|^{2}-3}{2}\big)\sqrt{\mu} \Delta^{-1}_N (\p_t c).
		\end{equation} 
		Note that $\Delta^{-1}_{N} (\p_{t} c)$ is well defined by energy conservation. When this test function combined with $\p_t f$, 
		\begin{equation*}
		\begin{split}
		\iint_{ U \times \mathbb{R}^{3}}  \p_t f \big(\frac{|v|^{2}-3}{2}\big)\sqrt{\mu} \Delta^{-1}_N (\p_t c) &= \| \nabla\Delta_N^{-1} (\p_t c) \|^2_{2} .
		\end{split}
		\end{equation*}
		On the other hand, considering other terms, except $\p_t f$,
		\begin{equation} \label{c time 12 required}
		\begin{split}
		\|\nabla\Delta_N^{-1} (\p_{t} c) \|^{2}_{2} &\lesssim \Big| \iint_{ U\times\mathbb{R}^3} \big(\frac{|v|^{2}-3}{2}\big) \sqrt{\mu} \Delta^{-1}_N (\p_t c) \Big( - v\cdot \nabla_{y} f + g - L f  \Big)  \Big| \\
		&\lesssim \Big| \iint_{ U\times\mathbb{R}^3} (\mathbf{P}f + (\mathbf{I}-\mathbf{P})f) \ \nabla\cdot\big( v \big(\frac{|v|^{2}-3}{2}\big) \sqrt{\mu} \Delta^{-1}_N (\p_t c) \big) \Big| \\
		&\quad + \ \Big| \iint_{ U\times\mathbb{R}^3} \big(\frac{|v|^{2}-3}{2}\big) \sqrt{\mu} \Delta^{-1}_N (\p_t c) \big( g - Lf \big) \Big| \\
		&\lesssim \varepsilon \|\nabla\Delta_N^{-1} (\p_{t} c) \|^{2}_{2} + \|b\|_{2}^{2} + \|c\|_{2}^{2} + \|\sqrt{\nu}(\mathbf{I}-\mathbf{P})f\|_{2}^{2} + \|g\|_{2}^{2} , \\
		\end{split}	
		\end{equation}
		where we used Poincare inequality and sufficiently small $\varepsilon \ll 1$ comes from Young's inequality. Therefore we get 
		\begin{equation} \label{mid step for c}
		\|\nabla\Delta_N^{-1} (\p_{t} c) \|^{2}_{2} \lesssim  \|b\|_{2}^{2} + \|c\|_{2}^{2} + \|\sqrt{\nu}(\mathbf{I}-\mathbf{P})f\|_{2}^{2} + \|g\|_{2}^{2}.
		\end{equation}
		Using (\ref{mid step for c}) and $\Vert \cdot \Vert_2 \leq \Vert \cdot \Vert_{\nu}$, we can estimate $(**)$ by
		\begin{equation} \label{c time est}
		\begin{split}
		\Big| \iint_{ U\times\mathbb{R}^{3}} f \p_{t}\Psi_{c} dydv\Big|  
		&\lesssim \Big| \iint_{ U\times\mathbb{R}^{3}} \mathbf{P}f \p_{t}\Psi_{c} dydv\Big| + \Big| \iint_{ U\times\mathbb{R}^{3}} (\mathbf{I-P})f \p_{t}\Psi_{c} dydv\Big|  \\
		&= \Big| \iint_{ U\times\mathbb{R}^{3}} (\mathbf{I-P})f \p_{t}\Psi_{c} dydv\Big|  \\
		&\lesssim \varepsilon\|\nabla\Delta_N^{-1} (\p_{t} c) \|^{2}_{2} + \|(\mathbf{I-P})f\|_{2}^{2}  \\
		&\lesssim \varepsilon\|b\|_{2}^{2} + \varepsilon\|c\|_{2}^{2} + \|\sqrt{\nu}(\mathbf{I}-\mathbf{P})f\|_{2}^{2} + \|g\|_{2}^{2},
		\end{split}
		\end{equation}
		where we used 
		\begin{equation}  
		\begin{split}
			\iint_{ U\times\mathbb{R}^{3}} \mathbf{P}f \p_{t}\Psi_{c} &= \sum_{i=1}^{3} \iint_{ U\times\mathbb{R}^{3}} \Big(a  + b\cdot v + c \frac{|v|^2-3}{2}\Big) {\mu} (|v|^{2} - \b_{c}) v_{i} \p_{i} u_{c}  \\
			&= \sum_{i=1}^{3} \iint_{ U\times\mathbb{R}^{3}} b_{i} {\mu} v_{i}^{2} (|v|^{2} - \b_{c}) \p_{i} u_{c} = 0,  \\
		\end{split}
		\end{equation} 
		because of our choice \eqref{b_c cst}. 
		
		Combining (\ref{energy bdry}), (\ref{energy int}), (\ref{energy RHS}), and (\ref{c time est}), we gain
		\begin{equation} \label{energy est}
		\begin{split}
		\|c\|_{2}^{2} &\lesssim \frac{d}{dt}G(t) + \varepsilon\|b\|_{2}^{2} + \|\sqrt{\nu}(\mathbf{I-P})f\|_{2}^{2} + \| g \|_{2}^{2} ,\quad \vert G(t)\vert \lesssim \|f(t)\|_{2}^{2},\quad \varepsilon \ll 1.  \\
		\end{split}
		\end{equation} 
		
		From (\ref{mass est}) and (\ref{energy est}), we prove Proposition \ref{prop_massenergy}.
		\end{proof} 
		
		\subsection{Momentum $b$ Estimate}
		The main difficulty comes from estimating $\|b\|_{2}$. In contrast with the mass and energy estimates, momentum $b$ is not entirely conserved and therefore we ought to distinguish the ``tangential'' momenta and ``normal'' momentum to the boundary as (\ref{b_xi}). 
		In this section we use the global triply orthogonal coordinate system \eqref{ann domain}. 
		
		For the momentum estimate, we express the weak formulation \eqref{glo IBP re} in the coordinate system of \eqref{ann domain}. Let us use $\mathbb{R}^{3}_{+}$ for the domain of $\eta$, namely $\eta^{-1}( U)$. We use a set of test functions $\Psi = \phi(\V) \psi(x)$, which depends on the coordinates and transformed velocity \eqref{v_x}. Now we expand interior integral of the LHS of (\ref{glo IBP re}) as 
		\begin{align}
		- \iint_{ U\times\mathbb{R}^{3}} \mathbf{P}f (v\cdot \nabla) \Psi  dy dv  
		&=	- \sum_{j,m} \int_{\mathbb{R}^{3}_{+} }   \Big( \int_{\mathbb{R}^{3}} \V_{j} \V_{m} \phi(\V) {\mu}  d \V\Big) b_{x_{m}} D_{j} [\psi(x)] \sqrt{|g|} dx \notag	\\   
		&\quad - \sum_{j,m} \int_{\mathbb{R}^{3}_{+} } \Big( \int_{\mathbb{R}^{3}} \V_{j} \V_{m}  D_{j} [\phi(\V)] \mu d \V  \Big) b_{x_{m}} \psi(x) \sqrt{|g|} dx.\label{Pf2}
		\end{align} 
	For the second term \eqref{Pf2}, we use the integration by parts in velocity, (\ref{D gamma def}), and Lemma \ref{dv/dx}: 
		\begin{align}
		- \iint_{ U\times\mathbb{R}^{3}} \mathbf{P}f (v\cdot \nabla) \Psi  dy dv&=	- \sum_{j,m} \int_{\mathbb{R}^{3}_{+} } \Big( \int_{\mathbb{R}^{3}} \V_{j} \V_{m} \phi(\V) {\mu} \Big) b_{x_{m}} D_{j} \psi(x) \sqrt{|g|} dx 	\\   
		&\quad + \sum_{j,m} \sum_{\ell\neq m} \int_{\mathbb{R}^{3}_{+} } \Big( \int_{\mathbb{R}^{3}} \phi(\V)  \V_{j} \V_{\ell}  {\mu} \Big) b_{x_{m}} \Gamma^{\ell}_{D,jm} \psi(x) \sqrt{|g|} dx 	\\  
		&\quad + \sum_{j,m} \sum_{\ell\neq j} \int_{\mathbb{R}^{3}_{+} } \Big( \int_{\mathbb{R}^{3}} \phi(\V)  \V_{m} \V_{\ell}  {\mu} \Big) b_{x_{m}} \Gamma^{\ell}_{D,jj}  \psi(x) \sqrt{|g|} dx .	 
		\end{align} 
		Using oddness in velocoity integration, we can expand as following.
		\begin{equation} \label{loc wk form} 
		\begin{split}
		& - \iint_{ U\times\mathbb{R}^{3}} \mathbf{P}f (v\cdot \nabla) \Psi  dy dv  \\
		&= - \sum_{j=1}^{3} \int_{\mathbb{R}^{3}_{+} } \Big( \int_{\mathbb{R}^{3}} \V^{2}_{j}  \phi(\V) {\mu} \Big) b_{x_{j}} D_{j} \psi(x) \sqrt{|g|} dx 	\quad (j=m)	\\   
		&\quad - \sum_{j\neq m} \int_{\mathbb{R}^{3}_{+} } \Big( \int_{\mathbb{R}^{3}} \V_{j} \V_{m} \phi(\V) {\mu} \Big) b_{x_{m}} D_{j} \psi(x) \sqrt{|g|} dx 	\quad (j\neq m)	\\ 
		&\quad  + \sum_{j=\ell\neq m} \int_{\mathbb{R}^{3}_{+} } \Big( \int_{\mathbb{R}^{3}} \phi(\V)  \V^{2}_{j}{\mu} \Big) b_{x_{m}} \Gamma^{j}_{D,jm} \psi(x) \sqrt{|g|} dx  \quad (j=\ell, \ \ell\neq m)		\\   
		&\quad + \sum_{j\neq\ell, \ell\neq m} \int_{\mathbb{R}^{3}_{+} } \Big( \int_{\mathbb{R}^{3}} \phi(\V)  \V_{j} \V_{\ell}  {\mu} \Big) b_{x_{m}} \Gamma^{\ell}_{D,jm} \psi(x) \sqrt{|g|} dx  \quad (j\neq\ell, \ \ell\neq m)		\\ 
		&\quad + \sum_{m=\ell\neq j} \int_{\mathbb{R}^{3}_{+} } \Big( \int_{\mathbb{R}^{3}} \phi(\V)  \V^{2}_{m}  {\mu} \Big) b_{x_{m}} \Gamma^{m}_{D,jj}  \psi(x) \sqrt{|g|} dx  \quad (m=\ell, \ \ell\neq j)		\\   
		&\quad + \sum_{m\neq\ell,\ell\neq j} \int_{\mathbb{R}^{3}_{+} } \Big( \int_{\mathbb{R}^{3}} \phi(\V)  \V_{m} \V_{\ell}  {\mu} \Big) b_{x_{m}} \Gamma^{\ell}_{D,jj}  \psi(x) \sqrt{|g|} dx . \quad (m\neq\ell, \ \ell\neq j) 		\\     
		\end{split}
		\end{equation}
		
		\subsubsection{Normal momentum $b_{x_{3}}$ estimate.}
		We try three types of test functions, Symmetric, Non-symmetric 1, and Non-symmetric 2.  \\
		\textit{(Symmetric case)}
		We try test function $\Psi_{3,1} = \phi_{3}(\V) \psi_{3}(x)$, where
		\begin{equation} \label{3D bn sym test}
		\begin{split}
		\phi_{3}(\V) &= (\V^{2}_{3} - \b_b) \sqrt{\mu},\quad\text{where}\quad \int_{\mathbb{R}^3} (\V^{2}_{i} - \b_b) \mu d\V = 0,	\\
		\psi_{3}(x) &= \zeta_{3}(x) D_{3}u_{3},
		\end{split}
		\end{equation}
		where $\zeta_{3}(x)$ to be determined and $u_{3}$ solves Dirichlet problem,
		\begin{equation} \label{u3 eq}
		-\Delta_{bel} u_{3} = b_{x_{3}},\quad u_{3}\vert_{\p U} = 0.
		\end{equation}
		Above $\phi_{3}$ yields only $\V^{2}_{3}\sqrt{\mu}$-terms in (\ref{loc wk form}). With multiplying proper normalization constants to each $\phi_{i}$'s, LHS of (\ref{glo IBP re}) becomes 
		\begin{equation} \label{3D bn sym}
		\begin{split}
		&\iint_{\p U\times\mathbb{R}^{3}} f^b \Psi_{3,1}^{b} (v\cdot n) dA dv - \iint_{ U\times\mathbb{R}^{3}} \mathbf{P}f (v\cdot \nabla_{y}) \Psi_{3,1} dy dv \\
		&= - \int_{\mathbb{R}^{3}_{+}} \big( \underbrace{ b_{x_{3}} \zeta_{3} D^{2}_{3} u_{3} }_{main} + b_{x_{3}} D_{3}\zeta_{3} D_{3} u_{3} \big) \sqrt{|g|} dx	\\
		&\quad + \int_{\mathbb{R}^{3}_{+}} \Big( - b_{x_{1}}\Gamma^{1}_{D,33} - b_{x_{2}}\Gamma^{2}_{D,33} + b_{x_{3}}\Gamma^{3}_{D,11} + b_{x_{3}}\Gamma^{3}_{D,22} \Big) \zeta_{3} D_{3} u_{3} \sqrt{|g|} dx	\\
		&\quad + C \int_{x_{3}=0} \Big(\underbrace{ \int_{\mathbb{R}^{3}} f^{b} \phi_{3}(\V) \V_{3} }_{ =0 \ \text{oddness in} \ \V_{3} }\Big) \psi_{3}(x) \sqrt{g_{11}g_{22}} dA .
		\end{split}
		\end{equation} 
		
		\noindent \textit{(Non-symmetric case 1)}
		We try test function $\Psi_{3,2} = \sum_{i=1,2} \phi_{i}(\V) \psi_{i}(x)$, where
		\begin{equation} \label{3D bn nonsym1 test}
		\begin{split}
		\phi_{i}(\V) &= \V_{i} \V_{3} \sqrt{\mu},	\\
		\psi_{i}(x) &= \zeta_{3}(x) D_{i}u_{3},
		\end{split}
		\end{equation}
		for $i=1,2$. Above $\phi_{1}$ and $\phi_{2}$ yield only $\V_{3}\V_{1}\sqrt{\mu}$ and $\V_{2}\V_{3}\sqrt{\mu}$ terms in (\ref{loc wk form}) respectively. With multiplying proper normalization constants to each $\phi_{i}$'s, LHS of (\ref{glo IBP re}) becomes 
		{\footnotesize
		\begin{equation} \label{3D bn nonsym1}
		\begin{split}
		&\iint_{\p U\times\mathbb{R}^{3}} f^b \Psi_{3,2}^{b} (v\cdot n) dA dv - \iint_{ U\times\mathbb{R}^{3}} \mathbf{P}f (v\cdot \nabla_{y}) \Psi_{3,2} dy dv \\
		&= - \int_{\mathbb{R}^{3}_{+}} \big( b_{x_{1}} \zeta_{3} D_{31} u_{3}   + b_{x_{1}} D_{3}\zeta_{3} D_{1}u_{3} +  b_{x_{3}} \zeta_{3} D_{11}u_{3}  + b_{x_{3}} D_{1}\zeta_{3} D_{1}u_{3} \big) \sqrt{|g|} dx	\\
		&\quad - \int_{\mathbb{R}^{3}_{+}} \big( b_{x_{2}} \zeta D_{32} u_{3} + b_{x_{2}} D_{3}\zeta_{3} D_{2}u_{3} + b_{x_{3}} \zeta_{3} D_{22}u_{3} + b_{x_{3}} D_{2}\zeta_{3} D_{2}u_{3} \big) \sqrt{|g|} dx	\\
		&\quad + \int_{\mathbb{R}^{3}_{+}} \Big( 2 b_{x_{3}}\Gamma^{1}_{D,33} + 2 b_{x_{1}}\Gamma^{3}_{D,11} + b_{x_{3}}\Gamma^{1}_{D,22} + b_{x_{1}}\Gamma^{3}_{D,22} + \underbrace{ b_{x_{2}}\Gamma^{3}_{D,12} + b_{x_{2}}\Gamma^{1}_{D,32} }_{=0 \ \text{by} \ (\ref{geo cond})} \Big) \zeta_{3} D_{1} u_{3} \sqrt{|g|} dx	\\
		&\quad + \int_{\mathbb{R}^{3}_{+}} \Big( 2 b_{x_{2}}\Gamma^{3}_{D,22} + 2 b_{x_{3}}\Gamma^{2}_{D,33} + \underbrace{ b_{x_{1}}\Gamma^{3}_{D,21} + b_{x_{1}}\Gamma^{2}_{D,31} }_{=0 \ \text{by} \ (\ref{geo cond})} + b_{x_{2}}\Gamma^{3}_{D,11} + b_{x_{3}}\Gamma^{2}_{D,11} \Big) \zeta_{3} D_{2} u_{3} \sqrt{|g|} dx	\\
		&\quad + C \int_{x_{3}=0} \Big(\int_{\mathbb{R}^{3}} f^{b} \phi_{1}(\V) \V_{3}\Big) \underbrace{\psi_{1}(x)}_{=0} \sqrt{g_{11}g_{22}} dA  + C \int_{x_{3}=0} \Big(\int_{\mathbb{R}^{3}} f^{b} \phi_{2}(\V) \V_{3}\Big) \underbrace{\psi_{2}(x)}_{=0} \sqrt{g_{11}g_{22}} dA ,
		\end{split}
		\end{equation} 
		}
		where last two terms are zero by Dirichlet boundary condition of $u_{3}$.  \\
		
		\noindent \textit{(Non-symmetric case 2)}
		We try test function $\Psi_{3,3} = \sum_{i=1,2} \phi_{i}(\V) \psi_{i}(x)$, where
		\begin{equation} \label{3D bn nonsym2 test}
		\begin{split}
		\phi_{i}(\V) &= (\V^{2}_{i} - \b_{b}) \sqrt{\mu},\quad\text{where}\quad \int_{\mathbb{R}^3} (\V^{2}_{i} - \b_{b}) \mu d\V = 0  ,	\\
		\psi_{i}(x) &= \zeta_{3}(x) D_{3}u_{3},
		\end{split}
		\end{equation}
		for $i=1,2$. Above $\phi_{1}$ and $\phi_{2}$ yield only $\V^{2}_{1}\sqrt{\mu}$ and $\V^{2}_{2}\sqrt{\mu}$ terms in (\ref{loc wk form}) respectively. With multiplying proper normalization constants to each $\phi_{i}$'s, LHS of (\ref{glo IBP re}) becomes
		{\footnotesize
		\begin{equation} \label{3D bn nonsym2}
		\begin{split}
		&\iint_{\p U\times\mathbb{R}^{3}} f^b \Psi_{3,3}^{b} (v\cdot n) dA dv - \iint_{ U\times\mathbb{R}^{3}} \mathbf{P}f (v\cdot \nabla_{y}) \Psi_{3,3} dy dv \\
		&= - \int_{\mathbb{R}^{3}_{+}} \big( b_{x_{1}} \zeta_{3} D_{13} u_{3} + b_{x_{1}} D_{1}\zeta_{3} D_{3} u_{3} \big) \sqrt{|g|} dx	\\
		&\quad - \int_{\mathbb{R}^{3}_{+}} \big( b_{x_{2}} \zeta D_{23} u_{3} + b_{x_{2}} D_{2}\zeta_{3} D_{3} u_{3} \big) \sqrt{|g|} dx	\\
		&\quad + \int_{\mathbb{R}^{3}_{+}} \Big( - b_{x_{2}}\Gamma^{2}_{D,11} - b_{x_{3}}\Gamma^{3}_{D,11} + b_{x_{1}}\Gamma^{1}_{D,22} + b_{x_{1}}\Gamma^{1}_{D,33} \Big) \zeta_{3} D_{3} u_{3} \sqrt{|g|} dx	\\
		&\quad + \int_{\mathbb{R}^{3}_{+}} \Big( - b_{x_{1}}\Gamma^{1}_{D,22} - b_{x_{3}}\Gamma^{3}_{D,22} + b_{x_{2}}\Gamma^{2}_{D,11} + b_{x_{2}}\Gamma^{2}_{D,33} \Big) \zeta_{3} D_{3} u_{3} \sqrt{|g|} dx	\\
		&\quad + C \int_{x_{3}=0} \Big(\underbrace{ \int_{\mathbb{R}^{3}} f^{b} \phi_{1}(\V) \V_{3} }_{ =0 \ \text{oddness in} \ \V_{3} }\Big) \psi_{1}(x) \sqrt{g_{11}g_{22}} dA + C \int_{x_{3}=0} \Big(\underbrace{ \int_{\mathbb{R}^{3}} f^{b} \phi_{2}(\V) \V_{3} }_{ =0 \ \text{oddness in} \ \V_{3} }\Big) \psi_{2}(x) \sqrt{g_{11}g_{22}} dA .
		\end{split}
		\end{equation} 
		}
		\subsubsection{Condition for $\zeta_{3}(x)$}
		We add (\ref{3D bn sym}), (\ref{3D bn nonsym1}), and $(-1)\times$(\ref{3D bn nonsym2}). This is equivalent to applying test function
		\begin{equation} \label{Psib3}
		\Psi_{b,3} := \Psi_{3,1} + \Psi_{3,2} - \Psi_{3,3},
		\end{equation}
		to (\ref{loc wk form}). Also using (\ref{noncommu}), 
		\begin{equation} \label{bn part}
		\begin{split}
		& - \iint_{ U\times\mathbb{R}^{3}} \mathbf{P}f (v\cdot \nabla) \Psi_{b,3}  dydv+ \iint_{\p U\times\mathbb{R}^{3}} f^b \Psi_{b,3}^{b} (v\cdot n) dAdv \\
		&= - \int_{\mathbb{R}^{3}_{+}} b_{x_{3}} \zeta_{3} \big( D_{1}^{2} + D_{2}^{2} + D_{3}^{2}  \big) u_{3} \sqrt{|g|} dx	\\
		& + \int_{\mathbb{R}^{3}_{+}} b_{x_{3}} \Big( 2\Gamma_{D,33}^{1}\zeta_{3} - D_{1}\zeta_{3} + \Gamma_{D,22}^{1}\zeta_{3} \Big) D_{1} u_{3} \sqrt{|g|} dx	\\
		& + \int_{\mathbb{R}^{3}_{+}} b_{x_{3}} \Big( 2\Gamma_{D,33}^{2}\zeta_{3} - D_{2}\zeta_{3} + \Gamma_{D,11}^{2}\zeta_{3} \Big) D_{2} u_{3} \sqrt{|g|} dx	\\
		& + \int_{\mathbb{R}^{3}_{+}} b_{x_{3}} \Big( 2\Gamma_{D,11}^{3}\zeta_{3} + 2\Gamma_{D,22}^{3}\zeta_{3}  - D_{3}\zeta_{3}  \Big) D_{3} u_{3} \sqrt{|g|} dx	\\
		& + \sum_{i=1,2} \int_{\mathbb{R}^{3}_{+}} b_{x_{i}} \Big( -\Gamma_{D,33}^{i}\zeta_{3} + D_{i}\zeta_{3} \Big)	D_{3} u_{3} \sqrt{|g|} dx	\\
		& + \sum_{i=1,2} \int_{\mathbb{R}^{3}_{+}} \Big( \Gamma_{D,22}^{3}\zeta_{3} + \Gamma_{D,11}^{3}\zeta_{3}  - D_{3}\zeta_{3} \Big) b_{x_{i}}D_{i}u_{3} \sqrt{|g|} dx .
		\end{split}
		\end{equation}
		If we match index $3\rightarrow r,1 \rightarrow \theta, 2 \rightarrow z$ and pick 
		\begin{align} \label{zeta r}
			\zeta_{3} = \zeta_{r} = \frac{1}{r} 
		\end{align}
		 $\zeta_{3}(x)$ satisfies
		\begin{equation} \label{bn zeta eq}
		\begin{cases}
		D_{1}\zeta_{3} &=  \Gamma_{D,33}^{1}\zeta_{3} 	\\
		D_{2}\zeta_{3} &= \Gamma_{D,33}^{2}\zeta_{3}	\\
		D_{3}\zeta_{3} &= \Gamma_{D,11}^{3}\zeta_{3} + \Gamma_{D,22}^{3}\zeta_{3},	\\
		\end{cases}
		\quad\text{and}\quad c \leq |\zeta_{3}| \leq C,
		\end{equation}
		for uniformly positive $c$ and $C$ due to \eqref{ann domain}, \eqref{diff operator} and \eqref{christoffel}. Then above system becomes,
		\begin{equation} \label{bn final sum}
		\begin{split}
		& - \iint_{ U\times\mathbb{R}^{3}} \mathbf{P}f (v\cdot \nabla) \Psi_{b,3}  dydv + \iint_{\p U\times\mathbb{R}^{3}} f^b \Psi_{b,3}^{b} (v\cdot n) dAdv \\
		&= - \int_{\mathbb{R}^{3}_{+}} b_{x_{3}} \zeta_{3} \big( D_{1}^{2} + D_{2}^{2} + D_{3}^{2}  \big) u_{3} \sqrt{|g|} dx + \int_{\mathbb{R}^{3}_{+}} b_{x_{3}}	\zeta_{3} \sum_{i=1}^{3} \sum_{k\neq i} \Gamma_{D,kk}^{i}D_{i}u_{3} \sqrt{|g|} dx	\\
		&= - \int_{\mathbb{R}^{3}_{+}} b_{x_{3}} \zeta_{3} \Delta_{bel} u_{3} \sqrt{|g|} dx = \int_{U} \zeta_{3} b^{2}_{x_{3}} dy,
		\end{split}
		\end{equation}
		where we used Lemma \ref{D bel difference} and (\ref{u3 eq}).  \\
		
		With $\Psi_{b,3}$, which is defined in (\ref{Psib3}), the RHS of (\ref{glo IBP re}) is controlled by 
		\begin{equation} \label{bn RHS}
		\begin{split}
		\text{RHS of (\ref{glo IBP re})} &\lesssim \Big| \iint_{ U\times\mathbb{R}^{3}} \Psi_{b,3} \p_{t} f dy \Big| + \varepsilon \|u_{3}\|_{H^{2}}^{2} + \|(\mathbf{I-P})f\|_{2}^{2} + \| g \|_{2}^{2} ,\quad \varepsilon \ll 1. \\
		&\lesssim \frac{d}{dt}G(t) + \underbrace{ \Big| \iint_{ U\times\mathbb{R}^{3}} f \p_{t}\Psi_{b,3} dy \Big| }_{(**)} + \varepsilon \|b_{x_{3}}\|_{2}^{2} + \|\sqrt{\nu}(\mathbf{I-P})f\|_{2}^{2} + \| g \|_{2}^{2}, \\
		\end{split}
		\end{equation}
		where $\vert G(t)\vert \lesssim \|f(t)\|_{2}^{2}$. Here, we used $\|u_{3}\|_{H^{2}}^2 \lesssim \|b_{x_{3}}\|_{2}^{2}$ and $\Vert \cdot \Vert_2 \leq \Vert \cdot \Vert_{\nu}$.  \\
		
		To estimate time part $(**)$, we apply another test function,
		\begin{equation} \label{b3 time test}
		\varphi_{3} = \V_{3} \sqrt{\mu} \Delta^{-1}_{D} \p_{t}b_{x_{3}}.
		\end{equation} 
		When this test function combined with $\p_t f$, we get
		\begin{equation*}
		\begin{split}
		\iint_{ U \times \mathbb{R}^{3}} \p_t f \V_{3} \sqrt{\mu} \Delta^{-1}_{D}  \p_{t}b_{x_{3}} &= \| \nabla\Delta_{D}^{-1} \p_{t}b_{x_{3}} \|^{2}_{2}.
		\end{split}
		\end{equation*}
		On the other hand, considering other terms, except $\p_t f$,
		\begin{equation} \label{b3 time 12 required}
		\begin{split}
		\|\nabla\Delta_{D}^{-1} \p_{t}b_{x_{3}} \|^{2}_{2} &\lesssim \Big| \iint_{ U\times\mathbb{R}^3} \V_{3} \sqrt{\mu} \Delta^{-1}_D \p_{t}b_{x_{3}} \Big( - v\cdot \nabla f + g - L f  \Big)  \Big| \\
		&\lesssim \Big| \iint_{ U\times\mathbb{R}^3} (\mathbf{P}f + (\mathbf{I}-\mathbf{P})f) \ \nabla\cdot\big( v \V_{3} \sqrt{\mu} \Delta^{-1}_{D} \p_{t}b_{x_{3}} \big) \Big| \\
		&\quad + \ \Big| \iint_{ U\times\mathbb{R}^3} \V_{3} \sqrt{\mu} \Delta^{-1}_{D} \p_{t}b_{x_{3}} \big( g - Lf \big) \Big| \\
		&\lesssim \varepsilon \|\nabla\Delta_{D}^{-1} \p_{t}b_{x_{3}} \|^{2}_{2} + \Big( \|a\|_{2}^{2} + \|c\|_{2}^{2} + \|\sqrt{\nu}(\mathbf{I}-\mathbf{P})f\|_{2}^{2} + \|g\|_{2}^{2} \Big),  \\
		\end{split}	
		\end{equation}
		from Poincare inequality and sufficiently small $\varepsilon \ll 1$ comes from Young's inequality. Therefore we get 
		\begin{equation} \label{mid step for b3}
		\|\nabla\Delta_{D}^{-1} \p_{t}b_{x_{3}} \|^{2}_{2} \lesssim  \|a\|_{2}^{2} + \|c\|_{2}^{2} + \|\sqrt{\nu}(\mathbf{I}-\mathbf{P})f\|_{2}^{2} + \|g\|_{2}^{2}.
		\end{equation}
		Using (\ref{mid step for b3}), we can estimate $(**)$ in (\ref{bn RHS}) by
		\begin{equation} \label{b3 time est}
		\begin{split}
		\Big| \iint_{ U\times\mathbb{R}^{3}} f \p_{t}\Psi_{b,3} dy \Big|  
		&\lesssim \Big| \iint_{ U\times\mathbb{R}^{3}} \mathbf{P}f \p_{t}\Psi_{b,3} dy \Big| + \Big| \iint_{ U\times\mathbb{R}^{3}} (\mathbf{I-P})f \p_{t}\Psi_{b,3} dy \Big|  \\
		&\lesssim \big( \|c\|_{2} \|\nabla\Delta_{D}^{-1} \p_{t}b_{x_{3}} \|_{2} \big) +  \big( \varepsilon\|\nabla\Delta_{D}^{-1} \p_{t}b_{x_{3}} \|^{2}_{2} + \|(\mathbf{I-P})f\|_{2}^{2} \big) \\
		&\lesssim \varepsilon\|a\|_{2}^{2} + \|c\|_{2}^{2} + \|\sqrt{\nu}(\mathbf{I}-\mathbf{P})f\|_{2}^{2} + \|g\|_{2}^{2},
		\end{split}
		\end{equation}
		where we used 
		\begin{equation}  
		\begin{split}
		&\iint_{ U\times\mathbb{R}^{3}} \mathbf{P}f \p_{t}\Psi_{b,3}  \\
		&=  \iint_{ U\times\mathbb{R}^{3}} \mathbf{P}f \p_{t}\Psi_{3,1} + \iint_{ U\times\mathbb{R}^{3}} \mathbf{P}f \p_{t}\Psi_{3,2} - \iint_{ U\times\mathbb{R}^{3}} \mathbf{P}f \p_{t}\Psi_{3,3}  \\
		&\leq \iint_{ U\times\mathbb{R}^{3}} \Big(a  + b\cdot v + c \frac{|v|^2-3}{2}\Big) \sqrt{\mu} (\V_{3}^{2} - \b_{b}) \sqrt{\mu} \zeta_{3} D_{3} \Delta_{D}^{-1} \p_{t}b_{x_{3}}  \\
		&\quad + \underbrace{ \sum_{i=1}^{2} \iint_{ U\times\mathbb{R}^{3}} \Big(a  + b_{x}\cdot \V + c \frac{|v|^2-3}{2}\Big) \sqrt{\mu} (\V_{i} \V_{3}) \sqrt{\mu} \zeta_{3} D_{3} \Delta_{D}^{-1} \p_{t}b_{x_{3}} }_{=0 \ \text{odd in velocity}}  \\
		&\quad - \sum_{i=1}^{2} \iint_{ U\times\mathbb{R}^{3}} \Big(a  + b\cdot v + c \frac{|v|^2-3}{2}\Big) \sqrt{\mu} (\V_{i}^{2} - \b_{b}) \sqrt{\mu} \zeta_{3} D_{3} \Delta_{D}^{-1} \p_{t}b_{x_{3}} \\
		&\lesssim \|c\|_{2} \|\nabla\Delta_{D}^{-1} \p_{t} b_{x_{3}} \|_{2} \\
		&\lesssim \varepsilon\|a\|_{2}^{2} + \|c\|_{2}^{2} + \|\sqrt{\nu}(\mathbf{I}-\mathbf{P})f\|_{2}^{2} + \|g\|_{2}^{2}.
		\end{split}
		\end{equation} 
		
		Combining (\ref{bn final sum}), (\ref{bn RHS}), and (\ref{b3 time est}), we gain
		\begin{equation} \label{b3 est}
		\begin{split}
		\|b_{x_{3}}\|_{2}^{2} &\lesssim \frac{d}{dt}G(t) + \varepsilon\|a\|_{2}^{2} + \|c\|_{2}^{2} + \|\sqrt{\nu}(\mathbf{I-P})f\|_{2}^{2} + \| g \|_{2}^{2} ,\quad \vert G(t) \vert \lesssim \|f(t)\|_{2}^{2}. \\
		\end{split}
		\end{equation}

		\subsubsection{Tangential momentum $b_{x_{1}}$ estimate} 
		\hfill
		
		\noindent \textit{(Symmetric case)}
		We try test function $\Psi_{1,1} = \phi_{1}(\V) \psi_{1}(x)$, where
		\begin{equation} \label{3D btan sym test}
		\begin{split}
		\phi_{1}(\V) &= (\V^{2}_{1} - \b_b) \sqrt{\mu},\quad\text{where}\quad \int_{\mathbb{R}^3} (\V^{2}_{i} - \b_b) \mu d\V = 0 ,	\\
		\psi_{1}(x) &= \zeta_{1}(x) D_{1}u_{1},
		\end{split}
		\end{equation}
		where $\zeta_{1}(x)$ to be determined and $u_{1}$ solves Neumann problem,
		\begin{equation} \label{u1 eq}
		-\Delta_{bel} u_{1} = \frac{1}{\zeta_1} b_{x_{1}} - \langle \frac{1}{\zeta_1} b_{x_{1}} \rangle := \widetilde{ \frac{1}{\zeta_1} b_{x_{1}} }  ,\quad \p_{n} u_{1}\vert_{\p U} = 0, \quad \text{where} \quad \langle f \rangle := \fint_{ U} f = \frac{1}{| U|}\int_{ U} f.
		\end{equation}
		Above $\phi_{1}$ yields only $\V_{1}^{2}\sqrt{\mu}$ terms in (\ref{loc wk form}). With multiplying proper normalization constants to each $\phi_{i}$'s, LHS of (\ref{glo IBP re}) becomes 
		\begin{equation} \label{3D btan sym}
		\begin{split}
		& - \iint_{ U\times\mathbb{R}^{3}} \mathbf{P}f (v\cdot \nabla) \Psi_{1,1}  dydv + \iint_{\p U\times\mathbb{R}^{3}} f^b \Psi_{1,1}^{b} (v\cdot n) dAdv \\
		&= - \int_{\mathbb{R}^{3}_{+}} \big( b_{x_{1}} \zeta_{1} D^{2}_{1} u_{1} + b_{x_{1}} D_{1}\zeta_{1} D_{1} u_{1} \big) \sqrt{|g|} dx	\\
		&\quad + \int_{\mathbb{R}^{3}_{+}} \Big( - b_{x_{2}}\Gamma^{2}_{D,11} - b_{x_{3}}\Gamma^{3}_{D,11} + b_{x_{1}}\Gamma^{1}_{D,22} + b_{x_{1}}\Gamma^{1}_{D,33} \Big) \zeta_{1} D_{1} u_{1} \sqrt{|g|} dx	\\
		&\quad + C \int_{x_{3}=0} \Big(\underbrace{ \int_{\mathbb{R}^{3}} f^{b} \phi_{1}(\V) \V_{3} }_{ =0 \ \text{oddness in} \ \V_{3} } \Big) \psi_{1}(x) \sqrt{g_{11}g_{22}} dA .
		\end{split}
		\end{equation} 
		
		\noindent \textit{(Non-symmetric case 1)}
		We try test function $\Psi_{1,2} = \sum_{i=2,3} \phi_{i}(\V) \psi_{i}(x)$, where
		\begin{equation} \label{3D btan nonsym1 test}
		\begin{split}
		\phi_{i}(\V) &= \V_{i} \V_{1} \sqrt{\mu},	\\
		\psi_{i}(x) &= \zeta_{1}(x) D_{i}u_{1},
		\end{split}
		\end{equation}
		for $i=2,3$. Above $\phi_{2}$ and $\phi_{3}$ yield only $\V_{1}\V_{2}\sqrt{\mu}$ and $\V_{3}\V_{1}\sqrt{\mu}$ terms in (\ref{loc wk form}) respectively. With multiplying proper normalization constants to each $\phi_{i}$'s, LHS of (\ref{glo IBP re}) becomes 
		{\footnotesize
		\begin{equation} \label{3D btan nonsym1}
		\begin{split}
		& - \iint_{ U\times\mathbb{R}^{3}} \mathbf{P}f (v\cdot \nabla) \Psi_{1,2}  dydv + \iint_{\p U\times\mathbb{R}^{3}} f^b \Psi_{1,2}^{b} (v\cdot n) dAdv \\
		&= - \int_{\mathbb{R}^{3}_{+}} \big( b_{x_{2}} \zeta_{1} D_{12} u_{1}  + b_{x_{2}} D_{1}\zeta_{1} D_{2}u_{1} + b_{x_{1}} \zeta_{1} D_{22}u_{1} + b_{x_{1}} D_{2}\zeta_{1} D_{2}u_{1} \big) \sqrt{|g|} dx	\\
		&\quad - \int_{\mathbb{R}^{3}_{+}} \big(  b_{x_{3}} \zeta_{1} D_{13} u_{1}   + b_{x_{3}} D_{1}\zeta_{1} D_{3}u_{1} + b_{x_{1}} \zeta_{1} D_{33}u_{1} + b_{x_{1}} D_{3}\zeta_{1} D_{3}u_{1} \big) \sqrt{|g|} dx	\\
		&\quad + \int_{\mathbb{R}^{3}_{+}} \Big( 2 b_{x_{1}}\Gamma^{2}_{D,11} + 2 b_{x_{2}}\Gamma^{1}_{D,22} + \underbrace{ b_{x_{3}}\Gamma^{2}_{D,13} + b_{x_{3}}\Gamma^{1}_{D,23} }_{=0 \ \text{by} \ (\ref{geo cond})} + b_{x_{1}}\Gamma^{2}_{D,33} + b_{x_{2}}\Gamma^{1}_{D,33} \Big) \zeta_{1} D_{2} u_{1} \sqrt{|g|} dx	\\
		&\quad + \int_{\mathbb{R}^{3}_{+}} \Big( 2 b_{x_{3}}\Gamma^{1}_{D,33} + 2 b_{x_{1}}\Gamma^{3}_{D,11} + b_{x_{3}}\Gamma^{1}_{D,22} + b_{x_{1}}\Gamma^{3}_{D,22} + \underbrace{ b_{x_{2}}\Gamma^{3}_{D,12} + b_{x_{2}}\Gamma^{1}_{D,32} }_{=0 \ \text{by} \ (\ref{geo cond})} \Big) \zeta_{1} D_{3} u_{1} \sqrt{|g|} dx	\\
		&\quad + C  \int_{x_{3}=0} \Big( \underbrace{ \int_{\mathbb{R}^{3}} f^{b} \phi_{2}(\V) \V_{3} }_{ =0 \ \text{oddness in} \ \V_{3} } \Big) \psi_{2}(x) \sqrt{g_{11}g_{22}} dA  + C \int_{x_{3}=0} \Big(\int_{\mathbb{R}^{3}} f^{b} \phi_{3}(\V) \V_{3}\Big) \underbrace{\psi_{3}(x)}_{=0 } \sqrt{g_{11}g_{22}} dA ,
		\end{split}
		\end{equation} 
		}
		where we used oddness in velocity integration and Neumann boundary condition.\\
		
		\noindent \textit{(Non-symmetric case 2)}
		We try test function $\Psi_{1,3} = \sum_{i=2,3} \phi_{i}(\V) \psi_{i}(x)$, where
		\begin{equation} \label{3D btan nonsym2 test}
		\begin{split}
		\phi_{i}(\V) &= (\V^{2}_{i} - \b_{b}) \sqrt{\mu},\quad\text{where}\quad \int_{\mathbb{R}^3} (\V^{2}_{i} - \b_{b}) \mu d\V = 0,	\\
		\psi_{i}(x) &= \zeta_{1}(x) D_{1}u_{1},
		\end{split}
		\end{equation}
		for $i=2,3$. Above $\phi_{2}$ and $\phi_{3}$ yield only $\V^{2}_{2}\sqrt{\mu}$ and $\V^{2}_{3}\sqrt{\mu}$ terms in (\ref{loc wk form}) respectively. Multiplying proper normalization constants to each $\phi_{i}$'s, LHS of (\ref{glo IBP re}) becomes
		{\footnotesize 
		\begin{equation} \label{3D btan nonsym2}
		\begin{split}
		& - \iint_{ U\times\mathbb{R}^{3}} \mathbf{P}f (v\cdot \nabla) \Psi_{1,3}  dydv + \iint_{\p U\times\mathbb{R}^{3}} f^b \Psi_{1,3}^{b} (v\cdot n) dAdv \\
		&= - \int_{\mathbb{R}^{3}_{+}} \big( b_{x_{2}} \zeta_{1} D_{21} u_{1} + b_{x_{2}} D_{2}\zeta_{1} D_{1} u_{1} \big) \sqrt{|g|} dx	\\
		&\quad - \int_{\mathbb{R}^{3}_{+}} \big(  b_{x_{3}} \zeta_{1} D_{31} u_{1}  + b_{x_{3}} D_{3}\zeta_{1} D_{1} u_{1} \big) \sqrt{|g|} dx	\\
		&\quad + \int_{\mathbb{R}^{3}_{+}} \Big( - b_{x_{1}}\Gamma^{1}_{D,22} - b_{x_{3}}\Gamma^{3}_{D,22} + b_{x_{2}}\Gamma^{2}_{D,11} + b_{x_{2}}\Gamma^{2}_{D,33} \Big) \zeta_{1} D_{1} u_{1} \sqrt{|g|} dx	\\
		&\quad + \int_{\mathbb{R}^{3}_{+}} \Big( - b_{x_{1}}\Gamma^{1}_{D,33} - b_{x_{2}}\Gamma^{2}_{D,33} + b_{x_{3}}\Gamma^{3}_{D,11} + b_{x_{3}}\Gamma^{3}_{D,22} \Big) \zeta_{1} D_{1} u_{1} \sqrt{|g|} dx	\\
		&\quad + C \int_{x_{3}=0} \Big(\underbrace{ \int_{\mathbb{R}^{3}} f^{b} \phi_{1}(\V) \V_{3} }_{ =0 \ \text{oddness in} \ \V_{3} }\Big) \psi_{1}(x) \sqrt{g_{11}g_{22}} dA + C \int_{x_{3}=0} \Big(\underbrace{ \int_{\mathbb{R}^{3}} f^{b} \phi_{2}(\V) \V_{3} }_{ =0 \ \text{oddness in} \ \V_{3} }\Big) \psi_{2}(x) \sqrt{g_{11}g_{22}} dA .
		\end{split}
		\end{equation} 
		}
		\subsubsection{Condition for $\zeta_{1}(x)$}
		We add (\ref{3D btan sym}), (\ref{3D btan nonsym1}), and $(-1)\times$(\ref{3D btan nonsym2}). This is equivalent to applying test function
		\begin{equation} \label{Psib1}
		\Psi_{b,1} := \Psi_{1,1} + \Psi_{1,2} - \Psi_{1,3},
		\end{equation}
		to (\ref{loc wk form}). Using (\ref{noncommu}),
		\begin{equation}  \label{btan part}
		\begin{split}
		&- \iint_{ U\times\mathbb{R}^{3}} \mathbf{P}f (v\cdot \nabla) \Psi_{b,1}  dydv + \iint_{\p U\times\mathbb{R}^{3}} f^b \Psi_{b,1}^{b} (v\cdot n)dAdv \\
		&= - \int_{\mathbb{R}^{3}_{+}} b_{x_{1}} \zeta_{1} \big( D_{1}^{2} + D_{2}^{2} + D_{3}^{2}  \big) u_{1} \sqrt{|g|} dx	\\
		& + \int_{\mathbb{R}^{3}_{+}} b_{x_{1}} \Big( 2\Gamma_{D,22}^{1}\zeta_{1} + 2\Gamma_{D,33}^{1}\zeta_{1} - D_{1}\zeta_{1} \Big) D_{1} u_{1} \sqrt{|g|} dx	\\
		& + \int_{\mathbb{R}^{3}_{+}} b_{x_{1}} \Big( 2\Gamma_{D,11}^{2}\zeta_{1} - D_{2}\zeta_{1} + \Gamma_{D,33}^{2}\zeta_{1} \Big) D_{2} u_{1} \sqrt{|g|} dx	\\
		& + \int_{\mathbb{R}^{3}_{+}} b_{x_{1}} \Big( 2\Gamma_{D,11}^{3}\zeta_{1} - D_{3}\zeta_{1} + \Gamma_{D,22}^{3}\zeta_{1} \Big) D_{3} u_{1} \sqrt{|g|} dx	\\
		& + \sum_{i=2,3} \int_{\mathbb{R}^{3}_{+}} b_{x_{i}} \Big( -\Gamma_{D,11}^{i}\zeta_{1} + D_{i}\zeta_{1} \Big)	D_{1} u_{1} \sqrt{|g|} dx	\\
		& + \sum_{i=2,3} \int_{\mathbb{R}^{3}_{+}} \Big( \Gamma_{D,22}^{1}\zeta_{1} + \Gamma_{D,33}^{1}\zeta_{1}  - D_{1}\zeta_{1} \Big) b_{x_{i}}D_{i}u_{1} \sqrt{|g|} dx .
		\end{split}
		\end{equation}	
		Similar to the normal momentum estimate, if we match index $3\rightarrow r,1 \rightarrow \theta, 2 \rightarrow z$ and pick 
		\begin{align} \label{zeta r}
			\zeta_{1} = \zeta_{\theta} = \frac{1}{r},
		\end{align}
		 $\zeta_{1}(x)$ satisfies:
		\begin{equation} \label{btan zeta eq}
		\begin{cases}
		D_{1}\zeta_{1} &= \Gamma_{D,22}^{1}\zeta_{1} + \Gamma_{D,33}^{1}\zeta_{1}	\\
		D_{2}\zeta_{1} &= \Gamma_{D,11}^{2}\zeta_{1}	\\
		D_{3}\zeta_{1} &= \Gamma_{D,11}^{3}\zeta_{1}	\\
		\end{cases},
		\quad\text{and}\quad c \leq |\zeta_{1}| \leq C, 
		\end{equation}
		for uniformly positive $c$ and $C$ from \eqref{ann domain}, \eqref{diff operator} and \eqref{christoffel}. Then above system becomes nearly same as (\ref{bn final sum}) and using Lemma \ref{D bel difference}, the LHS of (\ref{glo IBP re}) is
		\begin{equation} \label{btan final sum}
		\begin{split}
		&- \iint_{ U\times\mathbb{R}^{3}} \mathbf{P}f (v\cdot \nabla) \Psi_{b,1}  dydv + \iint_{\p U\times\mathbb{R}^{3}} f^b \Psi_{b,1}^{b} (v\cdot n) dAdv\\
		&= - \int_{\mathbb{R}^{3}_{+}} b_{x_{1}} \zeta_{1} \big( D_{1}^{2} + D_{2}^{2} + D_{3}^{2}  \big) u_{1} \sqrt{|g|} dx + \int_{\mathbb{R}^{3}_{+}} b_{x_{1}}	\zeta_{1} \sum_{i=1}^{3} \sum_{k\neq i} \Gamma_{D,kk}^{i}D_{i}u_{1} \sqrt{|g|} dx,	\\
		&= - \int_{\mathbb{R}^{3}_{+}} b_{x_{1}} \zeta_{1} \Delta_{bel} u_{1} \sqrt{|g|} dx  \\
		&=  \int_{ U} b_{x_{1}}^{2} dy - \frac{1}{| U|} \int_{ U} \zeta_{1} b_{x_{1}} \int_{ U} \frac{1}{\zeta_1}b_{x_{1}}.  \\
		\end{split}
		\end{equation}

		With $\Psi_{b,1}$, which is defined in (\ref{Psib1}), the RHS of (\ref{glo IBP re}) is controlled by 
		\begin{equation} \label{btan RHS}
		\begin{split}
		\text{RHS of (\ref{glo IBP re})} &\lesssim \Big| \iint_{ U\times\mathbb{R}^{3}} \Psi_{b,1} \p_{t} f dy \Big| + \varepsilon \|u_{1}\|_{H^{2}}^{2} + \|(\mathbf{I-P})f\|_{2}^{2} + \| g \|_{2}^{2} ,\quad \varepsilon \ll 1. \\
		&\lesssim \frac{d}{dt}G(t) + \underbrace{ \Big| \iint_{ U\times\mathbb{R}^{3}} f \p_{t}\Psi_{b,1} dy \Big| }_{(**)} + \varepsilon \|b_{x_{1}}\|_{2}^{2} + \|\sqrt{\nu}(\mathbf{I-P})f\|_{2}^{2} + \| g \|_{2}^{2},\\
		\end{split}
		\end{equation}
		where $\vert G(t)\vert \lesssim \|f(t)\|_{2}^{2}$. Here, we used $\|u_{1}\|_{H^{2}}^2 \lesssim \|b_{x_{1}}\|_{2}^{2}$ and $\Vert \cdot \Vert_{2} \leq \Vert \cdot \Vert_{\nu}$.  \\
		
		To estimate time part $(**)$ in (\ref{btan RHS}), we apply another test function,
		\begin{equation} \label{b1 time test}
		\varphi_{1} = \V_{1} \sqrt{\mu} \Delta^{-1}_{N} \widetilde{ \p_{t}b_{x_{1}} } .
		\end{equation} 
		When this test function is combined with $\p_t f$, we get
		\begin{equation*}
		\begin{split}
		&\iint_{ U \times \mathbb{R}^{3}} \p_t f \V_{1}\sqrt{\mu} \Delta^{-1}_{N} \widetilde{ \p_{t}b_{x_{1}} } = \int_{ U} \p_t b_{x_{1}} \Delta_{N}^{-1}\widetilde{ \p_{t}b_{x_{1}} } \\
		&= \int_{ U} \nabla\cdot\nabla \Delta_{N}^{-1} \widetilde{ \p_t b_{x_{1}} } \Delta_{N}^{-1} \widetilde{ \p_{t}b_{x_{1}} } + \langle \p_{t} b_{x_{1}} \rangle \underbrace{ \int_{ U} \Delta_{N}^{-1} \widetilde{ \p_{t}b_{x_{1}}} }_{=0} \\
		&= \| \nabla\Delta_{N}^{-1} \widetilde{ \p_{t}b_{x_{1}} }   \|^{2}_{2}.
		\end{split}
		\end{equation*}
		On the other hand, considering other terms except $\p_t f$,
		\begin{equation} \label{b1 time 12 required}
		\begin{split}
		\|\nabla\Delta_{N}^{-1} \widetilde{ \p_{t}b_{x_{1}} } \|^{2}_{2} 
		&\lesssim \Big| \iint_{ U\times\mathbb{R}^3} \V_{1}\sqrt{\mu} \Delta^{-1}_N \widetilde{ \p_{t}b_{x_{1}} } \Big( - v\cdot \nabla f + g - L f  \Big)  \Big| \\
		&\lesssim \Big| \iint_{ U\times\mathbb{R}^3} (\mathbf{P}f + (\mathbf{I}-\mathbf{P})f) \ \nabla\cdot\big( v \V_{1} \sqrt{\mu} \Delta^{-1}_{N} \widetilde{ \p_{t}b_{x_{1}} } \big) \Big| \\
		&\quad + \ \Big| \iint_{ U\times\mathbb{R}^3} \V_{1}\sqrt{\mu} \Delta^{-1}_{N} \widetilde{ \p_{t}b_{x_{1}} } \big( g - Lf \big) \Big| \\
		&\lesssim \varepsilon \|\nabla\Delta_{N}^{-1} \widetilde{  \p_{t}b_{x_{1}}  } \|^{2}_{2}  + \Big( \|a\|_{2}^{2} + \|c\|_{2}^{2} + \|\sqrt{\nu}(\mathbf{I}-\mathbf{P})f\|_{2}^{2} + \|g\|_{2}^{2} \Big) , \\
		\end{split}	
		\end{equation}
		where sufficiently small $\varepsilon \ll 1$ comes from Young's inequality and we used Poincare inequality. Therefore we get 
		\begin{equation} \label{mid step for b1}
		\|\nabla\Delta_{N}^{-1} \widetilde{ \p_{t}b_{x_{1}} } \|^{2}_{2} \lesssim  \|a\|_{2}^{2} + \|c\|_{2}^{2} + \|\sqrt{\nu}(\mathbf{I}-\mathbf{P})f\|_{2}^{2} + \|g\|_{2}^{2}.
		\end{equation}
		Using (\ref{mid step for b1}), we can estimate $(**)$ in (\ref{btan RHS}) by
		\begin{equation} \label{b1 time est}
		\begin{split}
		\Big| \iint_{ U\times\mathbb{R}^{3}} f \p_{t}\Psi_{b,1} dy \Big|  
		&\lesssim \Big| \iint_{ U\times\mathbb{R}^{3}} \mathbf{P}f \p_{t}\Psi_{b,1} dy \Big| + \Big| \iint_{ U\times\mathbb{R}^{3}} (\mathbf{I-P})f \p_{t}\Psi_{b,1} dy \Big|  \\
		&\lesssim \|c\|_{2} \|\nabla\Delta_{N}^{-1} \widetilde{ \p_{t}b_{x_{1}}} \|_{2}  +  \big( \varepsilon\|\nabla\Delta_{N}^{-1}  \widetilde{ \p_{t}b_{x_{1}} } \|^{2}_{2} + \|(\mathbf{I-P})f\|_{2}^{2} \big) \\
		&\lesssim \varepsilon\|a\|_{2}^{2} + \|c\|_{2}^{2} + \|\sqrt{\nu}(\mathbf{I}-\mathbf{P})f\|_{2}^{2} + \|g\|_{2}^{2},
		\end{split}
		\end{equation}
		where we used 
		\begin{equation}  
		\begin{split}
		&\iint_{ U\times\mathbb{R}^{3}} \mathbf{P}f \p_{t}\Psi_{b,1}  \\
		&=  \iint_{ U\times\mathbb{R}^{3}} \mathbf{P}f \p_{t}\Psi_{1,1} + \iint_{ U\times\mathbb{R}^{3}} \mathbf{P}f \p_{t}\Psi_{1,2} - \iint_{ U\times\mathbb{R}^{3}} \mathbf{P}f \p_{t}\Psi_{1,3}  \\
		&\leq \iint_{ U\times\mathbb{R}^{3}} \Big(a  + b\cdot v + c \frac{|v|^2-3}{2}\Big) \sqrt{\mu} (\V_{1}^{2} - \b_{b}) \sqrt{\mu} \zeta_{1} D_{1} \Delta_{N}^{-1} \widetilde{\p_{t} b_{x_{1}}}  \\
		&\quad + \underbrace{ \sum_{i=2}^{3} \iint_{ U\times\mathbb{R}^{3}} \Big(a  + b_{x}\cdot \V + c \frac{|v|^2-3}{2}\Big) \sqrt{\mu} (\V_{i} \V_{1}) \sqrt{\mu} \zeta_{1} D_{1} \Delta_{N}^{-1}\widetilde{\p_{t} b_{x_{1}}} }_{=0 \ \text{odd in velocity}}  \\
		&\quad - \sum_{i=2}^{3} \iint_{ U\times\mathbb{R}^{3}} \Big(a  + b\cdot v + c \frac{|v|^2-3}{2}\Big) \sqrt{\mu} (\V_{i}^{2} - \b_{b}) \sqrt{\mu} \zeta_{1} D_{1} \Delta_{N}^{-1}\widetilde{\p_{t} b_{x_{1}}}  \\
		&\lesssim \|c\|_{2} \|\nabla\Delta_{N}^{-1}  \widetilde{  \p_{t}b_{x_{1}} } \|_{2} \\
		&\lesssim \varepsilon\|a\|_{2}^{2} + \|c\|_{2}^{2} + \|\sqrt{\nu}(\mathbf{I}-\mathbf{P})f\|_{2}^{2} + \|g\|_{2}^{2} .
		\end{split}
		\end{equation} 
		
		Combining (\ref{btan zeta eq}), (\ref{btan final sum}), (\ref{btan RHS}), and (\ref{b1 time est}), we gain
		\begin{equation} \label{b1 est}
		\begin{split}
		\int_{ U} b_{x_{1}}^{2} dy &\lesssim \frac{d}{dt}G(t) + \varepsilon\|a\|_{2}^{2} + \|c\|_{2}^{2} + \|\sqrt{\nu}(\mathbf{I-P})f\|_{2}^{2} + \| g \|_{2}^{2} + \int_{ U} \zeta_{1} b_{x_{1}} \int_{ U} \frac{1}{\zeta_1} b_{x_{1}},  \\
		\end{split}
		\end{equation} 
		where $\vert G(t)\vert \lesssim \|f(t)\|_{2}^{2}$. 
		
		\subsubsection{Tangential momentum $b_{x_{2}}$ estimate.}
		This is nearly similar as what we computed for $b_{x_{1}}$ and $u_{1}$ case. We define
		\begin{equation} \label{u2 eq}
		-\Delta_{bel} u_{2} = \frac{1}{\zeta_2} b_{x_{2}} - \langle \frac{1}{\zeta_2} b_{x_{2}} \rangle := \widetilde{\frac{1}{\zeta_2}b_{x_{2}}},\quad \p_{n} u_{2}\vert_{\p U} = 0.
		\end{equation}
		\noindent \textit{(Symmetric case)} $\Psi_{2,1} = \phi_{2}(\V) \psi_{2}(x)$, where
		\begin{equation} \label{3D btan2 sym test}
		\begin{split}
		\phi_{2}(\V) &= (\V^{2}_{2} - \b_b) \sqrt{\mu},\quad\text{where}\quad \int_{\mathbb{R}^3} (\V^{2}_{i} - \b_b) \mu d\V = 0 ,	\\
		\psi_{2}(x) &= \zeta_{2}(x) D_{2}u_{2}.
		\end{split}
		\end{equation}
		\noindent \textit{(Non-symmetric case 1)}
		We try test function $\Psi_{2,2} = \sum_{i=1,3} \phi_{i}(\V) \psi_{i}(x)$, where
		\begin{equation} \label{3D btan2 nonsym1 test}
		\begin{split}
		\phi_{i}(\V) &= \V_{i} \V_{2} \sqrt{\mu},	\\
		\psi_{i}(x) &= \zeta_{2}(x) D_{i}u_{2},
		\end{split}
		\end{equation}
		for $i=1,3$. \\
		\noindent \textit{(Non-symmetric case 2)}
		We try test function $\Psi_{2,3} = \sum_{i=1,3} \phi_{i}(\V) \psi_{i}(x)$, where
		\begin{equation} \label{3D btan2 nonsym2 test}
		\begin{split}
		\phi_{i}(\V) &= (\V^{2}_{i} - \b_b) \sqrt{\mu},\quad\text{where}\quad \int_{\mathbb{R}^3} (\V^{2}_{i} - \b_{b}) \mu d\V = 0  ,	\\
		\psi_{i}(x) &= \zeta_{2}(x) D_{2}u_{2},
		\end{split}
		\end{equation}
		for $i=1,3$. 
		We just change role of index $1$ and $2$ to gain. If we match index $3\rightarrow r,1 \rightarrow \theta, 2 \rightarrow z$ and pick 
		\begin{align} \label{zeta r}
			\zeta_{2} = \zeta_{z} = 1,
		\end{align}
		 $\zeta_{2}(x)$ satisfies:
		\begin{equation} \label{btan2 zeta eq}
		\begin{cases}
		D_{1}\zeta_{2} &= \Gamma_{D,22}^{1}\zeta_{2}	\\
		D_{2}\zeta_{2} &= \Gamma_{D,11}^{2}\zeta_{2} + \Gamma_{D,33}^{2}\zeta_{2}	\\
		D_{3}\zeta_{2} &= \Gamma_{D,22}^{3}\zeta_{2}	\\
		\end{cases}
		\quad\text{and}\quad c \leq |\zeta_{2}| \leq C, 
		\end{equation}
		with uniform positive constants $c$ and $C$ due to \eqref{diff operator} and \eqref{christoffel}. Considering above three test function cases, by defining 
		\begin{equation} \label{Psib2}
		\Psi_{b,2} = \Psi_{2,1} + \Psi_{2,2} - \Psi_{2,3},
		\end{equation}
		we get
		\begin{equation} \label{btan2 final sum}
		\begin{split}
		&- \iint_{ U\times\mathbb{R}^{3}} \mathbf{P}f (v\cdot \nabla) \Psi_{b,2}  dydv + \iint_{\p U\times\mathbb{R}^{3}} f^b \Psi_{b,2}^{b} (v\cdot n) dAdv \\
		&=  \int_{ U}  b_{x_{2}}^{2} dy - \frac{1}{| U|} \int_{ U} \zeta_{2} b_{x_{2}} \int_{ U} \frac{1}{\zeta_2} b_{x_{2}}.  \\
		\end{split}
		\end{equation}
		
		With $\Psi_{b,2}$, similar as (\ref{b1 est}), we have   
		\begin{equation} \label{b2 est}
		\begin{split}
		\int_{ U}  b_{x_{2}}^{2} dy &\lesssim \frac{d}{dt}G(t) + \varepsilon\|a\|_{2}^{2} + \|c\|_{2}^{2} + \|\sqrt{\nu}(\mathbf{I-P})f\|_{2}^{2} + \| g \|_{2}^{2} + \int_{ U} \zeta_{2} b_{x_{2}} \int_{ U} \frac{1}{\zeta_2} b_{x_{2}}, \\
		\end{split}
		\end{equation} 
		where $\vert G(t)\vert \lesssim \|f(t)\|_{2}^{2}$.

		\hide
		\begin{remark}
		If $U$ is a periodic in-$x_2$ cylindrical domain with global triply orthogonal coordinate system, then it is parametrized by
		\begin{align*}
			\eta(x_1,x_2,x_3) := (\eta_1(x_1,x_3), \eta_2(x_1,x_3), x_2), \quad (x_1,x_3) \in [c_1,C_1] \times [c_2,C_2]\; \text{and }0\leq x_2 \leq H,
		\end{align*}
		for fixed $H>0$. Since 
		\begin{align*}
		\sqrt{g_{22}} =1, \quad \Gamma_{D,22}^{1}=\Gamma_{D,22}^{3}=\Gamma_{D,11}^{2}=\Gamma_{D,33}^{2}=0, 
		\end{align*}
		we obtain $\zeta_2=\sqrt{g_{22}}=1$ from 
		\begin{align*}
			\begin{cases}
		D_1 \zeta_2 = \Gamma_{D,22}^{1} \zeta_2 =0,\\ 
		D_2 \zeta_2 = (\Gamma_{D,11}^{2} + \Gamma_{D,33}^{2})\zeta_2 =0, \\
		D_3 \zeta_2 = \Gamma_{D,22}^{3}\zeta_2 =0. 
			\end{cases}
		\end{align*}
		If $U$ is an axis-symmetric domain with respect to $z$-axis, then it can be parametrized by 
		\begin{align*}
			\eta(x_1,x_2,x_3) = (\eta_1 (x_2,x_3) \cos x_1, \eta_1 (x_2,x_3) \sin x_1, \eta_2 (x_2, x_3)),		
		\end{align*}
		for $0\leq x_1 \leq 2\pi \; \text{and } (x_2,x_3) \in [d_1,D_1] \times [d_2,D_2].$ Since 
		\begin{align*}
			&\sqrt{g_{11}}= \eta_1, \sqrt{g_{22}}=\sqrt{(\p_2 \eta_1)^2 +(\p_2 \eta_2)^2}, \sqrt{g_{33}} = \sqrt{(\p_3 \eta_1)^2 + (\p_3 \eta_2)^2} \\
			&\Gamma_{D,11}^{2} = -\frac{\partial_2 \eta_1}{\eta_1 \sqrt{(\p_2 \eta_1)^2 + (\p_2 \eta_2)^2}}, \; \Gamma_{D,11}^{3} = -\frac{ \p_3 \eta_1}{\eta_1 \sqrt{(\p_3 \eta_1)^2 + (\p_3 \eta_2)^2}}, \; \Gamma_{D,22}^{1}=\Gamma_{D,33}^{1} =0,
		\end{align*}
		we obtain $\zeta_1= \frac{1}{\sqrt{g_{11}}}=\frac{1}{\eta_1}$ from 
		\begin{align*}
			\begin{cases}
				D_1 \zeta_1 = \frac{1}{\eta_1} \p_1 \zeta_1 = (\Gamma_{D,22}^1 + \Gamma_{D,33}^1) \zeta_1 = 0, \\
				D_2 \zeta_1 = \frac{1}{\sqrt{(\p_2 \eta_1)^2 + (\p_2 \eta_2)^2}} \p_2 \zeta_1 = \Gamma_{D,11}^2 \zeta_1 =-\frac{\partial_2 \eta_1}{\eta_1 \sqrt{(\p_2 \eta_1)^2 + (\p_2 \eta_2)^2}} \zeta_1,\\
				D_3 \zeta_1 = \frac{1}{\sqrt{(\p_3 \eta_1)^2 + (\p_3 \eta_2)^2}} \p_3 \zeta_1 = \Gamma_{D,11}^3 \zeta_1 =-\frac{\partial_3 \eta_1}{\eta_1 \sqrt{(\p_3 \eta_1)^2 + (\p_3 \eta_2)^2}} \zeta_1.
			\end{cases}
		\end{align*}
		\end{remark}
		\subsection{Mean estimate.}
		From (\ref{b1 est}) and (\ref{b2 est}), if we must control 
		\[
		\int_{ U} \zeta_{1} b_{x_{1}} \int_{ U} \frac{1}{\sqrt{g_{11}}} b_{x_{1}}   \quad\text{and}\quad  \int_{ U} \zeta_{2} b_{x_{2}} \int_{ U} \frac{1}{\sqrt{g_{22}}} b_{x_{2}}  ,  \\
		\]
		to close coercivity estimate, where these terms come from mean of (\ref{u1 eq}) and (\ref{u2 eq}). \\

		\begin{lemma} \label{lemma mean}
			Let us assume 
			\begin{align*}
			0 < c \leq \sqrt{g_{ii}} \leq C, \quad  \vert \zeta_i \vert \leq C, \; \text{and } \; \vert D\zeta_i \vert \leq C, \quad  i=1,2,
			\end{align*}
			for uniform positive constant $c$ and $C$, where $D$ is differential operator defined in \eqref{diff operator}.  For $\varepsilon>0$, we have the following mean estimate for tangential momentum for $i=1,2$,
			\begin{equation} \label{mean est}
			\int_{U} \zeta_{i} b_{x_{i}} \int_{U} \frac{1}{\sqrt{g_{ii}}} b_{x_{i}}
			\lesssim \frac{d}{dt} G(t) + \varepsilon\|a\|_{2}^{2} + \|c\|_{2}^{2} + \varepsilon\|b_{x_{i}}\|_{2}^{2} + \|\sqrt{\nu}(\mathbf{I-P})f\|_{2}^{2} + \|g\|_{2}^{2},
			\end{equation}
			where $G(t)$ satisfies $\vert G(t) \vert \lesssim \|f(t)\|_{2}^{2}$.
		\end{lemma}
		\begin{proof}
		When $i$ is index for tangential direction, we try test function 
		\[
			\phi_{i}(t,v,x) = \phi_{i}(v,x) = x_{i} \big( \frac{|v|^{2} - 3}{2} \big) \sqrt{\mu},
		\]
		to Boltzmann equation then we get
		\begin{equation} \label{mean wk}
		\begin{split}
		0 &= \iint_{U\times \mathbb{R}^{3}} \phi_{i} \p_{t} f + \iint_{U\times \mathbb{R}^{3}} \phi_{i} (v\cdot\nabla) f + \iint_{U\times \mathbb{R}^{3}} \phi_{i} \big( Lf - g \big) \\
		&= \frac{d}{dt}\iint_{U\times \mathbb{R}^{3}} f\phi_{i}(t) + \iint_{\p U\times \mathbb{R}^{3}} \phi_{i}(v\cdot n) f^{b} - \underbrace{ \iint_{U\times \mathbb{R}^{3}} f (v\cdot\nabla) x_{i} \big( \frac{|v|^{2} - 3}{2} \big) \sqrt{\mu} }_{(***)}  \\
		&\quad + \iint_{U\times \mathbb{R}^{3}} x_{i} \big( \frac{|v|^{2} - 3}{2} \big) \sqrt{\mu} \big( Lf - g \big).
		\end{split}
		\end{equation}	
		Note that $(***)$ becomes,
		\begin{equation} \label{3*}
		\begin{split}
		(***) 
		&= \iint_{U\times \mathbb{R}^{3}} f (\V\cdot D) x_{i} \big( \frac{|v|^{2} - 3}{2} \big) \sqrt{\mu}  \\
		&= \int_{U} \frac{1}{\sqrt{g_{ii}}} \int_{\mathbb{R}^{3}} f \V_{i} \sqrt{\mu} \big( \frac{|v|^{2} - 3}{2} \big)  \\
		&= \int_{U} \frac{1}{\sqrt{g_{ii}}} \Big( \int_{\mathbb{R}^{3}} \mathbf{P}f \frac{|v|^{2}-3}{2} \V_{i} \sqrt{\mu} dv + \int_{\mathbb{R}^{3}} (\mathbf{I-P})f \frac{|v|^{2}-3}{2} \V_{i} \sqrt{\mu} dv  \Big)   \\
		&= \int_{U} \frac{1}{\sqrt{g_{ii}}} \Big( \int_{\mathbb{R}^{3}} b_{x_{i}} \V_{i}^{2} \frac{|v|^{2}-3}{2} {\mu} dv + \int_{\mathbb{R}^{3}} (\mathbf{I-P})f \frac{|v|^{2}-3}{2} \V_{i} \sqrt{\mu} dv  \Big)   \\
		&= M \int_{U} \frac{1}{\sqrt{g_{ii}}} b_{x_{i}} + \int_{U} \frac{1}{\sqrt{g_{ii}}} \int_{\mathbb{R}^{3}} (\mathbf{I-P})f \frac{|v|^{2}-3}{2} \V_{i} \sqrt{\mu} dv  ,\quad M:= \int_{\mathbb{R}^{3}} \V_{i}^{2} \frac{|v|^{2}-3}{2} {\mu} dv > 0.   \\
		\end{split}
		\end{equation}
		From (\ref{3*}) and (\ref{mean wk}),
		\begin{equation} \label{star pre est}
		\begin{split}
		\int_{U} \frac{M}{\sqrt{g_{ii}}} b_{x_{i}}
		&= \frac{d}{dt}\iint_{U\times \mathbb{R}^{3}} f\phi_{i}(t) - \int_{U} \frac{1}{\sqrt{g_{ii}}} \int_{\mathbb{R}^{3}} (\mathbf{I-P})f \frac{|v|^{2}-3}{2} \V_{i} \sqrt{\mu} dv  \\
		&\quad + \iint_{U\times \mathbb{R}^{3}} x_{i} \big( \frac{|v|^{2} - 3}{2} \big) \sqrt{\mu} \big( Lf - g \big).
		\end{split}
		\end{equation}
		Multiplying $\int \zeta_{i}b_{x_{i}}$ to (\ref{star pre est}) and using uniform positivity of $\frac{M}{\sqrt{g_{ii}}}$ (since $0 < c \leq \sqrt{g_{ii}} \leq C$), 
		\begin{equation} \label{pre mean est}
		\begin{split}
			M \int_{U} \zeta_{i}b_{x_{i}} \int_{U} \frac{1}{\sqrt{g_{ii}}} b_{x_{i}}
			&\lesssim \int_{U} \zeta_{i}b_{x_{i}} \frac{d}{dt} \iint_{U\times \mathbb{R}^{3}} f\phi_{i} + \int_{U} \zeta_{i}b_{x_{i}} \big( \|(\mathbf{I-P})f\|_{2} + \|g\|_{2} \big)  \\
			&\lesssim \frac{d}{dt} \underbrace{ \Big( \int_{U} \zeta_{i}b_{x_{i}} \iint_{U\times \mathbb{R}^{3}} f\phi_{i} \Big) }_{\sim G(t)} - \iint_{U\times \mathbb{R}^{3}} f\phi_{i} \int_{U} \zeta_{i} \p_{t}b_{x_{i}} + \varepsilon\|b_{x_{i}}\|_{2}^{2} + \|(\mathbf{I-P})f\|_{2}^{2} + \|g\|_{2}^{2}  \\
			&\lesssim \frac{d}{dt} G(t) - \int_{U} c x_{i} \underbrace{ \int_{U} \zeta_{i} \p_{t}b_{x_{i}} }_{A} + \varepsilon\|b_{x_{i}}\|_{2}^{2} + \|(\mathbf{I-P})f\|_{2}^{2} + \|g\|_{2}^{2},
		\end{split}
		\end{equation}
		where $\varepsilon>0$ comes from Young's inequality. 
		To estimate $A$, we use test function $\tilde{\phi} = \zeta_{i} v_{x_{i}}\sqrt{\mu}$. Since $\tilde{\phi}$ is uniformly bounded, we obtain,
		\begin{equation} \label{dt b}
		\begin{split}
			A &= \int_{U} \zeta_{i} \p_{t}b_{x_{i}} = -\iint_{U\times \mathbb{R}^{3}} \tilde{\phi} (v\cdot\nabla) f - \iint_{U\times \mathbb{R}^{3}} \tilde{\phi}(Lf - g)  \\
			&= - \iint_{\p U\times \mathbb{R}^{3}} \zeta_{i} v_{x_{i}} (v\cdot n) f^{b} + \iint_{U\times \mathbb{R}^{3}} v_{x_{i}} f\sqrt{\mu} (v_{x}\cdot D)\zeta_{i} + \iint_{U\times \mathbb{R}^{3}} \zeta_{i} f\sqrt{\mu} (v_{x}\cdot D)v_{x_{i}} - \iint_{U\times \mathbb{R}^{3}} \tilde{\phi}(Lf - g)  \\
			&\lesssim \|a\|_{2} + \|c\|_{2} + \|(\mathbf{I-P})f\|_{2} + \|g\|_{2},
		\end{split}
		\end{equation}
		where we used $|\zeta_{i}|, \  |D\zeta_{i}| \leq C$. Note that we assumed $|\zeta_{i}| \leq C$ was assumed in (\ref{btan zeta eq}) and (\ref{btan2 zeta eq}) already. Second condition $|D\zeta_{i}| \leq C$ also follows from three equations of $\zeta_{i}$ in (\ref{btan zeta eq}) and (\ref{btan2 zeta eq}) with uniform bound of $|\zeta_{i}| \leq C$ and $\Gamma_{D,ij}^{k}$. 
		Now, from (\ref{pre mean est}),(\ref{dt b}), and $\Vert \cdot \Vert_2 \leq \Vert \cdot \Vert_{\nu}$, we obtain
		\begin{equation*}
			\int_{U} \zeta_{i}b_{x_{i}} \int_{U} \frac{1}{\sqrt{g_{ii}}} b_{x_{i}}
			\lesssim \frac{d}{dt} G(t) + \varepsilon\|a\|_{2}^{2} + \|c\|_{2}^{2} + \varepsilon\|b_{x_{i}}\|_{2}^{2} + \|\sqrt{\nu}(\mathbf{I-P})f\|_{2}^{2} + \|g\|_{2}^{2}.
		\end{equation*}
		\end{proof}
		\unhide
		
		\begin{proposition} \label{test coercivity}
			Let $ U$ be an annular cylinder domain \eqref{domain} which is periodic in vertical $z$-direction. We assume that $f$ solves the linearized Boltzmann equation \eqref{linear BE} with specular boundary condition \eqref{specular}. 
			\hide
			Assume that there exist uniform positive constants $c$ and $C$ such that $0 < c \leq \sqrt{g_{ii}} \leq C$ for all $i=1,2,3$. Assuming that $\zeta_{3}$ satisfies (\ref{bn zeta eq}), and if $\zeta_{1}$ and $\zeta_{2}$ satisfy (\ref{btan zeta eq}),(\ref{btan2 zeta eq}), and 
			\begin{align*}
				\vert D\zeta_i \vert \leq C, \quad i=1,2,
			\end{align*}
			\unhide
			 then we have coercivity estimate 
			\begin{equation} \label{ann Pf est}
			\begin{split}
			\|\mathbf{P}f\|_{\nu}^{2} &\lesssim \frac{d}{dt}G(t) + \|(\mathbf{I-P})f\|_{\nu}^{2} + \| g \|_{2}^{2} ,\quad \vert G(t) \vert \lesssim \|f(t)\|_{2}^{2}, \\
			\end{split}
			\end{equation}
			\hide
			Moreover, if $U$ is an axis-symmetric domain, \eqref{ann Pf est} also holds 
			\unhide
			under the assumption for angular momentum and momentum with respect to $z$-axis:
			\begin{align} \label{conserve}
				\int_{U \times \R^3} \{(y-y_0) \times z\} \cdot v f(t,y,v) \sqrt{\mu} dydv=0 \text{ and } \int_{U\times \R^3} v_3f(t,y,v) \sqrt{\mu} dydv=0,				\end{align}
		for all $t\geq0$. 
		\end{proposition}

		\begin{proof}
			From (\ref{b3 est}), (\ref{b1 est}), and (\ref{b2 est}), we have
			\begin{equation} \label{b with mean}
			\begin{split}
			\|b\|_{L^{2}_{y}}^{2} &\lesssim \frac{d}{dt}G(t) + \varepsilon\|a\|_{L^{2}_{y}}^{2} + \|c\|_{L^{2}_{y}}^{2} + \|\sqrt{\nu}(\mathbf{I-P})f\|_{2}^{2} + \| g \|_{2}^{2}  \\
			&\quad + \int_{ U} \zeta_{1} b_{x_{1}} \int_{ U} \frac{1}{\zeta_{1}} b_{x_{1}} + \int_{U} \zeta_{2} b_{x_{2}} \int_{ U} \frac{1}{\zeta_{2}} b_{x_{2}}, \quad \vert G(t) \vert \lesssim \|f(t)\|_{2}^{2}. \\
			\end{split}
			\end{equation} 
			WLOG, we set $y_0=(0,0,0)$ and $z=(0,0,1)$ in \eqref{conserve}. Then, we can rewrite \eqref{conserve} by using the global triply orthogonal coordinate system \eqref{ann domain}:
			\begin{align} \label{conserve zeta}
				\begin{split}
				&\int_{U\times \R^3} \{ (y-y_0) \times z\} \cdot v f(t,y,v) \sqrt{\mu} dydv \\
				&\quad = \int_{U \times \R^3} (r \sin \theta, -r \cos \theta,0) \cdot v f(t,y,v) \sqrt{\mu} dy dv =\int_{U} \frac{1}{\zeta_1} b_{x_1} =0,\\
				& \int_{U\times \R^3} v_3f(t,y,v) \sqrt{\mu} dydv\\ 
				&\quad =\int_{U} (0,0,1) \cdot v f(t,y,v) \sqrt{\mu}dydv = \int_{U} \frac{1}{\zeta_2} b_{x_{2}} =0. 
				\end{split}
			\end{align}
			Therefore, the last two terms in the RHS of (\ref{b with mean}) become 0, and then
			\hide
			Then, 
			\begin{align*}
				\Vert b \Vert_{L^{2}_{y}}^2 &\lesssim \frac{d}{dt} G(t) +\varepsilon \Vert a \Vert_{L^{2}_{y}}^2 +\Vert c \Vert_{L^{2}_{y}}^2 +  \|\sqrt{\nu}(\mathbf{I-P})f\|_{2}^{2} + \|g\|_{2}^{2}.
			\end{align*}
			\unhide
			 we obtain 
			\begin{align} \label{b esti}
				\Vert b \Vert_{L^{2}_{y}}^2 \lesssim \frac{d}{dt} G(t) +\varepsilon \Vert a \Vert_{L^{2}_{y}}^2 +\Vert c \Vert_{L^{2}_{y}}^2 +  \|\sqrt{\nu}(\mathbf{I-P})f\|_{2}^{2} + \|g\|_{2}^{2}.
			\end{align}
			Combining with Proposition \ref{prop_massenergy} and \eqref{b esti}, we have 
			\begin{align} \label{a,c esti}
				\begin{split}
				\Vert a \Vert_{L^{2}_{y}}^2 &\lesssim \frac{d}{dt} G(t) + \varepsilon \Vert a \Vert_{L^{2}_{y}}^2 +\Vert c \Vert_{L^{2}_{y}}^2 +  \|\sqrt{\nu}(\mathbf{I-P})f\|_{2}^{2} + \|g\|_{2}^{2},\\
				\Vert c \Vert_{L^{2}_{y}}^2 &\lesssim \frac{d}{dt} G(t) + \varepsilon^2 \Vert a \Vert_{L^{2}_{y}}^2 + \varepsilon \Vert c \Vert_{L^{2}_{y}}^2 +  \|\sqrt{\nu}(\mathbf{I-P})f\|_{2}^{2} + \|g\|_{2}^{2}.
				\end{split}
			\end{align}
			From the estimate for $a$ in \eqref{a,c esti}, if $\varepsilon <1/2$, we derive the following as:
			\begin{align} \label{a esti}
				\Vert a \Vert_{L^{2}_{y}}^2 \lesssim \frac{d}{dt} G(t) + \Vert c \Vert_{L^{2}_{y}}^2 + \|\sqrt{\nu}(\mathbf{I-P})f\|_{2}^{2} + \|g\|_{2}^{2}.
			\end{align}
			Using the above estimate and the estimate for $c$ in \eqref{a,c esti}, we can further estimate for $c$
			\begin{align*}
				\Vert c \Vert_{L^{2}_{y}}^2 \lesssim \frac{d}{dt} G(t) +\varepsilon \Vert c \Vert_{L^{2}_{y}}^2 + \|\sqrt{\nu}(\mathbf{I-P})f\|_{2}^{2} + \|g\|_{2}^{2},
			\end{align*}
			which implies 
			\begin{align} \label{final c}
				\Vert c \Vert_{L^{2}_{y}}^2 \lesssim \frac{d}{dt} G(t) + \|\sqrt{\nu}(\mathbf{I-P})f\|_{2}^{2} + \|g\|_{2}^{2}.
			\end{align}
			Furthermore, from \eqref{b esti}, \eqref{a esti}, and \eqref{final c}, we get
			\begin{align} \label{final a,b}
				\begin{split}
				\Vert a \Vert_{L^{2}_{y}}^2 \lesssim \frac{d}{dt} G(t) +\|\sqrt{\nu}(\mathbf{I-P})f\|_{2}^{2} + \|g\|_{2}^{2}, \\
				\Vert b \Vert_{L^{2}_{y}}^2 \lesssim \frac{d}{dt} G(t) +\|\sqrt{\nu}(\mathbf{I-P})f\|_{2}^{2} + \|g\|_{2}^{2}. 
				\end{split}
			\end{align}
			Using the fact that there exists $C_1>0$ such that $\Vert \mathbf{P}f \Vert_{\nu}^2 \leq C_1 \Vert \mathbf{P}f \Vert_{2}^2$, one obtains that 
			\begin{align*}
				\Vert \mathbf{P}f \Vert_{\nu}^2 \leq C_1 \Vert \mathbf{P}f \Vert_2^2=C_1 ( \Vert a \Vert_{L^{2}_{y}}^2 + \Vert b \Vert_{L^{2}_{y}}^2 + \Vert c \Vert_{L^{2}_{y}}^2) \lesssim \frac{d}{dt} G(t) +\|\sqrt{\nu}(\mathbf{I-P})f\|_{2}^{2} + \|g\|_{2}^{2},
			\end{align*}
			where the second inequality comes from \eqref{final c} and \eqref{final a,b}. 
			\end{proof}
\begin{remark}
In the proof of Proposition \ref{test coercivity}, two terms in \eqref{conserve zeta} 
\begin{align*}
	\int_{U} \frac{1}{\zeta_1} b_{x_1}  \quad \text{and} \quad \int_{U} \frac{1}{\zeta_2} b_{x_2} 
\end{align*}
are corresponding to angular momentum and momentum with respect to the $z$-axis. Without the assumption \eqref{conserve}, this explicitly shows that an estimate like \eqref{b esti} cannot be obtained. To ensure that the solution $F$ to the Boltzmann equation \eqref{Boltzmann} converges to the global Maxwellian $\mu$, it is necessary to assume \eqref{conserve}. 
\end{remark}
\begin{remark}
	To the best of our knowledge, this result is the first constructive proof of (hypo)coercivity of the linearized Boltzmann equation in the presence of specular reflection boundary, which is not trivial (e.g. a half space). The claim of \cite{YZ_JDE} could not hold as the angular momentum preserves in any axis-symmetric domains. 
\end{remark}
\begin{remark}
	In our proof we heavily rely on the orthogonal coordinate of a single chart. In general, we do not even have a global single chart of coordinate system. 
\end{remark}

\section*{Acknowledgement}
GK and DL are partly supported by the National Research Foundation of Korea(NRF) grant funded by the Korean government(MSIT)(No. NRF-2019R1C1C1010915 and No. RS-2023-00212304). CK is partly supported by NSF-DMS 1900923, NSF-CAREER 2047681, the Simons fellowship in Mathematics,and the Brain Pool fellowship funded by the Korean Ministry of Science and ICT. He thanks Yan Guo's comments about \cite{YZ_JDE} and kind hospitality of POSTECH during a stay in August/2022.
 

\end{document}